\documentclass{amsart}
\usepackage[utf8]{inputenc}
\usepackage{amssymb}
\usepackage{amsmath}
\usepackage{xcolor}
\usepackage{amscd}
\usepackage{comment}
\usepackage{accents}
\usepackage{enumitem}
\usepackage{tikz}
\usepackage{tikz-cd}
\usepackage[backend=biber, maxbibnames=99]{biblatex}
\usepackage{upgreek}

\usepackage{xr-hyper}
\usepackage[
pdftex,
bookmarks=false,
colorlinks=true,
debug=true,
pdfnewwindow=true]{hyperref}

\newcommand\iso{\,\vphantom{j^{X^2}}\smash{\overset{\sim}{\vphantom{\rule{0pt}{0.20em}}\smash{\longrightarrow}}}\,}

\title{Quantum Harish-Chandra bimodules at roots of unity and affine Hecke category}
\author{ Trung Vu}
\address{Department
of Mathematics, Yale University, New Haven CT USA}
\email{trung.vu@yale.edu}

\newtheorem{Thm}{Theorem}[section]
\newtheorem{Prop}[Thm]{Proposition}
\newtheorem{Cor}[Thm]{Corollary}

\newtheorem{Lem}[Thm]{Lemma}
\theoremstyle{definition}
\newtheorem{Ex}[Thm]{Example}
\newtheorem{defi}[Thm]{Definition}
\newtheorem{Rem}[Thm]{Remark}
\newtheorem{Conj}[Thm]{Conjecture}
\numberwithin{equation}{section}

\def\a{\alpha}
\def\b{\beta}

\def\e{\epsilon}

\def \utheta{\upvartheta}
\def\w{\omega}
\def\s{\sigma}

\def\g{\mathfrak{g}}
\def\tg{\tilde{\mathfrak{g}}}

\def\h{\mathfrak{h}}

\def \n{\mathfrak{n}}

\def \m{\mathfrak{m}}
\def \n{\mathfrak{n}}
\def \fu{\mathfrak{u}}
\def \fE{\mathfrak{E}}
\def \fJ{\mathfrak{J}}
\def \fP{\mathfrak{P}}
\def \fC{\mathfrak{C}}
\def \fR {\mathfrak{R}}
\def \fI {\mathfrak{I}}

\def \uchi{\underline{\chi}}
\def \uH {\underline{H}}

\def \op{\textnormal{op}}

\def \rs {\textnormal{rs}}

\def\x{\times}
\def\<{\langle}
\def\>{\rangle}

\def \({\left(}
\def \){\right)}

\newcommand{\tE}{\tilde{E}}
\newcommand{\tF}{\tilde{F}}

\newcommand{\te}{\tilde{e}}
\newcommand{\he}{\hat{e}}
\newcommand{\hf}{\hat{f}}
\newcommand{\tf}{\tilde{f}}

\newcommand{\tZ}{\tilde{Z}}
\newcommand{\tFr}{\tilde{\textnormal{Fr}}}

\newcommand{\St}{\mathsf{St}}
\newcommand{\Rmod}{\textnormal{-Rmod}}
\newcommand{\rmod}{\textnormal{-rmod}}
\newcommand{\Lmod}{\textnormal{-Lmod}}
\newcommand{\lmod}{\textnormal{-lmod}}
\newcommand{\Bimod}{\textnormal{-Bimod}}
\newcommand{\bimod}{\textnormal{-bimod}}
\newcommand{\Pbim}{\textnormal{Pbim}}

\newcommand{\Mod}{\textnormal{-mod}}
\newcommand{\Hilt}{\textnormal{HCTilt}}
\newcommand{\Tilt}{\textnormal{Tilt}}
\newcommand{\ad}{\operatorname{ad}}

\newcommand{\Ext}{\textnormal{Ext}}

\newcommand{\loc}{\mathrm{loc}}
\newcommand{\Hom}{\operatorname{Hom}}
\newcommand{\RHom}{\operatorname{RHom}}
\newcommand{\End}{\operatorname{End}}
\newcommand{\Id}{\mathrm{Id}}
\newcommand{\Rep}{\operatorname{Rep}}
\newcommand{\Spec}{\operatorname{Spec}}
\newcommand{\Mat}{\operatorname{Mat}}
\newcommand{\Supp}{\operatorname{Supp}}
\newcommand{\Stab}{\operatorname{Stab}}
\newcommand{\Coh}{\operatorname{Coh}}
\newcommand{\Proj}{\operatorname{Proj}}

\newcommand{\dU}{\dot{U}}
\newcommand{\mU}{\mathcal{U}}
\newcommand{\dmU}{\dot{\mathcal{U}}}

\newcommand{\cmU}{\check{\mathcal{U}}}

\newcommand{\cU}{\check{U}}

\newcommand{\CA}{\mathcal{A}}
\newcommand{\CB}{\mathcal{B}}
\newcommand{\CC}{\mathcal{C}}
\newcommand{\CE}{\mathcal{E}}
\newcommand{\CW}{\mathcal{W}}
\newcommand{\CP}{\mathcal{P}}
\newcommand{\CH}{\mathcal{H}}
\newcommand{\CO}{\mathcal{O}}
\newcommand{\CR}{\mathcal{R}}
\newcommand{\CS}{\mathcal{S}}
\newcommand{\CF}{\mathcal{F}}
\newcommand{\CG}{\mathcal{G}}

\newcommand{\bU}{{\mathbf{U}}}
\newcommand{\BZ}{{\mathbb{Z}}}
\newcommand{\BN}{{\mathbb{N}}}
\newcommand{\BQ}{{\mathbb{Q}}}
\newcommand{\BC}{{\mathbb{C}}}
\newcommand{\BR}{{\mathbb{R}}}

\newcommand{\BI}{\mathbb{I}}
\newcommand{\BJ}{\mathbb{J}}
\newcommand{\sN}{\mathsf{N}}
\newcommand{\sF}{\mathsf{F}}
\newcommand{\sR}{\mathsf{R}}
\newcommand{\sC}{\mathsf{C}}
\newcommand{\usR}{\underline{\mathsf{R}}}
\newcommand{\pr}{\mathsf{pr}}
\newcommand{\SB}{\mathsf{SB}}
\newcommand{\ASB}{\mathsf{ASB}}
\newcommand{\SA}{\mathsf{A}}
\newcommand{\BS}{\mathsf{BS}}
\newcommand{\Stn}{\mathsf{St}}
\newcommand{\uBS}{\underline{\mathsf{BS}}}
\newcommand{\uB}{\underline{B}}

\newcommand{\sd}{{\mathsf{d}}}
\newcommand{\Dyn}{\mathrm{Dyn}}
\newcommand{\Or}{\mathsf{Or}}
\newcommand{\sA}{\mathsf{p}}
\newcommand{\HC}{\mathsf{HC}}

\newcommand{\dash}{\textnormal{-}}

\newcommand{\uu}{\underline{\mathfrak{u}}}
\newcommand{\sslash}{\mathbin{/\mkern-6mu/}}

\newcommand{\bmat}[1]{\begin{bmatrix}#1\\\end{bmatrix}}

\makeatletter
\DeclareRobustCommand{\cev}[1]{%
  {\mathpalette\do@cev{#1}}%
}
\newcommand{\do@cev}[2]{%
  \vbox{\offinterlineskip
    \sbox\z@{$\m@th#1 x$}%
    \ialign{##\cr
      \hidewidth\reflectbox{$\m@th#1\vec{}\mkern4mu$}\hidewidth\cr
      \noalign{\kern-\ht\z@}
      $\m@th#1#2$\cr
    }%
  }%
}
\makeatother

\begin{document}

\begin{abstract}The category of Harish-Chandra bimodules for quantum groups was first appeared in the works about topological quantum field theory of surfaces. In this paper, we study  this category when the quantum parameter $q$ is an odd order root of unity. We relate the category to the category of affine Soergel bimodules and to non-commutative Springer resolution.    
\end{abstract}
\maketitle
{ \small \tableofcontents}

\section{Introduction}

Let $\g$ be a complex semisimple Lie algebra, $G$ be its simply connected algebraic group and  $U(\g)$ be its universal enveloping algebra. The category of Harish-Chandra $U(\g)$-bimodules is an important object in representation theory. It was first studied by Harish-Chandra in his works on infinite dimentional representions of $G$ viewed as a real group. Later, it was connected to category $O$ \cite{BG80} , character sheaves \cite{BFO12},  the finite $W$-algebras associated to $\g$ \cite{IL11} and so on.

The quantum group $\bU_q(\g)$ is a Hopf algebra deformation of $U(\g)$. The notion of Harish-Chandra bimodules for quantum groups was first appeared in works of David Jordan and his collaborators in their works about topological quantum field theory of surfaces \cite{BBJ18}. In this paper, we will study the category of quantum Harish-Chandra bimodules when the quantum  parameter $q$ is a root of unity of odd order $\ell$.

The quantum group $\bU_v(\g)$ is a Hopf algebra over $\BC(v)$ generated by $\{E_i, F_i, K^\lambda\}_{1\leq i \leq r}^{\lambda \in Q}$, here $Q$ is the root lattice of $\g$ and $r$ is the rank of $\g$, see \cite[$\mathsection 4$]{J96} or Section \ref{sec: set up}. Let $\CA:=\BC[v, v^{-1}]$. {\it The Lusztig form } $\cmU_v(\g)$ is the Hopf $\CA$-subalgebra of $\bU_v(\g)$ generated by $\{ E_i^{(n)}, F_i^{(n)}, K^\lambda\}_{1\leq i \leq r}^{\lambda \in Q, n \in \BN}$; {\it the De Concini-Kac form} $\mU_v(\g)$ is the Hopf $\CA$-subalgebra of $\bU_v(\g)$ generated by $\{ E_i, F_i, K^\lambda\}_{1\leq i \leq r}^{\lambda \in Q}$, see Section \ref{sec: set up}. By specifying $v \mapsto q \in \BC^\x$, we get Hopf algebras $\cmU_q(\g), \mU_q(\g)$ over $\BC$.

The Lusztig form can be viewed as  an analog of the algebraic group $G$, so that  we can define the category $\Rep(\cmU_q(\g))$ of rational representations of $\cmU_q(\g)$.  This category is a braided monoidal category. Inside $\Rep(\cmU_q(\g))$, there is an algebra object $O_q[G]$ called {\it reflection equation algebra} which is a quantization of $O(G)$, the ring of regular functions on $G$. As a vector space, $O_q[G]$ is spanned by the matrix coefficients of finite dimensional representations in $\Rep(\cmU_q(\g))$.
\begin{defi}The category of {\it quantum Harish-Chandra bimodules} is the category of finitely generated right $O_q[G]$-modules in the category $\Rep(\cmU_q(\g))$. We denote this category by $\HC_q$.
\end{defi}
In \cite{BBJ18}, the category $\HC_q$ is the factorization homology of  annulus so that $\HC_q$ is monoidal by stacking annuli.

In this paper, we use the setup from \cite{LTV}. The Lusztig form is modified by a Drinfeld twist and the modified one is denoted by $\cU_q(\g)$. The Drinfeld twist changes the braided monoidal  structure on  $\Rep(\cmU_q(\g))$ to get a slightly different but equivalent braided monoidal category $\Rep(\cU_q(\g))$. We then identify $O_q[G]$ with the $\cU_q(\g)$-locally finite part $U_q^{fin}$ of the {\it even subalgebra} $U_q^{ev}(\g)$, a modification of the De Concini-Kac form.  In the paper, we define the monoidal structure on $\HC_q$ in an algebraic way, see Section \ref{ssec: noncomplete HC}. This should be related to the monoidal structure obtained by stacking annuli as above.

In the rest of this section , $q=\e$ is a root of unity of odd order $\ell$ satisfying the condition \ref{eq: assumption on l} in Section \ref{sec: defi of  quantum HC}.

The $\cU_\e(\g)$-invariant part of $U_\e^{fin} \cong O_\e[G]$ is central in $U_\e^{fin}$, and we call it the {\it  Harish-Chandra center $\CW_\e$}. The {\it principal block} $\HC_\e(0,0)$ is  defined with respect to the trivial central character of $\CW_\e$, see Definition \ref{defi: integral blocks}.

{\bf Affine Soergel bimodules:}  The Weyl group $W$ of $\g$ acts on the Cartan subalgebra $\h$ of $\g$. Let $P, Q$ be the weight and root lattices of $\g$, respectively. Let $W_{aff}:=W\ltimes Q$ and $W_{ext}:=W\ltimes P$ be the affine Weyl group and the extended affine Weyl group, respectively. Let $S_{aff}$ be the set of simple reflections in $W_{aff}$. Note that $W_{ext}=\Lambda \ltimes W_{aff}$, where $\Lambda \cong P/Q$ is a finite abelian group.  Both $W_{ext}$ and $W_{aff}$  act on $\hat{\h}:= \h \oplus \BC \hbar$, here $\BC \hbar$ is a one dimensional vector space, see Section \ref{ssec: the algebra R} . Let $\sR$ be the completion of the symmetric algebra $S(\hat{\h})$ at zero. 

We define the {\it Bott-Samelson bimodules} $\BS_s:=\sR\otimes_{\sR^s} \sR$ for any $s\in S_{aff}$ and the { \it graph bimodules } $\sR_x$ for $x \in \Lambda$, see Section \ref{ssec: defi of Soergel bimod}. The category $\SB_\hbar$ is the full subcategory  of the category of $\sR$-bimodules generated by $\{\BS_s, \sR_x~|~ s\in S_{aff}, x\in \Lambda\}$ via taking tensor products, direct summands and direct sums. We then define the category $\SB$ which has the same objects as in $\SB_\hbar$ but morphism spaces in $\SB$ are obtained from those in $\SB_\hbar$ by taking the  quotient by $\hbar$, see Section \ref{ssec: defi of Soergel bimod}. The category $\SB$ can be realized as a variant of the category of Soergel bimodules introduced  by Abe in \cite{Abe1}.

\begin{Thm}[Corollary \ref{cor: Sb embed}]\label{thm: Sb embed to HC}There is a full embedding of monoidal categories $\fI_1: \SB \hookrightarrow \HC_\e(0,0)$. 
\end{Thm}

By the Kazhdan-Lusztig theory of canonical bases of Hecke algebras, the group $W_{ext}$ can be divided into a disjoint union of subsets called {\it two sided cells} with a certain order. There is a smallest two-sided cell $C_0$. This two-sided cell gives a corresponding full subcategory $\sC_0 \subset \SB$. The category $\sC_0$ is semi-monoidal, i.e, having  all structures of monoidal categories except the existence of unit.  Let $\Proj(\HC_\e(0,0))$ be the full subcategory of projective objects in $\HC_\e(0,0)$.
\begin{Thm}[Theorem \ref{thm: smalles two-sided cell vs proj}]\label{thm: C to Proj} The functor $\fI_1$ restrict to an equivalence of semi-monoidal categories \[\CC_0 \iso \Proj(\HC_\e(0,0)).\]
\end{Thm}
Furthermore, we show that
\begin{Thm}[Theorem \ref{thm: triangulate equivalence as HCtilt}]\label{thm: triang equivalence 1} The functor $\fI_1$ gives an equivalence of monoidal triangulated categories \[K^b(\SB) \iso D^b(\HC_\e(0,0)).\]
\end{Thm}

Theorem \ref{thm: triang equivalence 1} is proved by using  relations between  the categories in the theorem  with coherent sheaves on Steinberg variety and non-commutative Springer resolution. Recall the Grothendieck-Springer resolution $\tg \rightarrow \g^*$ and the Steinberg variety $\St:= \tg \x_{\g^*} \tg$. There is a $G$-equivariant vector bundle $\CE$ on $\tg$ with $\SA=\End_{\tg}(\CE)^{op}$ such that the derived global section functor
\[ \Gamma_\CE: D^b\Coh^G(\tg) \rightarrow D^b(\SA\Mod^G), \qquad \CF \mapsto \RHom_{\tg}(\CE, \CF),\] is an equivalence of triangulated categories. Then $\SA$ is called {\it noncommutative Springer resolution}. We introduce certain completed version $\St^\wedge$, $(\SA \otimes_{\BC[\g^*]} \SA^{op})^\wedge$ as in Section \ref{ssec: complete version}
 \begin{Thm}[Theorem \ref{thm: HC and NCS}-\ref{thm: triangulate equivalence as HCtilt}]\label{thm: triang equivalence 2}There is an equivalence of monoidal triangulated categories
 \[ K^b(\SB) \cong D^b(\HC_\e(0,0)) \cong D^b((\SA\otimes_{\BC[\g^*]} \SA^{op})^\wedge \Mod^G) \cong D^b\Coh^G(\St^\wedge).\]
 \end{Thm}
There are group homomorphisms  from $B_{aff}$ to the group of equivalence classes of invertible elements in $K^b(\SB)$ via Rouquier complexes and to the group of equivalence classes of invertible elements in $D^b\Coh^G(\St^\wedge)$ by Bezrukavnikov-Riche. The equivalence $K^b(\SB) \cong D^b\Coh^G(\St^\wedge)$ intertwines with these two  group homomorphisms.

The image of $\fI_1$ is the full subcategory $\Hilt_\e(0,0)$ of HC-tilting bimodules defined in Definition \ref{defi: diagonal bimod}, which are related to tilting modules in $\Rep(\cU_\e(\g))$. As a consequence of Theorem \ref{thm: triang equivalence 1}, we have
\begin{Thm}[Theorem \ref{thm: triangulate equivalence as HCtilt}.b] The functor $\fI_1$ restricts to an equivalence of monoidal additive categories 
\[ \SB \iso \Hilt_\e((0,0)).\]
\end{Thm}

\subsection{Ideas of proof} The first key idea in the proof of Theorem \ref{thm: Sb embed to HC} is to consider a deformation of $\HC_\e(0,0)$. We  will consider a category $\HC_q(0,0)$ where the parameter $q$ is equal to $\e e^{2\pi \sqrt{-1}\hbar /\ell} \in \BC[[\hbar]]$. The second key idea is the restriction functor $\bullet_\dag$ from $\HC_q(0,0)$ to the category of $\sR$-bimodules constructed in Section \ref{ssec: restriction functors}.  The functor $\bullet_\dag$ can be thought as a quantization of the process of restricting to a  tranversal slice of a symplectic leaf in a Poisson variety. The construction follows Losev's construction in  \cite{IL11}. Finally, we construct certain bimodules in $\HC_q(0,0)$ which are related to translation functors for quantum category $O$. We show that $\bullet_\dag$ is  fully faithful on these bimodules and the category $\SB_\hbar$ is contained in the images of these bimodules under the functor $\bullet_\dag$. Then we obtain an embedding $\SB_\hbar \hookrightarrow \HC_q(0,0)$, which induces the embedding $\fI_1$ by taking the quotient by $\hbar$.

Let us comment on the proof of Theorem \ref{thm: C to Proj}. The category $\sC_0$ contains a distinguished object $\uBS_{w_0}$ which is mapped to a projective object in $\HC_\e(0,0)$ under $\fI_1$. We will study the simple objects in $\HC_\e(0,0)$ to show that any projective object in $\HC_\e(0,0)$ is a direct summand of $\fI_1(B_1) \star \fI_1(\uBS_{w_0}) \star \fI_1(B_2) \cong \fI_1(B_1 \star \uBS_{w_0} \star B_2)$ for $B_1, B_2\in \SB$, here $\star$ stands for the monoidal products in the  corresponding categories. The latter description matches with a description of $\sC_0$ so that we obtain the equivalence in Theorem \ref{thm: C to Proj}.

To prove Theorem \ref{thm: triang equivalence 2}, we relate the category $\mathsf{NCS}:= (\SA\otimes_{\BC[\g^*]} \SA^{op})^\wedge \Mod^G$ to the Abe realization of $\SB$ via  the functor $\hat{\fR}$ constructed in \eqref{eq: functor R} of  Section \ref{ssec: complete version}. This functor is a quantum analog of the functor constructed by Bezrukavnikov and Riche in \cite{BR22}. We will obtain an embedding of monoidal categories $\fI_2: \SB \hookrightarrow \mathsf{NCS} $ so that it restricts to an equivalence $\sC_0 \iso \Proj(\mathsf{NCS})$, the latter is the full subcategory of projective objects in $\mathsf{NCS}$. From the functors $\fI_1$ and $ \fI_2$, we construct an equivalence of monoidal abelian categories $\HC_\e(0,0)\iso \mathsf{NCS}$ and use this equivalence to establish other equivalences in Theorem \ref{thm: triang equivalence 2}.

\subsection{Further directions}

\subsubsection{Other blocks} In Appendix \ref{append: enhanced version}, we discuss some results which can be used to study other blocks of quantum Harish-Chandra bimodules. The results for other blocks should be compared to results in \cite{GYZX25}.

\subsubsection{Even order roots of unity} The setup in \cite{LTV} works pretty well with the even order roots of unity.  Most of the results in this paper on quantum Harish-Chandra bimodules can be generalized to more general roots of unity cases, in particular, to the case where  the quantum parameter $q$ is an  even order root of unity, however, some technical details need to be resolved. 

\subsubsection{Toward localization theorem for quantum groups} One of  motivations behind this paper is the  hope of using  another approach to establish localization theorems for De Concini-Kac forms of quantum groups for more general parameter $q\in \BC^\x$, mentioned in \cite[$\mathsection 1.6.3$]{IL23}. Such theorem has been proved by \cite{EK08}, \cite{Tan21} for generic $q$ and for roots of unity of order $p^k$ for some prime number $p$.

\subsubsection{Positive characteristic} Since the combinatorics of $\HC_\e(0,0)$ depends on the order of roots of unity rather than the base field, the category $\HC_\e(0,0)$ should still behave well when we replace the complex number $\BC$ with  a field of {\it good} positive characteristic $p$. Roughly speaking, the characteristic $p$ is good if the group $G$ behaves as the group over $\BC$.

\begin{Conj}\label{conj}Assume the characteristic  $p$ is good and $\e$ is of large enough order.  There is an equivalence of monoidal triangulated categories
\[ D^b(\HC_\e(0,0)) \iso D^b\Coh^G(\St_m^\wedge).\]
\end{Conj}
Here $\St_m$ is the multiplicative analog of $\St$, where we replace $\g$ in the definition of $\St$ with  the algebraic group $G$. The right hand side appeared  in a conjectural equivalence of coherent-constructible realizations of affine Hecke categories  in a field of  {\it good} positive characteristic, see \cite[$\mathsection 1.4$]{BR24}. Conjecture  \ref{conj} may be helpful to establish the equivalence in a similar way that modular Harish-Chandra bimodules was used in {\it loc. cit.}

\subsection{Outline} In Section $1\dash3$, we recall many results in \cite{LTV}, where we introduced the even subalgebra $U_q^{ev}(\g)$ and its locally finite part $U_q^{fin}$.
Section $4$ is a technical section about completions of algebras. In Section $5$, we recall the quantum category $O$. In Section $6$, we introduce the quantum Harish-Chandra bimodules and study their basis properties. In Section $7$, we introduce the notion of Poisson  bimodules. In Section $8$, we construct the key tool in this paper, the restriction functors $\bullet_\dag$ that allow us  to relate quantum Harish-Chandra bimodules with affine Soergel bimodules. In Section $9$, we recall affine Soergel bimodules and their variant constructed by Abe in \cite{Abe1}. Then we establish Theorem \ref{thm: Sb embed to HC} in Section $11$. In Section $12$, we study simple objects in $\HC_\e(0,0)$ in order to study the projective objects in the principal block and establish Theorem \ref{thm: C to Proj}. Finally, in Section $13$, we recall results about non-commutative Springer resolution and establish Theorem \ref{thm: triang equivalence 2}.

\subsection{Acknowledgments} 

 I am deeply grateful to my advisor Ivan Losev for his invaluable support, guidance and encouragement throughout the course of this research. I am also benefited from discussions with Pablo Boixeda Alvarez, Roman Bezrukavnikov, Kien Do Hoang, Quan Situ, Yaochen Wu. This work was partially supported by an NSF Grant DMS-2001139.

\section{Set up} \label{sec: set up}

In this section, we establish notation and recall various results in \cite{LTV}. Let $\g$ be a semisimple Lie algebra with  Cartan subalgebra $\h$, simple roots $\{\a_i\}_{i=1}^r$, fundamental weights $\{\w_i\}_{i=1}^r$,  the weight lattice $P$ and the root lattice $Q$. We fix a nondegenerate invariant bilinear form $(\;,\;)$ on $\h^*$ such that $\sd_i:=\frac{(\a_i, \a_i)}{2}$ is equal to $1$ for short roots $\a_i$, particular, $\sd_i\in \{ 1,2,3\}$ for any $i$.  Define $\w^\vee_i:=\frac{\w_i}{\sd_i}$ and $\a^\vee_i:=\frac{\a_i}{\sd_i}$, the fundamental coweights and simple coroots. We have the Cartan matrix $(a_{ij})_{i,j=1}^n$ and the symmetrized Cartan matrix $(b_{ij})_{i,j=1}^r$,  
 where $a_{ij}=(\a^\vee_i, \a_j)=2(\a_i,\a_j)/(\a_i, \a_i) $ and $ b_{ij}:=(\a_i, \a_j)$.

Let $v$ be a formal variable and 
\[ \CA :=\BC[v, v^{-1}]\left[\left\{\frac{1}{v^{2k}-1}\right\}\right]_{1\leq k \leq \max\{\sd_i\}}\]

Let us define the following elements in $\CA$: 
\begin{align*}[s]_v&:=\frac{v^s-v^{-s}}{v-v^{-1}}, \qquad [s]_v !=[1]_v\dots [s]_v, \qquad \bmat{m\\s}_v:=\prod_{c=1}^s \frac{v^{m-c+1}-v^{-m+c-1}}{v^c-v^{-c}},\\
 (s)_v!&:=\frac{1-v^{-2s}}{1-v^{-2}}, \qquad (s)_v!:=(1)_v \dots (s)_v, \qquad \binom{m}{s}_v:=\prod_{c=1}^s \frac{1-v^{-2(m+1-c)}}{1-v^{-2c}}.
 \end{align*}

The quantum group $\bU_v(\g)$ is the Hopf algebra over $\BC(v)$ generated by  $\{ E_i, F_i, K^\mu\}_{1\leq i \leq r}^{\mu \in Q}$ subject to the following relations:
\begin{equation*}
    \begin{split}
        &K^\mu K^{\mu'}=K^{\mu+\mu'}, \qquad K^0=1,\\
        &K^\mu E_i K^{-\mu}=v^{(\mu, \a_i)}E_i, \qquad K^\mu F_i K^{-\mu}=v^{-(\mu, \a_i)} F_i,
    \end{split}
\end{equation*}
\begin{equation*}
    \begin{split}
            &[E_i, F_j] =\delta_{i,j} \frac{K_i-K_i^{-1}}{v_i-v_i^{-1}},\\
            &\sum_{m=0}^{1-a_{ij}} (-1)^m \bmat{1-a_{ij}\\m}_{v_i} E_i^{1-a_{ij}-m}E_j E_i^m =0\quad(i \neq j),\\
        &\sum_{m=0}^{1-a_{ij}}(-1)^m \bmat{1-a_{ij}\\m}_{v_i}F_i^{1-a_{ij}-m}F_j F_i^m =0 \quad (i\neq j).
    \end{split}
\end{equation*}
Here $K_i:=K^{\a_i}$ and $v_i:=v^{\sd_i}$. One equips $U_q(\g)$  with the Hopf structure as follows:
\begin{equation}\label{eq: usual hopf structure}
    \begin{split}
        &\Delta: E_i \mapsto E_i\otimes 1 + K_i \otimes E_i, \quad F_i \mapsto F_i \otimes K^{-1}_i+1\otimes F_i,\quad K^\mu \mapsto K^\mu\otimes K^\mu,\\
        &S: E_i \mapsto -K^{-1}_i E_i, \quad F_i \mapsto -F_i K_i, \quad K^\mu\mapsto K^{-\mu},\\
        &\varepsilon: E_i \mapsto 0, \quad F_i \mapsto 0, \quad K^\mu \mapsto 1.
    \end{split}
\end{equation}

There is the  left adjoint action of $\bU_v(\g)$ on itself defined by
\begin{equation}\label{eq: l-ad of DJquantum group}
\ad(x)(u)=\sum x_{(1)}u S(x_{(2)}), \qquad \forall x,u \in \bU_v(\g),
\end{equation}
here we use Sweedler's notation for the coproduct.

Let $\displaystyle E_i^{(n)}:=\frac{E_i^n}{(n)_{v_i}!}$, $\displaystyle F_i^{(n)}:=\frac{F_i^n}{(n)_{v_i}!}$.  The Lusztig form $\cmU_v(\g)$ is the $\CA$-subalgebra of $\bU_v(\g)$ generated by $\{E_i^{(n)}, F_i^{(n)}, K^{\a_i}\}_{1\leq i \leq r}^{n \in \BN}$, while the De Concini-Kac form $\mU_v(\g)$ is the $\CA$-subalgebra generated by $\{ E_i, F_i, K^{\a_i}\}_{1 \leq i \leq r}$. Both are Hopf $\CA$-subalgebras of $\bU_v(\g)$. However, the adjoint action \eqref{eq: l-ad of DJquantum group} does not restrict to an action of $\cmU_v(\g)$ on $\mU_v(\g)$. This issue and {\it other} issues were remedied in \cite{LTV}. Roughly speaking, we will twist the coproduct of $\bU_v(\g)$ so that the left adjoint action gives rise to an action of the \emph{(twisted) Lusztig form} $\cU_v(\g)$ on the \emph{even subalgebra} $U^{ev}_v(\g)$, which is a suitable modification of the De Concini-Kac form.

\subsection{Modified quantum group}
\label{ssec: twisted construction} Let us recall a construction in \cite{LTV}.
\begin{Prop}[Theorem 1 \cite{R}](a) For a (topological) Hopf algebra $(A,m,\Delta, S,\varepsilon)$ and $F\in A\widehat{\otimes} A$ satisfying 
\begin{equation}\label{eq: twist} (\Delta\otimes \Id)(F)=F_{13}F_{23}, \;\; (\Id\otimes \Delta)(F)=F_{13}F_{12}, \;\; F_{12}F_{13}F_{23}=F_{23}F_{13}F_{12}, \;\; F_{12}F_{21}=1,
\end{equation}
the formulas
\[ \Delta^{(F)}(a)=F\Delta(a)F^{-1}, \qquad S^{(F)}(a)=uS(a)u^{-1}, \qquad \varepsilon^{(F)}(a)=\varepsilon(a)\]
with $u:=m(\Id\otimes S)(F)$, endow $A$ with a new Hopf algebra structure $(A,m,\Delta^{(F)}, S^{(F)}, \varepsilon^{(F)})$.

\noindent
(b) If $(A,m,\Delta, S,\varepsilon)$ is a quasitriangular Hopf algebra with universal $R$-matrix $R\in A\otimes A$, then $(A,m, \Delta^{(F)}, S^{(F)}, \varepsilon^{(F)})$ is also a quasitriangular Hopf algebra with universal $R$-matrix:
\[ R^{(F)}=F^{-1}R F^{-1}.\]   
\end{Prop}

Let $\bU_v(\g, P/2)$ be a Hopf algebra over $\BC(v^{1/2})$ obtained from $\bU_v(\g)$ by extending the base ring to $\BC(v^{1/2})$ and enlarging the span of $\{K^\lambda\}_{\lambda \in Q}$ to  $\{K^\lambda\}_{\lambda \in P/2}$.

Let $\Dyn(\g)$ denote the graph obtained from the Dynkin diagram of $\g$ by replacing all multiple edges by simple ones, e.g., $\Dyn(\mathfrak{sp}_{2r})=\Dyn(\mathfrak{so}_{2r+1})=\Dyn(\mathfrak{sl}_{r+1})=A_r$. 
Let us  fix an orientation $\Or$  of $\Dyn(\g)$ and then associate with it a skew-symmetric  matrix $(\e_{ij})_{i,j=1}^r$ via 
\begin{equation*}
    \e_{ij}=\begin{cases} 0  \quad &\text{if $a_{ij} \geq 0$}\\
                          1     &\text{if $a_{ij} <0$ and $\Or$ contains an oriented edge $i \rightarrow j$}\\
                          -1  &\text{if $a_{ij} <0$ and $\Or$ contains an oriented edge $i \leftarrow j$}
            \end{cases}
\end{equation*}
Let us consider the skew-symmetric matrix $(\phi_{ij})_{i,j =1}^r$ where $\phi_{ij}=\e_{ij}\frac{(\a_i, \a_j)}{2}$. The twist
\begin{equation}
    \sF=v^{\sum_{ij} \phi_{ij} \w^\vee_i\otimes \w^\vee_j}
\end{equation}
satisfies condition \eqref{eq: twist}. This twist belongs to a certain topological completion of $\bU_v(\g, P/2)$. 
Let us define  
\begin{gather*}\nu^>_i = -\a_i+\sum_{j=1}^r\phi_{ij}\w^\vee_j,  \quad \nu^<_i = \sum_{j=1}^r \phi_{ij} \w^\vee_j, \quad
\zeta^>_i =\a_i -2\sum_{j=1}^r \phi_{ij} \w^\vee_j,\quad \zeta^<_i =-\a_i-2\sum_{j=1}^r \phi_{ij}\w^\vee_j.\quad \\
 \tE_i:= E_i K^{\nu^>_i}, \qquad \tF_i:=K^{-\nu^<_i} F_i.
 \end{gather*}
\begin{Rem} The elements $\zeta^<_i, \zeta^>_i$ belong to $2P$, see \cite[Lemma 3.4]{LTV}.
\end{Rem}

One can show that $\bU_v(\g, P/2)$ is generated by $\{\tE_i,\tF_i, K^\lambda\}_{1\leq i \leq r}^{\lambda \in P/2}$ subject to the relations:
\begin{equation}\label{eq:gen-rel-twistedDJ}
\begin{split}
  & K^{\mu} K^{\mu'}=K^{\mu+\mu'} \,, \qquad K^0=1 \,, \\
  & K^{\mu} \tE_i K^{-\mu}=v^{(\alpha_i,\mu)}\tE_i \,, \qquad K^{\mu} \tF_i K^{-\mu}=v^{-(\alpha_i,\mu)}\tF_i \,, \\
  & \tE_i\tF_j=v^{(\a_i, -\zeta^<_j)}\tF_j\tE_i \quad(i\neq j) \,, \qquad 
    \tE_i \tF_i - v_i^2\tF_i\tE_i=v_i \frac{1-K_i^{-2}}{1-v_i^{-2}} \,, \\
  & \sum_{m=0}^{1-a_{ij}}(-1)^m v^{m \epsilon_{ij}b_{ij}}\bmat{1-a_{ij}\\ m}_{v_i} \tE_i^{1-a_{ij}-m}\tE_j\tE_i^{m}=0 \quad (i\ne j) \,, \\
  & \sum_{m=0}^{1-a_{ij}}(-1)^m v^{m \epsilon_{ij}b_{ij}}\bmat{1-a_{ij}\\m}_{v_i} \tF_i^{1-a_{ij}-m}\tF_j\tF_i^{m}=0 \quad (i\ne j) \,,
\end{split}
\end{equation}
here $K_i:=K^{\a_i}$, $v_i=v^{\sd_i}$ as usual. Moreover, we have the Hopf structure obtained via twist $\sF$:
\begin{equation}\label{eq: twist-Hopf-2}
\begin{split}
  & \Delta'(K^\mu)=K^\mu\otimes K^\mu \,, \ \  \Delta'(\tE_i)=1\otimes \tE_i + \tE_i\otimes K^{-\zeta^>_i} \,, \ \ 
    \Delta'(\tF_i)=1\otimes \tF_i+\tF_i \otimes K^{\zeta^<_i} \,, \\
  & S'(K^\mu)=K^{-\mu} \,, \qquad S'(\tE_i)=-\tE_iK^{\zeta^>_i} \,, \qquad S'(\tF_i)=-\tF_iK^{-\zeta^<_i} \,.
\end{split}
\end{equation}

\begin{defi}\label{def: A-forms} (a) The \emph{(twisted) Lusztig form} $\cU_v(\g)$ is the $\CA$-subalgebra of $\bU_v(\g, P/2)$ generated by $\{ \tE_i^{(n)}, \tF_i^{(n)}, K^\lambda\}_{\lambda \in 2P}^{1\leq i\leq r}$ with $\tE_i^{(n)}:=\frac{\tE_i^n}{(n)_{v_i}!}$ and $\tF_i^{(n)}:=\frac{\tF_i^n}{(n)_{v_i}!}$. 

\noindent
(b) The \emph{even subalgebra} $U^{ev}_v(\g)$ is the $\CA$-subalgebra of $\bU_v(\g, P/2)$  generated by $\{ \tE_i, \tF_i, K^\lambda\}_{\lambda \in 2P}^{1\leq i \leq r}$.

\noindent
(c) The {\em mixed form} $U_v^{mix}$ is the $\CA$-subalgebra of $\bU_v(\g, P/2)$ generated by $\{ \tE_i^{(n)}, \tF_i, K^\lambda\}_{\lambda \in 2P}^{1\leq i \leq r}$.
\end{defi}
We emphasize the use the lattice $2P$ in the definition. These algebras are Hopf $\CA$-subalgebras of $\bU_v(\g,P/2)$ with the twisted Hopf structure. The name \emph{even subalgebra} comes from the fact that we only use the lattice $2P$ for the Cartan part in the set of generators of $U^{ev}_v(\g)$. 

We have the left adjoint action $\ad'_l$ of $\bU_v(\g, P/2)$ on itself, similar to \eqref{eq: l-ad of DJquantum group}. 
\begin{Prop}[Proposition 3.11 \cite{LTV}] The left adjoint action $\ad'_l$ of $\bU_v(\g, P/2)$ on itself restricts to an adjoint action of $\cU_v(\g)$ on $U^{ev}_v(\g)$.
\end{Prop}

For any algebra homomorphism $\CA\rightarrow R$, $v\mapsto q\in R^\x$, we define the specializations:
\begin{equation}\label{eq: specalized forms} \cU_q(\g):=\cU_v(\g) \otimes_\CA R, \qquad  U^{ev}_q(\g):=U^{ev}_v(\g) \otimes_\CA R, \qquad U_q^{mix}:= U_v^{mix} \otimes_\CA R.
\end{equation}

\begin{Rem}\label{rem: Uev to lusztig} We have the left adjoint action $\ad'_l: \cU_q(\g) \curvearrowright U^{ev}_q(\g)$.  The inclusion $\iota: U^{ev}_v(\g) \hookrightarrow \dU_v(\g)$ induces a Hopf algebra homomorphism $\iota: U^{ev}_q(\g) \rightarrow \cU_q(\g)$.
\end{Rem}

\begin{Rem}\label{rem: triangular decomposition}
By\cite[$\mathsection3.6$]{LTV}, we have  $\cU_q(\g)\cong \cU^<_q \otimes_R \cU^0_q \otimes_R \cU^>_q$, in which $\cU^<_q, \cU^>_q, \cU^0_q$ are generated by $\{ \tF_i^{(n)}\}^{1\leq i \leq r}_{n \in \BN}, \{ \tE_i^{(n)}\}^{1\leq i \leq r}_{n \in \BN}, \big \{K^\mu, \binom{K_i;0}{m}\big\}_{\mu \in 2P}^{1\leq i \leq r, m \in \BN}$, respectively. Here
\begin{equation}
    \binom{K_i;a}{n} = \frac{\prod_{s=1}^n (1-K_i^{-2}v_i^{2(s-a-1)})}{\prod_{s=1}^n(1-v_i^{-2s})}.
\end{equation}
\end{Rem}

\begin{Rem}We have  { \em the idempotented Lusztig form} $\dmU_q(\g, P)$ defined in \cite[Chapter 23]{l-book} generated by
\[ \{ E_i^{(n)}1_\lambda, F_i^{(n)}1_\lambda~|~ 1\leq i\leq r, n \geq 0, \lambda \in P\}.\]
Similarly, we have the idempotented Lusztig form $\dU_q(\g,P)$ defined with generators
\[ \{ \tE_i^{(n)}1_\lambda, \tF_i^{(n)}1_\lambda| 1\leq i \leq r, n \geq 0, \lambda \in P\}.\]
We equip $\dmU_q(\g, P)$ with a Hopf algebra structure via \eqref{eq: usual hopf structure} and  $\dU_q(\g,P)$ with a Hopf algebra structure by \eqref{eq: twist-Hopf-2}. We record the coproduct of $\dU_q(\g, P)$: 
\begin{equation}\label{EF coproduct}
\begin{split}
  & \Delta(\tE^{(r)}_i 1_\lambda)=
    \sum_{c=0}^r \prod_{\lambda'+\lambda''=\lambda}  q^{-(r-c)(\zeta^>_i, \lambda'')} 
    \tE_i^{(r-c)}1_{\lambda'}\otimes \tE_i^{(c)}1_{\lambda''} \,,\\
  & \Delta(\tF_i^{(r)}1_\lambda)=
    \sum_{c=0}^r \prod_{\lambda'+\lambda''=\lambda} q_i^{2c(r-c)}q^{c(\zeta^<_i, \lambda'')} 
    \tF_i^{(c)}1_{\lambda'}\otimes \tF_i^{(r-c)}1_{\lambda''} \,,
\end{split}
\end{equation}
For more discussions, see \cite[$\mathsection 5$]{LTV}. 
\end{Rem}
\subsection{Affine Weyl group $W_{aff}$}\

The {\em affine Weyl group} $W_{aff}:=W \ltimes Q$ and the {\em extended affine Weyl group} $W_{ext}:=W\ltimes P$ naturally act on $P$ and $\h^*$. Let $t_\lambda$ denote the element of $W_{ext}$ corresponding to $\lambda \in P$.

We have the  {\em dot action} of $W$ on $\h^*$, hence on $P$,  as follows: 
\[w \cdot \lambda= w(\lambda+\rho)-\rho\qquad  \text{for $\lambda \in \h^*, w\in W$},\]
here $\rho= \sum_i \w_i$. For the number $\ell$, we have the {\em $\ell$-shifted dot action} $\bullet_\ell$  of $W_{ext}$ in $\h^*$ as follows:
\begin{equation}\label{eq: l-shifted dot action}
w \bullet_\ell(\mu)=w(\mu+\rho)-\rho, \qquad t_\lambda \bullet_\ell \mu=\mu+\ell \lambda, \qquad \text{for $w\in W, \lambda\in P, \mu \in \h^*$}.
\end{equation}

The {\em fundamental alcove} with respect to the $\bullet_\ell$-action of $W_{aff}$ on $\h^*_\BR:= P\otimes_{\BZ} \BR$ is 
\[ C=\{ \lambda \in \h^*_\BR~|~ 0<\< \lambda +\rho, \a^\vee\> < \ell ~\text{for all $\a \in \Delta_+$}\}\]
Then the closure of $C$ is
\begin{equation}\label{eq: closure of C}\overline{C}=\{ \lambda \in  \h^*_\BR~|~0\leq \< \lambda+\rho, \a^\vee\> \leq \ell ~\text{for all $\a \in \Delta_+$}\}
\end{equation}
Let $\a_0$ be the highest positive root of $\g$. Then $\overline{C}$ is a polytope bounded by $r+1$ hyperplanes: 
\[ \<?+\rho, \a_i^\vee\>=0 ~\text{for all simple roots $\a_i$}, \qquad \< ?+\rho, \a_0^\vee\>=\ell.\]
\begin{Rem}\label{rem: not belong to affine hyper} If $\lambda\in \overline{C}$, then $\< \lambda+\rho, \a^\vee\> \in [0, \ell)$ if and only if $\lambda$ does not belong to the hyperplane $\< ?+\rho, \a_0^\vee\>=\ell$.
\end{Rem}
\subsection{Affine Weyl group $W_{*aff}$}\

Fix a positive integer $\ell$. For all positive roots $\a \in \Delta_+$ of $\g$, let
\begin{equation}\label{eq: li and la}
\ell_i :=\ell/\text{gcd}(2\sd_i, \ell)\quad (1\leq i \leq r), \qquad \ell_\a:= \ell/\text{gcd}((\a, \a), \ell),
\end{equation}

Let us consider the following data: 

\begin{itemize}
\item The lattices $P^*=\bigoplus_{i=1}^r \BZ \w^*_i$ and $ Q^*=\bigoplus_{i=1}^r \BZ \a^*_i$, where
$\w^*_i:=\ell_i \w_i$, $\a^*_i:=\ell_i \a_i$. Then set $\w^{*\vee}_i:=\w^\vee_i/\ell_i$, $\a^{*\vee}_i:=\a^\vee_i/\ell_i$ and $\a^{*\vee}:= \a^\vee/\ell_\a$ for all $\a \in \Delta_+$.
\item The new Cartan matrix with $(i,j)$-entry 
\begin{equation}\label{New Cartan entry} 
  a_{ij}^*=2(\a^*_i,\a^*_j)/(\a^*_i, \a^*_i)=2\ell_j(\a_i, \a_j)/\ell_i(\a_i, \a_i) \,.
\end{equation}
\item The bilinear form on $P^* \subset P$ is restricted from the bilinear form on $P$ by restriction.
\end{itemize}

So that $(a^*_{ij})$ is the Cartan matrix of a semisimple Lie algebra $\g^d$ by \cite[$\mathsection 2.2.4$]{l-book}. 
\begin{Rem}\label{rem: Q*} $Q^*=\{ \a \in Q~|~ (2\lambda, \a) \; \text{is divisible by $\ell$ for all $\lambda \in P$}\}$.
\end{Rem}
We have the affine Weyl group $W_{*aff}=W\ltimes Q^*$ and the extended affine Weyl group $W_{*ext}:=W\ltimes P^*$ . Both act on $P$ and $\h^*$. The {\em dot action} $\bullet$ of $W_{*ext}$ on $\h^*$ is defined by:
\begin{equation}
\label{eq: dot action W*}
w \bullet (\mu)=w(\mu+\rho)-\rho, \qquad t_\lambda \bullet \mu=\mu+\lambda, \qquad \text{for $w\in W, \lambda \in P^*, \mu \in \h^*$}.
\end{equation}
\begin{Rem}\label{rem: Waff for odd l} When $\ell$ is odd  (coprime to $3$ when $\g$ has  component of type $G_2$), then $P^*=\ell P$ and the dot action $\bullet$ of $ W_{*ext}$ on $\h^*$ coincides with the $\ell$-shifted dot action $\bullet_\ell$ of $ W_{ext}$ on $\h^*$. 
\end{Rem}

The {\em fundamental alcove} $C \subset \h^*_{\BR}:=P\otimes_{\BZ} \BR$ under the dot action of $W_{*aff}$ is defined as 
\begin{equation}\label{eq: fundamental alcove for W*} C:=\{ x\in P\otimes_{\BZ}\BR~|~ 0 < \< x+\rho, \a^{*\vee}\> <1 \;\; \text{for positive roots $\a \in \Delta_+$}\}.
\end{equation}

The next lemma is a slight generalization of \cite[Lemma 7.8]{J03} with the same proof.
\begin{Lem}\label{lem: Key lemma on weights}Let $\mu, \lambda$ be in the closure $\overline{C}$ of the fundamental alcove in $\h^*_{\BR}$  and $\nu_1$ be the unique dominant weight in the $W$-orbit of $\mu-\lambda$.

\noindent
(a) We have $\lambda+w\nu \not \in W_{*aff}\bullet \mu$ for all $w\in W$ and $\nu \in P_+$ such that $\nu<\nu_1$.

\noindent
(b) If $w\in W$ with $\lambda+w\nu_1 \in W_{*aff}\bullet \mu$ then there is some $w_1\in W_{*aff}$ such that $w_1 \bullet  \lambda=\lambda$ and $w_1 \bullet \mu=\lambda+w\nu_1$.
\end{Lem}

\subsection{The algebra $\sR$} \label{ssec: the algebra R}\

Let $\hat{\h}:=\h \oplus \BC \hbar$. The group $W_{ext}$ acts on $\hat{\h}$ as follows: $\hbar$ is invariant, $W$ acts on $\h$ by the default action, and for $\lambda \in P, x\in \h$ then $t_\lambda x=x+\< \lambda, x\> \hbar$.

We identify the dual $\hat{\h}^*$  with $\h^*\oplus \BC$ by 
\[ \< (\mu, z), x\>=\< \mu, x\>, \quad \<(\mu, z), \hbar\>= -z.\]
The action of $W_{ext}$ on $\hat{\h}^*$ becomes
\[ w (\mu, z)=(w\mu, z), \quad t_\lambda (\mu, z)= (\mu+z\lambda, z).\]
\begin{defi}Let $\sR:=\BC[[\hat{\h}^*]]$, the completion of $\BC[\hat{\h}^*]=S(\hat{\h})$ at $0$. Let $\m$ be the maximal ideal of $\sR$. Then $W_{ext}$ acts on $\sR$ as above.
\end{defi}

\begin{Rem}The quotient $\usR:=\sR/\hbar \sR $ is isomorphic to $\BC[[\h^*]]$, the completion of $\BC[\h^*]$ at $0$. Furthermore, the $W_{ext}$-action on $\sR$ induces a $W_{ext}$-action on $\usR$. This $W_{ext}$-action on $\usR$ comes from the $W_{ext}$-action on $\h^*$ via the quotient $W_{ext}\twoheadrightarrow W \curvearrowright \h^*$.
\end{Rem}




\section{Rational representations of $\cU_q(\g)$}

We recall the category of rational representations of the Lusztig form $\cU_q(\g)$ and discuss  projective and tilting objects in this category.

\subsection{Rational representations of $\cU_q(\g)$}\

\label{ssec: rational reps}

Recall the triangular decomposition of $\cU_q(\g)$ in  Remark \ref{rem: triangular decomposition}. Let $\sN \in \BZ_{>0}$ be such that $\frac{\sN}{2}(\frac{P}{2}, \frac{P}{2}) \in \BZ$. Fix $q^{1/\sN} \in R$ such that $(q^{1/\sN})^\sN =q$. For any $\lambda \in P$, let define $\chi_\lambda: \cU^0_q \rightarrow R$ by 
\begin{equation}\label{eq: character chi}
    \chi_\lambda(K^\mu)=q^{(\mu, \lambda)}, \qquad \chi_\lambda\left(\binom{K_i;0}{m}\right)=\binom{(\lambda, \a^\vee_i)}{m}_{q_i},\qquad \text{for $\mu \in 2P$ and $m \in \BN$.}
\end{equation}

 \begin{defi}[cf. \cite{APW1}, $\mathsection 1$] A $\cU_q(\g)$-module $M$ is a {\it rational representation} (of type $1$) if it satisfies the following conditions:
 \begin{enumerate}
     \item[(i)] $M$ is a weight module meaning  that there is a decomposition $M=\bigoplus_{\lambda \in P} M_\lambda$, where $u_0 m=\chi_\lambda(u_0)m$ for all $u_0\in \cU^0_q, m\in M_\lambda$. 
     \item[(ii)]For any $m \in M$, there is $k>0$  such that $\tE_i^{(s)} m=0$ for all $s>k$ and all $1\leq i \leq r$.
     \item[(iii)]For any $m \in M$, there is $k>0$ such that $\tF_i^{(s)}m=0$ for all $s>k$ and all $1 \leq i \leq r$.
 \end{enumerate}
 Let $\Rep(\cU_q(\g))$ denote the category of rational $\cU_q(\g)$-representations. Let $\Rep^{fd}(\cU_q(\g))$ denote the full subcategory of $\Rep(\cU_q(\g))$ consisting of objects that are finitely generated $R$-modules.
 \end{defi}
\begin{Rem}\label{rem: dual module}For any $V\in \Rep^{fd}(\cU_q(\g))$, we define $\cU_q(\g)$-representations $V^*$ and $ ^*V$ as follows:
\begin{itemize}
 \item Set $V^*:=\Hom_R(V, R)$ and $(uf)(v)=f(S(u)v)$ for $v\in V, f\in V^*, u\in \cU_q(\g)$.
\item Set $^*V:=\Hom_R(V,R) $ and  $(uf)(v)=f(S^{-1}(u)v)$ for $v\in V, f\in~^*V, u\in \cU_q(\g)$.
\end{itemize}
Note that $^* V\cong V^*$ since $S^2(u)=K^{\rho} uK^{-\rho}$ for all $u \in \cU_q(\g)$.
\end{Rem}
\begin{defi}(a) For any $\lambda \in P$, let $R_\lambda$ denote the representation of $\cU^{\geqslant}_q:= \cU^0_q \otimes_R \cU^>_q$ defined by $\cU^\geqslant_q \rightarrow \cU^0_q \xrightarrow[]{\chi_\lambda} R$. Then we define the {\em Verma module} $\Delta_q(\lambda):= \cU_q(\g) \otimes_{\cU^\geqslant_q} R_\lambda$.

\noindent
(b) For any $\lambda \in P_+$, the {\em Weyl module} $W_q(\lambda)$ is the maximal rational quotient of the Verma module $\Delta_q(\lambda)$. Let $1_\lambda$ be the image of $1\in \cU_q(\g)$ in $\Delta_q(\lambda)$. Then $W_q(\lambda)$ is the quotient of $\Delta_q(\lambda)$ by the left $\cU_q(\g)$-submodule generated by $\tF^{(s)}_i 1_\lambda$ for $s> (\lambda, \a^\vee_i)$ and $1\leq i \leq r$.
\end{defi}
For the existence of the maximal rational quotient of $\Delta_q(\lambda)$ and the description of $W_q(\lambda)$, see \cite[Proposition 1.14, 1.20]{APW1}. We remark that these results in \cite{APW1} hold for general ring $R$ as being argued in \cite[Remark 7.8, 7.13]{LTV}.

\begin{defi}A $\dU_q(\g,P)$-module $M$ is a { \it unital rational representation} if it satisfies
\begin{enumerate}
    \item[(i)] For any $m \in M$, we have  $1_\lambda m =0$ for all but finitely many $\lambda \in P$ and $\sum_\lambda 1_\lambda m=m$.
    \item[(ii)] For any $m \in M$, there is $k>0$ such that $\tE^{(s)}_i1_\lambda m=0$ for all $s>k$ and all $1\leq i \leq r$.
    \item[(iii)]For any $m \in M$, there is $k>0$ such that $\tF^{(s)}_1 1_\lambda m=0$ for all $s>k$ and all $1\leq i \leq r$.
\end{enumerate}
Let $\Rep(\dU_q(\g, P))$ denote the category of unital rational representations of $\dU_q(\g,P)$. The category $\Rep(\dmU_q(\g, P))$ is defined similarly.
\end{defi}
\begin{Rem}\label{rem: equivalence of Rat reps}There is a natural equivalence of braided monoidal categories
\begin{equation*}\Rep(\cU_q(\g)) \cong \Rep(\dU_q(\g, P)).
\end{equation*}
On the other hand, since  the Hopf structure of $\dU_q(\g, P)$ is obtained from  the Hopf structure of $\dmU_q(\g, P)$ by a Drinfeld twist, there is  a natural equivalence of braided monoidal categories:
\begin{equation*}
     \Rep(\dU_q(\g, P)) \cong \Rep(\dmU_q(\g, P)).
 \end{equation*}
\end{Rem}
\subsection{Quantum Frobenius homomorphism}\

In this section, we assume that  $(\textbf{A})$ $R=\BC$ and  $q=\e\in R$, a root of unity of order $\ell$ so that  $\ell_i \geq \max\{ 2, -a_{ij}\}_{1\leq i \leq r}$ for all $1\leq i \leq r$.

Let us recall the data \eqref{New Cartan entry} and the semisimple Lie algebra $\g^d$, which is either $\g$  or the Langland dual $\g^\vee$ of $\g$. Hence $\Dyn(\g^d)$ is the same graph as $\Dyn(\g)$. Let us fix the same orientation $\Or$ for $\Dyn(\g^d)$ as in $\Dyn(\g)$.

We form the $\BQ(v^{1/2})$-Hopf algebra $\bU^*(\g, P^*/2)$  with generators $\{\he_i, \hf_i, K^\mu\}_{1\leq i \leq r}^{\mu \in P^*/2}$ and the data \eqref{New Cartan entry}. We have the following twist with respect to the orientation $\Or$ of $\Dyn(\g^d)$: 
\[ \sF^*:=v^{\sum_{i,j} \phi^*_{ij} \w^{*\vee}_i \otimes \w^{*\vee}_j}, \qquad \text{here}\qquad \phi^*_{ij}=\e_{ij}\frac{(\a^*_i, \a^*_j)}{2}.\]
As in Section \ref{ssec: twisted construction}, we consider the following twisted generators:
\[ \te_i:= \he_i K^{\nu^{*>}_i}, \qquad \tf_i:= K^{-\nu^{*<}_i}\hf_i, \]
where 
 \[ \nu^{*>}_i := -\a^*_i+\sum_{1\leq j \leq r} \phi^*_{ij} \w^{*\vee}_j=\ell_i \nu^>_i, \qquad \nu^{*<}_i:= \sum_{1\leq j \leq r} \phi^*_{ij} \w^{*\vee}_j =\ell_i \nu^<_i.\]
Finally, we obtain the idempotented Lusztig form $\dU^*_q(\g, P^*)$ (after the base change $\CA \rightarrow R, v \mapsto q \in R^\x$) with generators:
\[ \{ \te^{(n)}_i 1_\lambda, \tf^{(n)}_i 1_\lambda| 1\leq i \leq r, \lambda \in P^*\}.\]
The next proposition is a twisted version of \cite[Theorem 35.1.9]{l-book}, see \cite[$\mathsection 5$]{LTV}. We consider the assumption $(\textbf{A})$ for simplicity.
\begin{Prop}\label{prop: Frobenius morphism}  There is a unique $\BC$-algebra homomorphism 
\[ \tFr: \dU_\e(\g,P)\rightarrow \dU^*_\e(\g, P^*)\]
such that 
\begin{itemize}
    \item $\tFr(\tE^{(n)}_i1_\lambda)$ equals $\te_i^{(n/\ell_i)}1_\lambda$ if $\lambda \in P^*$ and $n$ is divisible by $\ell_i$, and is zero otherwise.
    \item $\tFr(\tF^{(n)}_i 1_\lambda$ equals $\tf_i^{(n/\ell_i)}1_\lambda$ if $\lambda \in P^*$ and $n$ is divisible by $\ell_i$, and is zero otherwise.
\end{itemize}
Furthermore, this homomorphism  is compatible with comultiplications.
\end{Prop}
 \begin{Rem}The homomorphism $\tFr$ gives a rise to a monoidal functor of categories: 
\begin{equation}
    \tFr^*:  \Rep(\dU^*_\e(\g, P^*) \rightarrow \Rep(\dU_\e(\g, P)).
\end{equation}
\end{Rem}
Let $\cU_\BC(\g^d)$ be the enveloping algebra of $\g^d$. The next result follows by \cite[Proposition 17]{LTV}.
\begin{Prop}There is a unique $\BC$-linear homomorphism of Hopf algebras
\[ \tFr: \cU_\e(\g)\rightarrow \cU_\BC(\g^d)\]
 such that 
 \[ \tE_i^{(n)} \mapsto (\e^*_i)^{-n /\ell_i}e_i^{(n/\ell_i)}, \qquad \tF_i^{(n)} \mapsto f_i^{(n/\ell_i)}, \qquad K^\lambda \mapsto 1,\]
 where $\lambda \in 2P$ and we set $e_i^{(n/\ell_i)}=f_i^{(n/\ell_i)}=0$ if $\ell_i$ does not divide $n$.
\end{Prop}

\subsection{More on rational representations }\  \label{ssec: more on rational}

Until the end of this section, we consider the following two cases: 
\begin{enumerate}[label=(\textbf{\Alph*})]
    \item\label{case over C}$R=\BC$ and  $q=\e\in R$, a root of unity of order $\ell$. We assume that  $\ell_i \geq \max\{ 2, -a_{ij}\}_{1\leq i \leq r}$ for all $1\leq i \leq r$.
    \item\label{case over C[[h]]}$R=\BC[[\hbar]]$ and $q=\e e^{2\pi \sqrt{-1}\hbar/\ell} \in R$, where $\e$ as in $(\textbf{A})$.
\end{enumerate}
So we have a short exact sequence: $0 \rightarrow \cU_q(\g) \xrightarrow[]{\cdot \hbar} \cU_q(\g) \xrightarrow[]{/\hbar} \cU_\e(\g)\rightarrow 0$.

\begin{defi}\label{defi: Steinberg representation} Let $\lambda_{\St}:=\sum_{i}(\ell_i-1) \w_i$. The Steinberg module  $\St_\e\in \Rep(\cU_\e(\g))$ is the Weyl module $W_\e(\lambda_\St)$. The Steinberg module $\St_q\in \Rep(\cU_q(\g))$ is the Weyl module $W_q(\lambda_\St)$. 
\end{defi}
We have an equivalence of braided monoidal categories $\Rep(\cU_\e(\g))\cong \Rep(\dmU_\e(\g, P))$, see Remark \ref{rem: equivalence of Rat reps}. Hence, by \cite{N}, we have the following proposition: 
\begin{Prop}\label{prop: proj in Rep(Ue(g))}(a) The module $\St_\e$ is projective, injective and self-dual in $\Rep(\cU_\e(\g))$. 

\noindent
(b) The category $\Rep(\cU_\e(\g))$ has enough projectives and injectives.  Any object $M$ in $\Rep^{fd}(\cU_\e(\g))$ admits an epimorphism from a projective object of the form $\St_\e \otimes_\BC N$ with $N \in \Rep^{fd}(\cU_\e(\g))$.
\end{Prop}

\begin{Prop}\label{prop: proj in Rep(Uq(g))}(a) For any $N_q\in \Rep^{fd}(\cU_q(\g))$ that  is a free module of finite rank over $\BC[[\hbar]]$, the object $\St_q\otimes_{\BC[[\hbar]]} N_q$ is projective in $\Rep(\cU_q(\g))$.

\noindent
(b) Any object in $\Rep^{fd}(\cU_q(\g))$ admits a surjective morphism from some projective object of the form $\St_q\otimes_{\BC[[\hbar]]} N_q$  as in part (a). Furthermore, $N_q$ can be chosen so that $N_q=\oplus_{\lambda_i} W_q(\lambda_i)$ for some  dominant weights $\lambda_i$.

\noindent
(c) The category $\Rep(\cU_q(\g))$ has enough projectives.

\noindent
(d) $^*\St_q \cong \St_q^* \cong \St_q$.
\end{Prop}
\begin{proof}
(a) follows by Proposition \ref{prop: proj in Rep(Ue(g))} and the following claim:
\begin{itemize}\item[{}]
{\it Claim:} Suppose $V_q \in \Rep^{fd}(\cU_q(\g))$ is a free module over $\BC[[\hbar]]$ and $V_q/\hbar V_q$ is projective in $\Rep^{fd}(\cU_\e(\g))$, then $V_q$ is projective in $\Rep^{fd}(\cU_q(\g))$.
\end{itemize}

To prove the claim, we need to show that  $\Ext^1_{\Rep(\cU_q(\g))}(V_q, N)=0$ for any $N\in \Rep^{fd}(\cU_q(\g))$.

\noindent
{\it Step 1:} If $N \in \Rep^{fd}(\cU_\e(\g))$ then, by \cite[Lemma 7.20]{LTV},
\[ \Ext^1_{\Rep(\cU_q(\g))}(V_q, N)=\Ext^1_{\Rep(\cU_\e(\g))}(V_q/\hbar V_q, N),\]
and the right hand side is zero since $V_q/\hbar V_q$ is projective in $\Rep^{fd}(\cU_\e(\g))$.

\noindent
{\it Step 2:} Suppose $N \in \Rep^{fd}(\cU_q(\g))$ is flat over $\BC[[\hbar]]$. We get a short exact sequence  $0\rightarrow N \xrightarrow[]{\cdot \hbar} N \rightarrow N/\hbar N\rightarrow 0$ in $\Rep^{fd}(\cU_q(\g))$, which gives a long exact sequence of $\BC[[\hbar]]$-modules
\[ \dots \rightarrow \text{Ext}^1_{\Rep(\cU_q(\g))}(V_q, N)\xrightarrow[]{\cdot \hbar} \text{Ext}^1_{\Rep(\cU_q(\g))}(V_q, N) \rightarrow \text{Ext}^1_{\Rep(\cU_q(\g))}(V_q, N/\hbar N)\dots\]
By Step 1, the last term vanishes, so the map $ \text{Ext}^1_{\Rep(\cU_q(\g))}(V_q, N)\xrightarrow[]{\cdot \hbar} \text{Ext}^1_{\Rep(\cU_q(\g))}(V_q, N)$ is surjective. On the other hand, $\text{Ext}^1_{\Rep(\cU_q(\g))}(V_q, N)$ is finitely generated over $\BC[[\hbar]]$, one way to prove it is in \cite[Proposition $5.15$]{APW1}. By the  Nakayama lemma, $\text{Ext}^1_{\Rep(\cU_q(\g))}(V_q, N)=0$.

\noindent
{\it Step 3:} For any $N \in \Rep^{fd}(\cU_q(\g))$, let $N_{\textnormal{tor}}:=\{n \in N~|~\hbar^k n=0 \;\text{for some $k >0$}\}$. Then $N_{\textnormal{tor}}$ is a subobject of $N$ in $\Rep^{fd}(\cU_q(\g))$. Since $N_{\textnormal{tor}}$ is finitely generated over $\BC[[\hbar]]$, it admits a finite filtration whose subquotients are objects in $\Rep^{fd}(\cU_\e(\g))$. On the other hand, $N/N_{\textnormal{tor}}$ is flat over $\BC[[\hbar]]$. Therefore, by Step 1 and Step 2, we have $\text{Ext}^1_{\Rep(\cU_q(\g))}(V_q, N)=0$.

\noindent
(b) For $N \in \Rep^{fd}(\cU_q(\g))$, the quotient $N/\hbar N$ is in $ \Rep^{fd}(\cU_\e(\g))$. There is a surjective map: 
\begin{equation*}\St_\e \otimes_\BC \Big(\bigoplus_{\lambda_i} W_\e(\lambda_i) \Big)\twoheadrightarrow N/\hbar N, 
\end{equation*}
 for a finite collection of dominant weights $\{ \lambda_i\}$. Here we only need to use Weyl modules because of the surjective map $\St_\e \otimes W_\e(\lambda_i) \twoheadrightarrow \St_\e \otimes L_\e(\lambda_i)$ and  that any simple module $L_\e(\mu)$ can be covered by projective objects of the from $\St_\e \otimes L_\e(\lambda_i)$. If any simple object is covered by a direct sum of projectives $\St_\e\otimes_\BC W_\e(\lambda_i)$ then every object in $\Rep^{fd}(\cU_\e(\g))$ admits such a cover.
 
 Since $\St_q \otimes_{\BC[[\hbar]]} W_q(\lambda_i)$ is projective in $\Rep^{fd}(\cU_q(\g))$ by part (a), we have the following commutative diagram in $\Rep^{fd}(\cU_q(\g))$: 
\[ \begin{tikzcd} \St_q \otimes_{\BC[[\hbar]]}\left(\bigoplus_{\lambda_i} W_q(\lambda_i)\right) \arrow[d, " /\hbar"] \arrow[r, dashrightarrow ]& N \arrow[d, " /\hbar"']& \\
\St_\e \otimes_\BC \left(\bigoplus_{\lambda_i} W_\e(\lambda_i)\right) \arrow[r, two heads] & N/\hbar N    
\end{tikzcd}
\]
By the  Nakayama lemma, the upper horizontal arrow is surjective. This proves part (b).

\noindent
(c) follows from part (b) since any object in $\Rep(\cU_q(\g))$ is a union of objects in $\Rep^{fd}(\cU_q(\g))$.

\noindent
(d) We have $^* \St_q \xrightarrow[]{/\hbar} \;^*\St_\e\cong \St_\e \cong \St^*_\e \xleftarrow[]{/\hbar} \St^*_q$, where $^*\St_\e \cong \St_\e \cong \St^*_\e$ follows by Proposition \ref{prop: proj in Rep(Ue(g))}.(a). Since  $\St_q$ is projective in $\Rep(\cU_q(\g))$, we must have $^*\St_q\cong \St_q \cong \St^*_q$.
\end{proof}

\subsection{Tilting modules} \label{ssec: tilting module} We assume that $(R, q)$ is as  in Case \ref{case over C} or Case\ref{case over C[[h]]}.  

In \cite{APW1}, the induction functor $H^0_q: \Rep(\cU^\leq_q) \rightarrow \Rep(\cU_q(\g))$ was constructed. This functor is the right adjoint functor to the forgetful functor $\mathfrak{F}: \Rep(\cU_q(\g))\rightarrow \Rep(\cU^\leq_q)$. 

For a dominant weight $\lambda$, let $R_\lambda \in \Rep(\cU^\leq_q)$  be defined via $\cU^\leq_q \rightarrow \cU^0_q \xrightarrow[]{\chi_\lambda}  R$. We define the {\it dual Weyl module} $H^0_q(\lambda):= H^0_q(R_\lambda)$. The next lemma, proved in \cite[Proposition 3.3]{APW1} under some restriction on $R$, holds in general as was explained in \cite[Proposition 7.12]{LTV}
\begin{Lem}\label{lem: dual Weyl and Weyl}  $H^0_q(\lambda) \cong W_q(-w_0\lambda)^*$.
\end{Lem}

\begin{defi}\label{defi: good filtration} Let $M \in \Rep(\cU_q(\g))$. An exhaustive $\cU_q(\g)$-module filtration $\{0\}=M_0 \subset M_1 \subset \dots$ is called {\em good} if, for each $i$, we have $M_i/M_{i-1}\cong H^0_q(\lambda_i)\otimes_R P_i$ for some dominant weight $\lambda_i$ and some finitely generated projective $R$-module $P_i$.    
\end{defi}
\begin{Rem}In Case \ref{case over C} or Case \ref{case over C[[h]]}, any finitely generated  projective $R$-module $P_i$ is free.
\end{Rem}
\begin{defi}\label{defi: tilting module}$M \in \Rep^{fd}(\cU_q(\g))$ is a {\em tilting module} if $M$ and $M^*$ admit good filtrations. 
\end{defi}
The following result is standard 
\begin{Prop}(a) The set of indecomposable tilting modules in $\Rep^{fd}(\cU_q(\g))$ is in one-to-one correspondent to the set of dominant weights. Let $T_q(\lambda)$ denote the indecomposable tilting module associated to the dominant weight $\lambda$.

\noindent
(b)  The weight space $(T_q(\lambda))_\lambda$ is a free $R$-module of rank one. For any weight $\nu$ of $T_q(\lambda)$, we have $\nu_+<\lambda$ where $\nu_+$ is the unique dominant weight in $W\nu$.
\end{Prop}

\begin{Prop}\label{prop: proj is tilt} Any projective object in $\Rep^{fd}(\cU_q(\g))$ is tilting.
\end{Prop}
\begin{proof}It is the fact that any direct summand of a tilting module is tilting. By Proposition \ref{prop: proj in Rep(Uq(g))}, it is enough to show that $\St_q \otimes_R N_q$ is tilting for $N_q \in \Rep^{fd}(\cU_q(\g))$ which is a free $R$-module. 

First, we show that $\St_q\otimes_R N_q$ has a good filtration. Using \cite[Lemma 7.28]{LTV}, it is enough to show that $\Ext^i_{\Rep(\cU_q(\g))}(W_q(\lambda), \St_q\otimes_R N_q)=0$ for $i>0$ and all dominant weight $\lambda$. Indeed,   
\begin{align*}\Ext^i_{\Rep(\cU_q(\g))}(W_q(\lambda), \St_q\otimes_R N_q) &\cong \Ext^i_{\Rep(\cU_q(\g)}(\St_q^*\otimes_R W_q(\lambda), N_q)\\
&\cong \Ext^i_{\Rep(\cU_q(\g)}(\St_q\otimes_R W_q(\lambda), N_q)=0, \quad \text{for $i>0$},
\end{align*}
here we use $\St_q^* \cong \St_q$ and $\St_q\otimes_R W_q(\lambda)$ is projective in $\Rep(\cU_q(\g))$ by Proposition \ref{prop: proj in Rep(Uq(g))}.   

On the other hand, $(\St_q\otimes_R N_q)^*\cong N_q^*\otimes_R \St_q^* \cong \St_q \otimes_R N_q^*$, here we use the braiding of $\Rep(\cU_q(\g))$ and $\St_q^*\cong \St_q$ in the second isomorphism. Since $N_q^*$ is a free $R$-module, $(\St_q\otimes_R N_q)^*$ also has a good filtration as above. So $\St_q\otimes_R N_q$ is tilting. This finishes the proof.
\end{proof}



\section{The even subalgebra  $U^{ev}_q(\g)$}
\label{ssec: even part algebra}



In this section, we discuss the structure of the even subalgebra $U_q^{ev}(\g)$ in more detail. There is Lusztig's braid group action on $\bU_v(\g, P/2)$ defined as follows , see \cite[Part VI]{l-book}, \cite[$\mathsection 4.9$]{J96},
\begin{equation*}\label{eq:braid group}
\begin{split}
   T_i(K^\mu)&=K^{s_{\a_i}\mu} \,, \qquad T_i(E_i)=-F_iK^{\a_i} \,, \qquad T_i(F_i)=-K^{-\a_i}E_i \,, \\
   T_i(E_j)&=\sum_{k=0}^{-a_{ij}} (-1)^k\frac{v_i^{-k}}{[-a_{ij}-k]_{v_i}![k]_{v_i}!}E_i^{-a_{ij}-k}E_jE_i^k \,, \\
   T_i(F_j)&=\sum_{k=0}^{-a_{ij}} (-1)^k\frac{v_i^{k}}{[-a_{ij}-k]_{v_i}![k]_{v_i}!}F_i^kF_jF_i^{-a_{ij}-k} \,.
\end{split}
\end{equation*}

Let us pick a reduced expression of the longest element $w_0=s_{i_1}s_{i_2} \dots s_{i_N}$ in the Weyl group $W$, here $N$ is the cardinality of the positive root system $\Delta_+$. Then $\Delta_{+}=\{ \b_k=s_{i_1}\dots s_{i_{k-1}} \a_{i_k}\}_{k=1}^N$. We define the root vectors $\{ E_\b, F_\b\}_{\b \in \Delta_+}$ in a standard way:
\begin{equation*}\label{eq:root-generator}
  E_{\beta_{k}} =\, T_{i_1}\cdots T_{i_{k-1}} E_{i_k} \,, \qquad 
  F_{\beta_{k}} =\, T_{i_1}\cdots T_{i_{k-1}} F_{i_k} =\tau(E_{\b_k})
  \qquad \forall\, 1\leq k\leq N \,.
\end{equation*}

For a positive root $\b=\sum a_i \a_i$, let $\nu^>_\b:= \sum_i a_i \nu^>_i, \nu^<_\b:=\sum_i a_i \nu^<_i$ and set 
\begin{equation*}
 \tE_{\b_k}:= v^{b^>_{\b_k}}E_{\b_k} K^{\nu^>_{\b_k}}, \qquad \tF_{\b_k}:= v^{b^<_{\b_k}}K^{-\nu^<_{\b_k}} F_{\b_k},
\end{equation*}
here $b^>_{\b_k}, b^<_{\b_k}\in \BZ/2$ are normalized factors, see \cite[$\mathsection 3.6$]{LTV}, so that $\tE_{\b_k}, \tF_{\b_k} \in U^{ev}_\CA(\g)$.  Set
\[ \tE^{\vec{k}}:=\tE_{\b_1}^{k_1}\dots \tE_{\b_N}^{k_N} \,, \qquad \tF^{\vec{k}}:=\tF_{\b_1}^{k_1}\dots \tF_{\b_N}^{k_N} \,, \qquad 
    \tE^{\cev{k}}:=\tE_{\b_N}^{k_N}\dots \tE_{\b_1}^{k_1} \,, \qquad \tF^{\cev{k}}:=\tF_{\b_N}^{k_N}\dots \tF_{\b_1}^{k_1} \,,\]
Let $U_\CA^{ev>}, U^{ev 0}_\CA, U^{ev<}_\CA$ be the $\CA$-subalgebras of $U^{ev}_\CA(\g)$ generated by $\{\tE_i\}_{i=1}^r, \{ K^\lambda\}_{\lambda \in 2P}, \{ \tF_i\}_{i=1}^r$, respectively. We have the PBW basis for $U^{ev}_\CA(\g)$ by \cite[Lemma 3.10]{LTV}:
\begin{Lem}\label{lem: PBW-basis for even part}
(a) There is the triangular decomposition $U_\CA^{ev}(\g) \cong U^{ev<}_\CA \otimes_\CA U^{ev0}_\CA \otimes_\CA U^{ev>}_\CA$.

\noindent
(b) The sets $\{ \tE^{\vec{k}}\}_{\vec{k}\in \BZ^N_{\geq0}}, \{ \tE^{\cev{k}}\}_{\vec{k}\in \BZ^N_{\geq 0}}$ are $\CA$-bases of $U^{ev>}_\CA$.

\noindent
(c) The sets $\{ \tF^{\vec{k}}\}_{\vec{k}\in \BZ^N_{\geq 0}}, \{\tF^{\cev{k}}\}_{\vec{k}\in \BZ^N_{\geq 0}}$ are $\CA$-bases of $U^{ev<}_\CA$.

\noindent
(d) The set $\{K^\mu\}_{\mu \in 2P}$ is an $\CA$-basis of $U^{ev 0}_\CA$.
\end{Lem}
By base change, we have the corresponding PBW-bases for  $ U^{ev}_q(\g)$.

\subsection{Harish-Chandra center $Z_{HC}$}\
\label{ssec: HC center}

\begin{defi}The Harish-Chandra center $Z_{HC}$ of $U^{ev}_q(\g)$ is the subset of all $\cU_q(\g)$-invariant elements in  $U^{ev}_q(\g)$.
\end{defi}

By \cite[Lemma 9.29]{LTV},  $Z_{HC}$ is central in $U^{ev}_q(\g)$. Let $\pi_-: U_q^{ev<}=R\bigoplus\oplus_{0\neq \a\in Q_+}U_{q, -\a}^{ev <}\rightarrow R$ and $\pi_+: U_q^{ev >}=R\bigoplus \oplus_{0\neq \a \in Q_+}U_{q, \a}^{ev>}\rightarrow R$ be the projections. Let us consider the natural map:
\[ \pi: Z_{HC}\hookrightarrow U^{ev}_q(\g) \cong U^{ev<}_q \otimes_R U^{ev0}_q \otimes_R U^{ev>}_q \xrightarrow[]{\pi_-\otimes \text{Id}\otimes \pi_+} U^{ev 0}_q\cong R\<K^{2\lambda}\>_{\lambda \in P},\]
here $R\< K^{2\lambda}\>_{\lambda \in P}$ is the polynomial algebra of the lattice $2P$. 
\begin{Prop}[Theorem 9.30 \cite{LTV}] \label{prop: HC center} The map $\pi$ restricts to an isomorphism of algebras
\[ \pi: Z_{HC}\iso R\< K^{2\lambda}\>_{\lambda \in P}^{W_\bullet},\]
where the dot-action of the Weyl group $W$  on $R\< K^{2\lambda}\>_{\lambda \in P}$ is defined via:
\[ w_\bullet(K^\mu)=q^{(w^{-1}\rho-\rho, \mu)}K^{w(\mu)} \qquad \text{for all $x\in W,~ \mu \in 2P$}.\]
\end{Prop}
\begin{Rem}\label{rem: renormalize HC map} Let $\gamma_{-\rho}: R\< K^{2\lambda}\>_{\lambda \in P}\rightarrow R\< K^{2\lambda}\>_{\lambda \in P}$ is defined by $\gamma_{-\rho}(K^{2\lambda})= q^{(-\rho, 2\lambda)}K^{2\lambda}$ for all $\lambda \in P$. Then we have an isomorphism 
\[ \gamma_{-\rho}\circ \pi: Z_{HC}\iso R\< K^{2\lambda}\>_{\lambda \in P}^W.\]
\end{Rem}

\subsubsection{Harish-Chandra center when $R=\BC$}\

Let us consider the case when $R=\BC$ and $q=e^a \in \BC^\x$ for some $ a\in \BC$, then we set $q^z:=e^{az}$ for any $z \in \BC$. Any $\lambda \in \h^*$ defines a closed point on $T:= \Spec \BC\<K^{2\mu}\>_{\mu \in P}$  corresponding to the homomorphism $K^{2\mu} \mapsto q^{(2\mu, \lambda)}$. This defines a map $ \h^* \rightarrow T$. 

We recall the isomorphism $\Spec Z_{HC} \cong T/(W, \bullet)$, where $(W, \bullet)$ means that $W$ acts on $T$ via the $\bullet$-action defined in Proposition \ref{prop: HC center}. So we have a composition
\begin{equation}\label{eq: tau map}\tau: \h^* \rightarrow T \rightarrow T/(W, \bullet).
\end{equation}
Let 
\[ X:=\{ \lambda \in \h^*~|~q^{(2\mu, \lambda)}=1 \; \text{for all $\mu \in P$}\}.\]
\begin{Lem}\label{lem: closed points on HC center}
(a) Let $\lambda, \mu \in \h^*$ then  $\tau(\mu)=\tau(\lambda)$ if and only if $\lambda -w\bullet \mu \in X$ for some $w\in W$.

\noindent
(b) Suppose $q$ is a root of unity of order $\ell$.  If $\lambda \in \mu +Q$ then $\tau(\lambda) =\tau(\mu)$ implies that $\lambda$ is contained in the $W_{*aff}$-orbit of $\mu$ under the dot action $\bullet$ in \eqref{eq: dot action W*}.
\end{Lem}
\begin{proof}
(a) Straightforward. (b) This is equivalent to showing that $X\cap Q=Q^*$. Since $q$ is a root of unity of order $\ell$, for any $a\in \BZ$, the equality $q^a=1$ holds if and only if $\ell~|~ a$.  Since $(2\mu, \nu) \in \BZ$ for all $\nu\in Q$ and $\mu \in P$, it follows that  $\nu \in Q\cap X$ if and only if $\ell ~|~(2\mu, \lambda)$ for all $\mu \in P$, equivalently, $\lambda \in Q^*$ by Remark \ref{rem: Q*}.
\end{proof}
\subsection{The Frobenius center $Z_{Fr}$}\
\label{ssec: Frobenius center}

In case \ref{case over C},  the algebra $U^{ev}_\e(\g)$ contains a large central subalgebra.

\begin{defi}The {\it Frobenius center}  $Z_{Fr}$ of $U^{ev}_\e(\g)$ is the subalgebra generated by 
\[\{\tE_{\a}^{\ell_\a}, \tF_\a^{\ell_\a}, K^\mu\}_{\a \in \Delta_+}^{ \mu \in 2P^*}.\]
\end{defi}
\begin{Rem}In \cite[$\mathsection 6$]{LTV}, we gave a conceptual definition of $Z_{Fr}$. That definition uses a $\cU_\e(\g)$-adjoint invariant pairing $U^{ev}_\e(\g) \x \dU_\e(\g,P)\rightarrow \BC$ which is non-degenerate in the first argument. Then $Z_{Fr}$ is defined to be the orthogonal complement of the kernel of the Frobenius map $\tFr$ form Proposition \ref{prop: Frobenius morphism} under this pairing. Since the construction of the pairing is involved and we will not need it in this paper, we refer readers to \cite[$\mathsection 6$]{LTV} for details.
\end{Rem}
\begin{Prop}(a) $Z_{Fr}$ is stable under the adjoint action of $\cU_\e(\g)$ on $U^{ev}_\e(\g)$. Furthermore, the action of $\cU_\e(\g)$ on $Z_{Fr}$ factors through the epimorphism $\tFr:  \cU_\e(\g) \rightarrow \cU_\BC(\g^d)$.

\noindent
(b) We have an isomorphism $Z_{Fr} \cong \BC[\tE_\a^{\ell_\a}]_{\a \in \Delta_+}\otimes_\BC \left(\bigoplus_{\lambda \in 2P^*} \BC K^{\lambda}\right) \otimes_{\BC} \BC[\tF_\a^{\ell_\a}]_{\a \in \Delta_+}$. Here, $\BC[\tE_\a^{\ell_\a}]_{\a \in \Delta_+}, \BC[\tF_\a^{\ell_\a}]_{\a \in \Delta_+}$ are the  polynomial algebras in the corresponding variables.  
\end{Prop}
Let us define two linear maps $\kappa, \gamma \in \End(\h^*)$ as follows:
\begin{equation}\label{kappa and gamma}
  \kappa(\alpha_i):=\alpha_i+\sum_{j=1}^r 2\phi_{ij}\omega^\vee_j=-\zeta^<_i \qquad \mathrm{and} \qquad 
  \gamma(\alpha_i):=\alpha_i-\sum_{j=1}^r 2\phi_{ij}\omega^\vee_j=\zeta^>_i \,,
\end{equation}
Let us define
\[ \tZ^>_{Fr}:=\BC[\tE^{\ell_\a}_\a K^{\ell_\a\gamma(\a)}]_{\a \in \Delta_+}, \qquad \tZ^<_{Fr}:= \BC[\tF^{\ell_\a}_\a K^{\ell_\a\kappa(\a)}]_{\a \in \Delta_+}, \qquad \tZ^0_{Fr}:= \bigoplus_{\lambda \in 2P^*} \BC K^{\lambda}.\]
\begin{defi} \label{defi: group Gd}Let $G^d$ be the simply connected algebraic group with the Lie algebra $\g^d$. Let $G^d_0$ be the open Bruhat cell of the algebraic group $G^d$.
\end{defi}
\begin{Prop}[Proposition 6.10 \cite{LTV}]\label{prop: F-center vs openBruhat}There is a $\cU_\BC(\g^d)$-linear algebra homomorphism:
\[\varphi\colon Z_{Fr} \,\iso\, \BC[G^d_0] \simeq \BC[U^d_-]\otimes_\BC \BC[T^d]\otimes_\BC \BC[U^d_+] \,. \]
Furthermore, under this isomorphism, we have $\tZ^>_{Fr} \cong \BC[U^d_-],$ $\tZ^0_{Fr} \cong \BC[U^d_0],$ $ \tZ^<_{Fr} \cong \BC[U^d_+]$.
\end{Prop}

\subsection{Center $Z$ of $U_\e^{ev}(\g)$}\
\label{ssec: whole center}

In this section, we will study the entire center $Z$ of $U^{ev}_\e(\g)$ in case \ref{case over C}. 
Let $\varphi_\e$ denote the natural surjection  $\varphi_\e: U^{ev}_{\CA}(\g) \rightarrow U^{ev}_\e(\g)$ corresponding the algebra homomorphism $\CA \rightarrow  \BC$ sending $v$ to $\e$.  For any $a, b\in Z$, pick arbitrary lifts $\hat{a} , \hat{b} $ in $U^{ev}_{\CA}(\g)$, then define 
 \begin{equation}\label{eq: poisson bracket}\{ a, b\}:=\varphi_\e\left(\frac{[\hat{a}, \hat{b}]}{v-\e}\right), \qquad \text{for all $ a, b\in Z$}.
 \end{equation}
The following results are in \cite[$\mathsection 10$]{LTV}.
\begin{Prop} \label{prop: Frobenius center}(a) The Frobenius center $Z_{Fr}$ is closed under the Poisson bracket \eqref{eq: poisson bracket}. Moreover, $Z_{Fr}$ is generated by $\{ \tE_i^{\ell_i}, \tF_i^{\ell_i}, K^{\lambda}\}_{\lambda \in 2P^*}^{1 \leq i \leq r}$ as a Poisson algebra.

\noindent
(b) The symplective leaves of $\Spec Z_{Fr} \cong G^d_0$ coincide with the intersections of the conjugacy classes of $G^d$ with the open Bruhat cell $G^d_0$.

\noindent
(c) The fibers of $U^{ev}_\e(\g)$ over points in the same symplectic leaf are isomorphic as algebras.
\end{Prop}
The last part follows by general results about Poisson orders in \cite{BG03}. Let $Z_\cap:= Z_{Fr}\cap Z_{HC}$.
\begin{Lem}Under the isomorphism $\gamma_{-\rho}\circ \pi: Z_{HC}\iso \BC\< K^{2\lambda}\>_{\lambda \in P}^W$, the algebra $Z_\cap$ is identified with $\BC\< K^{2\lambda}\>_{\lambda\in P^*}^W$.
\end{Lem}
So the inclusion $Z_\cap \hookrightarrow Z_{HC}$ gives rise to the finite morphism $\bullet^l: T/W \rightarrow T^d/W$. On the other hand, the inclusion $Z_\cap \hookrightarrow Z_{Fr}$ corresponds to the composition $G^d_0 \hookrightarrow G^d \rightarrow G^d\sslash G^d \cong T^d/W$, the categorical quotient. So we can form the fiber product $G^d_0 \x_{T^d/W} T/W$. 

\begin{Prop}[Theorem 10.30 \cite{LTV}]\label{prop: description of Z} $Z\cong Z_{Fr}\otimes_{Z_\cap}Z_{HC}$ so that $\Spec Z\cong G^d_0 \x_{T^d/W} T/W$. 
\end{Prop}
 Let $G^{d, reg}$ be the set of regular elements in $G^d$. Let us consider the projection
 \[ \sA_1: \Spec Z \cong G^d_0 \x_{T^d/W} T/W \rightarrow G^d_0 \hookrightarrow G^d.\]
 \begin{Thm}[Theorem 10.40 \cite{LTV}]\label{thm: Azumaya locus} The Azumaya locus of $U^{ev}_\e(\g)$ over $Z$ contains $\sA_1^{-1}(G^{d, reg})$, i.e., the restriction of $U^{ev}_\e(\g)$ to the open subset $\sA_1^{-1}(G^{d, reg})$ is a sheaf of Azumaya algebras. The dimension of irreducible modules of $U_\e^{ev}(\g)$ over its Azumaya locus is $\mathsf{d}=\prod_{\a\in \Delta_+} \ell_\a$.

 \end{Thm}

\begin{Rem}The Poisson structure on $Z_{HC}$ is trivial and $Z_\cap$ is the Poisson center of $Z_{Fr}$.
\end{Rem}

\subsection{The locally finite parts $U^{fin}_q, U^{fin}_\e$}Let consider the following subalgebra of $U_q^{ev}(\g)$:
\begin{equation}
    U^{fin}_q:= \{ u \in U^{ev}_q(\g)|\ad'(\cU_q(\g))(u)~\text{is a finitely generated $R$-module}\}
\end{equation}
Let $O_q[G] \subset \Hom_R(\cU_q(\g), R)$ consist of all matrix coefficients of representations in the category $\Rep^{fd}(\cU_q(\g))$. We can equip $O_q[G]$ with  an algebra structure called the {\it reflection equation algebra}, see \cite[$\mathsection 8$]{LTV},  so that :
\begin{Prop}\label{prop: Ufin and other algberas}(a) The algebra $U^{ev}_q(\g)$ is obtained from $U_q^{fin}$ by localizing $\{K^{-\w_1}, \dots K^{-\w_r}\}$.

\noindent
(b) There is an $R$-linear isomorphism of $\cU_q(\g)$-module algebras $O_q[G]\iso U^{fin}_q$.
\end{Prop}

\begin{Rem}The isomorphism  is constructed based on the adjoint invariant pairing $U^{ev}_q(\g) \x \cU_q(\g) \rightarrow R$, see \cite[$\mathsection 9$]{LTV}. The structure of the $\cU_q(\g)$-module $U^{fin}_q$ can be understood through the $ \cU_q(\g)$-module $O_q[G]$, which can be studied via properties of the category $\Rep(\cU_q(\g))$.
\end{Rem}
\begin{Lem}[Proposition 7.32 \cite{LTV}] $U_q^{fin} \cong O_q[G]$ has a good filtration.
\end{Lem}
Now we focus on  $U^{fin}_\e$ in the case \ref{case over C}. Set
\[ Z^{fin}_{Fr} := U^{fin}_\e \cap Z_{Fr}\]
\begin{Prop}\label{prop: properties of Ufin}
(a) Under the isomorphism $Z_{Fr} \iso\BC[G^d_0]$, we have $Z^{fin}_{Fr} \cong \BC[G^d]$.

\noindent
(b) $U^{fin}_\e$ is a  finitely generated projective $Z^{fin}_{Fr}$-module. Moreover, $U^{ev}_\e(\g) \cong U^{fin}_\e \otimes_{Z^{fin}_{Fr}} Z_{Fr}$.

\noindent
(c) The center $Z^{fin}$ of $U^{fin}_\e$ is $Z^{fin}_{Fr} \otimes_{Z_\cap} Z_{HC}$.
\end{Prop}



\section{Completion}


This technical section discusses completions of several algebras of interest. In this section, the specialization $U_q^{ev}(\g)$ refers to the case \ref{case over C[[h]]} while the specialization $U_\e^{ev}(\g)$ refers to the case \ref{case over C}; the same convention applies to the other algebras.  We will set $\CW_q:=Z_{q, HC}, \CW_\e:= Z_{\e, HC}$, these are the Harish-Chandra centers of $U^{ev}_q(\g)$, $U_\e^{ev}(\g)$, respectively. 

\subsection{Completion}\ \label{ssec: completion}

\begin{Lem}The algebra $U^{ev}_q(\g)$ is Noetherian. Let $\CW_q^{\wedge}$ be a completion of $\CW_q$ with respect to some ideal,  then $U^{ev}_q(\g) \otimes_{\CW_q} \CW_q^{\wedge}$ is also Noetherian.
\end{Lem}
\begin{proof}$U^{ev}_q(\g)$ has a $\BZ^{2N+1}_{\geq 0}$-filtration whose associated graded algebra is a twisted polynomial algebra over the Noetherian ring $R$, \cite[Proposition 10.7]{LTV}. The latter is Noetherian hence so is $U^{ev}_q(\g)$.

Because of the surjective map $U^{ev}_q(\g) \otimes_R \CW_q^{\wedge}\rightarrow U^{ev}_q(\g) \otimes_{\CW_q} \CW^{\wedge}_q$, it suffices to show that $U^{ev}_q(\g)\otimes_R \CW^\wedge_q$ is Noetherian. Since $\CW_q$ is free over $R$,  $\CW_q^\wedge$ is flat over $R$. Tensoring $-\otimes_R \CW_q^\wedge$ with the above filtration of $U^{ev}_q(\g)$ gives us a filtration on $U^{ev}_q(\g) \otimes_R \CW_q^\wedge$ so that the associated graded algebra is a twisted polynomial algebra over $\CW_q^\wedge$, which is Noetherian.
\end{proof}
Consider the epimorphism  $\varphi_\e: U^{ev}_q(\g) \rightarrow U^{ev}_\e(\g)$. We now describe the procedure to produce various completions of $U^{ev}_q(\g)$. Suppose $I$ is an ideal of the center $Z$ of $U^{ev}_\e(\g)$. Let $J=\varphi^{-1}_\e(I)$ then $U^{ev}_q(\g)J=JU^{ev}_q(\g)=\varphi_\e^{-1}(U^{ev}_\e(\g) I)$ which implies $U^{ev}_q(\g)J^k=J^k U^{ev}_q(\g)=(U^{ev}_q(\g)J)^k$, in particular, $U^{ev}_q(\g)J^k$ is a two-sided ideal for any $k \geq 1$. Let $U^{ev\wedge_J}_q$ denote the completion of $U^{ev}_q(\g)$ with respect to the two-sided ideal $U^{ev}_q(\g)J$.

The next lemma summarizes some properties of $U^{ev\wedge_J}_q:=\varprojlim U_q^{ev}/(U_q^{ev}J)^k$, which are proved via arguments used to prove the Artin-Ree lemma in Commutative Algebra, see \cite[$\mathsection 2.4$]{IL11}.

\begin{Lem}\label{lem: completion of Uev}(a) $U^{ev\wedge_J}_q$ is a flat (left and right) $U^{ev}_q(\g)$-module.

\noindent
(b) $U^{ev\wedge_J}_q$ is Noetherian.

\noindent
(c) $U^{ev\wedge_J}_q$ is complete and separated in the $U^{ev}_qJ$-adic topology. In particular, $U^{ev\wedge_J}_q$  is complete and separated in the $\hbar$-topology.

\noindent
(d)  The completion functor $M\mapsto M^\wedge:= \varprojlim M/(U^{ev}_qJ)^k M$ from the category of finitely generated left $U^{ev}_q(\g)$-modules to the category of left $U^{ev\wedge_J}_q$-modules  is exact. Moreover, $M^\wedge$ is canonically isomorphic to $U^{ev\wedge_J}_q\otimes_{U^{ev}_q(\g)} M$.

\noindent
(e) We have a natural short exact sequence $0 \rightarrow U^{ev\wedge_J}_q \xrightarrow[]{\cdot \hbar} U^{ev\wedge_J}_q\rightarrow U^{ev\wedge_I}_\e\rightarrow 0$, here $U^{ev\wedge_I}_\e= U^{ev}_\e(\g) \otimes_Z Z^{\wedge_I}$ is the completion of $U^{ev}_\e(\g)$ at the two-sided ideal $U^{ev}_\e(\g)I$. So $U^{ev\wedge_J}_q$ is a $\BC[[\hbar]]$-flat deformation of $U^{ev\wedge_I}_\e$. 
\end{Lem}
\begin{proof} Part $(e)$ is obtained by applying part $(d)$ to the short exact sequence $0 \rightarrow  U_q^{ev}(\g) \xrightarrow[]{\cdot \hbar} U_q^{ev}(\g) \rightarrow U_q^{ev}(\g) \rightarrow 0$. The other part will be proved analogously to the similar statements for commutative algebras, e.g., see \cite[$\mathsection 7$]{Ei95}. To simplify the notation, let us denote  $U_q^{ev}:=U_q^{ev}(\g)$.

{\em Step 0: } If $I=0$  then $JU_q^{ev}=\hbar U_q^{ev}$. The blow up algebra $\text{Bl}_{(\hbar)}U_q^{ev}:= U_q^{ev}\oplus U_q^{ev}\hbar\oplus \dots =U_q^{ev}[\hbar]$ is Noetherian since $U_q^{ev}$ is Noetherian. Therefore, the completion $U_q^{ev\wedge_\hbar}$ of $U_q^{ev}$ with respect to the two-sided ideal $\hbar U_q^{ev}$ satisfies all properties of the lemma.

{\em Step 1:} Suppose $I \subset Z$ satisfies $\{ I,I\} \subset I$, then $[J,J] \subset \hbar J$. Since $\hbar U_q^{ev} \subset J U_q^{ev}$, by Step 0, the completion of $U_q^{ev\wedge_\hbar}$ with respect to the two-sided ideal $JU_q^{ev\wedge_\hbar}$ is the same as $U_q^{ev\wedge_J}$. So it is enough to show that  the blow up algebra $\text{Bl}_{(J)}(U_q^{ev\wedge_\hbar}):=U_q^{ev\wedge_\hbar}\oplus JU_q^{ev\wedge_\hbar} \oplus J^2 U_q^{ev\wedge_\hbar}\oplus \dots$ is Noetherian. Note that $(JU_q^{ev\wedge_\hbar})^k=J^kU_q^{ev\wedge_\hbar}=U_q^{ev\wedge_\hbar}J^k$ since $\hbar \in J$. 

Consider the epimorphism $\phi_\hbar: U_q^{ev\wedge_\hbar} \twoheadrightarrow U_\e^{ev}$. Let $A:=\phi_\hbar^{-1}(Z)$ then $A$ is complete and seperated in the $\hbar$-adic topology. Moreover, $[A,A]\subset \hbar A$ hence $A/\hbar A$ is commutative. Since $[J,J]\subset \hbar J$, we have that $[JA,JA]\subset \hbar JA$. Since $\hbar \in J$ then $(JA)^k=J^kA=AJ^k$, particularly, $JA$ is a two-sided ideal containing $\hbar$ in $A$.

Now $\text{Bl}_{(J)}(U_q^{ev\wedge_\hbar})=U_q^{ev\wedge_\hbar}\text{Bl}_{JA}(A)=\text{Bl}_{JA}(A)U_q^{ev\wedge_\hbar}$, here $\text{Bl}_{JA}(A):= A\oplus JA \oplus (JA)^2 \oplus \dots$. Since $U_q^{ev\wedge_\hbar}$ is finitely generated as both left and right $A$-modules, it suffices to show that $\text{Bl}_{JA}(A)$ is Noetherian. The latter  is proved by the following result in \cite[Lemma 2.4.2]{IL11}:
\begin{itemize}[label={}]
\item {\em Let $A$ be a $\BC[\hbar]$-algebra and $J$ be a two-sided ideal of $A$ containing $\hbar$. Suppose that $A$ is complete and separated in the $\hbar$-adic topology, the algebra $A/\hbar A$ is Noetherian commutative and $[J,J] \subset \hbar J$. Then the algebra $\text{Bl}_{J}(A)$ is Noetherian.}
\end{itemize}
To apply the quoted result, we are left to  show that $A/\hbar A$ is Noetherian.
\begin{equation}
\label{eq: SES of A/hA}
0\rightarrow \hbar U_q^{ev\wedge_\hbar}/\hbar A \rightarrow A/\hbar A \rightarrow Z\rightarrow 0.
\end{equation}
Since  $\hbar U_q^{ev\wedge_\hbar} \subset A$, the finitely generated left (or right) $A$-module $\hbar U_q^{ev\wedge_\hbar} /\hbar A$ is a finitely generated left (or right) $Z$-module, hence  is Noetherian. Hence, $A/\hbar A$ is Noetherian by \eqref{eq: SES of A/hA}.

{\em Step 2:} Consider an arbitrary  ideal $I \subset Z$.  Let $\mathcal{J}:=\phi^{-1}(I^2)$. Note that $J^k \supset \mathcal{J}^k \supset J^{2k}$ and $U_q^{ev}J^k=(U_q^{ev}J)^k, U_q^{ev}\mathcal{J}^k=(U_q^{ev}\mathcal{J})^k$ for any $k \geq 1$. Therefore, the $U_q^{ev}J$-adic filtration and the  $U_q^{ev} \mathcal{J}$-adic filtration of $U_q^{ev}$ are compatible with each other. Since $\{ I^2, I^2\}\subset I^2$, the completion $U_q^{ev\wedge_{\mathcal{J}}}$ satisfies all properties of the lemma, hence so does $U_q^{ev\wedge_J}$.
\end{proof}



\begin{Lem}\label{lem: module over completion}
Let $\widehat{J}:=U_q^{ev\wedge_J}J$. Any finitely generated module over $U_q^{ev\wedge_J}$ is complete and separated in the $\widehat{J}$-adic topology.
\end{Lem}
\begin{proof} By Lemma \ref{lem: completion of Uev}, $U_q^{ev\wedge_J}$ is complete and separated in the $\widehat{J}$-adic topology. The pair $(U_q^{ev\wedge_J}, \widehat{J})$ satisfies all properties of Lemma \ref{lem: completion of Uev} by the same proof. In particular, the properties in Lemma \ref{lem: completion of Uev}.d says that any finitely generated module over $U_q^{ev\wedge_J}$ is complete and separated in the $\widehat{J}$-adic topology.
\end{proof}
\begin{Lem} \label{lem: equiv action on completion}The adjoint action of $\cU_\e(\g)$ on $U_\e^{ev}(\g)$ extends uniquely to an action of $\cU_\e(\g)$ on $U_\e^{ev\wedge_I}$ by continuity. The adjoint action of $\cU_q(\g)$ on $U_q^{ev}(\g)$ extends uniquely to an action of $\cU_q(\g)$ on  $U_q^{ev\wedge_J}$ by continuity.
\end{Lem}
\begin{proof}The adjoint action of $\cU_\e(\g)$ on $Z$ factors through the map $\tFr: \cU_\e(\g) \rightarrow \cU_\BC(\g^d)$ and $U_\e^{ev}(\g)$ is a $\cU_\e(\g)$-module algebra. Then direct computations show
\begin{equation*}\label{eq: 1-action on I^k} x(U_\e^{ev}I) \subset U_\e^{ev}I \qquad \text{ for $x\in \{ K^\lambda, \tE_i, \tF_i\} \subset \cU_\e(\g)$}.
\end{equation*}
This implies that 
\begin{equation}\label{eq: 2-action on Ik} \begin{cases}
    x(U_\e^{ev}I^k) \subset U_\e^{ev}I^k  \qquad &\text{for $x\in \{K^\lambda, \tE_i, \tF_i\} \subset \cU_\e(\g)$}\\
    x(U_\e^{ev}I^k) \subset U_\e^{ev}I^{k-1} \qquad  &\text{for $x \in \{ \tE_i^{(\ell_i)}, \tF_i^{(\ell_i)}\}$} \subset \cU_e(\g).
    \end{cases}
\end{equation}
Since $\cU_\e(\g)$ is generated by $\{ K^\lambda, \tE_i, \tF_i, \tE_i^{(\ell_i)}, \tF_i^{(\ell_i)}\}$, the equation \eqref{eq: 2-action on Ik} implies the first statement of the lemma.

Let $\phi$ denote both epimorphisms $\cU_q(\g) \twoheadrightarrow \cU_\e(\g)$ and $U_q^{ev}\twoheadrightarrow U_\e^{ev}$. Since $\phi(xm)=\phi(x)\phi(m)$ for $x\in \cU_q(\g), m \in U_q^{ev}$, by \eqref{eq: 2-action on Ik} for $k=1,2$ we have 
\[ \begin{cases}  x(U_q^{ev}J) \subset U_q^{ev}J \qquad &\text{for $x\in \{K^\lambda, \tE_i, \tF_i\} \subset \cU_q(\g)$}\\
    x(U_q^{ev}J^2) \subset U_q^{ev}J\qquad  &\text{for $x \in \{ \tE_i^{(\ell_i)}, \tF_i^{(\ell_i)}\}$} \subset \cU_q(\g).
    \end{cases}\]
Using the fact that $U_q^{ev}$ is a $\cU_q(\g)$-module algebra, one can show  that 
\begin{equation}\label{eq: action of Jk} \begin{cases}x(U_q^{ev}J^k) \subset U_q^{ev}J^k  \qquad &\text{for $x\in \{K^\lambda, \tE_i, \tF_i\} \subset \cU_q(\g)$}\\
    x(U_q^{ev}J^k) \subset U_q^{ev}J^{k-1} \qquad  &\text{for $x \in \{ \tE_i^{(\ell_i)}, \tF_i^{(\ell_i)}\}$} \subset \cU_q(\g).
    \end{cases}
\end{equation}
Since $\cU_q(\g)$ is generated by $\{K^\lambda, \tE_i, \tF_i, \tE_i^{(\ell_i)}, \tF_i^{(\ell_i)}\}$, the equation \eqref{eq: action of Jk} implies the second statement of the lemma.
\end{proof}

\begin{Lem} \label{lem: Hom is fg over HC} For any representation $V\in \Rep^{fd}(\cU_\e(\g))$, the space $\Hom_{\cU_\e(\g)}(V, U^{fin}_\e)$ is a finitely generated module over $Z_\cap$ and hence is a finitely generated module over $\CW_\e$.
\end{Lem}
Let $Z^{fin}_{Fr}\Mod^{G_\e}$ be the category of finitely generated $Z^{fin}_{Fr}$-modules in $\Rep(\cU_\e(\g))$.

\begin{proof}Since $\Rep^{fd}(\cU_\e(\g))$ has enough projectives, it is enough to prove the lemma when $V$ is projective in $\Rep^{fd}(\cU_\e(\g))$. Furthermore, since $U_\e^{fin} \in Z^{fin}_{Fr}\Mod^{G_\e}$, there is a projective object  $V_1\in \Rep^{fd}(\cU_\e(\g))$ with a surjective morphism $V_1 \otimes_\BC Z^{fin}_{Fr} \rightarrow U^{fin}_\e$ in $Z^{fin}_{Fr}\Mod^{G_\e}$.
Since $V$ is projective in $\Rep(\cU_\e(\g)$, 
\begin{equation}\label{eq: morphism of Z-cap} \Hom_{\cU_\e(\g)}(V, V_1 \otimes_\BC Z^{fin}_{Fr}) \rightarrow \Hom_{\cU_\e(\g)}(V, U_\e^{fin}),
\end{equation}
is a surjective. The morphism \eqref{eq: morphism of Z-cap} is a morphism of $Z_\cap$-modules. On the other hand
\[ \Hom_{\cU_\e(\g)}(V, V_1 \otimes_\BC Z^{fin}_{Fr})=\Big((V^* \otimes V_1)^{\fu_\e} \otimes_\BC Z^{fin}_{Fr}\Big)^{\cU_\BC(\g^d)}.\]
By standard arguments in Invariant Theory, the right-hand side is a finitely generated module over $(Z^{fin}_{Fr})^{\cU_\BC(\g^d)}$, which is just $Z_\cap$. Therefore, by \eqref{eq: morphism of Z-cap}, $\Hom_{\cU_\e(\g)}(V, U_\e^{fin})$ is finitely generated over $Z_\cap$.
\end{proof}

\subsection{Completion of Poisson algebra $Z$}\

Recall the isomorphism  $\Spec Z \cong G^d_0 \x_{T^d/W} T/W$ from Proposition \ref{prop: description of Z} . Let $\chi$ be a regular point in $G^d_0$ and let $\xi=(\chi, \utheta)$ be a point in $\Spec Z$. Consider the following ideals in $Z$: 

\begin{itemize}
    \item $\m_\xi$, the maximal ideal of $Z$ at $\xi$.
    \item $Z \m_\chi$, here $\m_\chi$ is the maximal ideal of $Z_{Fr}$ at $\chi$.
\end{itemize}

The completion $Z_{Fr}^{\wedge_\chi}$ of $Z_{Fr}$ at $\chi$ is naturally a Poisson algebra. Let $V$ be the cotangent space at $\chi$ of the conjugacy class of $\chi \in \Spec Z_{Fr}\cong G^d_0$ , which is the symplectic leaf containing $\chi$, see Proposition \ref{prop: Frobenius center}. Then $V$ carries a natural non-degenerate skew-symmetric bilinear form. This bilinear form induces a Poisson structure on $\BC[[V]]$, the completion of symmetric tensor $S(V)$ at the maximal ideal generated by $V$. Let $\uchi$ be the image of $\chi$ under the natural map $\Spec Z_{Fr} \rightarrow \Spec Z_\cap$, then we form the completion $Z^{\wedge_{\uchi}}_\cap$.

\begin{Lem}\label{lem: Decompose Poisson center} There is an isomorphism of complete local Poisson algebras $Z^{\wedge_\chi}_{Fr} \cong \BC[[V]]\widehat{\otimes} Z^{\wedge_{\uchi}}_\cap$.    
\end{Lem}
 \begin{proof} 
Let us recall  the Steinberg slice $S$ in $G^d$, see \cite{St74}. We can move this slice by conjugation such that $S$ contains $\chi$. The natural map $S \rightarrow Z_\cap$ is an isomorphism of varieties. Since $S$ is the transversal slice of the conjugacy class $G^d\chi$ and all varieties $S, G^d\chi, G^d_0$ are smooth, the natural morphism $Z_{Fr}^{\wedge_\chi}\rightarrow \BC[G^d\chi]^{\wedge_\chi}\widehat{\otimes} \BC[S]^{\wedge_\chi}$ is an isomorphism of complete local algebras.

The natural morphism $Z_{Fr}^{\wedge_\chi} \twoheadrightarrow \BC[G^d\chi]^{\wedge_\chi}$ is a surjective morphism of Poisson algebras. Moreover, $\BC[[V]]\cong \BC[G^d\chi]^{\wedge_\chi}$ as complete local Poisson algebras. By \cite[Proposition 3.3]{Ka06}, there is an embedding of Poisson algebras $\BC[[V]]\rightarrow Z_{Fr}^{\wedge_\chi}$ such that the composition $\BC[[V]]\hookrightarrow Z_{Fr}^{\wedge_\chi} \twoheadrightarrow \BC[G^d\chi]^{\wedge_\chi}$ is an isomorphism of complete local Poisson algebras.

Since $Z_{Fr}$ and $Z_\cap$ are smooth, the natural morphism $Z_\cap^{\wedge_{\uchi}} \rightarrow Z_{Fr}^{\wedge_\chi}$ is injective. Furthermore, $Z_\cap^{\wedge_{\uchi}}$ is contained in the Poisson center of $Z_{Fr}^{\wedge_\chi}$. Hence , we have a morphism of Poisson algebras
\begin{equation}\label{eq: Decomposition of Poisson algeras}
\BC[[V]]\widehat{\otimes} Z_{\cap}^{\wedge_{\uchi}}\rightarrow Z_{Fr}^{\wedge_\chi}.
\end{equation}
The morphism \ref{eq: Decomposition of Poisson algeras} fits into the following composition
\[ \BC[[V]]\widehat{\otimes} Z_\cap^{\wedge_\chi} \rightarrow Z_{Fr}^{\wedge_\chi} \iso \BC[G^d\chi]^{\wedge_\chi}\widehat{\otimes} \BC[S]^{\wedge_\chi}.\]
Since both homomorphisms $Z_\cap^{\wedge_{\uchi}} \rightarrow \BC[S]^{\wedge_\chi}$ and $\BC[[V]]\rightarrow \BC[G^d\chi]^{\wedge_\chi}$ are isomorphisms of algebras, the homomorphism of Poisson algebras \eqref{eq: Decomposition of Poisson algeras} must also be an isomorphism.
\end{proof}

Let $Z^{\wedge_\xi}$ and $Z^{\wedge_\chi}$ be the completions of $Z$ at the ideals $\m_\xi$ and $Z \m_\chi$, respectively. Since $\Spec \CW_\e \rightarrow \Spec Z_\cap$ is a finite morphism,  there is a natural isomorphism of Poisson algebras:
\begin{equation} \label{eq: decompose Z} Z^{\wedge_\chi} \cong \prod_{\xi =(\chi, \utheta)}Z^{\wedge_\xi},
\end{equation}
note that there are only finitely many $\xi$ of the form $(\chi, \utheta)$. Let $\m_{\uchi}$ be the maximal ideal of $Z_\cap$ at $\uchi$, and  $\m_\utheta$ be the maximal ideal of $\CW_\e$ at the point $\utheta$. Let $\CW_\e^{\wedge_{\uchi}}$ and $\CW_\e^{\wedge_\utheta}$ be the completions of the  Harish-Chandra center $\CW_\e$ at the ideals $\CW_\e\m_{\uchi}$ and $\m_\utheta$, respectively. We have the following isomorphisms of algebras:
\begin{equation}
    Z^{\wedge_\chi} \cong Z_{Fr}^{\wedge_\chi}\widehat{\otimes}_{Z^{\wedge_{\uchi}}_\cap}\CW_\e^{\wedge_{\uchi}}, \qquad  Z^{\wedge_\xi}\cong Z_{Fr}^{\wedge_\chi}\widehat{\otimes}_{Z^{\wedge_{\uchi}}_\cap} \CW_\e^{\wedge_\utheta}.
\end{equation}
Lemma \ref{lem: Decompose Poisson center} implies the following corollary:
\begin{Cor} \label{cor: Decompose Poisson center} There is a decomposition of Poisson algebras $Z^{\wedge_\chi} \cong \BC[[V]]\widehat{\otimes}\CW_\e^{\wedge_{\uchi}}$. Fix one such decomposition, it gives  rise a family of inclusions of Poisson algebras $\BC[[V]]\hookrightarrow Z^{\wedge_\xi}$ and then a family of isomorphisms of Poisson algebras $Z^{\wedge_\xi} \cong \BC[[V]]\widehat{\otimes} \CW_\e^{\wedge_\utheta}$. 
\end{Cor}

\subsection{Structural results about $U^{ev\wedge_\chi}_q, U^{ev\wedge_\xi}_q$}\ \label{ssec: completion of Uev} 

 Let $U^{ev\wedge_\chi}_\e$ and $U^{ev\wedge_\xi}_\e$ be the completions of $U^{ev}_\e(\g)$ with respect to the ideals $Z\m_\chi$ and $\m_\xi$, respectively. By \eqref{eq: decompose Z}, we have an isomorphism:
\begin{equation}\label{eq: decompose Uev}U^{ev\wedge_\chi}_\e \cong \prod_{\xi=(\chi, \utheta)} U^{ev\wedge_\xi}_\e.
\end{equation}

Since  $\chi \in G^{d, reg}_0 \subset \Spec Z_{Fr}$ is regular, by Theorem \ref{thm: Azumaya locus}, there is an isomorphism of algebras $U^{ev\wedge_\xi}_\e \cong \Mat_{\sd}(Z^{\wedge_\xi})$.
Therefore, $U^{ev \wedge_\chi}_\e \cong \Mat_{\sd}(Z^{\wedge_\chi})$. Moreover, if we fix an isomorphism $U^{ev \wedge_\chi}_\e \cong \Mat_{\sd}(Z^{\wedge_\chi})$ then it induces a family  of algebra  isomorphisms $U^{ev \wedge_\xi}_\e \cong \Mat_{\sd}(Z^{\wedge_\xi})$.

Let $U_q^{ev\wedge_\chi}, U_q^{ev\wedge_\xi}$ be the completions of $U_q^{ev}(\g)$ with respect to the ideals $Z\m_\chi, \m_\xi$, constructed as in Section \ref{ssec: completion}. Since $Z\m_\chi \subset \m_\xi$, we have an algebra homomorphism $U_q^{ev\wedge_\chi} \rightarrow U_q^{ev\wedge_\xi}$, hence we have 
\[ U_q^{ev\wedge_\chi}\rightarrow \prod_{\xi=(\chi, \utheta)} U_q^{ev\wedge_\xi}.\]

 Consider the natural map $\psi: \CW_q \xrightarrow[]{/\hbar} \CW_\e$. Let us define
\begin{equation}\label{eq: complete of HC center}\fJ_\utheta:= \psi^{-1}(\m_\utheta)\qquad \text{and} \qquad \fJ_{\uchi}:= \psi^{-1}(\CW_\e \m_{\uchi}).
\end{equation}
Let $\CW_q^{\wedge_\utheta}$ and $\CW_q^{\wedge_{\uchi}}$ denote the completions of $\CW_q$ with respect to the ideals $\fJ_\utheta$ and $\fJ_{\uchi}$. Since $\CW_\e \m_{\uchi} \subset \m_\utheta$, we have an algebra homomorphism $\CW_q^{\wedge_{\uchi}} \rightarrow \CW_q^{\wedge_\utheta}$, hence we have 
\[ \CW_q^{\wedge_{\uchi}}\rightarrow \prod_{\xi=(\chi, \utheta)} \CW_q^{\wedge_\utheta}.\]

\begin{Lem}\label{lem: decomposition of Uev and W}These maps are isomorphisms:
\[ U^{ev\wedge_\chi}_q \iso \prod_{\xi=(\chi, \utheta)} U^{ev\wedge_\xi}_q, \qquad \CW_q^{\wedge_{\uchi}}\iso \prod_{\utheta}\CW_q^{\wedge_\utheta},\]
here $\utheta$ runs over all preimages of $\uchi$ under the map $\Spec \CW_\e  \rightarrow \Spec Z_\cap$.
\end{Lem}
\begin{proof} Both sides of the first map are complete in the $\hbar$-adic topology and flat over $\BC[[\hbar]]$ by Lemma \ref{lem: completion of Uev}. Morevover, the $\hbar$-quotient of the first map gives the isomorphism \eqref{eq: decompose Uev}, hence the first map is an isomorphism. The proof that the second map is an isomorphism  is similar.
\end{proof}

Let $T_\hbar(V)$ be the tensor power over $\BC[\hbar]$ generated by the cotangent space $V$ at $\chi$ of the symplectic leaf containing $\chi$. 
With the symplectic form on $V$,  we form the formal Weyl algebra \[\operatorname{Weyl}_\hbar[V]:= T_\hbar(V)/\< v_1\otimes v_2-v_2\otimes v_1-\hbar\{ v_1, v_2\}~|~v_1, v_2\in V\>,\]
here  $\< v_1\otimes v_2-v_2\otimes v_1-\hbar\{v_1, v_2\}~|~v_1, v_2\in V\>$ is the two-sided ideal in $T_\hbar (V)$. Let $\CA^\wedge_q$ be the completion of $\operatorname{Weyl}_\hbar[V]$ at the maximal ideal generated by $V$ and $\hbar$.
\begin{Prop}\label{prop: Decomposition of completions}Let $\chi \in G^{d, reg}_0 \subset \Spec Z_{Fr}$. There are isomorphisms:
\begin{align*}
 U^{ev \wedge_\xi}_q &\iso \Mat_{\sd}(\BC)\otimes_\BC (\CA^\wedge_q \widehat{\otimes}_{\BC[[\hbar]]} \CW_q^{\wedge_\utheta}),\\
    U^{ev\wedge_{\uchi}}_q & \iso \Mat_{\sd}(\BC)\otimes_\BC(\CA^\wedge_q\widehat{\otimes}_{\BC[[\hbar]]} \CW_q^{\wedge_{\uchi}}).
\end{align*}
\end{Prop}
\begin{proof} Let us prove the first isomorphism; the proof for the second isomorphism is similar.

Recall the epimorphism $\phi^\wedge: U_q^{ev\wedge_\xi} \xrightarrow[]{/\hbar} U_\e^{ev\wedge_\xi}$. We fix an isomorphism $U_\e^{ev\wedge_\xi} \iso \text{Mat}_{\sd}(Z^{\wedge_\xi})$. By Lemma \ref{lem: completion of Uev}, $U_q^{ev\wedge_\xi}$ is complete and separated in the $\hbar$-adic topology and flat over $\BC[[\hbar]]$. By lifting idempotents,  
\[U_q^{ev\wedge_\xi} \cong \text{Mat}_{\sd}(\BC) \otimes_\BC R_\hbar^\xi\]
so that $R_\hbar^\xi$ is a formal deformation of $Z^{\wedge_\xi}$, i.e.,  $R_\hbar^\xi$ is complete and separated in the $\hbar$-adic topology and flat over $\BC[[\hbar]]$. Furthermore, under the natural map $\phi_\hbar: R_\hbar^\xi  \xrightarrow[]{\hbar} Z^{\wedge_\xi}$, let $J:=\phi_\hbar^{-1}(\m_\xi)$ then $R_\hbar^\xi$ is complete in the $(R_\hbar J)$-adic topology by Lemma \ref{lem: completion of Uev}(c).

Fix a Poisson isomorphism $Z^{\wedge_\xi}\iso \BC[[V]]\widehat{\otimes} \CW_\e^{\wedge_\utheta}$ as in Corollary \ref{cor: Decompose Poisson center}. We can lift the embedding $\BC[[V]]\hookrightarrow Z^{\wedge_\xi}$ to an embedding $\CA_q^\wedge \hookrightarrow R_\hbar^\xi$, see \cite[Lemma 4.1]{IL15}. On the other hand, the embedding $\CW_q^{\wedge_\utheta} \hookrightarrow U_q^{ev\wedge_\xi}$  induces an embedding $\CW_q^{\wedge_\utheta} \hookrightarrow R_\hbar^\xi$ since $\CW_q^{\wedge_\utheta}$ is central in $U_q^{ev\wedge_\xi}$. Therefore, we have an algebra morphism $\CA_q^\wedge \widehat{\otimes}_{\BC[[\hbar]]} \CW_q^{\wedge_\utheta} \rightarrow R_\hbar^\xi$ which fits into the following diagram:
    \[ \begin{tikzcd}\CA_q^\wedge \widehat{\otimes}_{\BC[[\hbar]]} \CW_q^{\wedge_\utheta} \arrow[d, "/\hbar"] \arrow[r] & R_\hbar^\xi \arrow[d, "/\hbar"] &\\
    \BC[[V]]\widehat{\otimes} \CW_\e^{\wedge_\utheta} \arrow[r, "\cong"]& Z^{\wedge _\xi}
    \end{tikzcd} \]
The completed tensor product $\CA_q^\wedge\widehat{\otimes}_{\BC[[\hbar]]} \CW_q^{\wedge_\utheta}$ and $R_\hbar^\xi$ are complete and separated in the $\hbar$-adic topology, and $R_\hbar^{\xi}$ is flat over $\BC[[\hbar]]$, therefore, $\CA_q^\wedge \widehat{\otimes}_{\BC[[\hbar]]} \CW_q^{\wedge_\utheta} \rightarrow R_\hbar^\xi$ is an isomorphism.   
\end{proof}

\subsection{Other complete algebras}\


For $\utheta \in \Spec \CW_\e$ and $\uchi \in \Spec Z_\cap$, we defined the completions $\CW_q^{\wedge_{\uchi}}, \CW_q^{\wedge_\utheta}$ following \eqref{eq: complete of HC center}. Let us define the following algebras:
\begin{equation}\label{eq: partly complete algebras}
\begin{split}
    U^{fin, \uchi}_\e := U^{fin}_\e \otimes_{\CW_\e} \CW_\e^{\wedge_{\uchi}}, \qquad
    U^{fin, \uchi}_q := \big(U^{fin}_q \otimes_{\CW_q} \CW^{\wedge_{\uchi}}_q\big)/\big(\cap_{k=1}^\infty \hbar^k U^{fin}_q\otimes_{\CW_q} \CW_q^{\wedge_{\uchi}}\big),  \\
    U^{fin, \utheta}_\e:= U^{fin}_\e \otimes_{\CW_\e}  \CW_\e^{\wedge_\utheta} , \qquad  U^{fin, \utheta}_q:= \big( U^{fin}_q \otimes_{\CW_q} \CW_q^{\wedge_\utheta}\big)/\big(\cap_{k=1}^\infty \hbar^k U^{fin}_q\otimes_{\CW_q} \CW_q^{\wedge_\utheta}\big)
 \end{split}  
\end{equation}
\begin{Rem}$U^{fin, \uchi}_\e = U^{fin}_\e \otimes_{Z_\cap} Z^{\wedge_{\uchi}}_\cap$. Furthermore,
\begin{equation}\label{eq: decomposition of Ufin}  U^{fin, \uchi}_\e= \prod_{\utheta} U^{fin, \utheta}_\e \qquad \text{and} \qquad U^{fin, \uchi}_q= \prod_{\utheta} U^{fin, \utheta}_q,
\end{equation}
where $\utheta$ runs over the preimages of $\uchi$ under the map $\Spec \CW_\e \rightarrow \Spec Z_\cap$. 
\end{Rem}
\begin{Rem}  Since $U_q^{fin}$ is flat over $\BC[[\hbar]]$ and $\CW_q^{\wedge_?}$ is flat over $\CW_q$, it follows that $U_q^{fin}\otimes_{\CW_q} \CW_q^{\wedge_?}$ is flat over $\BC[[\hbar]]$. We expect $U^{fin}_q\otimes_{\CW_q} \CW_q^{\wedge_?}$ to be separated in the $\hbar$-adic topology but cannot prove it. The quotients in the definition of $U^{fin, \uchi}_q$ and $U^{fin ,\utheta}_q$ are needed to make sure  that these algebras are separated in the $\hbar$-adic topology. Then it is easy to show that  these two algebras are flat over $\BC[[\hbar]]$ since $U_q^{fin}\otimes_{\CW_q} \CW_q^{\wedge_?}$ is flat over $\BC[[\hbar]]$. Moreover, they are the maximal rational subrepresentations of the completions $U^{fin \wedge_{\uchi}}_q$ and $ U^{fin \wedge_\utheta}_q$, respectively, but we will not use this fact.
\end{Rem}


\section{Quantum category O}
 Recall the complete local algebra $\sR$ and its maximal ideal $\m$ in Section \ref{ssec: the algebra R}.  Recall the mixed form $U_v^{mix}$ in Definition \ref{def: A-forms} and its specializations $U_\e^{mix}, U_q^{mix}$ in the cases \ref{case over C} and \ref{case over C[[h]]}, respectively. We have the triangular decompositions
\[ U_q^{mix} \cong U_q^{ev<}\otimes_{\BC[[\hbar]]} \cU_q^0 \otimes_{\BC[[\hbar]]} \cU_q^>, \qquad U_\e^{mix}\cong U_\e^{ev<} \otimes_\BC \cU_\e^0 \otimes_{\BC} \cU_\e^>,\]
see Remark \ref{rem: triangular decomposition} for the definitions of $\cU^0_q, \cU^>_q$ and Lemma \ref{lem: PBW-basis for even part} for the definition of $U_q^{ev <}$.

As in \cite{IL23, Si24}, we introduce the (integral block) categories $O_q$.

\begin{defi}The category $O_q$ is the full subcategory of finitely generated left $U_q^{mix}$-modules consisting of all modules $M$ with an $R$-module decomposition $M=\oplus_\lambda M_\lambda$ satisfying
\begin{itemize}
    \item $\cU_q^0$ acts on $M_\lambda$ via $\chi_{\lambda}$ defined in \eqref{eq: character chi}.
    \item $M_\lambda$ is a finitely generated module over $R$.
    \item The weights are bounded from above: there is a finite collection of weights $\{ \lambda_i\}_{i \in I}$ so that $M_\lambda \neq \{0\}$ implies $\lambda \leq \lambda_i$ for some $i\in I$ under the dominance order on $P$.
\end{itemize}
\end{defi}
\begin{Rem}In the case \ref{case over C}, we denote $O_q$ by $O_\e$. The category $O_\e$ is embedded into $O_q $ via the pullback under the epimorphism $U^{mix}_q \twoheadrightarrow U^{mix}_\e$. 
\end{Rem}
\begin{Ex}Let $\BC[[\hbar]]_\lambda$ be $\BC[[\hbar]]$ with the  $\cU_q^\geq$-module structure defined  via $\cU_q^{\geq} \rightarrow \cU_q^0 \xrightarrow[]{\chi_\lambda} \BC[[\hbar]]$, then the  {\em Verma module} $\Delta_q(\lambda):=U_q^{mix} \otimes_{\cU_q^\geq} \BC[[\hbar]]_\lambda$ belongs to $O_q$. Similarly, the  {\em Verma module} $\Delta_\e(\lambda):=U_\e^{mix}\otimes_{\cU_\e^\geq} \BC_\lambda$ belongs to $O_\e$.
\end{Ex}
 
We have the embedding $\iota: \h^* \hookrightarrow \sR$ by sending $\nu \in \h^*$ to  the element $(\nu, \cdot)\in \h \subset \sR$. For any $\lambda \in P$, there is an algebra  homomorphism $\chi_{\lambda, \sR}: \cU_q^0\rightarrow \sR$ satisfying:
\begin{equation}\label{eq: deformed char}
 K^\nu \mapsto q^{(\lambda, \nu)}e^{2\pi \sqrt{-1}\iota(\nu)}, \qquad (\nu \in 2P).
\end{equation}
Since $q$ is not a root of unity in $\sR$, the image of the element $\binom{K_i;a}{m} \in \cU_q^0$ under $\chi_{\lambda, \sR}$ is determined by the images of $K^\nu$.
\begin{defi}A {\em deformed weight module} over $U_q^{mix}\otimes_{\BC[[\hbar]]} \sR$ is a $U_q^{mix}\otimes_{\BC[[\hbar]]} \sR$-module $M$ with an $\sR$-module decomposition $M=\oplus_\lambda M_\lambda$ so that $\cU_q^0$ acts on $M_\lambda$ via $\chi_{\lambda, \sR}$,  and $\tF_i M_\lambda \subset M_{\lambda-\a_i}, \tE_i^{(n)}M_\lambda \subset M_{\lambda+n \a_i}$ for $1\leq i \leq r, n\in \BZ_{\geq 0}, \lambda \in P$.  
\end{defi}

\begin{defi}The deformed category $O_{q, \sR}$ is the full subcategory of finitely generated left $U_q^{mix}\otimes_{\BC[[\hbar]]} \sR$-modules consisting of all deformed weight modules $M=\oplus_\lambda M_\lambda$ satisfying
\begin{itemize}
    \item $M_\lambda$ is a finitely generated module over $\sR$.
    \item The weights are bounded from above: there is a finite collection of weights $\{ \lambda_i\}_{i \in I}$ so that $M_\lambda \neq \{0\}$ implies $\lambda \leq \lambda_i$ for some $i\in I$ under the dominance order on $P$.
\end{itemize}   
\end{defi}

\begin{Ex}Let $\sR_\lambda$ be $\sR$ with the $\cU_q^\geq \otimes_{\BC[[\hbar]]} \sR$-module structure on which $\cU_q^\geq$ acts via $\cU_q^\geq \rightarrow \cU_q^0 \xrightarrow[]{\chi_{\lambda, \sR}}\sR$. Then the {\em Verma module} $\Delta_{q, \sR}(\lambda):= U_q^{mix}\otimes_{\cU_q^\geq} \sR_\lambda$ belongs to $O_{q, \sR}$.
\end{Ex}

Recall the map $
\tau: \h^* \rightarrow \Spec \CW_\e$ in \eqref{eq: tau map} with a choice of $\e=e^a$ for some $a\in \BC$.
\begin{defi}The equivalent relation $\overset{\tau}{\sim}$ on $P$ is defined  as $\lambda \overset{\tau}{\sim}\mu \Leftrightarrow \tau(\lambda)=\tau(\mu)$. Let $P/\tau$ denote the equivalent classes of $P$ under  $\overset{\tau}{\sim}$. Let $[\lambda]$ denote the class of $\lambda$ in $P/\tau$.
\end{defi}
\begin{Rem}\label{rem: Ptau under assumption on l}Under the assumption \ref{eq: assumption on l} on $\ell$ in the next section, we have $P/\tau =P/(W_{ext}, \bullet_\ell)$.
\end{Rem}
\begin{Lem}\label{lem: block decomposition} There are infinitesimal block  decompositions:
\[O_\e =\bigoplus_{[\lambda] \in P/\tau} O_\e^{[\lambda]}, \qquad O_q=\bigoplus_{[\lambda] \in P/\tau} O_q^{[\lambda]}, \qquad O_{q, \sR} =\bigoplus_{[\lambda] \in P/\tau}O_{q, \sR}^{[\lambda]},\]  
where $O_\e^{[\lambda]}$ is the full subcategory of $O_\e$ spanned by Verma modules $\Delta_{\e}(\nu)$ for all  $\nu \in [\lambda]\subset P/\tau$; the categories $O_q^{[\lambda]}, O_{q, \sR}^{[\lambda]}$ are similarly defined.
\end{Lem}
\begin{proof} We prove the third decomposition. Let $M\in O_{q, \sR}$. 

 Let us recall the Harish-Chandra center $\CW_q$. Let us consider the commutative algebra $\CW_q\otimes_{\BC[[\hbar]]} \sR$. The support of $\Delta_{q, \sR}$ in $\CW_q\otimes_{\BC[[\hbar]]} \sR$ contains the unique closed point $(\lambda, \m)$ corresponding to the morphism $\CW_q\otimes_{\BC[[\hbar]]} \sR\xrightarrow[]{/\m} \CW_\e\xrightarrow[]{\lambda}\BC$, see \eqref{eq: tau map}. Two closed points $(\lambda, \m)$ and $(\mu, \m)$ are the same if and only if $\tau(\lambda)=\tau(\mu)$. 

 On the other hand, since $M$ is finitely generated over $U_q^{mix}\otimes_{\BC[[\hbar]]} \sR$ and its weights are bounded from above, $M$ has a finite filtration whose successive quotients are quotient of Verma modules $\Delta_{q, \sR}(\lambda)$ for some $\lambda$. Therefore, $M=\oplus_\lambda M^{[\lambda]}$, where $M^{[\lambda]}$ is the component whose support in $\CW_q\otimes_{\BC[[\hbar]]} \sR$ contains the unique closed point $(\lambda, \m)$\footnote{to be more rigorious, one note that each weight space $M_\mu$ of $M$ is a finitely generated submodule over $\CW_q\otimes_{\BC[[\hbar]]} \sR$, then one decomposes $M_\lambda$ first then takes direct sum.}. This finishes the lemma.
\end{proof}

Let $\pr_{[\lambda]}: O_\e \rightarrow O_\e^{[\lambda]}$ be the natural projection. We use the same notation for the projections of the other two decompositions in Lemma \ref{lem: block decomposition}. Recall that each $\lambda \in P$ gives a point $\lambda \in \Spec \CW_\e$.

 \begin{Lem}\label{lem: extension action on cat O}Any object in $O_{q, \sR}^{[\lambda]}$ carries a natural action of $U^{ev}_q\otimes_{\CW_q} \CW_q^{\wedge_\lambda}$ by continuity. This gives  actions of the algebra $U_q^{fin, \lambda}$ defined in \eqref{eq: partly complete algebras} and  the algebra $U_q^{ev, \lambda}:= \big(U^{ev}_q\otimes_{\CW_q} \CW_q^{\wedge_\lambda}\big)/\big(\cap_{k=1}^\infty \hbar^k U_q^{ev}\otimes_{\CW_q} \CW_q^{\wedge_\lambda}\big)$ on it. 
 \end{Lem}
\begin{proof}{\it Step 1:} Let $\m_\lambda$ be the maximal ideal of $\lambda \in \Spec \CW_\e$, see \eqref{eq: tau map}.  Since any object in $O_\e^{[\lambda]}$ has a finite filtration whose subquotients are quotients of the Verma modules $\Delta_\e(\mu)$ for $\mu \in  [\lambda]$, any object in $O_\e^{[\lambda]}$ is killed by some power $\m_\lambda$.

{\it Step 2:} Let $M \in O_q^{[\lambda]}$.  Recall that $\m$ is the maximal ideal of $\sR$. Recall the epimorphism  $\phi: \CW_q\xrightarrow[]{/\hbar} \CW_\e$ and let $\fJ_\lambda:=\phi^{-1}(\m_\lambda)$. Each weight space $M_\mu$ of $M$ is stable under the $\CW_q$-action. Since $M/ M \m^k$ has a finite filtration whose subquotients are contained in $O_\e^{[0]}$, by Step $1$, there is $s_k>0$ such that $\fJ_\lambda^{s_k} M \subset M \m^k$. On the other hand, each weight space $M_\mu$ is a finitely generated module over $\sR$, therefore is complete and separated in the $\m$-adic topology. Hence, the action of $\CW_q$ on $M_\mu$ extends uniquely to an action of $\CW_q^{\wedge_\lambda}$ by continuity. This proves the first part of the lemma.

{\it Step 3:} Since each weight space $M_\mu$ is separated in the $\hbar$-adic topology,  the second part of the lemma follows.
\end{proof}

\begin{defi}\label{defi: Weyl(mu-lambda)}(a) For $\mu, \lambda \in P_+$, let $\nu$ be the unique dominant weight in the orbit $W(\mu-\lambda)$. Then we define Weyl module $W_q(\mu|\lambda):=W_q(\nu)$ and the indecomposable tilting module $T_q(\mu|\lambda):= T_q(\nu)$.

\noindent
(b) Let $W_\lambda$ be the stabilizer of $\lambda$ in $W_{*aff}$ under the dot action $\bullet$, see \eqref{eq: dot action W*}. 
\end{defi}
\begin{Lem}\label{lem: Verma under translation} Let $\lambda, \mu$ be in the closure of the fundamental alcove  $\overline{C}$ in \eqref{eq: fundamental alcove for W*} such that $W_\lambda \subset W_\mu$. Then  
\begin{gather*}\pr_{[\mu]}\big(W_q(\mu|\lambda)\otimes_{\BC[[\hbar]]} \Delta_\e(\lambda)\big)=\Delta_\e(\mu), \\
\pr_{[\mu]}\big(W_q(\mu|\lambda)\otimes_{\BC[[\hbar]}\Delta_q(\lambda)\big)=\Delta_q(\mu), \\
\pr_{[\mu]}\big(W_q(\mu|\lambda)\otimes_{\BC[[\hbar]]}\Delta_{q, \sR}(\lambda)\big)=\Delta_{q, \sR}(\mu).
\end{gather*}
The same equalities hold if we replace the Weyl module $W_q(\mu|\lambda)$ by the tilting module $T_q(\mu|\lambda)$.
\end{Lem}
\begin{proof}

Let us prove the third equality, the others are proved similarly. The object $W_q(\mu|\lambda)\otimes_{\BC[[\hbar]]} \Delta_{q, \sR}(\lambda)$ has a finite filtration whose subquotients are Verma modules $\Delta_{q, \sR}(\lambda+\nu)$, where $\nu$ runs over the weights of $W_q(\mu|\lambda)$, with the multiplicity equal to the weight multiplicity of $\nu$ in $W_q(\mu|\lambda)$.

By Lemma \ref{lem: block decomposition}, only subquotients $\Delta_{q, \sR}(\lambda+\nu)$ with $\tau(\lambda+\nu)=\tau(\mu)$ will survive after applying $\pr_{[\mu]}$ to $W_q(\mu|\lambda)\otimes_{\BC[[\hbar]]} \Delta_{q, \sR}(\lambda)$. Since $\lambda+\nu \in \mu+Q$, by Lemma \ref{lem: closed points on HC center}, $\tau(\lambda+\nu)=\tau(\mu)$ if and only if $\lambda+\nu \in W_{*aff}\bullet\mu$. Hence by Lemma \ref{lem: Key lemma on weights}, using the condition $W_\lambda \subset W_\mu$,  we conclude that only the subquotient $\Delta_{q, \sR}(\mu)$ will survive after applying $\pr_{[\mu]}$, i.e., the third equality holds.
\end{proof}

\begin{Rem}\label{rem: right hand version}We also consider the right module versions  $O^r_\e, O^r_q$ and $O^r_{q, \sR}$ with analogous results as in this section.
\end{Rem}


\section{Quantum Harish-Chandra bimodules} \label{sec: defi of  quantum HC}

We are going to define the main categories studied in this paper. To simplify the exposition, we will consider only  the cases \ref{case over C} or \ref{case over C[[h]]} from Section \ref{ssec: more on rational}. For any $\cU_q(\g)$-module $M$, let $M^{rat}$ be the maximal rational subrepresentation of $\cU_q(\g)$ in $M$. We also assume that 
\begin{enumerate}[label=(\Alph*)]
\setcounter{enumi}{2}
\item \label{eq: assumption on l} $\ell$ is an odd number, bigger than the Coxeter number of $\g$ and coprime to $\mathsf{e}:= |P/Q|$. Then we choose $\e$ with $\e^{1/\mathsf{e}}$ such that $\e^{\ell(\lambda, \mu)}=(\e^{1/\mathsf{e}})^{\mathsf{e}\ell(\lambda, \mu)}=1$ for all $\lambda, \mu \in P$.
\end{enumerate}
\begin{Rem}\label{rem: a choice of e}For example, we can choose $\e=e^{2\pi \sqrt{-1}\mathsf{e}/\ell}$ and $\e^{1/\mathsf{e}}=e^{2\pi \sqrt{-1}/\ell}$.
\end{Rem}
\begin{Rem}The assumption \ref{eq: assumption on l} on $\ell$ and the choice of $\e$ as in Remark \ref{rem: a choice of e} are to ensure that Lemma \ref{lem: left and right Zfin-actions coincide} holds for any quantum Harish-Chandra bimodule with the weights containing in the weight lattice $P$. The assumption is also used where the results about representations of small quantum groups are used in Section \ref{sec: simple HC}. Lemma \ref{lem: left and right Zfin-actions coincide} is crucial in order to construct the restriction functor $\bullet_\dag$ in Section \ref{ssec: restriction functors}. Otherwise, Lemma \ref{lem: left and right Zfin-actions coincide}, hence the construction of $\bullet_\dag$, only holds on quantum Harish-Chandra bimodules with the weights containing in a proper sublattice between $Q$ and  $P$. 
\end{Rem}
\subsection{Non-completed version}\label{ssec: noncomplete HC} Since $U^{ev}_q(\g)$ is a $\cU_q(\g)$-module algebra, we can define the following categories:
\begin{defi}Let $U^{ev}_q\Rmod^{\cU_q}$, $U^{ev}_q\Lmod^{\cU_q}$ and $U^{ev}_q\Bimod^{\cU_q}$ be the categories  of right $U^{ev}_q(\g)$-modules, left $U^{ev}_q(\g)$-modules and $U^{ev}_q(\g)$-bimodules in the category of $\cU_q(\g)$-modules, respectively. 
\end{defi}
\begin{Rem} \label{rem: left action on equiv Uev-mod} Recall $\iota: U^{ev}_q(\g) \rightarrow \cU_q(\g)$ in Remark \ref{rem: Uev to lusztig}. By Appendix \ref{append: equivariant modules}, any $M \in U^{ev}_q\Rmod^{\cU_q}$ is an object in $U_q^{ev}\Bimod^{\cU_q}$ with the left $U_q^{ev}(\g)$-module structure defined by 
\begin{equation}\label{eq: left action} hm =\sum (\iota(h_{(1)}) \cdot m) h_{(2)}, \qquad \text{for $h\in U_q^{ev}(\g), m \in M$}
\end{equation}
here $\cdot$ represents the action of $\cU_q(\g)$ on $M$. Similarly, any $N\in U^{ev}_q\Lmod^{\cU_q}$ is naturally an object in $U^{ev}_q\Bimod^{\cU_q}$ with the right $U^{ev}_q(\g)$-module structure  defined by:
\begin{equation}\label{eq: right action}  nh =\sum h_{(2)}(\iota(S^{-1}h_{(1)}) \cdot n), \qquad \text{for $h\in U_q^{ev}(\g), n \in N$.}
\end{equation}
\end{Rem}

Similarly, since $U^{fin}_q$ is an algebra object in $\Rep(\cU_q(\g))$,  we can define the following categories:
\begin{defi}Let $U^{fin}_q\Rmod^{G_q}$, $U^{fin}_q\Lmod^{G_q}$ and $U^{fin}_q\Bimod^{G_q}$ be the categories of right $U^{fin}_q$-modules, left $U^{fin}_q$-modules and $U^{fin}_q$-bimodules in the category $\Rep(\cU_q(\g))$, respectively.
\end{defi}
\begin{Ex}For any $V\in \Rep(\cU_q(\g))$, the object $V\otimes_R U^{fin}_q$ is naturally an object in $U^{fin}_q\Rmod^{G_q}$: the right $U^{fin}_q$-module structure comes from the right $U^{fin}_q$-action on $U^{fin}_q$ while $\cU_q(\g)$ acts on $V\otimes_R U^{fin}_q$ via tensor product. Similarly, $U^{fin}_q \otimes_R V$ is naturally an object in $U^{fin}_q\Lmod^{G_q}$.
\end{Ex}
\begin{Rem}If $V_q$ is projective in $ \Rep(\cU_q(\g))$ then $V_q\otimes_R U_q^{fin}$ is projective in $U_q^{fin}\Rmod^{G_q}$, meanwhile $U_q^{fin}\otimes_R V_q$ is projective in $U_q^{fin}\Lmod^{G_q}$.
\end{Rem}
\begin{Lem}\label{lem: left action on HCmod} There are fully faithful functors: 
\[ U^{fin}_q\Rmod^{G_q}\rightarrow U^{fin}_q\Bimod^{G_q}, \qquad U^{fin}_q\Lmod^{G_q}\rightarrow U^{fin}_q\Bimod^{G_q}.\]
\end{Lem}
\begin{proof}We will construct the first functor, the proof for the second functor is the same.

\noindent
{\it Step 1:} We define the left $U_q^{fin}$-action on $V\otimes_R U^{fin}_q$, where $V \in \Rep^{fd}(\cU_q(\g))$ is projective over $R$. First, we have 
\[ V\otimes_R U^{fin}_q =(V\otimes_R U^{ev}_q(\g))^{rat}.\]
Indeed, since $V$ is a finitely generated projective module over $R$, we have 
\begin{multline*} \Hom_{\cU_q(\g)}(V_1, V\otimes_R U^{ev}_q(\g)) \cong \Hom_R(V_1, V\otimes_R U^{ev}_q(\g))^{\cU_q(\g)}
\cong \Hom_R( ^*V \otimes_R V_1, U^{ev}_q(\g))^{\cU_q(\g)}\\
\cong \Hom_R(^*V\otimes_R V_1, U^{fin}_q)^{\cU_q(\g)}
\cong \Hom_{\cU_q(\g)}(V_1, V\otimes_R U^{fin}_q),
\end{multline*}
for all $V_1 \in \Rep(\cU_q(\g))$, see  Remark \ref{rem: dual module} for definition of $^*V$.

By Remark \ref{rem: left action on equiv Uev-mod}, there is a left $U^{ev}_q(\g)$-action on $V\otimes_R U^{ev}_q(\g)$ so that $V\otimes_R U^{ev}_q(\g) \in U^{ev}_q\Bimod^{\cU_q(\g)}$. The left $U^{fin}_q$-action on $V\otimes_R U^{ev}_q(\g)$ preserves$(V\otimes_R U^{ev}_q(\g))^{rat}$, hence we have a natural left $U^{fin}_q$-action on $V\otimes_R U^{fin}_q$ so that the object belongs to $U^{fin}_q\Bimod^{G_q}$.

\noindent
{\it Step 1.1:} For any $V_1, V_2\in \Rep^{fd}(\cU_q(\g))$ which are projective over $R$, the forgetting  map 
\[ \Hom_{U_q^{fin}\Bimod^{G_q}}(V_1\otimes_R U_q^{fin}, V_2\otimes_R U_q^{fin}) \rightarrow \Hom_{U_q^{fin}\Rmod^{G_q}}(V_1\otimes_R U_q^{fin}, V_2\otimes_R U_q^{fin})\]
 is bijective. This map is obviously injective. Let us show that it is surjective. Let $f: V_1\otimes_R U_q^{fin}\rightarrow V_2\otimes_R U_q^{fin}$ belong to the right hand side. We need to show that $f$ is $U_q^{fin}$-linear with respect to the left $U_q^{fin}$-actions on both objects. By tensoring $-\otimes_{U_q^{fin}}U_q^{ev}$, we have the map $f_{loc}: V_1\otimes_R U_q^{ev}\rightarrow V_2\otimes_R U_q^{ev}$ in the category $U_q^{ev}\Rmod^{\cU_q}$ so that we recover $f$ by taking the maximal rational $\cU_q$-subrepresentation parts. Since the left $U_q^{ev}$-actions on both objects of the map $f_{loc}$ are determined by the $\cU_q(\g)$-actions and the right $U_q^{ev}$-actions as in Remark \ref{rem: left action on equiv Uev-mod}, we have that $f_{loc}$ is $U_q^{ev}$-linear with respect to the left $U_q^{ev}$-actions on both objects. This implies that $f$ is $U_q^{fin}$-linear with respect to the left $U_q^{fin}$-actions on both objects.

\noindent
{\it Step 2:} By Step 1, we can equip $(\oplus_i V_i)\otimes_R U_q^{fin}$ with the left $U_q^{fin}$-module structure when $V_i\in \Rep^{fd}(\cU_q(\g))$ is projective over $R$ for all $i$, here we also consider infinite direct sums. Furthermore, the following map is bijective
\begin{multline}\Hom_{U_q^{fin}\Rmod^{G_q}}((\oplus_i V_i)\otimes_R U_q^{fin}, (\oplus V'_j) \otimes_R U_q^{fin}) \\
\rightarrow \Hom_{U_q^{fin}\Bimod^{G_q}}((\oplus_i V_i)\otimes_R U_q^{fin}, (\oplus V'_j)\otimes_R U_q^{fin}).
\end{multline}

\noindent
{\it Step 3:} The category $\Rep(\cU_q(\g))$ has enough projective objects and projective objects are projective over $R$ in either case \ref{case over C} or  case \ref{case over C[[h]]}. Hence for any $M\in U_q^{fin}\Rmod^{G_q}$, there are collections of projective objects $\{V_i\}, \{V'_j\}$ in $\Rep^{fd}(\cU_q(\g))$ and an exact sequence
\[ (\oplus V'_j)\otimes_R U_q^{fin}\xrightarrow[]{\phi} (\oplus V_i)\otimes_R U_q^{fin} \xrightarrow[]{\pi} M\rightarrow 0.\]

Then we define the left $U_q^{fin}$-action on $M$ by $zm=\pi(zm')$ for $m' \in (\oplus V_i)\otimes_R U_q^{fin}$ such that $\pi(m')=m$.
This action is well-defined. Indeed, if $\pi(m'_1)=\pi(m'_2)$ then $m'_1-m'_2=\phi(m'')$ for some $m'' \in (\oplus V'_j)\otimes_R U_q^{fin}$ then  
\[\pi(zm'_1)-\pi(zm'_2)=\pi(z(m'_1-m'_2))=\pi(z\phi(m''))=\pi\circ\phi(zm''))=0.\]

The left $U_q^{fin}$-action on $M$ does not depend on the choice of $(\oplus V_i)\otimes_R U_q^{fin} \xrightarrow[]{\pi}M$. Indeed, let $\{W_j\}$ be another set of projective objects in $\Rep^{fd}(\cU_q(\g))$ with a surjective map $(\oplus W_j)\otimes_R U_q^{fin}\xrightarrow[]{\pi'} M$. Since $(\oplus V_i)\otimes_R U_q^{fin}$ is projective in $U_q^{fin}\Rmod^{G_q}$, there is a morphism $g: (\oplus V_i)\otimes_R U_q^{fin}\rightarrow (\oplus W_j)\otimes_R U_q^{fin}$ such that $\pi=\pi'\circ g$. This implies the claim.

Now we show that for any $f\in \Hom_{U_q^{fin}\Rmod^{G_q}}(M,N)$ then $f\in \Hom_{U_q^{fin}\Bimod^{G_q}}(M,N)$. This will conclude that we have  a functor $U_q^{fin}\Rmod^{G_q} \rightarrow U_q^{fin}\Bimod^{G_q}$ which is fully faithful. Consider the following commutative diagram
\begin{equation*}
    \begin{tikzcd} (\oplus V_i)\otimes_R U_q^{fin}\arrow[d, two heads,  "\pi"] \arrow[r, "f' "]& (\oplus W_j)\otimes_R U_q^{fin}\arrow[d, two heads,  "\pi'"]&\\
    M\arrow[r, "f"] & N,
    \end{tikzcd}
\end{equation*}
where the map $f'$ exists because $(\oplus V_i)\otimes_R U_q^{fin}$ is projective in $U_q^{fin}\Rmod^{G_q}$. Now
\[ f(zm)= f\circ \pi(z m')=\pi'\circ f'(zm')=\pi' \circ (z f'(m'))=z(\pi'\circ f'(m'))=zf(m),\]
where $m'\in (\oplus V_i) \otimes_R U_q^{fin}$ such that $\pi(m')=m$.
\end{proof}
\begin{Lem}(a) For any $M\in U_q^{fin}\Rmod^{G_q}$, the natural map $M \rightarrow M\otimes_{U_q^{fin}} U_q^{ev}$ is a morphism of $\cU_q$-equivariant $U_q^{fin}$-bimodules.

\noindent
(b) If $M\in U^{fin}_q\Rmod^{G_q}$ is finitely generated as a right $U^{fin}_q$-module then $M$ is also finitely generated as a left $U^{fin}_q$-module. Similarly, if $M\in U^{fin}_q\Lmod^{G_q}$ is finitely generated as a left $U^{fin}_q$-module then $M$ is also finitely generated as a right $U^{fin}_q$-module.
\end{Lem}
\begin{proof} (a) It is obvious that the map is a morphism of $\cU_q$-equivariant right $U_q^{fin}$-modules. We need to show that the map is $U_q^{fin}$-linear with respect to the left $U_q^{fin}$-actions. This is true for $M=V\otimes_R U_q^{fin}$ where $V$ is a projective in $\Rep^{fd}(\cU_q(\g))$ since in this case the natural map becomes the inclusion $V\otimes_R U_q^{fin}\hookrightarrow V\otimes_R U_q^{ev}$. Then the statement for general $M$ follows since $M$ admits an exact sequence $(\oplus_j V'_j)\otimes_R U_q^{fin} \rightarrow (\oplus_i V_i) \otimes_R U_q^{fin} \rightarrow M$ for collections $\{V_i\}_{i\in I}, \{V'_j\}_{j\in J}$ of projectives in $\Rep^{fd}(\Rep(\cU_q(\g))$.

\noindent
(b) We will prove the first statement only since the proof for the second statement is the same. For any $V\in \Rep^{fd}(\cU_q(\g))$, we have that $V\otimes_R U^{fin}_q \in U_q^{fin} \Rmod^{G_q}$ and $U^{fin}_q\otimes_R V \in U_q^{fin}\Lmod^{G_q}$ are objects in $U^{fin}_q\Bimod^{G_q}$ by Lemma \ref{lem: left action on HCmod}. 

The morphisms of $\cU_q(\g)$-modules $V\rightarrow V\otimes_R U^{fin}_q, v\mapsto v\otimes 1$ and $ V\rightarrow U^{fin}_q\otimes_R V, v\mapsto 1\otimes v$ give rise to morphisms:
\begin{align*} p_1: U^{fin}_q \otimes_R V\rightarrow V\otimes_R U^{fin}_q \qquad \text{in $U_q^{fin}\Lmod^{G_q}$}\\
p_2: V\otimes_R U^{fin}_q \rightarrow U^{fin}_q\otimes_R V \qquad \text{in $U_q^{fin}\Rmod^{G_q}$},
\end{align*}
which are morphisms in $U_q^{fin}\Bimod^{G_q}$ by Lemma \ref{lem: left action on HCmod}. We will show that $p_1, p_2$ are mutually inverse.  Indeed, $p_2\circ p_1: U_q^{fin}\otimes_R V \rightarrow U_q^{fin}\otimes_R V$ maps $1\otimes v$ to $1\otimes v$, hence must be an isomorphism, similarly, $p_1\circ p_2$ is an isomorphism. So  $V\otimes_R U^{fin}_q \cong U^{fin}_q\otimes_R V$ in $U^{fin}_q \Bimod^{G_q}$.

If $M\in U^{fin}_q\Rmod^{G_q}$ is such that $M$ is finitely generated as a right $U^{fin}_q$-module, then there is  $V\in \Rep^{fd}(\cU_q(\g))$ with a surjective map $V\otimes_R U^{fin}_q \twoheadrightarrow M$ in $U^{fin}_q \Rmod^{G_q}$. By the above paragraph, $V\otimes_R U^{fin}_q$ is finitely generated as a left $U^{fin}_q$-module, hence $M$ is also finitely generated as a left $U^{fin}_q$-module.
\end{proof}

\begin{defi}The category of {\it quantum Harish-Chandra bimodules} is the full subcategory $U^{fin}_q\rmod^{G_q}$ of the category $U^{fin}_q\Rmod^{G_q}$ consisting of all objects which are finitely generated right modules over $U^{fin}_q$. We denote this category by $\HC_q$.
\end{defi}
\begin{Rem} It is not clear in the case \ref{case over C[[h]]} that  $U^{fin}_q$ is Noetherian so we are not sure if $\HC_q$ is an abelian category. Nevertheless, we will later be interested in some completed versions of $\HC_q$ which will be proved to be abelian categories.
\end{Rem}
\begin{Rem}Thank to Lemma \ref{lem: left action on HCmod} we have a monoidal structure on  the category $\HC_q$. 
\end{Rem}

Let us consider the case \ref{case over C} under the additional assumption \ref{eq: assumption on l}.
\begin{Lem}\label{lem: left and right Zfin-actions coincide} The left and right actions of $Z^{fin}_{Fr}$ on any object of $U^{fin}_\e\Rmod^{G_\e}$ coincide.\footnote{The same is true for weight modules in $U^{ev}_\e \rmod^{\cU_\e(\g)}$: the left and right $Z_{Fr}$-actions coincide. By the proof, the left and right actions of $Z^{fin}_{Fr}, K^\lambda (\lambda \in 2P^*)$ coincide and then we use that  $Z_{Fr}=Z^{fin}_{Fr}[K^{\lambda_0}]$ where $\lambda_0= 2\sum \ell_i \w_i$.}
\end{Lem}
\begin{proof}By \cite[Lemma 6.11(c)]{LTV}, $Z^{fin}_{Fr}=\bigoplus_{\lambda \in P^*_+} \ad'(\cU_\BC(\g^d)) K^{-2\lambda}$. Let $M \in  U^{fin}_\e \Rmod^{G_\e}$.  In what follows, let $\cdot$ denote the actions of $\cU_\e(\g)$.

{\it Step 1:} For any $m \in M$, by construction \eqref{eq: left action}
\[ K^{-2\lambda}m= (K^{-2\lambda} \cdot m) K^{-2\lambda} =m K^{-2\lambda},\]
here under the assumption \ref{eq: assumption on l} on $\ell$, for all $\lambda \in P^*_+$, the action of $K^{-2\lambda}$ on any rational representation in $\Rep(\cU_\e(\g))$ is trivial.

{\it Step 2:} Let $u \in Z^{fin}_{Fr}$ be such that $um=mu$ for all $m \in M$. We will show that 
\begin{equation} \label{eq: induction on equal Z-action} 
(\tE_i^{(\ell_i)}\cdot u ) m= m (\tE_i^{(\ell_i)} \cdot u), \qquad (\tF_i^{(\ell_i)}\cdot u) m =m (\tF_i^{(\ell_i)}\cdot u),
\end{equation}
for all $1\leq i \leq r$. Indeed, we have 
\begin{equation*}
    \begin{split}
 \tE_i^{(\ell_i)} \cdot (um)&=(\tE^{(\ell_i)} \cdot u)(K^{-\ell_i \zeta^>_i} \cdot m)+ u(\tE_i^{(\ell_i)} \cdot m)= (\tE^{(\ell_i)}_i \cdot u)m +u(\tE^{(\ell_i)} \cdot m)\\
 \tE_i^{(\ell_i)}\cdot (mu)&=(\tE_i^{(\ell_i)} \cdot m)(K^{-\ell_i \zeta_i^>}\cdot u) +m(\tE^{(\ell_i)}_i \cdot u)= ( \tE_i^{(\ell_i)} \cdot m) u + m(\tE^{(\ell_i)}_i \cdot u)
 \end{split}
 \end{equation*}
then the first equality of \eqref{eq: induction on equal Z-action} follows. The proof for the second equality of \eqref{eq: induction on equal Z-action} is the same.\\
The lemma follows by  using both steps and the equality $Z^{fin}_{Fr}=\bigoplus_{\lambda \in P^*_+} \ad'(\cU_\BC(\g^d)) K^{-2\lambda}$. 
\end{proof}
\subsection{Remarks about completions} To define and study completed  Harish-Chandra bimodules, we need some remarks about commutative algebras.

 Let $R$ be a commutative ring and $I$ be an ideal of $R$. Let $R^{\wedge_I}$ be the completion of $R$ with respect to the ideal $I$. Let $M$ be a $R$-module. We say that a $R^{\wedge_I}$-module structure on $M$ is an extension of the $R$-module structure if the latter can be recovered from the former via the natural map $R\rightarrow R^{\wedge_I}$. One can see that if $M$ is separated in the $I$-adic topology, i.e, $\cap_k I^k M=0$, then  the $R^{\wedge_I}$-module structure on $M$, if it exists, is uniquely recovered from the $R$-module structure by continuity in the $I$-adic topology. In general, extended $R^{\wedge_I}$-module structures on $M$ may not exist or exist but are not unique.

\begin{Lem}\label{lem fg quotient imply fg}(a) Let $R$ be a Noetherian ring. Let $I\subset R$ be an ideal. Let $M$ be an $R$-module. Suppose the following: (i) $R$ is complete in the $I$-adic topology, (ii) $M$ is separated in the $I$-adic topology, (iii) $M/IM$ is finitely generated over $R/I$. Then $M$ is finitely generated over $R$.

\noindent
(b) Let $R, S$ and $T$ be (not necessarily commutative) rings. Assume $R$ is left Noetherian. Let $M$ is a finitely generated left $R$-module and $N$ be an $(R,S)$-bimodule. Then for any $(S,T)$-bimodule $V$ such that  $V$ is flat as a left $S$-module, we have an isomorphism of right $T$-modules:
\[ \Hom_R(M,N\otimes_S V) \cong \Hom_R(M,N)\otimes_S V.\]
\end{Lem}
\begin{proof}
(a) There is $m_1, \dots, m_n\in M$ such that $M=IM+N$ where $N=Rm_1+\dots +Rm_n$. Since $R$ is complete in the $I$-adic topology hence so is $N$. Inductively, we have $M=I^k M+N$ for all $k \geq 0$. Since $M$ is seperated in the $I$-adic topology and $N$ is complete in the $I$-adic topology, we must have $M=N$.

\noindent
(b) Since $R$ is left Noetherian and $M$ is a finitely generated left $R$-module, there is  a presentation $R^{\oplus n}\rightarrow R^{\oplus m}\rightarrow M\rightarrow 0$. Then there are exact sequences
\[0\rightarrow \Hom_R(M,N \otimes_S V)\rightarrow \Hom_R(R^{\oplus n}, N\otimes_S V)\rightarrow \Hom_R(R^{\oplus m}, N\otimes_S V) \]
and 
\[ 0\rightarrow \Hom_R(M,N) \otimes_S V \rightarrow\Hom_R(R^{\oplus n}, N)\otimes_S V\rightarrow \Hom_R(R^{\oplus m}, N)\otimes_S V,\]
where we use the flatness of $V$. The second and the third terms in these exact sequence are identified, and we deduce an identification of the first terms.
\end{proof}

\subsection{Completed version}\ \label{ssec: complete HCbim}

Let $\chi\in \Spec Z_{Fr}$ and $\utheta, \utheta'\in \Spec \CW_\e$. Assume that the image of $\utheta$ and $\utheta'$ under the map $\Spec \CW_\e \rightarrow \Spec Z_\cap $ are $\uchi$ which is the image of $\chi$ under the map  $\Spec Z_{Fr} \rightarrow \Spec Z_\cap$. 

Let us consider the algebras $U_q^{fin, \uchi}, U_q^{fin, \utheta}, U_\e^{fin, \uchi}$ and $U_\e^{fin, \utheta}$ in \eqref{eq: partly complete algebras}. Let $U_\e^{fin, \utheta}\rmod^{G_\e}$ (resp., $U_\e^{fin, \utheta}\lmod^{G_\e})$ be the category of finitely generated $U_\e^{fin, \utheta}$-right  modules (resp., -left modules) in $\Rep(\cU_\e)$. Other categories for other  considered algebras are defined similarly. 

\begin{Lem}\label{lem: unique left Ue-uchi}For any $M\in U_\e^{fin, \uchi}\rmod^{G_\e}$, the left $U_\e^{fin}$-action extends uniquely to a left action of $U_\e^{fin, \uchi}$ by continuity in the left $\m_{\uchi}$-adic topology on $M$ so that $M\in U_\e^{fin, \uchi}\bimod^{G_\e}$.
\end{Lem}
\begin{proof}By Lemma \ref{lem: left and right Zfin-actions coincide}, $\m_{\uchi}M=M\m_{\uchi}$, hence we can define a left $Z_\cap^{\wedge_{\uchi}}$-action on $M$ to coincide with the right $Z_\cap^{\wedge_{\uchi}}$-action on $M$. So the left and right $Z^{fin}_{Fr}\otimes_{Z_\cap}Z_\cap^{\wedge_{\uchi}}$-actions on $M$ coincide. Since $U_\e^{fin}$ is a finitely generated module  over $Z^{fin}_{Fr}$,   $M$ is finitely generated over $Z^{fin}_{Fr} \otimes_{Z_\cap} Z_\cap^{\wedge_{\uchi}}$. This implies that $M\in U_\e^{fin, \uchi}\bimod^{G_\e}$. 

To show that the extended left $U_\e^{fin, \uchi}$-action is uniquely recovered by continuity in the left $\m_{\uchi}$-adic topology on $M$, it is enough to show that $\cap_k \m_{\uchi}^k M=0$. The left hand side is equal to $\cap_k M \m_{\uchi}^k$.  Let $V$ be any projective object in $\Rep^{fd}(\cU_\e)$, then
\[ \Hom_{\cU_\e}(V,  \cap_k M\m_{\uchi}^k)=\cap_k \Hom_{\cU_\e}(V, M \m_{\uchi}^k)=\cap_k \Hom_{\cU_\e}(V, M)\m_{\uchi}^k.\]
By Lemma \ref{lem: Hom is fg over HC}, $\Hom_{\cU_\e}(V,M)$ is a finitely generated module over $\CW_\e^{\wedge_{\uchi}}$, hence,  $\cap_k \Hom_{\cU_\e}(V, M)\m_{\uchi}^k=0$. So $\Hom_{\cU_\e}(V, \cap_k \m_{\uchi}^k M)=0$ for all projective modules $V$ in $ \Rep^{fd}(\cU_\e)$. Since $\Rep(\cU_\e)$ has enough projectives, it follows that $\cap_k \m_{\uchi}^k M=0$.  
\end{proof}

\begin{Lem}\label{lem: unique left utheta-action} Let $M\in U_\e^{fin, \utheta'}\rmod^{G_\e}$ so that the left $U_\e^{fin}$-action on $M$ is extended to a left $U_\e^{fin, \utheta}$-action. Then $M$ is separated in the left $\m_{\utheta}$-adic topology and the left $U_\e^{fin, \utheta}$-action is uniquely recovered by continuity in the left $\m_\utheta$-adic topology. Furthermore, $M \in U_\e^{fin, \utheta}\lmod^{G_\e}$.
\end{Lem}
\begin{proof} Under the projection $U_\e^{fin, \uchi} \twoheadrightarrow U_\e^{fin, \utheta}$, we have a left $U_\e^{fin, \uchi}$-action on $M$. By Lemma \ref{lem: unique left Ue-uchi}, such left $U_\e^{fin, \uchi}$-action is unique so that $M \in U_\e^{fin, \uchi}\lmod^{G_\e}$. Hence $M \in U_\e^{fin, \utheta}\lmod^{G_\e}$, in particular, $M$ is finitely generated over $U_\e^{fin, \utheta}$.  By Lemma \ref{lem: Hom is fg over HC}, for any  $V\in \Rep^{fd}(\cU_\e)$, we have that   $\Hom_{\cU_\e}(V, M)$ is a finitely generated left $\CW_\e^{\wedge_{\utheta}}$-module.
So for any projective $V$ in $\Rep^{fd}(\cU_\e)$, we have that  $\Hom_{\cU_\e}(V, \cap_k \m_\utheta^k M) =\cap_k \m_\utheta^k \Hom_{\cU_\e}(V,M)=0$. Since $\Rep(\cU_\e)$ has enough projectives, $\cap_k \m_\utheta^k M=0$, i.e., $M$ is separated in the left $\m_\utheta$-adic topology.     
\end{proof}
\begin{defi} Let $\HC_\e(\utheta, \utheta')$ denote the full subcategory of $U_\e^{fin, \utheta'}\rmod^{G_\e}$  such that the left $U^{fin}_\e$-action  on any object in $\HC_\e(\utheta, \utheta')$ extends to an action of $U^{fin, \utheta}_\e$.
\end{defi}
\begin{Rem} By decomposition \eqref{eq: decomposition of Ufin}, we have a  functor
\begin{equation}
    \pr_{\utheta, \utheta'}:  U^{fin, \utheta'}_\e \rmod^{G_\e} \rightarrow \HC_\e(\utheta, \utheta'),
\end{equation}
by projecting to the direct summand with the left $U^{fin, \utheta}_\e$-action. 
\end{Rem}

Recall $\fJ_{\uchi}=\psi_\e^{-1}(\CW_\e \m_{\uchi})$ and $\fJ_\utheta=\psi^{-1}_\e(\m_{\utheta})$ where $\psi_\e: \CW_q\rightarrow \CW_\e$. 


\begin{Lem}\label{lem: extension of left-actions} Let $M\in U^{fin, \uchi}_q\rmod^{G_q}$.

\noindent
(a) For any $V_q\in \Rep^{fd}(\cU_q(\g))$, the space $\Hom_{\cU_q(\g)}(V_q, M)$ is a finitely generated right  module over $\CW_q^{\wedge_{\uchi}}$. Here the right $\CW_q^{\wedge_{\uchi}}$-module structure comes from the right $\CW_q^{\wedge_{\uchi}}$-module structure on $M$, where we use that $\cU_q(\g)$ acts trivially on $\CW_q^{\wedge_{\uchi}}$.

\noindent
(b) The left and right $\fJ_{\uchi}$-adic topology on any object $M \in U_q^{fin, \uchi}\rmod^{G_q}$ coincides. Moreover,  $M$ is separated in the right (or left) $\fJ_{\uchi}$-adic topology, hence,  separated in the $\hbar$-adic topology.

\noindent
(c) The left action of $\CW_q$ on $M$ extends uniquely to a left action of $\CW_q^{\wedge_{\uchi}}$ by continuity in the $\fJ_{\uchi}$-adic topology. So $M$ is naturally an object in $U^{fin, \uchi}_q\bimod^{G_q}$.
\end{Lem}
We need a simple lemma
\begin{Lem}\label{lem: Im divided by h} Let $V$ be a projective object in $\Rep(\cU_q(\g))$ and $M\in \Rep(\cU_q(\g))$. Suppose $f\in \Hom_{\cU_q}(V,M)$ is such that $\text{Im} f \subset \hbar M$ then there is $f'\in \Hom_{\cU_q(\g)}(V, M)$ such that $f=\hbar f'$.
\end{Lem}
\begin{proof}Consider the following diagram
\[ \begin{tikzcd} & V \arrow[dl, dashed,  "f' "'] \arrow[d, "f"]&\\
  M \arrow[r, "\cdot \hbar"]& M &  
\end{tikzcd}
\]
Since $\text{Im} f \subset \hbar M$ and $V$ is projective in $\Rep(\cU_q(\g))$, there is $f'\in \Hom_{\cU_q}(V,M)$ which makes the above diagram commutative. This implies the lemma.
\end{proof}

\begin{proof}[Proof of Lemma \ref{lem: extension of left-actions}](a) Since $\Rep^{fd}(\cU_q(\g))$ has enough projectives, we can assume $V_q$ is projective in $\Rep^{fd}(\cU_q(\g))$ and $M=V'_q\otimes_{\BC[[\hbar]]} U^{fin, \uchi}_q$ for some $V'_q$ free over $\BC[[\hbar]]$ in $\Rep^{fd}(\cU_q(\g))$. 
\[ \Hom_{\cU_q(\g)}(V_q, V'_q\otimes_{\BC[[\hbar]]} U^{fin, \uchi}_q) =\Hom_{\cU_q(\g)}(^*V'_q \otimes_{\BC[[\hbar]]} V_q, U^{fin, \uchi}_q),\]
here $^*V'_q$ is the left dual of $V'_q$. Therefore, we reduce to proving that $\Hom_{\cU_q(\g)}(V_q, U^{fin, \uchi}_q)$ is finitely generated over $\CW_q^{\wedge_{\uchi}}$ for any $V_q\in \Rep^{fd}(\cU_q(\g))$.

We have a short exact sequence 
\[ 0\rightarrow \Hom_{\cU_q(\g)}(V_q, U_q^{fin, \uchi})\xrightarrow[]{\cdot \hbar} \Hom_{\cU_q(\g)}(V_q, U_q^{fin, \uchi})\rightarrow \Hom_{\cU_q(\g)}(V_q, U^{fin, \uchi}_\e).\]

Since $U_\e^{fin}$ is a $(\cU_\e, \CW_\e)$-bimodule and $\CW_\e^{\wedge_{\uchi}}$ is flat over $\CW_\e$, by Lemma \ref{lem fg quotient imply fg}.b, 
\[ \Hom_{\cU_q(\g)}(V_q, U^{fin, \uchi}_\e)=\Hom_{\cU_\e(\g)}(V_\e, U^{fin, \uchi}_\e) \cong \Hom_{\cU_\e(\g)}(V_\e, U^{fin}_\e)\otimes_{\CW_\e}\CW^{\wedge_{\uchi}}_\e,\]
here $V_\e:=V_q/\hbar V_q$. By Lemma \ref{lem: Hom is fg over HC}, $\Hom_{\cU_q(\g)}(V_q, U^{fin, \uchi}_\e)$ is finitely generated over $\CW_\e^{\wedge_{\uchi}}$ and
\[ \Hom_{\cU_q(\g)}(V_q, U^{fin, \uchi}_q)/\hbar \Hom_{\cU_q(\g)}(V_q, U^{fin, \uchi}_q)\]
is a finitely generated module  over $\CW_\e^{\wedge_{\uchi}}$. Moreover, $\CW_q^{\wedge_{\uchi}}$ is complete in the $\hbar$-adic topology and $\Hom_{\cU_q(\g)}(V_q, U^{fin, \uchi}_q)$ is separated in the $\hbar$-adic topology (since $U_q^{fin, \uchi}$ is separated). Therefore, $\Hom_{\cU_q(\g)}(V_q, U^{fin, \uchi}_q)$ is a finitely generated module over $\CW_q^{\wedge_{\uchi}}$ by Lemma \ref{lem fg quotient imply fg}.a.

\noindent
(b) Since $M/\hbar M \in U_\e^{fin, \uchi}\rmod^{G_\e}$ and $\hbar \in \fJ_{\uchi}$, it follows that  $\fJ_{\uchi}M=M \fJ_{\uchi}$. This implies the first part.  Let $V_q$ be a projective object in $\Rep^{fd}(\cU_q(\g))$. Then 
\[\Hom_{\cU_q(\g)}(V_q, \bigcap M(\fJ_{\uchi})^k ) \iso \bigcap  \Hom_{\cU_q(\g)}(V_q, M)(\fJ_{\uchi})^k=0,\]
where $\Hom_{\cU_q(\g)}(V_q, M)$ is a finitely generated module over $\CW_q^{\wedge_{\uchi}}$ by part a.

Since $\Hom_{\cU_q(\g)}(V_q, \cap M(\fJ_{\uchi})^k )=0$ for all projective objects $V_q$ in $\Rep^{fd}(\cU_q(\g))$ and the latter category has enough projectives, it follows that $\bigcap  M(\fJ_{\uchi})^k=0$.

\noindent
(c) 
Let $m \in M$. Let $V_q$ be a projective object in $\Rep^{fd}(\cU_q(\g))$ with a surjective map $V_q \twoheadrightarrow \cU_q(\g) m$. We note that $\Hom_{\cU_q(\g)}(V_q, M)$ is a module over $\CW_q\otimes \CW_q^{\wedge_{\uchi}}$ and the natural map
\[ V_q\otimes_{\BC[[\hbar]]} \Hom_{\cU_q(\g)}(V_q, M)\rightarrow M,\]
is a homomorphism of $\CW_q\otimes \CW_q^{\wedge_{\uchi}}$-modules with the image containing $m$. Let $M'$ denote the image of this homomorphism. Then $M'$ is a finitely generated right $\CW_q^{\wedge_{\uchi}}$-module, hence complete and seperated in the right $\fJ_{\uchi}$-adic topology.

On the other hand, the left and right $\fJ_{\uchi}$-adic topologies on $M'$ coincide, i.e., $\fJ_{\uchi}M'=M' \fJ_{\uchi}$. Indeed, since $V$ is projective in $\Rep(\cU_q(\g))$, it follows that 
\[ \fJ_{\uchi}\Hom_{\cU_q}(V,M)=\Hom_{\cU_q}(V, \fJ_{\uchi}M)=\Hom_{\cU_q}(V, M \fJ_{\uchi})=\Hom_{\cU_q}(V,M)\fJ_{\uchi}.\]
The left $\CW_q$-action on $M'$ can be uniquely extended to a left action of $\CW_q^{\wedge_{\uchi}}$ by continuity in the $\fJ_{\uchi}$-adic topology. Then we use the fact that $M$ is separated in the left $\fJ_{\uchi}$-adic topology to conclude that the left $\CW_q^{\wedge_{\uchi}}$-action on various $M'$ define a unique left $\CW_q^{\wedge_{\uchi}}$-action on $M$. 


So $M$ has a left $U^{fin}_q\otimes_{\CW_q} \CW_q^{\wedge_{\uchi}}$-action, which then factors through $U^{fin, \uchi}_q$ since $M$ is separated in the $\hbar$-adic topology.  By continuity in the left $\fJ_{\uchi}$-adic topology, $M$ is a rational $\cU_q$-equivariant left $U_q^{fin, \uchi}$-module.

To show that $M\in U_q^{fin, \uchi}\bimod^{G_q}$, it is left to prove that $M$ is a finitely generated left $U_q^{fin, \uchi}$-module. Note that $M/\hbar M \in U_e^{fin, \uchi}\lmod^{G_\e}$ by Lemma \ref{lem: unique left Ue-uchi}, hence $\Hom_{\cU_q}(V_q, M/\hbar M)$ is a finitely generated left $\CW_\e^{\wedge_{\uchi}}$-module for any $V_q\in \Rep^{fd}(\cU_q)$. Arguing as in part a),  $\Hom_{\cU_q}(V_q, M)$ is a finitely generated left $\CW_q^{\wedge_{\uchi}}$-module for any $V_q$ in $\Rep^{fd}(\cU_q)$. Now  we use Lemma \ref{lem: Criterion for HCq} for left $U_q^{fin, \uchi}$-modules to conclude that $M$ is a finitely generated left $U_q^{fin, \uchi}$-module. 
\end{proof}
\begin{Lem}Let $M\in U_q^{fin, \utheta'}\rmod^{G_q}$ so that the left $U_q^{fin}$-action on $M$ is extended to a left $U_q^{fin, \utheta}$-action. Then $M$ is separated in the left $\fJ_\utheta$-adic topology and the left $U_q^{fin, \utheta}$-action is uniquely recovered by continuity in the left $\fJ_\utheta$-adic topology. Furthermore, $M\in U_q^{fin, \utheta}\lmod^{G_q}$.
\end{Lem}
\begin{proof}The proof is similar to that of Lemma \ref{lem: unique left utheta-action} using the surjective map $U_q^{fin, \uchi}\twoheadrightarrow U_q^{fin, \utheta}$. 
\end{proof}
\begin{defi}\label{defi: HCq(.,.)}   Let $\HC_q(\utheta, \utheta')$ denote the full subcategory of $U_q^{fin, \utheta'}\rmod^{G_\e}$  such that the left $U^{fin}_q$-action  on any object in $\HC_q(\utheta, \utheta')$ extends to an action of $U^{fin, \utheta}_q$.
\end{defi}
\begin{Rem}
Thanks to  decomposition \eqref{eq: decomposition of Ufin}, we have a functor
\begin{equation}
    \pr_{\utheta, \utheta'}:  U^{fin, \utheta'}_q \rmod^{G_q} \rightarrow \HC_q(\utheta, \utheta'),
\end{equation}
by projecting  to the direct summand with the left $U^{fin, \utheta}_q$-action.
\end{Rem}
\begin{Rem}We can define $\HC_\e(\utheta, \utheta')$ to be the full subcategory of $U_\e^{fin, \utheta}\lmod^{G_\e}$ such that the right $U_\e^{fin}$-action on any object in $\HC_\e(\utheta, \utheta')$ extends uniquely to an action of $U_\e^{fin, \utheta'}$ . There is a similar statement for the category $\HC_q(\utheta, \utheta')$.
\end{Rem}
We see that 
\begin{equation}\label{eq: decomposition of cat}
    U^{fin, \uchi}_\e \rmod^{G_\e} \cong \prod_{(\utheta, \utheta')} \HC_\e(\utheta, \utheta'), \qquad U^{fin, \uchi}_q\rmod^{G_q} \cong \prod_{(\utheta, \utheta')} \HC_q(\utheta, \utheta'),
\end{equation}
 where $(\utheta, \utheta')$ runs over all pairs such that images of $\utheta, \utheta'$ under $\Spec \CW_\e \rightarrow \Spec Z_\cap$ are $\uchi$.

\begin{Lem} \label{lem: complete HCq is abelian} The categories $\HC_\e(\utheta, \utheta')$ and $\HC_q(\utheta, \utheta')$ are abelian.
\end{Lem}
\begin{proof}See Appendix \ref{ssec: proof of abelian}
\end{proof}
For any $\lambda \in P$, we also use $\lambda$ to denote a closed point in $ \Spec(\CW_\e)$ as follows:
\begin{equation}\label{eq: closed point of HC}
\CW_\e \subset \BC[K^{2\lambda}]_{\lambda \in P }\rightarrow \BC \qquad K^{2\mu}\mapsto \e^{(2\mu, \lambda)} \; \text{for $\mu \in P$}.
\end{equation}

\begin{defi}\label{defi: integral blocks} By {\it integral blocks} of quantum Harish-Chandra bimodules, we mean $\HC_\e(\lambda, \lambda')$ and $\HC_q(\lambda, \lambda')$ for $\lambda, \lambda'\in P$. 
\end{defi}
\begin{Rem}By the assumption \ref{eq: assumption on l} on $\ell$, the image of $\lambda \in \Spec\CW_\e$ under the map $\Spec \CW_\e \rightarrow \Spec Z_\cap$ is  the point $1\in T/W \cong \Spec Z_\cap$.
\end{Rem}

\begin{defi}\label{defi: diagonal bimod}(a)  For any $V_q\in \Rep(\cU_q(\g))$ which is free of finite rank over $\BC[[\hbar]]$, let $P^{\utheta, \utheta'}(V_q)$ be the direct summand of $V_q\otimes_R U_q^{fin, \utheta'}$ in $\HC_q(\utheta, \utheta')$. We call $P^{\utheta,\utheta'}(V_q)$, their direct sums and direct summands,  the {\it diagonal bimodules}.

\noindent
(b) The {\it HC-tilting bimodules} in $\HC_q(\utheta, \utheta')$ are direct summands of direct sums of objects of the form $P^{\utheta, \utheta'}(V_q)$ for some tilting module $V_q$ in $\Rep^{fd}(\cU_q(\g))$. Let $\Hilt_q(\utheta, \utheta')$ denote the full additive subcategories of HC-tilting modules.

\noindent
(c) For $\lambda, \mu $ in the closure of the fundamental alcove  $\overline{C}$ in \eqref{eq: closure of C} , let $P^{\mu, \lambda}_q:= P^{\mu, \lambda}(T_q(\mu|\lambda))$. These are the {\it translation bimodules}. Here $T_q(\mu|\lambda)$ is defined in Definition \ref{defi: tilting module}.

\end{defi}

\begin{Rem} $P^{\utheta, \utheta'}(V_q)$ is the direct summand of $U^{fin, \utheta}_q\otimes_R V_q$ in $\HC_q(\utheta, \utheta')$ as well as the direct summand of $V_q\otimes_R U^{fin, \uchi}_q$ in $\HC_q(\utheta, \utheta')$.
\end{Rem}


\begin{Rem}Recall the projection $\pr_{[\lambda]}: O_q \rightarrow O_q^{[\lambda]}$ after Lemma \ref{lem: block decomposition}. We have
\[\pr_{[\mu]}\Big(T_q(\mu|\lambda)\otimes_{\BC[[\hbar]]}\Delta_q(\lambda)\Big)\cong P_q^{\mu, \lambda}\otimes_{U_q^{fin, \lambda}} \Delta_q(\lambda)\]
because $P_q^{\mu, \lambda}$ is the direct summand of $T_q(\mu|\lambda)\otimes_{\BC[[\hbar]]} U_q^{fin, \lambda}$ in the category $\HC_q(\mu, \lambda)$. The same holds for $\Delta_\e(\lambda)$ and $ \Delta_{q, \sR}(\lambda)$.
\end{Rem}

\begin{Rem}We avoid to use {\it tilting bimodules} since $\HC_q(\utheta, \utheta')$ has no highest weight structure. Tensor products  of HC-tilting bimodules are HC-tilting.  The translation bimodules $P_q^{\mu, \lambda}$ are HC-tilting by definition.
\end{Rem}
\begin{Lem}\label{lem: proj is hilt}(a) For any $M \in \HC_q(\utheta, \utheta')$, we have 
\[ \Hom_{\HC_q(\utheta, \utheta')}(P^{\utheta, \utheta'}(V_q), M) \cong \Hom_{\Rep(\cU_q(\g))}(V_q, M).\]

\noindent
(b) $\HC_q(\utheta, \utheta')$ has enough projectives. Any projective object in $\HC_q(\utheta, \utheta')$ is a direct summand of $P^{\utheta, \utheta'}(V_q)$ for some projective object $V_q \in \Rep^{fd}(\cU_q(\g))$, hence is  HC-tilting.

There are analogous statements for $\HC_\e(\utheta, \utheta')$.
\end{Lem}
\begin{proof}(a) follows by \eqref{eq: decomposition of cat}. 

\noindent
(b)  Since $\Rep(\cU_q(\g))$ has enough projective, for any $M \in \HC_q(\utheta, \utheta')$, there is a projective object $V_q \in \Rep^{fd}(\cU_q(\g))$ with a surjective map $P^{\utheta, \utheta'}(V_q) \twoheadrightarrow M$. On the other hand, by (a), the object $P^{\utheta, \utheta'}(V_q)$ is projective in $\HC_q(\utheta, \utheta')$. Hence the first half of (b) follows. The second half follows since any projective object in $\Rep^{fd}(\cU_q(\g))$ is tilting by Proposition \ref{prop: proj is tilt}.
\end{proof}

\begin{Lem}[Adjointness]\label{lem: adjoint}(a) For any $V\in \Rep^{fd}(\cU_\e(\g))$ then
\[\Hom_{\HC_\e(\utheta, \utheta')}(P^{\utheta, \utheta_1}(V)\otimes_{U_\e^{fin, \utheta_1}} M, N) \cong \Hom_{\HC_\e(\utheta_1, \utheta')}(M, P^{\utheta_1, \utheta}(V^*)\otimes_{U_\e^{fin, \utheta}} N)\]
for $M\in \HC_\e(\utheta_1, \utheta'), N\in \HC_\e(\utheta, \utheta')$ and 
\[\Hom_{\HC_\e(\utheta, \utheta_1)}(M \otimes_{U_\e^{fin, \utheta'}} P^{\utheta', \utheta_1}(V), N')\cong \Hom_{\HC_\e(\utheta, \utheta')}(M, N'\otimes_{U_\e^{fin,\utheta_1}} P^{\utheta_1, \utheta'}(V^*))\]
for $M\in \HC_\e(\utheta, \utheta')$ and  $N'\in \HC_\e(\utheta, \utheta_1)$.

\noindent
(b) There are similar statements for $V_q$ free over $\BC[[\hbar]]$ in $ \Rep^{fd}(\cU_q(\g))$ and suitable $\HC_q(?,?)$.
\end{Lem}
\begin{proof}(a) We have
\begin{align*}
    \Hom_{\HC_\e(\utheta, \utheta')}(P^{\utheta, \utheta_1}(V)\otimes_{U_\e^{fin, \utheta_1}} M,N) & \cong \Hom_{U_\e^{fin, \utheta'}\rmod^{G_\e}}(V\otimes_\BC M, N)\\
    &\cong \Hom_{U_\e^{fin, \utheta'}\rmod^{G_\e}}(M,V^*\otimes_\BC N)\\
    &\cong \Hom_{\HC_\e(\utheta, \utheta')}(M, P^{\utheta_1, \utheta}(V^*)\otimes_{U_\e^{fin, \utheta}} N).
\end{align*}
Similarly,
\begin{align*}
    \Hom_{\HC_\e(\utheta, \utheta_1)}(M\otimes_{U_\e^{fin, \utheta'}} P^{\utheta', \utheta_1}(V), N') &\cong \Hom_{U_\e^{fin, \utheta}\lmod^{G_\e}}(M\otimes_\BC V, N')\\
    &\cong \Hom_{U_\e^{fin, \utheta}\lmod^{G_\e}}(M, N'\otimes_{\BC} V^*) \\
    &\cong \Hom_{\HC_\e(\utheta, \utheta')}(M, N'\otimes_{U_\e^{fin, \utheta_1}} P^{\utheta_1, \utheta'}(V^*)).
\end{align*}

\noindent
(b) This is proved in the same way.    
\end{proof}

\begin{Prop}\label{prop: embed of Hilt} The  following natural functor of triangulated categories   is an embedding:
\[K^b(\Hilt_\e(\utheta, \utheta'))\rightarrow D^b(\HC_\e(\utheta, \utheta')).\]
\end{Prop}
\begin{proof}Let $M,N\in \Hilt_\e(\utheta, \utheta')$, then it is enough to show that $\Ext^i_{\HC_\e(\utheta,\utheta')}(M,N)=0$ for all $i \geq 1$. We then only need to prove this  in the case when $M=P_\e^{\utheta, \utheta'}(V), N=P_\e^{\utheta, \utheta'}(V')$ for some tilting modules $V,V'\in \Rep^{fd}(\cU_\e(\g))$.

In this case, 
\begin{equation}\label{eq: Ext(Hilt)} \Ext^i_{\HC_\e(\utheta, \utheta')}(P_\e^{\utheta, \utheta'}(V), P_\e^{\utheta, \utheta'}(V'))\cong \Ext^i_{\Rep(\cU_\e(\g))}(V, P_\e^{\utheta, \utheta'}(V')).
\end{equation}
The latter is a direct summand of
\begin{align*} \Ext^i_{\Rep(\cU_\e(\g))}(V, V'\otimes_\BC U_\e^{fin}\otimes_{\CW_\e} \CW_\e^{\wedge_{\utheta'}})& \cong \Ext^i_{\Rep(\cU_\e(\g))}(V, V'\otimes_\BC U_\e^{fin}) \otimes_{\CW_\e} \CW_\e^{\wedge_{\utheta'}}\\
&\cong \Ext^i_{\Rep(\cU_\e(\g))}(^*V'\otimes_\BC V, U_\e^{fin})\otimes_{\CW_\e} \CW_\e^{\wedge_{\utheta'}}
\end{align*}
in which the first isomorphism follows, since $\CW_\e^{\wedge_{\utheta'}}$ is flat over $\CW_\e$.

By Proposition \ref{prop: Ufin and other algberas}, $U_\e^{fin}$ has a good filtration.  Since $V$ and $ V'$ are tilting modules in $\Rep^{fd}(\cU_\e(\g))$, the tensor product $^*V'\otimes_\BC V$ is also a tilting module. Therefore, 
\[\Ext^i_{\Rep(\cU_\e(\g))}(^*V'\otimes_{\BC} V, U_\e^{fin})=0 \qquad \qquad \text{for $i\geq 1$}.\]
Hence,   \eqref{eq: Ext(Hilt)}  is equal to zero. This finishes the proof.
\end{proof}

\begin{Rem}\label{rem: Krull-Schmidt} By \cite[Theorem 21.35]{Lam91}, both $\HC_q(\utheta, \utheta')$ and $\HC_\e(\utheta, \utheta')$ are Krull-Schmidt categories since Hom spaces in these categories are finitely generated modules over complete local rings. 
\end{Rem}

\begin{Lem}\label{lem: proj cover of simple}For any simple object $L\in \HC_q(\utheta, \utheta')$, there is an indecomposable projective object $P_L$ with a surjective map $P_L \xrightarrow[]{\pi} L$. Furthermore, if there is another projective object $P'_L$ with a surjective map $P'_L \xrightarrow[]{\pi'} L$ then $P_L \cong P'_L$. In the other  word, $P_L$ is a projective cover of $L$. The analogous claim holds for $\HC_\e(\utheta, \utheta')$.
\end{Lem}
\begin{proof} Since $\HC_q(\utheta, \utheta')$ has enough projective, there is an indecomposable projective object $P_L$ with a surjective morphism $P_L \xrightarrow[]{\pi} L$.

Assume $f\in \End_{\HC_q(\utheta, \utheta')}(P_L,P_L)$ is  such that $f\circ \pi=\pi$, we will show that $f$ is an isomorphism. Let $\fJ_{\utheta'}$ be the maximal ideal of $\CW_q^{\wedge_{\utheta'}}$. Since $P_L$ is projective,
\[ \End_{\HC_q(\utheta, \utheta')}(P_L)/\End_{\HC_q(\utheta, \utheta')}(P_L)\fJ_{\utheta'} \cong \End_{\HC_q(\utheta, \utheta')}(P_L/P_L \fJ_{\utheta'}).\]
Hence $\End_{\HC_q(\utheta, \utheta')}(P_L/P_L \fJ_{\utheta'})$ is finite dimensional over $\BC$.  Let $\underline{\pi}$ be the induced morphism $P_L/P_L\fJ_{\utheta'} \rightarrow L$. Let $\underline{f}$ be the induced morphism $P_L/P_L \fJ_{\utheta'} \rightarrow P_L/P_L\fJ_{\utheta'}$. Then  $\underline{f}\circ \underline{\pi} =\underline{\pi}$.

Since $P_L$ is indecomposable, $\End_{\HC_q(\utheta, \utheta')}(P_L)$ is a local ring, hence $\End_{\HC_q(\utheta, \utheta')}(P_L/P_L \fJ_{\utheta'})$  is also a local ring. Suppose $f$ is not an isomorphism, then $f$ belongs to the maximal ideal  of $\End_{\HC_q(\utheta,\utheta')}(P_L)$. Then the image $\underline{f}$ belongs to the maximal ideal of $\End_{\HC_q(\utheta, \utheta')}(P_L/P_L \fJ_{\utheta'})$. The latter ring is a finite dimensional local $\BC$-algebra. Therefore, $\underline{f}$ is a nilpotent operator. This is a contradiction since $\underline{f}\circ \underline{\pi}=\underline{\pi}$.

Now we have the following commutative diagram

\[ \begin{tikzcd}
    P'_L \arrow[r,"f_1"]\arrow[dr, "\pi'"'] & P_L\arrow[d, "\pi"'] \arrow[r, "f"] & P'_L \arrow[r, "f_2"]\arrow[dl, "\pi' "'] & P_L \arrow[dll, "\pi"] \\
    & L &&
\end{tikzcd}
\]
 By the above discussion, $f\circ f_1$ and $f_2 \circ f$ are isomorphisms, hence $f$ is an isomorphism, i.e, $P_L \cong P'_L$. This finishes the proof.
\end{proof}



\section{Poisson bimodules} \label{sec: Poisson bimodules}

This section follows the exposition in \cite[$\mathsection 3$]{IL12}.

\begin{defi}
Let $\CB$ be an associative $\BC[[\hbar]]$-algebra. By {\it the noncommutative Poisson structure} on $\CB$ we mean a pair $(\CB, \CP)$, where $\CP$ is a $\BC[[\hbar]]$-subalgebra of $\CB$ containing $\hbar \CB$, along with a $\BC[[\hbar]]$-bilinear map $\{ \;,\;\}: \CP\otimes \CB \rightarrow \CB$ such that $\CP$ is closed with respect to $\{\; , \;\}$ and 
\begin{enumerate}
    \item $\{ z,z\}=0$,
    \item $\{ \hbar a, b\}=[a,b]$, 
    \item $\{ z, ab\}=\{z,a\} b+a\{ z, b\}$,
    \item $\{ z_1z_2, a\}=\{ z_1, a\} z_2+z_1\{ z_2, a\}$,
    \item $\{ \{ z_1, z_2\}, a\}=\{ z_1, \{ z_2, a\}\}-\{ z_2, \{ z_1, a\}\}$,
\end{enumerate}
for all $z, z_1, z_2 \in \CP$ and $a, b\in \CB$.
\end{defi}

\begin{defi} Let $M$ be a $\CB$-bimodule such that the left and right actions of $\BC[[\hbar]]$ coincide. We say that $M$ is a Poisson $\CB$-bimodule if it is equipped with a $\BC[[\hbar]]$-bilinear map $\{\;,\;\}: \CP \otimes M \rightarrow M$ satisfying the following equalities:
\begin{itemize}
    \item $\{ \hbar a, m\}=[a,m]$,
    \item $\{ z, am\}=\{ z,a\} m+a \{ z, m\},$ $\{z, ma\}=\{ z,m\}a+m \{ z, a\}$,
    \item $\{ z_1 z_2, m\}=\{ z_1, m\} z_2+z_1\{ z_2, m\}$,
    \item $\{ \{ z_1, z_2\}, m\} =\{ z_1,\{ z_2, m\}\}-\{ z_2, \{ z_1, m\}\}$
\end{itemize}
for all $z, z_1, z_2\in \CP$ and $a\in \CB, m \in M$. Denote $\Pbim(\CB)$ the category of Poisson bimodules with respect to the pair $(\CB, \CP)$.
\end{defi}
\begin{Rem} \label{rem: mor between flat Pbim} By the condition  $(2)$, $\CP$ must satisfy that $[\CP, \CB]\subset \hbar \CB$. Moreover, if $\CB$ is flat over $\BC[[\hbar]]$, the Poisson bracket is uniquely recovered by $\{ z, a\}=\hbar^{-1}[z,a]$ if it exists. Similarly,  if $M$ is flat over $\BC[[\hbar]]$ then the Poisson bracket is uniquely recovered by $\{ z, m\}=\hbar^{-1}[z,m]$. Furthermore,  for $M, N \in \Pbim(\CB)$ such that $M,N$ are flat over $\BC[[\hbar]]$, the forgetful map 
\[ \Hom_{\Pbim(\CB)}(M,N) \rightarrow \Hom_{\CB \bimod}(M,N)\]
is an isomorphism.
\end{Rem}
\begin{Rem}Morphisms between Poisson bimodules $M_1, M_2$ are morphisms of bimodules $f: M_1\rightarrow M_2$ such that $f\{ z, m\}=\{ z, f(m)\}$ for any $z\in \CP$ and $m \in M_1$. The tensor product $M_1\otimes_\CB M_2$ of two Poisson bimodules is naturally a Poisson bimodule with the bracket defined by $\{ z, m\otimes n\}=\{ z,m\}\otimes n +m\otimes \{ z, n\}$. Then $\Pbim(\CB)$ is a monoidal category with unit object $\CB$.
\end{Rem}

Let $\chi \in G^{d, reg}_0 \subset \Spec Z_{Fr}$. Recall the following formal deformation in Proposition \ref{prop: Decomposition of completions}:
\[ \phi_\chi: U_q^{ev \wedge_\chi}\cong \Mat_{\sd}(R_\hbar)\twoheadrightarrow U_\e^{ev \wedge_\chi}\cong \Mat_{\sd}(Z^{\wedge_\chi}), \qquad \phi_\hbar: R_\hbar \rightarrow Z^{\wedge_\chi}.\]
here $R_\hbar \cong \CA^\wedge_q\widehat{\otimes}_{\BC[[\hbar]]} \CW_q^{\wedge_{\uchi}}$ is a formal deformation of $Z^{\wedge_\chi} \cong \BC[[V]]\widehat{\otimes}_{\BC} \CW_\e^{\wedge_{\uchi}}$. Recall that $\CA_q^\wedge$ is the completion of the formal Weyl algebra and $\CW_q^{\wedge_{\uchi}}$ is the completion of the Harish-Chandra center $\CW_q$ introduced in Section \ref{ssec: completion of Uev}.

Let us define
\[ P_\hbar:=\phi^{-1}_\chi(Z^{\wedge_\chi}_{Fr}), \qquad C_\hbar:=\phi^{-1}_\hbar(Z^{\wedge_\chi}_{Fr})\]
then $P_\hbar=C_\hbar+\hbar U_q^{ev \wedge_{\chi}}$, where $C_\hbar$ is embedded into $U_q^{ev\wedge_{\chi}} \cong \Mat_{\sd}(\BC)\otimes_{\BC} R_\hbar$ via $c \mapsto \text{Id} \otimes c$ for $c\in C_\hbar$ and $\text{Id}$ is the identity matrix in $\Mat_{\sd}(\BC)$.

Consider the map $\psi_\chi: \CW_q^{\wedge_{\uchi}} \xrightarrow[]{/\hbar}\CW_\e^{\wedge_{\uchi}}$ and let $B_\hbar:= \psi_\chi^{-1}(Z_\cap^{\wedge_{\uchi}})$. Then $C_\hbar=\CA^\wedge_q \widehat{\otimes}_{\BC[[\hbar]]} B_\hbar$. 
\begin{Lem} $(U_q^{ev\wedge_\chi}, P_\hbar)$, $(R_\hbar, C_\hbar)$, $(\CW_q^{\wedge_{\uchi}}, B_\hbar)$ are noncommutative Poisson structures on the corresponding $\BC[[\hbar]]$-algebras.
\end{Lem}
\begin{proof}These hold because $Z^{\wedge_\chi}_{Fr}$ is closed under the Poisson bracket on $Z^{\wedge_\chi}$ while the Poisson bracket on $Z^{\wedge_{\uchi}}_{\cap}$ is trivial. 
\end{proof}

\begin{defi} Let $ \Pbim(U_q^{ev\wedge_\chi})$ be the category of Poisson $U_q^{ev\wedge_\chi}$-bimodules which are finitely generated both as left and right modules over $U_q^{ev\wedge_\chi}$. The categories $ \Pbim(R_\hbar)$ and $ \Pbim(\CW_q^{\wedge_{\uchi}})$ are defined in  similar ways.
\end{defi}
\begin{Rem}$U_q^{ev\wedge_\chi}$, $R_\hbar$, and $\CW_q^{\wedge_{\uchi}}$ are Noetherian since they are formal deformations of $U_\e^{ev\wedge_\chi}, Z^{\wedge_\chi}, \CW_\e^{\wedge_{\uchi}}$, which are Noetherian. Therefore,  these three categories are abelian.
\end{Rem}

Let $e=E_{11}$ be the matrix unit in $U_q^{ev\wedge_\chi}\cong \Mat_{\sd}(R_\hbar)$ then $eU_q^{\wedge_\chi}e\cong R_\hbar$ and $eP_\hbar e\cong C_\hbar$. For $M\in \Pbim(U_q^{ev\wedge_\chi})$, the space $eMe$ is naturally a Poisson $R_\hbar$-bimodule with the Poisson structure: $\{ epe, eme\}=e\{ p,eme\}e$ for $p \in P_\hbar$ and $m \in M$. This bracket is well-defined. Indeed, if $epe=0$, then $p\in \hbar U_q^{ev\wedge_\chi}$ by the equality $P_\hbar=C_\hbar+\hbar U_q^{ev\wedge_\chi}$. Then $p=\hbar p'$ for some $p'\in U_q^{ev\wedge_\chi}$, hence, $ep'e=0$ since $U_q^{ev\wedge_\chi}$ is flat over $\BC[[\hbar]]$ and then 
\[ e\{p, eme\}e=e[p', eme]e=[ep'e, eme]=0.\]
This construction gives us an equivalence of monoidal abelian categories:
\begin{equation}\label{eq: equiv Poisson cat 1} \fP_1: \Pbim(U_q^{ev\wedge_\chi}) \iso \Pbim(R_\hbar).
\end{equation}

On the other hand, using the arguments in the proof of \cite[Proposition 3.3.1]{IL11}, for any $M\in\Pbim(R_\hbar)$, the space  $M^{ad}:= \{ m\in M~|~ \{z, m\}=0 ~\text{for $z\in \CA^\wedge_q$}\}$ is naturally an object of   $\Pbim(\CW_q^{\wedge_{\uchi}})$ and there is a  natural isomorphism $M \cong \CA^\wedge_q \widehat{\otimes}_{\BC[[\hbar]]} M^{ad}$ in $ \Pbim(R_\hbar)$.  Then the assignment $M \mapsto M^{ad}$ gives us an equivalence of monoidal abelian categories:
\begin{equation}\label{eq: equiv Poisson cat 2}\fP_2: \Pbim(R_\hbar) \iso \Pbim(\CW_q^{\wedge_{\uchi}}).
\end{equation}
\begin{Rem}\label{rem: Poisson on A x M} The map $\CA^\wedge_q\widehat{\otimes}_{\BC[[\hbar]]}M^{ad}$ is defined by $u\otimes m\mapsto um$ for $u\in \CA^\wedge_q, m\in M^{ad}$. The Poisson structure on $\CA^{\wedge}_q$  is given by $\{u,v\}=\hbar^{-1}[u,v]$ for $u, v\in \CA^\wedge_q$. Hence $\CA^{\wedge}_q$ is a Poisson $\CA^\wedge_q$-bimodule. Therefore, $\CA^\wedge_q\widehat{\otimes}_{\BC[[\hbar]]}M^{ad}$ belongs to $\Pbim(R_\hbar)$ via complete tensor product for any $M^{ad}\in \Pbim(\CW_q^{\wedge_{\uchi'}})$. Recall $R_\hbar \cong \CA^\wedge_q\widehat{\otimes}_{\BC[[\hbar]]} \CW_q^{\wedge_{\uchi}}$ and $C_\hbar=\CA^\wedge_q\widehat{\otimes}_{\BC[[\hbar]]} B_\hbar$.
\[ \{u\otimes x, v\otimes m\}= \{u,v\}\otimes xm+vu\otimes\{x, m\} \quad \text{ for $u,v \in \CA^\wedge_q$, $m\in M^{ad}$ and $x\in B_\hbar$}.\]
\end{Rem}
We recall the following decompositions in Lemma \ref{lem: decomposition of Uev and W}:
\begin{equation}\label{eq: decompositions}  U_q^{ev\wedge_\chi}\cong \prod_{\xi=(\chi, \utheta)} U_q^{ev \wedge_\xi}, \qquad \CW_q^{\wedge_{\uchi}}= \prod_{\utheta} \CW_q^{\wedge_\utheta},
\end{equation}
here $\utheta$ runs over preiamges of $\uchi$  under the map $\Spec \CW_\e\rightarrow \Spec Z_\cap$. These two decompositions come from the natural surjections $U_q^{ev\wedge_\chi}\twoheadrightarrow U_q^{ev \wedge_\xi}$ and $\CW_q^{\wedge_{\uchi}} \twoheadrightarrow \CW_q^{\wedge_\utheta}$.

Let $\xi=(\chi, \utheta)$ and $\xi'=(\chi, \utheta')$ then any $(U_q^{ev\wedge_\xi}, U_q^{ev\wedge_{\xi'}})$-bimodule can be viewed as a $U_q^{ev\wedge_\chi}$-bimodule, while any $(\CW_q^{\wedge_\utheta}, \CW_q^{\wedge_{\utheta'}})$-bimodule can be viewed as a $\CW_q^{\wedge_{\uchi}}$-bimodule.

\begin{defi} A $(U_q^{ev\wedge_\xi}, U_q^{ev\wedge_{\xi'}})$-bimodule is called Poisson if it is a Poisson $U_q^{ev\wedge_\chi}$-bimodule. A $(\CW_q^{\wedge_\utheta}, \CW_q^{\wedge_{\utheta'}})$-bimodule is called Poisson if it is a Poisson $\CW_q^{\wedge_{\uchi}}$-bimodule.
\end{defi}
\begin{defi} Let $\Pbim(U_q^{\xi, \xi'})$ denote the category of Poisson $(U_q^{ev\wedge_\xi}, U_q^{ev\wedge_{\xi'}})$-bimodules which are finitely generated as left and right modules. Let $\Pbim(\CW_q^{\utheta, \utheta'})$ denote the category of Poisson $(\CW_q^{\wedge_\utheta}, \CW_q^{\wedge_{\utheta'}})$-bimodules which are finitely generated as left and right modules.
\end{defi}

\begin{Rem}\label{rem: idempotent splits of functor}The algebras in the decompositions \eqref{eq: decompositions} come with complete systems of idempotents so that one maps to the other under homomorphism $\CW_q^{\wedge_{\uchi}} \hookrightarrow U_q^{ev \wedge_\chi}$. Furthermore, $\fP_2 \circ \fP_1$ is also compatible with these two systems of complete idempotents, so that $\fP_2\circ \fP_1$ gives a rise to a family of equivalences of abelian categories:
\begin{equation}\label{eq: functor fP}
    \fP:  \Pbim(U_q^{\xi, \xi'}) \iso \Pbim(\CW_q^{\utheta, \utheta'}).
\end{equation}
\end{Rem}
\subsection{The functor $\bullet^{\wedge_\chi}: \Pbim(U_q^{ev, \uchi}) \rightarrow \Pbim(U_q^{ev\wedge_\chi})$}\ \label{ssec: bullet-complete functor}

Let us consider the epimorphism $\phi_{1, \hbar}: U_q^{ev, \uchi} \xrightarrow[]{/\hbar} U_\e^{ev, \uchi}$. Let $P_{1, \hbar}:= \phi^{-1}_{1, \hbar}(Z_{Fr} \otimes_{Z_\cap} Z_{\cap}^{\wedge_{\uchi}})$. Then the pair $(U_q^{ev, \uchi}, P_{1,\hbar})$ is a noncommutative Poisson structure on $U_q^{ev, \uchi}$. So we have the category $\Pbim(U_q^{ev,\uchi})$ of finitely generated Poisson bimodules over $U_q^{ev, \uchi}$.

Let $J_\chi:=\phi^{-1}_{1, \hbar}(\m_\chi)$, where $\m_\chi$ is the maximal ideal of $Z_{Fr}$ at $\chi$.
\begin{Rem} \label{rem: completion of U(ev, chi)}The completion of finitely generated $U_q^{ev, \uchi}$-modules with respect to the two-sided ideal $U_q^{ev, \uchi}J_\chi$ satisifies all properties of Lemma \ref{lem: completion of Uev} by the same reasoning. Furthermore, the completion of $U_q^{ev, \uchi}$ with respect to the two-sided ideal $U_q^{ev, \uchi}J_\chi$ is isomorphic to $U_q^{ev\wedge_\chi}$.
\end{Rem}
\begin{Lem}\label{lem: bullet-complete functor}For any $N\in \Pbim(U_q^{ev, \uchi})$, the assignment $N \mapsto N^{\wedge_\chi}:=\varprojlim N/N(U_q^{ev, \uchi}J_\chi)^k$ defines the following functor
\begin{equation} \bullet^{\wedge_\chi}: \Pbim(U_q^{ev, \uchi}) \rightarrow \Pbim(U_q^{ev\wedge_\chi}).
\end{equation}
This functor  is exact, monoidal and $(\CW_q^{\wedge_{\uchi}}, \CW_q^{\wedge_{\uchi}})$-bilinear.
\end{Lem}
\begin{proof} By the axiom (2) in the definition of Poisson bimodules, for any $z\in P_{1, \hbar}$ and $n \in N$, we have $zn-nz \in \hbar N$. This implies that $N(U_q^{ev, \uchi}J_\chi)=(U_q^{ev, \uchi}J_\chi)N$. Therefore, $N^{\wedge_\chi}=\varprojlim N/N(J_\chi U_q^{ev, \uchi})^k =\varprojlim N/(U_q^{ev, \uchi}J_\chi)^k N$. By Remark \ref{rem: completion of U(ev, chi)}, $N^{\wedge_\chi}$ is finitely generated as a left and as a right $U_q^{ev\wedge_\chi}$-modules, furthermore, the functor $\bullet^{\wedge_\chi}$ is exact if it exists.

The Poisson bracket $\{\;,\;\}: P_{1, \hbar} \x N \rightarrow N$ has a unique continous extension to a Poisson bracket $\{\;,\;\}: \overline{P_{1, \hbar}} \x N^{\wedge_\chi} \rightarrow N^{\wedge_\chi}$, where $\overline{P_{1, \hbar}}$ is the closure of $P_{1, \hbar}$ in $U_q^{ev\wedge_\chi}$. Since $\hbar U_q^{ev, \uchi} \subset P_{1, \hbar}$, the closure $\overline{P_{1, \hbar}}$ contains $\hbar U_1^{ev\wedge_\chi}$. One can show that the image of $\overline{P_{1, \hbar}}$ under $\phi_{\hbar}: U_q^{ev\wedge_\chi}\rightarrow U_\e^{ev\wedge_\chi}$ is $Z_{Fr}^{\wedge_\chi}$. Therefore, $\overline{P_{1, \hbar}}=P_\hbar$. So $N^{\wedge_\chi}$ is naturally an object in $\Pbim(U_q^{ev\wedge_\chi})$. 

Since $N(U_q^{ev, \uchi}J_\chi)=(U_q^{ev, \uchi}J_\chi)N$, it is clear how to equip $\bullet^{\wedge_\chi}$ with a monoidal structure. The $(\CW_q^{\wedge_{\uchi}}, \CW_q^{\wedge_{\uchi}})$-bilinearity   is also clear.  
\end{proof}

 \begin{Rem} \label{rem: lamda-graded Poisson bimod}Let $\Lambda:= P/Q$. Then we have  the categories $\Pbim^\Lambda (A)$ of $\Lambda$-graded finitely generated Poisson bimodules over various algebras $A$ considered above and the functors between these categories. Here, the $\Lambda$-grading on $A$ is trivial.
\end{Rem}


\section{Restriction functor}

In this section, we will construct the key functors in this paper. We recall the categories $U_q^{fin, \uchi}\rmod^{G_q}$ and $ \HC_q(\utheta, \utheta')$ defined  in Section \ref{ssec: complete HCbim}. We  recall the categories $\Pbim^\Lambda(U_q^{ev, \uchi}),$ $  \Pbim^\Lambda(U_q^{ev\wedge_{\chi}}),$ $ \Pbim^\Lambda(W_q^{\wedge_{\uchi}})$ and $\Pbim^\Lambda(\CW_q^{\utheta, \utheta'})$ defined  in Section \ref{sec: Poisson bimodules}, see Remark \ref{rem: lamda-graded Poisson bimod}.

\subsection{The functor $\bullet_\dag: \HC_q(\utheta, \utheta') \rightarrow \Pbim^\Lambda(\CW_q^{\utheta, \utheta'})$}\
\label{ssec: restriction functors}

Let us define a functor:
\[ \fC: U_q^{fin, \uchi}\rmod^{G_q} \rightarrow  \Pbim^\Lambda(U_q^{ev \wedge_\chi}).\]
Let $M\in U_q^{fin, \uchi}\rmod^{G_q}$. This functor is a composition of the following two functors

\begin{itemize}
    \item Let $M_{loc}:=M\otimes_{U^{fin}_q}U_q^{ev}$. Then $M_{loc}$ will naturally be an object in $\Pbim^\Lambda(U_q^{ev, \uchi})$ so that we have a functor $\bullet_{loc}: U_q^{fin, \uchi}\rmod^{G_q} \rightarrow \Pbim^\Lambda(U_q^{ev, \uchi})$. See Lemma \ref{lem: bullet-loc functor} for the existence of this functor.
    \item The functor $\bullet^{\wedge_\chi}: \Pbim^\Lambda(U_q^{ev, \uchi})\rightarrow \Pbim^\Lambda(U_q^{ev\wedge_\chi})$ is constructed in Section \ref{ssec: bullet-complete functor}.
\end{itemize}
Then $\fC:= \bullet^{\wedge_\chi} \circ \bullet_{loc}$.

\begin{Lem}\label{lem: bullet-loc functor} The functor $\bullet_\loc$ exists. It is exact, monoidal and $(\CW_q^{\wedge_{\uchi}}, \CW_q^{\wedge_{\uchi}})$-bilinear.
\end{Lem}
\begin{proof}{\em Step 1: Existence.} Note that $U_q^{fin, \uchi}\otimes_{U_q^{fin}} U_q^{ev} \cong U_q^{ev,\uchi}$. As in Remark \ref{rem: left action on equiv Uev-mod},  $M_{loc}$ is an object in $U_q^{ev}\rmod^{\cU_q}$. Furthermore, $M\rightarrow M_{loc}$ is a morphism of $\cU_q$-equivariant $U_q^{fin}$-bimodules. Therefore, the left $U_q^{fin, \uchi}$-module structure on $M$ gives a left $U_q^{ev, \uchi}$-module structure on $M_{loc}$. The left and right actions of $Z_{Fr}\otimes_{Z_\cap}Z_\cap^{\wedge_{\uchi}}$ on $M_{loc}/\hbar M_{loc}$ coincide by Lemma \ref{lem: left and right Zfin-actions coincide}. This implies that for any $z\in P_{1, \hbar} $ and $m\in M_{loc}$, then $zm-mz \in \hbar M_{loc}$.

Assume $M$ is flat over $\BC[[\hbar]]$, then $M_{loc}$ is also flat over $\BC[[\hbar]]$. In this case, $M_{loc}$ has a natural Poisson bimodule structure over $U_q^{ev, \uchi}$ by defining $\{z, m\}=\hbar^{-1}[z,m]$ for $ z\in P_{1, \hbar}, m\in M_{loc}$. The $P$-grading on $M_{loc}$ gives a $\Lambda$-grading on $M_{loc}$. Then $M_{loc}$ is an object in $\Pbim^\Lambda(U_q^{ev, \uchi})$. Furthermore, for any morphism $M\rightarrow N$ of $\BC[[\hbar]]$-flat objects in $U_q^{fin, \uchi}\rmod^{G_q}$, the morphism $M_{loc}\rightarrow N_{loc}$ is a morphism in the category $\Pbim^\Lambda(U_q^{ev,\uchi})$.

For general $M$, there are $\BC[[\hbar]]$-flat objects $P_1, P_2\in U_q^{fin, \uchi}\rmod^{G_q}$ and an exact sequence $P_2\xrightarrow[]{\phi} P_1\xrightarrow[]{\pi} M \rightarrow 0$, e.g, we can choose a  projective presentation of $M$. This gives us an exact sequence $P_{2, loc} \xrightarrow[]{\phi} P_{1, loc} \xrightarrow[]{\pi} M_{loc} \rightarrow 0$. This ensures the following Poisson structure on $M_{loc}$ is well-defined: for $m\in M_{loc}$ and $z\in P_{1, \hbar}$ then $\{z, m\}=\pi(\hbar^{-1}[z,p])$ where $\pi(p)=m$.
Indeed, if $\pi(p_1)=\pi(p_2)$ then $p_1-p_2=\phi(p')$ and then  $\pi(\{z, m_1-m_2\})=\pi(\{ z, \phi(p')\})=\pi \circ \phi(\{z, p'\})=0$. Now the $P$-grading on $M_{loc}$ gives the $\Lambda$-grading on $M_{loc}$. So $M_{loc}$ is an object in $\Pbim^\Lambda(U_q^{ev, \uchi})$.

The Poisson structure on $M_{loc}$ is independent on the choice of $P_1 \rightarrow M\rightarrow 0$. Indeed, let us consider other morphism $P_2\rightarrow M \rightarrow 0$ and  the fibered product $P_1 \x_{M} P_2$. Then both Poisson structures on $M_{loc}$ coming from $P_1\rightarrow M$ and $P_2\rightarrow M$ will coincide with the Poisson structure on $M_{loc}$ coming from $P_1 \x_M P_2\rightarrow M$.

Now we show that for any $f: M \rightarrow N$ then the map $f_{loc}: M_{loc}\rightarrow N_{loc}$ is a morphism in $\Pbim^\Lambda(U_q^{ev, \uchi})$. There are projective objects $P_1, P_2$ in $U_q^{fin, \uchi}\rmod^{G_q}$ and the following diagram 
\[ \begin{tikzcd}P_{1, loc} \arrow[d, "\pi_1"] \arrow[r, "f'"]& P_{2, loc} \arrow[d, "\pi_2"]&\\
M_{loc} \arrow[r, "f_{loc}"] &N_{loc}   
\end{tikzcd}
\]
Then $f_{loc}(\{z, m\})==f_{loc}\circ \pi_1(\{z, p\})=\pi_2 \circ f'(\{z, p\})=\{z, \pi_2\circ f'(p)\}=\{ z, f_{loc} \circ \pi_1(p)\}=\{z, f_{loc}(m)\}$ for $p \in P_{1, loc}$ such that $\pi_1(p)=m$.

{\em Step 2: Properties.} Since $U_q^{ev}$ is a localization of $U_q^{fin}$ by Proposition \ref{prop: Ufin and other algberas}, the functor $\bullet_{loc}$ is exact.
It is clear that $\bullet_{loc}$ is $(\CW_q^{\wedge_{\uchi}}, \CW_q^{\wedge_{\uchi}})$-bilinear. The functor $\bullet_{loc}$ is monoidal as follows: 
\begin{equation}\label{eq: iso-monoid}
\begin{split}M_{loc}\otimes_{U_q^{ev, \uchi}} N_{loc} &\cong (M\otimes_{U_q^{fin}}U_q^{ev}) \otimes_{U_q^{ev, \uchi}}(N\otimes_{U_q^{fin}} U_q^{ev})\\
&\cong (M\otimes_{U_q^{fin, \uchi}}U_q^{ev, \uchi})\otimes_{U_q^{ev, \uchi}} (N\otimes_{U_q^{fin}} U_q^{ev}) \\
&\cong M\otimes_{U_q^{fin, \uchi}} N \otimes_{U_q^{fin}} U_q^{ev} =(M\otimes_{U_q^{fin, \uchi}} N)_{loc}.
\end{split}
\end{equation}
It is left to show that \eqref{eq: iso-monoid} is an isomorphism of Poisson bimodules in $\Pbim^\Lambda(U_q^{ev, \uchi})$. It is done by choosing two surjective morphisms $P_1\twoheadrightarrow M, P_2\twoheadrightarrow N$  from $\BC[[\hbar]]$-flat objects $P_1, P_2$ in $U_q^{fin, \uchi}\rmod^{G_q}$ and arguing as in  the last paragraph of Step 1.   
\end{proof}

Composing with the equivalence $\fP: \Pbim^\Lambda(U_q^{ev\wedge_\chi}) \iso \Pbim^\Lambda( \CW_q^{\wedge_{\uchi}})$ in \eqref{eq: functor fP}, we obtain 
\begin{equation}\label{eq: res functor for chi}
    \bullet_\dag: U_q^{fin, \uchi}\rmod^{G_q} \xrightarrow[]{\fC} \Pbim^\Lambda(U_q^{ev\wedge_\chi}) \xrightarrow[]{\fP} \Pbim^\Lambda(\CW_q^{\wedge_{\uchi}}).
\end{equation}

The decompositions of the categories give us the functor $\bullet_\dag: \HC_q(\utheta, \utheta') \rightarrow \Pbim^\Lambda(\CW_q^{\utheta, \utheta'})$.

\begin{Prop}The functor $\bullet_\dag$ in \eqref{eq: res functor for chi} is exact, monoidal and $(\CW_q^{\wedge_{\uchi}}, \CW_q^{\wedge_{\uchi}})$-linear.
\end{Prop}
\begin{proof}Follows by Lemma \ref{lem: bullet-loc functor} and Lemma \ref{lem: bullet-complete functor}.
\end{proof}
Recall the regular $\chi \in \Spec(Z_{Fr}) \cong G_0$, the open Bruhat cell of $G$. Let $G_\chi$ be the stabilizer of $\chi$ under the conjugation action of $G$ on itself. Let $C(\chi)$ be the component group of $G_\chi$, then we have the natural map $Z(G) \rightarrow C(\chi)$ where $Z(G)$ is the center of $G$.

\begin{Prop}\label{prop: ff on diagonal bimods} Assume the map $Z(G) \rightarrow C(\chi)$ is surjective. Then the functor $\bullet_\dag$  in \eqref{eq: res functor for chi} is fully faithful on the diagonal bimodules, which are defined in Definition \ref{defi: diagonal bimod}.
\end{Prop}
\begin{Rem}The map $Z(G) \rightarrow C(\chi)$ is surjective when $\chi$ is a regular unipotent element in $G$. The surjectivity  always holds in type $A$ but not hold in general. The enhanced versions of $\bullet_\dag$ and Proposition \ref{prop: ff on diagonal bimods} are discussed in Appendix \ref{append: enhanced version}
\end{Rem}
We will need the following lemmas under the condition that $Z(G) \rightarrow C(\chi)$ is surjective.

\begin{Lem}\label{lem: equal isogeny} For any $V_\e \in \Rep^{fd}(\cU_\e(\g))$ which has weights contained  in the root lattice $Q$, the following natural map is bijective
\begin{equation}\label{eq: equal isogeny} \Hom_{\cU_\e}(V_\e, U_\e^{fin, \uchi}) \rightarrow \Hom_{\cU_\e}(V_\e, U_\e^{ev\wedge_\chi}).
\end{equation}
\end{Lem}
\begin{proof}See Appendix \ref{append: equal isogeny}.
\end{proof}
\begin{Lem} \label{lem: Hom(, Ucomplete at chi)}Let $V_q$ be an object in $\Rep^{fd}(\cU_q(\g))$ which has weights contained in the root lattice $Q$. Then $\Hom_{\cU_q}(V_q, U_q^{ev\wedge_\chi})$ is finitely generated over $\CW_q^{\wedge_{\uchi}}$. 
\end{Lem}
\begin{proof} By Lemma \ref{lem: equal isogeny}, the space $\Hom_{\cU_q}(V_q, U_\e^{ev\wedge_\chi}) \cong \Hom_{\cU_\e}(V_q/\hbar V_q, U_\e^{ev\wedge_\chi})$ is finitely generated over $\CW_\e^{\wedge_{\uchi}}$. We have an exact sequence
\[ 0\rightarrow \Hom_{\cU_q}(V_q, U_q^{ev\wedge_\chi})\xrightarrow[]{\cdot \hbar} \Hom_{\cU_q}(V_q, U_q^{ev\wedge_\chi}) \rightarrow \Hom_{\cU_q}(V_q, U_\e^{ev\wedge_\chi}).\]
Hence $\Hom_{\cU_q}(V_q, U_q^{ev\wedge_\chi})/\hbar \Hom_{\cU_q}(V_q, U_q^{ev\wedge_\chi})$ is finitely generated over $\CW_\e^{\wedge_{\uchi}}$.  On the other hand, $\CW_q^{\wedge_{\uchi}}$ is complete and separated in the $\hbar$-adic topology, while $\Hom_{\cU_q}(V_q, U_q^{ev\wedge_\chi})$ is separated in the $\hbar$-adic topology since $U_q^{ev\wedge_\chi}$ is separated in the $\hbar$-adic topology. Therefore, by Lemma \ref{lem fg quotient imply fg}, the space $\Hom_{\cU_q}(V_q, U_q^{ev\wedge_\chi})$ is finitely generated over $\CW_q^{\wedge_{\uchi}}$.   
\end{proof}




\begin{proof}[Proof of Proposition \ref{prop: ff on diagonal bimods}]It is enough to show the following map is bijective
\begin{multline}\label{eq: full-faith}
\Hom_{U_q^{fin, \uchi}\rmod^{G_q}}(V_q\otimes_{\BC[[\hbar]]} U_q^{fin, \uchi}, W_q\otimes_{\BC[[\hbar]]} U_q^{fin, \uchi})\\ \rightarrow \Hom_{\Pbim^\Lambda(U_q^{ev\wedge_\chi})}(V_q\otimes_{\BC[[\hbar]]}U_q^{ev\wedge_\chi}, W_q\otimes_{\BC[[\hbar]]} U_q^{ev\wedge_\chi})
\end{multline}
for $V_q, W_q\in \Rep(\cU_q(\g))$ which are free of finite rank over $\BC[[\hbar]]$.

$\bullet$ Since both $V_q\otimes_{\BC[[\hbar]]}U_q^{ev\wedge_\chi}$ and $W_q\otimes_{\BC[[\hbar]]}U_q^{ev\wedge_\chi}$ are flat over $\BC[[\hbar]]$, by Remark \ref{rem: mor between flat Pbim}, the right hand side of \eqref{eq: full-faith} is equal to 
\begin{equation}\label{eq: Hom of bimod} \Hom_{U_q^{ev\wedge_\chi}\bimod^\Lambda}(V_q\otimes_{\BC[[\hbar]]}U_q^{ev\wedge_\chi}, W_q\otimes_{\BC[[\hbar]]} U_q^{ev\wedge_\chi}).
\end{equation}

$\bullet$ $V_q\otimes_{\BC[[\hbar]]} U_q^{ev\wedge_\chi}$ is also an object in $U_q^{ev\wedge_\chi}\rmod^{\cU_q, \Lambda}$, the category of $\Lambda$-graded $\cU_q$-equivariant right $U_q^{ev\wedge_\chi}$-modules. We claim that \eqref{eq: Hom of bimod} is equal to 
\begin{equation}
    \label{eq: Hom of equiv-mod}
    \Hom_{U_q^{ev\wedge_\chi}\rmod^{\cU_q, \Lambda}}(V_q\otimes_{\BC[[\hbar]]} U_q^{ev\wedge_\chi}, W_q\otimes_{\BC[[\hbar]]} U_q^{ev\wedge_\chi})
\end{equation}
We have a forgeful map from $\eqref{eq: Hom of equiv-mod}$ to $ \eqref{eq: Hom of bimod}$: indeed, any morphism in \eqref{eq: Hom of equiv-mod} is a morphism of $\Lambda$-graded $(U_q^{ev}, U_q^{ev\wedge_\chi})$-bimodules, which then is a morphism of $\Lambda$-graded $U_q^{ev\wedge_\chi}$-bimodules by continuity. Now this forgetful map is injecitve.  This map is also surjective as follows: Any bimodule map is equivariant under the adjoint $U_q^{ev}(\g)$-actions \footnote{The adjoint $U_q^{ev}(\g)$-action is defined by $x \cdot m =\sum x_{(1)}m S(x_{(2)})$ for $m \in M$ and $x\in U_q^{ev}(\g)$.}. Since both $V_q\otimes_{\BC[[\hbar]]} U_q^{ev\wedge_\chi}$ and $W_q\otimes_{\BC[[\hbar]]}U_q^{ev\wedge_\chi}$ are flat over $\BC[[\hbar]]$, the  $\cU_q(\g)$-actions on these two modules are uniquely recovered from the adjoint $U_q^{ev}(\g)$-actions. 

$\bullet$ Now to prove that  \eqref{eq: full-faith} is an isomorphism, it is enough to prove that the following map is an isomorphism:
\begin{equation}\label{eq: full-faith2}
\Hom_{\cU_q(\g)}(V_q, U_q^{fin, \uchi})\rightarrow \Hom_{\cU_q(\g)}(V_q, U_q^{ev\wedge_\chi})
\end{equation}
for $V_q \in \Rep(\cU_q(\g))$ which is free of finite rank over $\BC[[\hbar]]$ and has weights contained in the root lattice $Q$ (the $\Lambda$-grading on any object in $U_q^{ev\wedge_\chi}\rmod^{\cU_q, \Lambda}$ is used here where $U_q^{ev\wedge_\chi}$ is in the degree $0 \in \Lambda$.)

$\bullet$ Since $\Rep(\cU_q(\g))$ has enough projective objects, we can assume that $V_q$ is projective. Since both $U_q^{fin , \uchi}$ and $U_q^{ev\wedge_\chi}$ are flat over $\BC[[\hbar]]$, we have 
\[ \begin{tikzcd} 0 \arrow[r] & \Hom_{\cU_q}(V_q, U_q^{fin, \uchi}) \arrow[d] \arrow[r, "\cdot \hbar"] & \Hom_{\cU_q}(V_q, U_q^{fin, \uchi})\arrow[d] \arrow[r]& \Hom_{\cU_q}(V_q, U_\e^{fin, \uchi}) \arrow[d]\arrow[r]& 0\\
0 \arrow[r]& \Hom_{\cU_q}(V_q, U_q^{ev\wedge_\chi}) \arrow[r, "\cdot \hbar"]& \Hom_{\cU_q}(V_q, U_q^{ev\wedge_\chi}) \arrow[r] & \Hom_{\cU_q}(V_q, U_\e^{ev\wedge_\chi})    
\end{tikzcd}
\]
The rightmost vertical arrow is the same as
\begin{equation}\label{eq: quotient iso} \Hom_{\cU_\e}(V_q/\hbar V_q, U_\e^{fin, \uchi}) \rightarrow \Hom_{\cU_\e}(V_q/\hbar V_q, U_\e^{ev\wedge_\chi})
\end{equation}
which is an isomorphism by Lemma \ref{lem: equal isogeny}. Therefore, the second row is also a short exact sequence. 

Both $\Hom_{\cU_q}(V_q, U_q^{fin, \uchi})$ and $\Hom_{\cU_q}(V_q, U_q^{ev\wedge_\chi})$ are finitely generated over $\CW_q^{\wedge_{\uchi}}$ by Lemma \ref{lem: extension of left-actions} and Lemma \ref{lem: Hom(, Ucomplete at chi)}. Moreover, $\Hom_{\cU_q}(V_q, U_q^{ev\wedge_\chi})$ is flat over $\BC[[\hbar]]$. Therefore, isomorphism \eqref{eq: quotient iso} implies that \eqref{eq: full-faith2} is an isomorphism.
\end{proof}

\subsection{The functor $\bullet_\dag: O^{[0]}_q \rightarrow \CW_q^{\wedge_{0}}\Mod^\Lambda$}\
Here $\CW_q^{\wedge_0} \Mod^\Lambda$ is the category of $\Lambda$-graded $\CW_q^{\wedge_0}$-modules.


Let us recall the identification in Proposition  \ref{prop: F-center vs openBruhat}: 
\[ Z_{Fr}= \BC[\tE_\a^\ell K^{\ell \gamma(\a)}]_{\a \in \Delta_+} \otimes_{\BC} \bigoplus_{\lambda \in \ell P} \BC K^{2\lambda} \otimes_{\BC} \BC[\tF^\ell_\a K^{\ell\kappa(\a)}]_{\a \in \Delta_+} \cong \BC[U_{-}]\otimes_{\BC} \BC[T]\otimes_{\BC} \BC[U_+]\]

Let $\chi$ be a regular unipotent element in $U_+ \subset G_0$.  Then  $ \{ K^{2\ell \lambda}-1, \tE_\a^\ell~|~\lambda \in P, \a \in \Delta_+\} \subset \m_{\chi}$, the maximal ideal of $\chi \in G_0= \Spec Z_{Fr}$. Let $\mathfrak{I}$ be the ideal of $Z_{Fr}$ generated by $K^{2\ell \lambda}-1$ and $\tE_\a^\ell$ for all $\a \in \Delta_+$ and $\lambda \in P$. Let us consider the map 
\begin{equation}\label{eq: iota for tangent space}\iota: \mathfrak{I} \subset \m_\chi \rightarrow \m_\chi/\m_\chi^2.
\end{equation}
On $\m_\chi/\m^2_\chi$ , we have the skew-symmetric bilinear form $\m_\chi /\m^2_\chi \x \m_\chi/ \m^2_\chi \rightarrow Z_{Fr}/\m_\chi \cong \BC$ as follows: $\{ f+\m^2_\chi, g+\m^2_\chi\} =\{ f, g\} +\m_\chi$.

\begin{Lem}\label{lem: J poisson closed}(a) $\{ K^{2\ell \lambda} -1, \tE_\a^\ell\} = a K^{2\ell \lambda} \tE_\a^\ell \in \mathfrak{I}$ and $\{ \tE_\a^\ell, \tE_\b^\ell\} \in Z^>_{Fr, \a+\b} \subset \mathfrak{I}$ for  $a\in \BC$ and $\lambda \in P; \a, \b \in \Delta_+$. Therefore, $\mathfrak{I}$ is closed under the Poisson bracket of $Z_{Fr}$.

\noindent
(b) The ideal $\fJ$ acts as zero on any object in $O_\e$.
\end{Lem}
\begin{proof}(a) In $U_q^{ev}(\g)$, we have $[K^{2\ell \lambda}-1, \tE_\a^\ell]=(q^{2\ell(2\lambda, \a)}-1)K^{2\ell \lambda} \tE_\a^\ell$. On the other hand, $[\tE_\a^\ell, \tE_\b^\ell] \in U^{ev>}_{\e, \a+\b}$. Therefore, two inclusions follow by the construction of the Poisson bracket on $Z$ and $Z_{Fr}$ in \eqref{eq: poisson bracket}. 

\noindent
(b) Let $M\in O_\e$ then $\tE_\a^\ell m=[\ell]_\e!~\tE_\a^{(\ell)} m=0$ for all $m \in M$. On the other hand, for any weight vector $m \in M_\mu$ then under the assumption \ref{eq: assumption on l}, $(K^{2\ell \lambda} -1)m=(\e^{(2\ell \lambda, \mu)}-1)m=0$. 
\end{proof}
 Recall that $N$ is the number of positive roots of $\g$. By Lemma \ref{lem: J poisson closed}, the image of $\iota$ in \eqref{eq: iota for tangent space} is an isotopic subspace of dimension at least $ N +r$ in $\m_\chi /\m^2_\chi$. On the other hand $\dim_\BC \m_\chi /\m^2_\chi =2 N + r$. The natural surjective map  from  $\m_\chi /\m^2_\chi$ to the cotangent space of the conjugacy class at $\chi$ is compatible with the skew-symmetric bilinear forms on these two spaces. Let $V$ be a maximal symplectic subspace in $\m_\chi/\m^2_\chi$ which is a lift  of the cotangent space. Since $\chi$ is a regular element, the maximal symplectic subspace of $\m_\chi /\m^2_\chi$ is of dimension $2N$. Therefore, the image of $\iota$  in \eqref{eq: iota for tangent space} is the maximal isotropic subspace of $\m_\chi /\m^2_\chi$, and the intersection $\fu:= \text{Im}(\iota) \cap V$ is a Lagrangian subspace of $V$.

\begin{Prop}\label{prop: factorization of complete poisson}Let $(A, \m)$ be a Noetherian complete local Poisson algebra. Let $\n\subset \m$ be an ideal closed under the Poisson bracket. Then the image  $\mathfrak{b}$ of  the map $\theta: \n \rightarrow \m \rightarrow \m/\m^2$ is an isotropic subspace of $\m/\m^2$.

Let $V \subset \m /\m^2$ be  a symplectic subspace of $\m/\m^2$ such that  $\fu:= V\cap \mathfrak{b}$ is a Lagrangian subspace of $V$. Then we can find a lift $V\rightarrow A$ so that $\fu \hookrightarrow \n$ and it extends to a Poisson embedding $\BC[[V]] \hookrightarrow A$.
\end{Prop}

\begin{proof} Let $\{x_1, \dots, x_d\}$ be a $\BC$-basis of $\fu$ and $\{x_1, \dots, x_d, y_1, \dots, y_d\}$ be a $\BC$-basis of $V$ such that $\{x_i, y_j\}=\delta_{ij}, \{x_i, x_j\}=\{y_i, y_j\}=0$ for $1\leq i,j \leq d$.

{\it Step 1:} Let $f\in \n$ such that $\theta(f)=x_1$. There is $g\in \m$ such that $\theta(g)=y_1$ and $\{f, g\}=1$. Indeed, let $g'\in \m$ such that $\theta(g')=y_1$, then we can choose $g= \sum_{i=1}^\infty \frac{a_i}{i!}(g')^i$ with $a_1=\frac{1}{\{f, g'\}}, a_{i+1}=-\frac{\{f, a_i\}}{\{f, g'\}}$ for $i \geq 2$. The element $g$ is well-defined since  $\{ f, g'\} \in 1+\m$  is invertible and $A$ is complete in the $\m$-adic topology.

So we have an embedding of complete Poisson algebras $\BC[[f,g]]\rightarrow A$. Let $A_1:=\{ a\in A~|~ \{f, a\}=\{g, a\}=0\}$. As in \cite[Proposition 3.5]{Ka06}, there is an isomorphism of complete local Poisson algebras 
\begin{equation}\label{eq: factorization}\BC[[f,g]]\widehat{\otimes} A_1 \rightarrow A,
\end{equation}
and $A_1$ is a Noetherian complete local Poisson algebra with the maximal ideal $\m_1=\m \cap A_1$. 

{\it Step 2:} We will show that the argument in Step 1 can be applied to $(A_1, \m_1, \n_1:=\n \cap A_1)$. Hence, the lemma follows by induction. 

By \eqref{eq: factorization}, the map $\m_1/(\m_1)^2 \rightarrow \m/\m^2$ is injective onto the skew-orthogonal subspace of $\BC x_1 \oplus \BC y_1$, in particular, $x_2, \dots, x_d, y_2, \dots, y_d$ are contained in $\m_1/(\m_1)^2$. So it is left to show that $x_2, \dots, x_d$ is contained in the image of $\n_1$ under the map $\theta$. Let $c' \in \n$ such that $\theta(c')=x_2$. Let us consider the following element: 
\[ c=c'+ \sum_{i,j \geq 0, (i,j) \neq (0,0)} \frac{f^i g^j}{i!j!}c_{ij},\]
\[ c_{0,1}=\{c', f\}, \qquad c_{1,0}=-\{ c', g\}, \qquad c_{i,j}=\{c_{i,j-1}, f\}=-\{ c_{i-1, j}, g\},\]
then $c\in A_1$. Since $A$ is Noetherian, $\n$ is complete in the $\m$-adic topology. The element $c_{0,i}\in \n$ since $\n$ is closed under Poisson bracket. Therefore, $c \in \n$. This shows that $x_2$ is contained in the image of $\n_1$ under the map $\theta$, e.g, $\theta(c)=x_2$. The proofs for $x_3, \dots , x_d$ are the same.
\end{proof}
Recall the map  $\phi_\hbar: R_\hbar \xrightarrow[]{\hbar} Z^{\wedge_\chi}$. The following lemma refines Proposition \ref{prop: Decomposition of completions}. 
\begin{Lem}[cf. Lemma 4.1\cite{IL12}]\label{lem: lift V to Rh}We can find a lift $\iota: V\rightarrow R_\hbar$ so that $\fu \hookrightarrow  \phi_\hbar^{-1}(\mathfrak{I})$ and $[\iota(v_1), \iota(v_2)]=\hbar\{v_1, v_2\}$ for $v_1, v_2\in V$. The lift gives the decomposition $R_\hbar \cong \CA_q\widehat{\otimes}_{\BC[[\hbar]]} \CW_q^{\wedge_{\uchi}}$. 
\end{Lem}

Recall the decomposition $U_q^{ev\wedge_\chi} \cong \Mat_{\sd}(\BC) \otimes R_\hbar$  and the matrix unit $e= E_{11} \in \Mat_{\sd}(\BC)$. By Lemma \ref{lem: extension action on cat O}, any $M\in O_q$ is a finitely generated $U_q^{ev, \uchi}$-module.  For any $M \in O_q$, let $M^{\wedge_\chi}:= U_q^{ev\wedge_\chi}\otimes_{U^{ev, \uchi}_q} M$. Then  $eM^{\wedge_\chi}$ is a finitely generated module over $R_\hbar$. Fix a lift $V\rightarrow R_\hbar$ as in Lemma \ref{lem: lift V to Rh}.

\begin{Lem} \label{lem: action of u}(a) For any $M \in O_q$, there is  a bilinear map $\{ \;, \; \}: \fu \x eM^{\wedge_\chi} \rightarrow eM^{\wedge_\chi}$ satisfying
\begin{itemize}
 \item $\hbar \{ u, m\} =um$ for $u \in \fu $ and $ m \in eM^{\wedge_\chi}$.
 \item $\{u, xm\}=\{u,x\}m+x \{ u, m\}$ for $u\in \fu, x \in R_\hbar $ and $ m \in eM^{\wedge_\chi}$.
\end{itemize}
We call the bilinear map $\{\;,\;\}$ on $eM^{\wedge_\chi}$ is a $\fu$-Poisson structure on $eM^{\wedge_\chi}$.  

For any $f\in \Hom_{O_q}(M,N)$, the map $f: eM^{\wedge_\chi} \rightarrow e N^{\wedge_\chi}$ satisfies $f\{u,m\}=\{u, f(m)\}$. We call the map with latter property a Poisson map.

\noindent
(b) For any $M\in O_q$, we have a decomposition of an $R_\hbar:= \CA^{\wedge}_q \widehat{\otimes}_{\BC[[\hbar]]} \CW_q^{\wedge_{\uchi}}$-module
\[ eM^{\wedge_\chi}\cong (\CA^\wedge_q/\CA^\wedge_q\fu)\widehat{\otimes}_{\BC[[\hbar]]} M_\dag,\]
here $M_\dag:=\{ m \in eM^{\wedge_\chi}~|~ \{ u, m\}=0 ~\forall~ u \in \fu\}$. 
\begin{Rem}\label{rem: poi-tensor on O}The map $\pi: (\CA^\wedge_q/\CA^\wedge_q\fu)\widehat{\otimes}_{\BC[[\hbar]]}M_\dag\rightarrow M$ is defined by $x\otimes m\mapsto \underline{x}m$ for $\underline{x}\in \CA^\wedge_q$ representing $x\in \CA^\wedge_q/\CA^\wedge_q \fu$ and $m\in M_\dag$. This is well-defined because $\fu M_\dag=0$. We also have a $\fu$-Poisson structure on pairing $(\CA^\wedge_q/\CA^\wedge_q\fu)\widehat{\otimes}_{\BC[[\hbar]]}M_\dag$ defined by $\{u, x\otimes m\}=\{u, x\}\otimes m$ for $u\in \fu, x\in \CA^\wedge_q/\CA^\wedge_q\fu$ and $m\in M_\dag$. Then the map $\pi$ is a Poisson map.
\end{Rem}
\end{Lem}
\begin{proof}(a) For any $M\in O_q$, the ideal $\mathfrak{I}$ acts as zero on $M/\hbar M$ by Lemma \ref{lem: J poisson closed}.b), hence, $u eM^{\wedge_\chi}  \in \hbar eM^{\wedge_\chi}$ for all $u \in \fu$. 

For any $M\in O_q$, we can find a $\BC[[\hbar]]$-flat object $N\in O_q$ with a surjective map $N \twoheadrightarrow M$. Indeed, the truncated category  $O_q^{\leq \nu}$ has enough projectives, see \cite[$\mathsection 2.3.2$]{Si24}. On the other hand, $M=\oplus_\lambda  M^\lambda$ for $M^\lambda \in O_q^{[\lambda]}$. Each  $M^\lambda$ belongs to $ O_q^{\leq \nu}$ for some $\nu$ so that  we can find a projective $P^\lambda \in O_q^{\leq \nu}$ and a surjective map $P^\lambda \twoheadrightarrow M^\lambda$. Such $P^\lambda$ is flat over $\BC[[\hbar]]$ by {\it loc.cit}. Hence, we can choose $N=\oplus P^\lambda$ with the surjective map $\oplus_\lambda P^\lambda \twoheadrightarrow \oplus M^\lambda =M$.

So we can find $\BC[[\hbar]]$-flat objects $N_2, N_1\in O_q$ with an exact sequence $N_2\rightarrow N_1\rightarrow M \rightarrow 0$. This gives us  an exact sequence by Remark \ref{rem: completion of U(ev, chi)}
\[ eN_2^{\wedge_\chi} \xrightarrow[]{\phi} e N_1^{\wedge_\chi} \xrightarrow[]{\pi} eM^{\wedge_\chi} \rightarrow 0.\]
Since $N_i$ for $i=1,2$ is flat over $\BC[[\hbar]]$, the completion $N_i^{\wedge_\chi}$ is also flat over $\BC[[\hbar]]$, hence so is $eN_i^{\wedge_\chi}$. Let us define $\{\;,\;\}: \fu \x eM^{\wedge_\chi} \rightarrow eM^{\wedge_\chi}$  by $\{u, m\}=\pi(\hbar^{-1}un)$ for $u\in \fu$, $m\in eM^{\wedge_\chi}$ and  any $n \in e N_1^{\wedge_\chi}$ such that $\pi(n)=m$. 
\begin{itemize}
    \item This bilinear map is well-defined: for $n_1, n_2$ such that $\pi(n_1)=\pi(n_2)$, we have $\pi(\hbar^{-1}un_1)-\pi(\hbar^{-1}u n_2)= \pi(\hbar^{-1} u \phi(n'))=\pi(\phi(\hbar^{-1}un'))=0$ for some $n'\in eN_2^{\wedge_\chi}$. 
    \item This bilinear map does not depend on the choice of the surjective map $N_1\twoheadrightarrow M$. Indeed let us consider another surjective map $N'_1\twoheadrightarrow M$. The fiber product $N_1 \x_{M} N'_1\subset N_1 \oplus N'_1$ is an object in $O_q$ ( it is equal to $\bigoplus_{\lambda} N_{1, \lambda} \x_{M_\lambda} N'_{1, \lambda}$). It is flat over $\BC[[\hbar]]$. So both bilinear maps defined by $N_1, N'_1$ coincide with the bilinear map defined by  $N_1 \x_{M} N'_1$.
    \end{itemize}
It is then straightforward to verify the properties of this bilinear map.

 The truncated category  $O_q^{\leq \nu}$ has enough projectives, see \cite[$\mathsection 2.3.2$]{Si24}. So in the category $O_q^{\leq \nu}$ with $\nu$ large enough, we can find a commutative diagram 
    \begin{equation*}\begin{tikzcd} M' \arrow[r, "f' "] \arrow[d, two heads, "\pi_1"] & N' \arrow[d, two heads, "\pi_2"]&\\
    M \arrow[r, "f"] & N 
    \end{tikzcd}
    \end{equation*}
where $M', N'$ are flat over $\BC[[\hbar]]$. Then $f\{u, m\}=f\circ \pi_1\{ u, m'\}=\pi_2 \circ f'\{u, m'\}=\{ u, \pi_2 \circ f'(m')\}=\{u, f(m)\}$, where $m'\in M'$ such that $\pi_1(m')=m$.

(b) By using part (a), this part follows the proof of \cite[Lemma 4.2]{IL15}. 
\end{proof}
Then we define the restriction functor $\bullet_\dag: O_q \rightarrow \CW_q^{\wedge_{\uchi}}\Mod^\Lambda$ by $M\mapsto M_\dag$. The $\Lambda$-grading on $M_\dag$  comes from the $\Lambda$-grading on $M$.

\begin{Prop}(a) The functor $\bullet_\dag: O_q\rightarrow \CW_q^{\wedge_{\uchi}}\Mod^\Lambda$ is exact and $\CW_q^{\wedge_{\uchi}}$-linear.

\noindent
(b) $(\Delta_\e(\lambda))_\dag \cong \BC$ for all $\lambda \in P$.
\end{Prop}
\begin{proof}
(a)  It is easy to see that $\bullet_\dag$ is $\CW_q^{\wedge_{\uchi}}$-linear. By Remark \ref{rem: completion of U(ev, chi)}, the functor $e(-)^{\wedge_\chi}$ is exact, hence $\bullet_\dag$ is left exact by construction. We only need to show that if $f\in \Hom_{O_q}(M,N)$ is surjective then $f_\dag: M_\dag \rightarrow N_\dag$ is also surjective. Note that $f^\wedge: eM^{\wedge_\chi}\rightarrow eN^{\wedge_\chi}$ is surjective. Since $f^\wedge$ is $R_\hbar$-linear, $\text{Im}(f^\wedge) \subset (\CA^\wedge_q/\CA^\wedge_q\fu)\widehat{\otimes}_{\BC[[\hbar]]} \text{Im}(f_\dag)$. Since $\CA^\wedge/\CA^\wedge_q\fu$ is topological free over $\BC[[\hbar]]$, in order to $\text{Im}(f^\wedge)=eN^{\wedge_\chi}$, we must have $\text{Im}(f_\dag)=N_\dag$, i.e., $f_\dag$ is surjective.

\noindent
(b) By Lemma \ref{lem: action of u}.b), for any $M\in O_\e \subset O_q$ then $eM^{\wedge_\chi} \cong (\CA^\wedge_q/\CA^\wedge_q\<\fu, \hbar\>)\widehat{\otimes}_{\BC} M_\dag$. Therefore, $M_\dag$ is isomorphic to the fiber of $eM^{\wedge_\chi}$ at $\chi$. This implies that $\dim_\BC M_\dag$ is equal to the dimension of fiber of $M$ at $\chi$ divided by $\sd=\ell^{N}$. Now the computation follows since $\Delta_\e(\lambda)$ is a free sheaf of rank $\sd$ over $U_+$.
\end{proof}

\begin{Lem}\label{lem: resfunctor on HC and O are compatible} For any $X\in U_q^{fin,\uchi}\rmod^{G_q}$ and $N \in O_q$, we have an isomorphism of $\CW_q^{\wedge_{\uchi}}$-modules
\[ (X\otimes_{U_q^{fin, \uchi}}N)_\dag \cong X_\dag \otimes_{\CW_q^{\wedge_{\uchi}}} N_\dag.\]
\end{Lem}

\begin{proof} We have that  $X\otimes_{U_q^{fin, \uchi}}N \cong X_{loc} \otimes_{U_q^{ev, \uchi}} N$, here $X_{loc}=X\otimes_{U_q^{fin, \uchi}} U_q^{ev, \uchi}$. Therefore, $(X\otimes_{U_q^{fin, \uchi}} N)^{\wedge_\chi} \cong X_{loc}^{\wedge_\chi}\otimes_{U_q^{ev\wedge_\chi}} N^{\wedge_\chi}$. On the other hand, $U_q^{ev\wedge_\chi}\cong \Mat_{\sd}(\BC) \otimes_{\BC} R_\hbar$ and $R_\hbar=eU_q^{ev\wedge_\chi}e$ where $e=E_{11}$ the matrix unit of $U_q^{ev\wedge_\chi}$. Therefore, $X^{\wedge_\chi}_{loc} \cong U_q^{ev\wedge_\chi} \otimes_{R_\hbar} eX^{\wedge_\chi}_{loc} e\otimes_{R_\hbar} eU_q^{ev\wedge_\chi}$ as Poisson bimodules. Meanwhile, $N^{\wedge_\chi} \cong U_q^{ev\wedge_\chi}e \otimes_{R_\hbar} eN^{\wedge_\chi}$. This implies that $(*)\;\;e(X\otimes_{U_q^{fin, \uchi}} N)^{\wedge_\chi}\cong eX^{\wedge_\chi}_{loc}e\otimes_{R_\hbar} eN^{\wedge_\chi}$as $R_\hbar$-modules.

Both $e(X\otimes_{U_q^{fin, \uchi}}N)^{\wedge_\chi}$ and $eN^{\wedge_\chi}$ are $R_\hbar$-modules with $\fu$-Poisson structures. Then tensor product $eX^{\wedge_\chi}_{loc}e\otimes_{R_\hbar} eN^{\wedge_\chi}$ is a $R_\hbar$-module with $\fu$-Poisson structure as follows: $\{u, x\otimes n\}=\{u,x\}\otimes n +x \otimes \{u,n\}$ for $u\in \fu, x\in eX^{\wedge_\chi}_{loc} e$ and $n \in eN^{\wedge_\chi}$. Then the isomorphism $(*)$ is an isomorphism of $R_\hbar$-modules with $\fu$-Poisson structures. This is proved by choosing a projective object $P_1 \in U_q^{fin, \uchi}\Rmod^{G_q}$ and a $\BC[[\hbar]]$-flat object $Q_1\in O_q$ with surjective maps $P_1\twoheadrightarrow X$ and $Q_1\twoheadrightarrow N$ and the constructions of Poisson structure in Lemma \ref{lem: bullet-loc functor} and Lemma \ref{lem: action of u}, we left the detail to the reader. The object $P_1$ exists since $U_q^{fin, \uchi}\Rmod^{G_q}$ has enough projective while the existence of $Q_1$ is proved in Lemma \ref{lem: action of u}.a.

Recall that $R_\hbar \cong \CA^{\wedge}_q\widehat{\otimes}_{\BC[[\hbar]]} \CW_q^{\wedge_{\uchi}}$. Now, $eX^{\wedge_\chi}_{loc}e \cong \CA^{\wedge}_q \widehat{\otimes}_{\BC[[\hbar]]} X_\dag$ in $\Pbim(R_\hbar)$ and $eN^{\wedge_\chi} \cong (\CA^\wedge_q/\CA^\wedge_q\fu)\widehat{\otimes}_{\BC[[\hbar]]} N_\dag$ as $R_\hbar$-modules with $\fu$-Poisson structures. Therefore, $e(X\otimes_{U_q^{fin, \uchi}}N)^{\wedge_\chi}\cong (\CA^\wedge_q/\CA^\wedge_q\fu)\widehat{\otimes}_{\BC[[\hbar]]}(X_\dag\otimes_{\CW_q^{\wedge_{\uchi}}} N_\dag)$ as $R_\hbar$-modules with $\fu$-Poisson structures.
This implies that $(X\otimes_{U_q^{fin, \uchi}} N)_\dag \cong X_\dag \otimes_{\CW_q^{\wedge_{\uchi}}} N_\dag$.
\end{proof}
\begin{Rem}There is the restriction functor for $O_{q, \sR}$. Since it is technical and we will not use it in this paper, the detail will be covered elsewhere.
\end{Rem}


\section{Soergel Bimodules}

\subsection{Soergel's definition}\ \label{ssec: defi of Soergel bimod}

 Recall $\Lambda:= P/Q$, here $P$ and $Q$ are the  weight and root lattices of $\g$. Let $S_{aff}$ be the set of simple reflections of $W_{aff}:=W\ltimes Q$. Recall the complete local algebra $(\sR, \m)$ with the action of $W_{ext}:=W \ltimes P$ on it in Section \ref{ssec: the algebra R}. 
 
 For any $s\in S_{aff}$, the {\em Bott-Samelson bimodule}  $\BS_s$ is $ \sR\otimes_{\sR^s} \sR$, here $\sR^s$ is the $s$-invariant part of $\sR$. For any $x\in W_{ext}$, the {\em graph bimodule} $\sR_x$ is isomorphic to $\sR$ as left $\sR$-modules while the right $\sR$-module structure  on $\sR_x$ is defined by $mr=x(r)m$ for $m\in \sR_x$ and $r\in \sR$.

 Let $\sR\bimod^\Lambda$ be the category of $\Lambda$-graded $\sR$-bimodules, where $\sR$ is in degree $0 \in \Lambda$.
 
\begin{defi}The category of  (extended) affine Soergel bimodules $\SB_\hbar$ is the full Karoubian monoidal subcategory of $\sR \bimod^\Lambda$ generated by Bott-Samelson bimodules $\sR\otimes_{\sR^s} \sR$ for $s\in S_{aff}$ (in degree $0 \in \Lambda$) and the graph bimodules $ \sR_x$ for $x\in \Lambda$ (in  degree $x \in \lambda$). 
\end{defi}

\begin{Rem}Since the action of $W_{ext}$ on $\h \oplus \BC \hbar$ is faithful, it follows that $\SB_\hbar$ is a full subcategory of the category of  $\sR$-bimodules as well.
\end{Rem}
The following result of Soergel is well known:
\begin{Prop}There is one-to-one correspondence between indecomposable objects in $\SB_\hbar$ and $W_{ext}$: $B_x \leftrightarrow x \in W_{ext}$. In which, $B_1 \cong \sR, B_s\cong \BS_s$ for $s\in S_{aff}$.
\end{Prop}

\begin{defi}Let $\SB$ be the category with the same collection of objects as in $\SB_\hbar$ but the morphisms  given  by 
\[\Hom_{\SB}(M,N)=\Hom_{\SB_\hbar}(M,N)/\hbar \Hom_{\SB_\hbar}(M,N).\]
\end{defi}
\begin{Lem}\label{lem: SB category} The object $B_x, (x\in W_{ext}),$ is still indecomposable in $\SB$. The category $\SB$ is Karoubian, Krull-Schmidt with the set of indecomposable objects $\{B_x, x\in W_{ext}\}$.
\end{Lem}
\begin{proof} Let $p \in \End_{\SB}(M)$ be an idempotent. Since $\End_{\SB_\hbar}(M)$ is finitely generated over $\sR$ hence complete and separated in the $\hbar$-adic topology, we can lift the idempotent $p$ to an idempotent element $p_\hbar \in \End_{\SB_{\hbar}}(M)$. Since $\SB_\hbar$ is Karoubian, there is a direct sum decomposition $M\cong P\oplus Q$ and $p_\hbar$ is the projection to $P$. This implies that $p \in \End_{\SB}(M)$ splits. In the other word, $\SB$ is Karoubian.

It is easy to see that  $\Hom_{\SB}(M,N)$ is a finitely generated module over the complete local ring $\usR:= \sR/\hbar \sR$ for any $M, N \in \SB$ hence $\SB$ is Krull-Schmidt. Moreover, since $B_x$ is indecomposable in $\SB_\hbar$, $\End_{\SB_\hbar}(B_x)$ is a local ring hence $\End_{\SB}(B_x)$ is also a local ring. Therefore, $B_x$ is indecomposable in $\SB$.
\end{proof}
\subsection{Abe's realization}\label{ssec: abe realization}\ 

Let $\usR:= \sR/\hbar \sR$ then $\usR =\BC[[\h^*]]$ and $W_{ext}$ acts on $\usR$ through the projection $W_{ext}\twoheadrightarrow W$. Since this action of $W_{ext}$ on $\h^*$ is not faithful, Soergel's definition will not yield a good category. A fix to Soergel's approach was proposed in \cite{Abe1}.

Let $Q$ be the fraction field of $\usR$. We define the category $\mathcal{C}$ following \cite{Abe1}. Its objects are $\usR$-bimodules $M$ with a decomposition of $(\usR, Q)$-bimodules  $M\otimes_{\usR} Q = \bigoplus_{w\in W_{ext}} M_{Q, w}$ such that
\begin{itemize}
     \item $M$ is finitely generated as an $\usR$-bimodule and flat as a right $\usR$-module.
    \item $M^w_Q \neq 0$ for only finitely many $w\in W_{ext}$.
    \item If $m\in M_{Q,w}$ and $f\in \usR$ then $mf=w(f)m$.
\end{itemize}
Inside $\mathcal{C}$, we construct the following objects
\begin{itemize}
    \item The graph bimodule $\usR_w$ with the decomposition $\usR_w\otimes_{\usR} Q=(\usR_w\otimes_{\usR} Q)_w$.
    \item For $s\in S_{aff}$, the Bott-Samelson bimodule $\uBS_s:= \usR \otimes_{\usR^s} \usR$ with the decomposition $\uBS_s\otimes_{\usR} Q= Q\oplus Q_s$ as in \cite[(2.1)]{Abe1}.
    
\end{itemize}
\begin{defi}Let $\ASB$ be the full Karoubian subcategory of $\mathcal{C}$ generated by the objects $\uBS_s, \usR_x$ with $s\in S_{aff}, x \in \Lambda$ under taking tensor products, direct sums and direct summands.
\end{defi}
\begin{Rem}The paper \cite{Abe1} works with graded bimodules over $\BC[\h^*]$ but their results also hold in our setting of bimodules over the complete local ring $\BC[[\h^*]]$. Moreover, the  {\em loc. cit.} is about the Coxeter group $W_{aff}$ but it is straightforward to extend their results to $W_{ext}$. 
\end{Rem}
\begin{Prop}[Theorem 4.1 \cite{Abe1}] There is one-to-one correspondence between indecomposable objects in $\ASB$ and $W_{ext}$: $\uB_x \leftrightarrow x\in W_{ext}$. Here $\uB_1\cong \usR, \uB_s \cong \uBS_s$ for $s\in S_{aff}$.
\end{Prop}

\begin{Prop}\label{prop: abe realization}There is an equivalence of monoidal categories $\SB \cong \ASB$.
\end{Prop}
\begin{proof} {\it Step 1:} Let us construct a monoidal functor $\mathfrak{Q}_\hbar: \SB_\hbar \rightarrow \ASB$.  Let $\sR_\hbar$ be the localization of $\sR$ at the prime ideal $\hbar \sR$, then $\sR_\hbar/\hbar \sR_\hbar \cong Q$. Since $W_{ext}$ fixes $\hbar$, it acts on $\sR_\hbar$. 

First, similar to Abe's construction in \cite{Abe1}, we define the category  $\CC_\hbar$: its object are $\sR$-bimodules $M$ with a decomposition $M\otimes_{\sR} \sR_\hbar=\bigoplus_{w\in W_{ext}} M_{\hbar,w}$ such that 
\begin{itemize}
    \item $M$ is finitely generated as a $\sR$-bimodules and flat as a right $\sR$-module
    \item $M_{\hbar,w} \neq 0$ for only finitely many $w\in W_{ext}$.\
    \item If $m \in M_{\hbar,w}$ and $f\in \sR_\hbar$ then $mf=w(f)m$.
\end{itemize}
Then it is clear that we have a monoidal functor $F_1: \CC_\hbar\rightarrow \CC$ by sending $M \in \CC_\hbar$ to $M/\hbar M$.

The graph bimodule $\sR_w$ belongs to $\CC_\hbar$ with the decomposition $\sR_w\otimes_{\sR} \sR_\hbar =(\sR_w\otimes_{\sR} \sR_\hbar)_w$. The Bott-Samelson bimodule $\BS_s, (s\in S_{aff})$ admits a short exact sequence of $\sR$-bimodules: 
\[0\rightarrow \sR \xrightarrow[]{\cdot(\a_s\otimes 1+1\otimes \a_s)} \BS_s \xrightarrow[]{m\otimes n \mapsto ms(n)} \sR_s\rightarrow 0,\]
where $\a_s \in \sR$ such that $s(\a_s)=-\a_s$ and $\sR=\sR^s \oplus \sR^s \a_s$. Moreover, the following composition 
\[  \sR_s \xrightarrow[]{\cdot(\a_s\otimes1-1\otimes \a_s)} \BS_s \xrightarrow[]{m\otimes n \mapsto m s(n)} \sR_s\]
is equal to multiplication by $2\a_s$  which is invertible in $\sR_\hbar$. The discussion so far implies that $\BS_s$ admits a decomposition $\BS_s\otimes_\sR \sR_\hbar=\sR_\hbar \oplus \sR_{\hbar,s}$, and then naturally belongs to $\CC_\hbar$.

So there is a natural monoidal functor $F_2: \SB_\hbar \rightarrow \CC_\hbar$. Because the action of $W_{ext}$ on $\sR$ is faithful, the functor $F_2$ is an embedding. 

Under the  composition $F_2\circ F_1: \SB_\hbar \rightarrow \CC$, the image of $\SB_\hbar$ is contained in $\ASB$. Therefore, we obtained the desired monoidal  functor $\mathfrak{Q}_\hbar: \SB_\hbar \rightarrow \ASB$.

{\it Step 2:} It is clear that the monoidal functor $\mathfrak{Q}_\hbar$ factors through a monoidal functor $\mathfrak{Q}: \SB \rightarrow \ASB$. The objects $\uBS_{s} (s\in S_{aff})$ and $\usR_w (w\in \Lambda)$ belong to the image of $\mathfrak{Q}$. Moreover, $\SB$ is Karoubian by Lemma \ref{lem: SB category}. Therefore, $\mathfrak{Q}$ is an equivalence if $\mathfrak{Q}$ is fully faithful.

Let us show that $\mathfrak{Q}$ is fully faithful. In the category $\SB_\hbar$, the left and right adjoint of $\BS_s \otimes -$ are both isomorphic to $\BS_s \otimes -$ for $s\in S_{aff}$, meanwhile the left and right adjoint of $\sR_w\otimes-$ are both isomorphic to $\sR_{w^{-1}} \otimes-$ for $w\in \Lambda$. Similarly, in the category $\ASB$, the left and right adjoint of $\uBS_s\otimes-$ are both isomorphic to $\uBS_s\otimes-$ for $s\in S_{aff}$, meanwhile the left and right adjoint of $\usR_{w}\otimes-$ are both isomorphic to $\usR_{w^{-1}}\otimes-$ for $w\in \Lambda$. The statements for $\sR_w$ and $\usR_w$ are easy to see. See \cite[Lemma 2.15]{Abe1} for the proof of statement for $\BS_s$ and $\uBS_s$.

Therefore, we reduce to show that  the following map is an isomorphism
\begin{equation}\label{eq: Hom of SB} \Hom_{\SB_\hbar}(\sR_w\otimes \BS_{s_1}\otimes \dots \BS_{s_k}, \sR)/\hbar \rightarrow \Hom_{\ASB}(\usR_w\otimes\uBS_{s_1}\otimes \dots \otimes \uBS_{s_k}, \usR),
\end{equation}
for $s_1, \dots, s_k \in S_{aff}$ and $w\in \Lambda$. Moreover, we only need to care about the case when $w$ is the unit $e\in \Lambda$, because otherwise both sides of \eqref{eq: Hom of SB} are zero.  In the case when $w=e$, by analyzing the {\em light leaves} in \cite[Theorem 3.12(3)]{Abe1}, both sides of \eqref{eq: Hom of SB} are free module over $\usR$ of same rank and \eqref{eq: Hom of SB} is surjective. This implies that \eqref{eq: Hom of SB} is an isomorphism.
\end{proof}
\subsection{Two-sided cells in $\SB_\hbar$, $\SB$}\

We have the Bruhat order $\prec$ on $W_{aff}$ and $W_{ext}$.
The indecomposable objects $\{ B_x| x\in W_{ext}\}$ in $\SB_\hbar$ can be constructed as follows. For $x=e$ then $B_e=\sR$ and for $s\in S_{aff}$, then $B_s=\BS_s$. For any $w\in W_{aff}$ with a reduced expression $w=s_1 \dots s_{r}$ then $\BS_{s_1}\otimes \dots \otimes \BS_{s_r}=B_w \bigoplus \oplus_{u \prec w} B_u^{\oplus_?}$, so $B_w$ is a certain indecomposable summand of this tensor product of Bott-Samelson bimodules. If $\pi \in \Lambda$ and $w\in W_{aff}$ then $B_{\pi w} =\sR_\pi  \otimes B_w$.

The group algebra $\BC[W_{aff}]$ has the standard basis $\{ T_w| w\in W_{aff}\}$ and the Kazhdan-Lusztig basis $\{C_w| w\in W_{aff}\}$. The group algebra $\BC[W_{ext}]$ has the standard basis $\{\pi T_w| \pi \in \Lambda, w\in W_{aff}\}$ and the Kazhdan-Lusztig basis $\{ \pi C_w| \pi\in \Lambda, w \in W_{aff}\}$.

The following result follows by Soergel's categorification theorem by specialization to $q=1$
\begin{Thm}\label{thm: So-categorify}
There is an isomorphism of algebras $K_0(\SB_\hbar)\cong \BC[W_{ext}]$ by  $[B_w]\mapsto C_w$.
\end{Thm}

For $u, w\in W_{ext}$, we define $u<^{LR} w$ if there $y, y' \in W_{ext}$ such that $C_yC_wC_{y'}=c^{y,y'}_{u,w}C_u +\dots$ with $c^{u,w}_{y,y'} \neq 0$. This defines  an equivalence relation  on $W_{ext}$: $u \sim^{LR} w$ if and only if $u<^{LR} w$ and $w<^{LR}u$, and the equivalence classes in $W_{ext}$ are called {\em two-sided cells}. 

We also have the order on the set of two-sided cells as follows: Let $C,C'$ be two-sided cells in $W_{ext}$ then $C < C'$ if there $x\in C$ and $y \in C'$ such that $x<^{LR} y$. It turns out that there is a unique smallest two-sided cell in this order. We denote this smallest two-sided cell by $C_{0}$. Let $w_0$ be the longest element in the finite Weyl group $W$, then $w_0\in C_{0}$, see \cite{Shi87}.

\begin{Rem} \label{rem: smallest two-sided cell} The above discussion combined with Theorem \ref{thm: So-categorify} implies that there is a full subcategory of $\SB_\hbar$, to be denoted by $\sC_{\hbar,0}$, consisting of all objects which are direct sums of $B_w, w\in C_0$. The full subcategory $\sC_{\hbar, 0}$ is closed under tensoring with any object of $\SB_\hbar$ on both sides. It contains $B_{w_0}=\sR\otimes_{\sR^W} \sR$ since $w_0 \in C_0$. Furthermore, for any $u,w \in C_0$, there are $y, y'\in W_{ext}$ such that $B_u$ is a direct summand of $B_y\otimes B_w\otimes B_{y'}$.
\end{Rem}


\begin{Rem}\label{rem: smallest twosided in SB}The above discussion is also applied to $\SB \cong \ASB$ by using Proposition \ref{prop: abe realization} and \cite{Abe1}. We  use $\sC_0$ to denote the full subcategory of $\SB\cong \ASB$ corresponding the smallest two-sided cell $C_0$.
\end{Rem}

\subsection{Anti-spherical module $M^{asph}$ and anti-spherical category $\mathcal{D}^{asph}$}\

Let $M^{asph}$ be the antispherical right module of the extended affine Weyl group $W_{ext}$, i.e., 
\[M^{asph}:= \mathsf{sign}\otimes_{\BC[W]} \BC[ W_{ext}]\]
where $\mathsf{sign}$ is the sign representation of $W$. Let $\{ T_x\}$ be the standard basis of $\BC[W_{ext}]$ and  $\{C_x\}$ be Kazhdan-Lusztig basis of $\BC[W_{ext}]$. Let $W_{ext}^f$ be the subset of $W_{ext}$ containing  all elements $x$ which have minimal length in the right coset $Wx$. Then $\{ 1\otimes T_x| x\in W_{ext}^f\}$ and $\{ 1\otimes C_x| x\in W_{ext}^f\}$ are the standard basis and Kazhdan-Lusztig basis of $M^{asph}$.

\begin{defi}Let $\mathcal{D}^{asph}$ be the category with the same objects as of $\SB \cong \ASB$ but let any morphism that factors through the direct sum of  some $B_w,~(w\not \in W_{ext}^f),$ to be zero.
\end{defi}
The category $\SB$ acts on $\mathcal{D}^{asph}$. Then $K_0(\mathcal{D}^{asph}) \cong M^{asph}$ as right $\BC[W_{ext}]$-modules, see \cite[$\mathsection 4$]{RW15}.


\section{Embedding $\SB_\hbar \rightarrow \Hilt_q(0,0)$ and $\SB \rightarrow \Hilt_\e(0,0)$}

We are still under the assumption \ref{eq: assumption on l}. Let us define the map $\iota: \h^*\rightarrow \sR$ via $\iota(\lambda)=(\lambda, -) \in \h \subset \sR$. For each $\lambda \in P$, we define the map 
\begin{equation*}\label{eq: varepsilon map}\varepsilon_\lambda: \CW_q \subset \BC[[\hbar]][K^{2\nu}]_{\nu \in P} \rightarrow \sR, \qquad K^{2\nu} \mapsto q^{(\lambda, 2\nu)} e^{2\pi \sqrt{-1}\iota(2\nu)}.
\end{equation*}
\begin{Lem}The map $\varepsilon_\lambda$ extends to an isomorphism $\CW_q^{\wedge_{\lambda}} \rightarrow \sR^{W_\lambda}$.
\end{Lem}
\begin{proof}See \cite[Lemma 5.5.9]{IL23}.
\end{proof}

Under identifications $\varepsilon_\mu: \CW_q^{\wedge_\mu} \iso \sR^{W_\mu}$ and $\varepsilon_\lambda: \CW_q^{\wedge_\lambda} \iso \sR^{W_\lambda}$, we have the functor 
\[\bullet_\dag: \HC_q(\mu, \lambda) \rightarrow (\sR^{W_\mu}, \sR^{W_\lambda})\bimod^\Lambda.\]
\begin{Prop}\label{prop: image of translation bimods}Let $\mu, \lambda$ be  in the closure of the fundamental alcove $\overline{C}$ in \eqref{eq: closure of C}. Assume $W_\lambda \subset W_\mu$. Recall the bimodule $P^{\mu, \lambda}_q$ defined in Definition \ref{defi: diagonal bimod}. Then

\noindent
(a) We have an isomorphism $P^{\mu, \lambda}_{q,\dag} \cong \sR^{W_\lambda}$ in the category $(\sR^{W_\mu}, \sR^{W_\lambda})\bimod^\Lambda$ , here on $\sR^{W_\lambda}$, the left $\sR^{W_\mu}$-action comes from the inclusion $\sR^{W_\mu} \hookrightarrow \sR^{W_\lambda}$, while the right $\sR^{W_\lambda}$-action comes from the right multiplication.

\noindent
(b) We have an isomorphism  $P^{\lambda, \mu}_{q,\dag} \cong \sR^{W_\lambda}$ in the category $(\sR^{W_\lambda}, \sR^{W_\mu})\bimod^\Lambda$, here on $\sR^{W_\lambda}$, the left $\sR^{W_\lambda}$-action comes from the left multiplication, while the right $\sR^{W_\mu}$-action comes from the inclusion $\sR^{W_\mu} \hookrightarrow \sR^{W_\lambda}$.
\end{Prop}

\begin{proof}(a) {\it Step 1:} Since $\bullet_\dag: \HC_q(\mu, \lambda) \rightarrow (\sR^{W_\mu}, \sR^{W_\lambda})\bimod^\Lambda$ is exact and $(\sR^{W_\mu}, \sR^{W_\lambda})$-linear, 
\[ P^{\mu, \lambda}_{q,\dag}/P^{\mu,\lambda}_{q, \dag}\fJ_\lambda \cong (P^{\mu, \lambda}_q/P^{\mu, \lambda}_q \fJ_\lambda)_\dag,\]
here $\fJ_\lambda$ is the maximal ideal of $\CW_q^{\wedge_\lambda} \iso \sR^{W_\lambda}$.

 By Lemma \ref{lem: resfunctor on HC and O are compatible}, $(P^{\mu, \lambda}_q/P^{\mu, \lambda}_q \fJ_\lambda)_\dag \otimes_{\BC} \Delta_\e(\lambda)_\dag \cong \Delta_\e(\mu)_\dag$. Hence $(P^{\mu, \lambda}_q/P^{\mu, \lambda}_q\fJ_\lambda)_\dag \cong \BC$. 
 
 By Proposition \ref{prop: ff on diagonal bimods}, 
\begin{equation}\label{eq: Hom on diagonal bimods} \Hom_{\HC_q(\mu, \lambda)} (P^{\mu, \lambda}_q, P^{\mu, \lambda}_q) \cong \Hom_{(\sR^{W_\mu}, \sR^{W_\lambda})\bimod^\Lambda}(P^{\mu, \lambda}_{q,\dag}, P^{\mu, \lambda}_{q,\dag}),
\end{equation}
Since $\bullet_\dag$ is $(\sR^{W_\mu}, \sR^{W_\lambda})$-linear and $P_q^{\mu, \lambda}$ is torsion free over $\CW_q^{\wedge_\lambda} \cong \sR^{W_\lambda}$, it follows that $P^{\mu, \lambda}_{q,\dag}$ is torsion free over $\sR^{W_\lambda}$. Moreover, $P^{\mu, \lambda}_{q, \dag}$ is finitely generated as a right $\sR^{W_\lambda}$-module, hence, $P^{\mu, \lambda}_{q,\dag} \cong \sR^{W_\lambda}$ as right $\sR^{W_\lambda}$-modules.

{\it Step 2:} From \eqref{eq: Hom on diagonal bimods} and Step 1, it follows that 
the right multiplication of $\CW_q^{\wedge_\lambda} \cong \sR^{W_\lambda}$ must induce an isomorphism 
\[ \sR^{W_\lambda} \iso \Hom_{\HC_q(\mu, \lambda)}(P^{\mu, \lambda}_q, P^{\mu, \lambda}_q).\]

{\it Step 3:} From the isomorphism $P^{\mu, \lambda}_q\otimes_{U^{fin, \lambda}_q} \Delta_{q, \sR}(\lambda) \cong \Delta_{q,\sR}(\mu)$, we have a $(\CW_q^{\wedge_\mu}, \CW_q^{\wedge_\lambda})$-linear map:
\begin{equation}\label{eq: Hom(HC) to Hom(O)}\Hom_{\HC_q(\mu, \lambda)}(P^{\mu, \lambda}_q, P^{\mu, \lambda}_q) \rightarrow \Hom_{O_{q, \sR}}(\Delta_{q,\sR}(\mu), \Delta_{q,\sR}(\mu)).
\end{equation}
Under the map  \eqref{eq: Hom(HC) to Hom(O)}, the right multiplication of $\CW_q^{\wedge_\lambda} \iso \sR^{W_\lambda}$ on $P^{\mu, \lambda}_q$ becomes the right multiplication of $\sR^{W_\lambda}$ on $\Delta_{q,\sR}(\mu)$. Combining this observation  with Step $2$, we see that \eqref{eq: Hom(HC) to Hom(O)} is injective.

Therefore, under the map  \eqref{eq: Hom(HC) to Hom(O)},  the left multiplication of $\CW^{\wedge_\mu}$ on $P^{\mu, \lambda}_q$  becomes the left multiplication of $\CW_q^{\wedge_\mu}$ on $\Delta_{q,\sR}(\mu)$. Note that $\varepsilon_\mu: \CW_q^{\wedge_\mu} \iso \sR^{W_\mu}$ identifies the left $\CW_q^{\wedge_\mu}$-action on $\Delta_{q,\sR}(\mu)$ with the right $\sR^{W_\mu}$-action on $\Delta_{q,\sR}(\mu)$. This implies that the left  and right $\sR^{W_\mu}$-actions on $P^{\mu, \lambda}_q$ coincide,  hence we obtain an isomorphism $P^{\mu, \lambda}_{q,\dag}\cong \sR^{W_\lambda}$ in $(\sR^{W_\mu}, \sR^{W_\lambda})\bimod^\Lambda$.

\noindent
(b) The proof is the same as in part $(a)$ but involving the right versions $ O^r_\e, O^r_q, O^r_{q, \sR}$.
\end{proof}

\begin{Thm}\label{thm: Sb-hbar embed} There is a full embedding of monoidal additive  categories $\SB_\hbar \rightarrow \Hilt_q(0,0)$.
\end{Thm}

\begin{proof}For each $s\in S_{aff}$, let $\lambda_s$ be a weight contained in the facet of the closure of the fundamental alcove associated to $s$. For each $x \in \Lambda \hookrightarrow W_{ext}= \Lambda \ltimes W_{aff}$, then $x \bullet_\ell 0$ is contained in the fundamental alcove $C$. The following bimodules are contained in $\Hilt_q(0,0)$: 
\[ P^{x, 0}_q:= P^{x\bullet_\ell 0, 0}_q, \quad P^{0,x}_q:= P^{0, x \bullet_\ell 0}_q, \quad P^{0, \lambda_s}_q\otimes_{U_q^{fin, \lambda_s}} P^{\lambda_s, 0}_q.\]
By Proposition \ref{prop: image of translation bimods}, the images of these bimodules under $\bullet_\dag: \HC_q(0,0) \rightarrow \sR\bimod^\Lambda$ are 
\[ \sR_x , \quad \sR_{x^{-1}}, \quad \sR\otimes_{\sR^s} \sR ,\]
in the degree $x, x^{-1}, 0 \in \Lambda$, respectively.

Combining this  with Proposition \ref{prop: ff on diagonal bimods},  we obtain a full embedding $\SB_\hbar \rightarrow \Hilt_q(0,0)$.  
\end{proof}

\begin{Cor}\label{cor: Sb embed}(a) There is a monoidal embedding  $\fI_1: \SB\rightarrow \Hilt_\e(0,0)$. Furthermore, tensoring with objects in $\SB \subset \Hilt_\e(0,0)$ on the left or on the right give exact endofunctors of $\HC_\e(0,0)$.

\noindent
(b) This  gives us an embedding of monoidal triangulated categories $K^b(\SB) \hookrightarrow D^b(\HC_\e(0,0))$
\end{Cor}
\begin{proof}
(a) There is a functor $\Hilt_q(0,0)\rightarrow \Hilt_\e(0,0)$ defined  via $M_q \mapsto M_\e:= M_q/\hbar M_q$. 

We will show that for any HC-tilting bimodule $M_q, N_q \in \Hilt_q(0,0)$, we have
\[ \Hom_{\HC_\e(0,0)}(M_\e, N_\e) \cong \Hom_{\HC_q(0,0)}(M_q, N_q)/\hbar \Hom_{\HC_q(0,0)}(M_q, N_q).\]
It is enough to prove that for any tilting modules $V_q,W_q$   in $\Rep^{fd}(\cU_q(\g))$, we have
\begin{multline*}\Hom_{U_\e^{fin, \uchi}\rmod^{G_\e}}(V_\e\otimes_{\BC} U_\e^{fin, \uchi}, W_\e \otimes_{\BC}U_\e^{fin, \uchi}) \cong \\
\Hom_{U_q^{fin, \uchi} \rmod^{G_q}}(V_q\otimes_{\BC[[\hbar]]}U_q^{fin, \uchi}, W_q\otimes_{\BC[[\hbar]]} U_q^{fin, \uchi})/\hbar
\end{multline*}
which amounts to proving the following isomorphism for any tilting module $V_q$
\begin{equation}\label{eq: Hom(V, Ufin)}
\Hom_{\cU_q}(V_\e, U_\e^{fin, \uchi}) \cong \Hom_{\cU_q}(V_q, U_q^{fin, \uchi})/\hbar \Hom_{\cU_q}(V_q, U_q^{fin, \uchi})
\end{equation}
Recall that $U_q^{fin, \uchi}=\widehat{U}/\bigcap_k \hbar^k \widehat{U}$ where $\widehat{U}:=U_q^{fin}\otimes_{\CW_q}\CW_q^{\wedge_{\uchi}}$. 
We consider the following diagram 
\[\begin{tikzcd}
     0 \arrow[r]&\Hom_{\cU_q}(V_q, \widehat{U}) \arrow[d] \arrow[r, hook, " \cdot \hbar"] & \Hom_{\cU_q}(V_q, \widehat{U}) \arrow[d] \arrow[r, two heads, " \pi"] & \Hom_{\cU_\e} (V_\e, U_\e^{fin, \uchi})\arrow[d, "\cong"]\arrow[r] &0\\
     0\arrow[r]& \Hom_{\cU_q}(V_q, U_q^{fin, \uchi}) \arrow[r, hook,  "\cdot \hbar"] & \Hom_{\cU_q}(V_q, U_q^{fin, \uchi}) \arrow[r] & \Hom_{\cU_\e} (V_\e, U_\e^{fin, \uchi}) \arrow[r]& 0 
\end{tikzcd}
\]
The first row is exact since $U_q^{fin}$ has an exhaustive good filtration, $V_q$ is tilting and $\CW_q^{\wedge_{\uchi}}$ is flat over $\CW_q$. Therefore, the second row is also exact.  This proves \eqref{eq: Hom(V, Ufin)}. 

\noindent
(b) Follow by part (a) and Proposition \ref{prop: embed of Hilt}.
\end{proof}

 \begin{defi}\label{defi: Hq and He}
 Let $\CH_q$ be the image of $\SB_\hbar$ in $\Hilt_q(0,0)$. Let $\CH_\e$ be the image of $\SB$ in $\Hilt_\e(0,0)$.
 \end{defi}


\section{Simple Harish-Chandra bimodules}  \label{sec: simple HC}


We are under the assumption \ref{eq: assumption on l}. We will classify simple Harish-Chandra bimodules in order to study projective objects in $\HC_\e(0,0)$ and $\HC_q(0,0)$. 

To simplify the notation, we replace the tensor product $-\otimes_{U_\e^{fin, ?}}-$ by $-\star-$.  We write $\Rep(G_\e)$ for $\Rep^{fd}(\cU_\e(\g))$. Under the assumption \ref{eq: assumption on l}, $\Rep^{fd}(\cU^*_\e(\g))$ is $\Rep(G)$, the category of finite dimensional  $G$-representations. Recall the Frobenius  functor $\tFr^*: \Rep(G) \rightarrow \Rep(G_\e)$.

Let $\chi_u$ be a regular unipotent element in $G_0$, the open Bruhat cell in $G$. Let $\HC_u$ denote the category $U_\e^{fin, \uchi_u}\rmod^{G_\e}$. Under the identification $Z_\cap \cong \BC[T/W]$ in Proposition \ref{prop: description of Z} (note that under the assumption \ref{eq: assumption on l}, $T^d \cong T$), the completion $Z_\cap^{\wedge_{\uchi_u}}$ corresponds to the completion of $\BC[T/W]$ at the point $1$.

\begin{Lem}\label{lem: diagonal bimodule and Rep(G)} For any diagonal bimodule $D \in \HC_\e(\mu, \lambda)$ and $M\in \Rep_{[\lambda]}(G_\e)$, $V\in \Rep(G)$, we have
\[ D\star (\tFr^*(V) \otimes M) \cong \tFr^*(V) \otimes(D \star M) \]   
\end{Lem}
\begin{proof}
    It is enough to prove this  when  $D=P^{\mu, \lambda}_\e(N)$ for some $N\in \Rep(G_\e)$. In this case, 
\[P_\e^{\mu, \lambda}(N)\star(\tFr^*(V)\otimes M) \cong \pr_{[\mu]}(N\otimes \tFr^*(V)\otimes M) \cong \tFr^*(V) \otimes \pr_{[\mu]}(N\otimes M)
\cong \tFr^*(V) \otimes (P_\e^{\mu, \lambda}(N)\star M).\]
\end{proof}

\subsection{Left (right)-trivial Harish-Chandra bimodules}\

Let us recall the longest element $w_0$ in $W$. The action of $-w_0$ on $P$ induces an action of $-w_0$ on $P/(W_{ext}, \bullet_\ell)$ since if $\lambda=wt_\mu \bullet_\ell \lambda'$ then $-w_0 \lambda= w_0ww_0t_{-w_0\mu}\bullet_\ell(-w_0 \lambda')$. For any $\lambda \in P$ let $\lambda^*=-w_0\lambda$.

Recall the map $\iota: U_\e^{ev}(\g) \rightarrow \cU_\e(\g)$. Let $\varepsilon: U_\e^{ev}(\g)\rightarrow \BC$ be the counit of $U_\e^{ev}(\g)$.
\begin{defi}[Definition/Lemma] Let $V\in \Rep_{[\lambda]}(G_\e)$

\noindent
(a)  $V$ can be viewed as an object in $\HC_\e(\lambda, 0)$ as follows:
\[ uv=\iota(u) v, \qquad vu =\varepsilon(u)v,\]
here $u \in U_\e^{fin} \subset U_\e^{ev}$ and $v\in V$. We call this bimodule structure on $V$ the {\it right-trivial Harish-Chandra bimodule} and denote it by $V^r$.

\noindent
(b) $V$ can be viewed as an object in $\HC_\e(0, \lambda^*)$ as follows:
\[ uv=\varepsilon(u) v, \qquad vu =\iota(S^{-1}(u))v,\]
here $u \in U_\e^{fin}$ and $v\in V$. We call this bimodule structure on $V$ the {\it left-trivial Harish-Chandra bimodule} and denote it by $V^{l}$.
\end{defi}
\begin{proof}We will explain why $V^l \in \HC_\e(0, \lambda^*)$, other parts are similar.  Let $u \in \CW_\e$, the Harish-Chandra center. It is enough to show that  $(*) ~V^l (u -\chi_{\lambda^*}(u))^k=0$ for some $k$. Note that $S^{-1}(u)$ still is  a central element in $U_\e^{ev}(\g)$. Furthermore, both $\iota(u)$ and $\iota(S^{-1}(u))$ are central in $\cU_\e(\g)$, which is proved by using the computation $(10.6)$ in \cite{LTV} and equality $\ad'_l(x)(\iota(u))=\varepsilon(x) \iota(u)$ for all $x\in \cU_\e(\g)$. On the other hand, $u=u_0+\sum_{\a\in Q_+} u_\a$ where $u_\a \in U_{\e,-\a}^{ev <} U_\e^{ev 0} U_{\e, \a}^{ev >}$ and $u_0\in U_\e^{ev, 0}$. Since $V\in \Rep_{[\lambda]}(G_\e)$, it admits a finite filtration whose successive quotients  are of the form $L_\e(\mu)$ for $\mu \in W_{ext} \bullet_\ell \lambda$. We then see that $\iota(S^{-1}(u))$ acts on $L_\e(\mu)$ as a multiplication by $\chi_{w_0\mu}(S^{-1}(u_0))=\chi_{\lambda^*}(u_0)=\chi_{\lambda^*}(u)$. This implies $(*)$ and finishes the proof.
\end{proof}
\begin{Lem}\label{lem: left-right trivial bimodules}
(a) The following functors are fully faithful:
\begin{equation*}
\bullet^r~:~ \Rep_{[\lambda]}(G_\e) \rightarrow \HC_\e(\lambda, 0) \qquad  V\mapsto V^r; \qquad \qquad \bullet^l ~:~ \Rep_{[\lambda]}(G_\e) \rightarrow \HC_\e(0, \lambda^*)\qquad V\mapsto V^l.
\end{equation*}


\noindent
(b) For any $V\in \Rep(G)$ then $\tFr^*(V)^r \cong \tFr^*(V)^l \in \HC_\e(0,0)$.
\end{Lem}

The proof is straightforward, part (b) needs to use the assumption \ref{eq: assumption on l}. The next lemma  concerns the action of diagonal bimodules on left (right)-trivial bimodules.
\begin{Lem}\label{lem: translation and left-right trivial bimod}
(a) For any $V \in \Rep(G_\e)$ and $V_1, V_2 \in \Rep_{[\lambda]}(G_\e)$ we have 
\begin{align*} P_\e^{\mu, \lambda}(V) \star V_1^r &\cong (\pr_{[\mu]}(V\otimes V_1))^r \cong (P_\e^{\mu, \lambda}(V) \star V_1)^r, \\
V_2^l \star P_\e^{\lambda^*, \mu}(V) &\cong (\pr_{[\mu^*]}(V\otimes V_2))^l\cong (P_\e^{\mu^*, \lambda}(V)\star V_2)^l.
\end{align*}

\noindent
(b) For any $V\in \Rep_{[\lambda]}(G_\e)$, we have 
\[ P_\e^{\mu, \lambda}\star V^r \cong (P_\e^{\mu, \lambda} \star V)^r, \qquad V^l \star P_\e^{\lambda^*, \mu^*}\cong (P_\e^{\mu, \lambda} \star V)^l.\]
\end{Lem}

\begin{proof}(a) The first isomorphism follows from
\[ V\otimes V_1^r \cong (V\otimes V_1)^r \cong \left(\bigoplus_{[\mu]} \pr_{[\mu]}(V\otimes V_1)\right)^r \cong \left(\bigoplus_{[\mu]} P_\e^{\mu, \lambda}\star V\right)^r\]
The proof for the second isomorphism is similar.

\noindent
(b) By definition $P_\e^{\mu, \lambda}=P_\e^{\mu, \lambda}(T_\e(\mu|\lambda))$ where the tilting module $T_\e(\mu|\lambda)$ is defined in Definition \ref{defi: Weyl(mu-lambda)}. Note that $T_\e(\lambda^*|\mu^*)=T_\e(\mu|\lambda)$. Hence part (b) is a special case of part (a).
\end{proof}

 \subsection{Hecke action on $\Rep_{[0]}(G_\e)$}\ 
 

Let $\text{Tilt}_{[0]}(G_\e)$ be the subcategory of tilting modules in $\Rep_{[0]}(G_\e)$. Recall the Steinberg module $\St_\e=W_\e((\ell-1)\rho)$. Let 
\begin{equation} (P/\ell P)_+ :=\{ \lambda \in P~|~0\leq (\lambda, \a_i^\vee) \leq \ell-1 ~\forall ~ 1\leq i \leq r\}
\end{equation}

\begin{Rem}\label{rem: rho block}The block $\Rep_{[-\rho]}(G_\e)$ is semisimple with simple objects $\St_\e \otimes \tFr^*(V)$ for any irreducible module $V \in \Rep(G)$. 
\end{Rem}
Recall the tilting module $T_\e(\mu|\lambda)$ defined in Definition \ref{defi: Weyl(mu-lambda)}.  
 \begin{defi}
 (a) Let $\mu, \lambda$ be in the closure of the fundamental alcove $C$. The {\em translation functor} $  T_{\lambda}^\mu: \Rep_{[\lambda]}(G_\e) \rightarrow \Rep_{[\mu]}(G_\e)$ is defined by $V \mapsto \pr_{[\lambda]}(T_\e(\mu|\lambda) \otimes V)\cong P_\e^{\mu, \lambda}\star V$ for any $V \in \Rep_{[\lambda]}(G_\e)$.

 \noindent
 (b) For $s\in S_{aff}$, let $\lambda_s$ be a weight contained in the facet associated to $s$ in the closure of the fundamental alcove. The {\em reflection functor}  $\Theta_s: \Rep_{[0]}(G_\e) \rightarrow \Rep_{[0]}(G_\e)$ is defined by $\Theta_s:=T_{\lambda_s}^0 \circ T_{0}^{\lambda_s}$.
 \end{defi}

 \begin{Rem}For $x \in \Lambda$, the element $x\bullet_\ell 0$ belongs to the fundamental alcove under the assumption \ref{eq: assumption on l}. Therefore, we have the following endofunctors of $\Rep_{[0]}(G_\e)$: $T_0^{x\bullet_\ell 0}, T_{x\bullet_\ell 0}^0, \Theta_s$ for $x\in \Lambda, s\in S_{aff}$. With definition of $\CH_\e$  in Definition \ref{defi: Hq and He} and the construction of objects in Theorem \ref{thm: Sb-hbar embed}, we obtain an action of $\SB$ on $\Rep_{[0]}(G_\e)$ via the above endofunctors.
 \end{Rem}
\begin{Rem}Since direct summands of tilting modules are tilting and tensor products of tilting modules are tilting, it is clear that the action of $\SB$ on $\Rep_{[0]}(G_\e)$ preserves $\Tilt_{[0]}(G_\e)$.
\end{Rem}

\begin{Lem}(a) There is an equivalence of categories:
\[ \mathcal{D}^{asph}\cong \Tilt_{[0]}(G_\e).\]

\noindent
(b) We have an isomorphism of right $\BC[W_{ext}]$-modules:  $M^{asph} \iso K_0(\Rep_{[0]}(G_\e)) $ defined by $1\otimes H_x \mapsto [W_\e(x\bullet_\ell 0)] =[H^0_\e(x\bullet_\ell 0)]$. Under this isomorphism, $1\otimes \uH_x \mapsto [T_\e(x\bullet_\ell 0)]$.
\end{Lem}
\begin{proof}(a) With the embedding $\SB\rightarrow \Hilt_\e(0,0)$, the proof follows the strategy in \cite[$\mathsection 5$]{RW15}.

\noindent
(b) Follow by  part (a) via taking $K$-theory and using $K_0(\Rep_{[0]}(G_\e)) \cong K_0(\Tilt_{[0]}(G_\e))$.
\end{proof}

\begin{Lem}\label{lem: translation 0 to -rho} The map $T:=T_0^{-\rho}: K_0(\Rep_{[0]}(G_\e)) \rightarrow K_0(\Rep_{[-\rho]}(G_\e))$ can be identified with the following quotient  map as $\BC$-linear maps
\begin{equation*}\label{eq: T(0 to -rho)} M^{asph} \twoheadrightarrow M_W^{asph}:=M^{asph}/\< mg-m~|~g\in W\>.
\end{equation*}
\end{Lem}
Let $\Delta(\lambda):=W_\e(\lambda)$ and $\nabla(\lambda) :=H^0_\e(\lambda)$ for dominant weight $\lambda$. If $\lambda$ is not dominant then $\nabla(\lambda) =\Delta(\lambda)=0$.  Denote $x \bullet_\ell \lambda$ by $x \cdot \lambda$ for $x\in W_{ext}$ and $\lambda \in \h^*$.

In the below proof, the computations of images of various $\nabla(\lambda)$ under the translation functors and reflection functors are done by the strategy in \cite[Proposition 7.13]{J03}, we omit the details.
\begin{proof}
Under the assumption \ref{eq: assumption on l}, $0$ is contained in the fundamental alcove.

First, we  note that $K_0(\Rep_{[0]}(G_\e))$ has a $\BC$-basis $\{[\nabla(t_\mu w\cdot 0)]~|~ t_\mu w \cdot 0 \in P_+\}$. The condition $t_\mu w \cdot 0 \in P_+$ implies that $\mu$ is dominant. Indeed, otherwise, there is  a finite simple reflection $s\in S$ such that $\< \mu, \a_s^\vee\> \leq -1$ then $\< t_\mu w \cdot 0, \a_s^\vee\>= \ell \< \mu, \a_s^\vee\>+\< w\cdot 0, \a_s^\vee\> <0$ since $\< w\cdot 0, \a_s^\vee\> \in [-\ell-1, \ell-1]$ due to $0$ is contained in the fundamental alcove.

\noindent
{ \it Step 1:} We will show that $T([V]g)=T([V])$ for $g\in W$ and $V\in \Rep_{[0]}(G_\e)$.

By definition, $[V](s+1)=[\Theta_s(V)]$ for a finite simple reflection $s\in S$.  It suffices to show that $T([\Theta_s\nabla(x \cdot 0)]=2T([\nabla(x\cdot 0)])$ for all $ x=t_\mu w$ such that $t_\mu w \cdot 0 \in P_+$

If $x= t_{\mu}w$ with $\mu$ not regular then $t_\mu w \cdot (-\rho)$ is not dominant, hence,  $T([\nabla(x \cdot 0)])=T(\nabla(xs \cdot 0))=0$. But $\Theta_s\nabla(x\cdot 0)$ has a filtration whose subquotients are of the form $\nabla(xs \cdot 0)$ or $\nabla(x \cdot 0)$, hence $T(\Theta_s\nabla(x \cdot 0))=0=2 T(\nabla(x\cdot 0))$.

If $x=t_\mu w$ with $\mu$ regular, we want to show that $[\Theta_s\nabla(x \cdot 0)]=[\nabla(x \cdot 0)]+[\nabla(xs \cdot 0)]$ because this implies $T([\Theta_s\nabla(x\cdot 0)])=2[\nabla(\ell\mu-\rho)]=2T([\nabla(x\cdot 0)])$. It is required to show that $x\cdot \lambda_s$ and $xs \cdot 0$ are dominant.
\begin{itemize}
    \item $t_\mu w \cdot \lambda_s=\ell  \mu -\rho+ w(\lambda_s+\rho)$. Note that $\<w( \lambda_s+\rho), \a_i^\vee\>=\< \lambda_s+\rho, w^{-1}(\a_i^\vee)\>\in [-\ell+1, \ell-1]$ for a finite simple reflection $s\in S$ by Remark \ref{rem: not belong to affine hyper} .Therefore, $t_\mu w \cdot \lambda_s$ is dominant.
    \item $t_\mu ws \cdot 0=\ell \mu -\rho+ws(\rho)$. Since $0$ is contained in the fundamental alcove, $\< ws (\rho), \a_i^\vee\>=\< \rho, (ws)^{-1}(\a_i^\vee)\> \in [-\ell+1, \ell-1]$. Therefore, $xs \cdot 0$ is dominant.
\end{itemize}

\noindent
{\it Step 2:} Recall that $K_0(\Rep_{[0]}(G_\e))$ has a $\BC$-basis $\{[\nabla(t_\mu w\cdot 0)]~|~ t_\mu w \cdot 0 \in P_+\}$. The lemma clearly follows by the following claims:
\begin{enumerate}[label=(\arabic*)]
\item  The set $\big\{T([\nabla(t_\mu w \cdot 0)])=[\nabla(\ell \mu-\rho)]\big\}$ (with all regular dominant $\mu$ such that $t_\mu w \cdot 0$ dominant) forms a basis of $K_0(\Rep_{[-\rho]}(G_\e))$.
\item $T([\nabla(t_\mu w\cdot 0])=0$ for $\mu$ dominant but not regular such that $t_\mu w \cdot 0$ is dominant. 
\item $1\otimes t_\mu w \in K=\< mg-m\>$ for $\mu$ dominant but not regular such that $t_\mu w \cdot 0$ dominant.
\end{enumerate}

Claim $(1)$ follows by Remark \ref{rem: rho block}. Let us prove Claim $(3)$.  Note that 
\[ M_W^{asph}=\mathsf{sign}\otimes_{\BC[W]} \BC[W_{ext}]\otimes_{\BC[W]} \mathsf{triv}.\]
If $\mu$ is dominant but not regular then there is a simple reflection $s\in S$ so that $s\mu=\mu$. Then $t_\mu s=st_\mu$. Then in $M_W^{asph}$, we have
\[ 1_{\mathsf{sign}} \otimes t_\mu  \otimes 1_{\mathsf{triv}}=1_{\mathsf{sign}}\otimes t_\mu s \otimes 1_{\mathsf{triv}}=1_{\mathsf{sign}} \otimes s t_\mu \otimes 1_{\mathsf{triv}}=-1_{\mathsf{sign}}\otimes t_\mu \otimes 1_{\mathsf{triv}},\]
Hence $1_{\mathsf{sign}}\otimes t_\mu \otimes 1_{\mathsf{triv}}=0$, and then $1_{\mathsf{sign}} \otimes t_\mu w \otimes 1_{\mathsf{triv}}=1_{\mathsf{sign}} \otimes t_\mu \otimes 1_{\mathsf{triv}}=0$ for all $w\in W$. This proves Claim $(3)$. The proof completes.
\end{proof}
\begin{Cor}\label{cor: Hecke action on Rep(Ge)}
(a) For any simple object $L_\e(\lambda) \in \Rep_{[0]}(G_\e)$, there is $P\in \CH_\e$ such that $P_\e^{-\rho, 0} \star P \star L_\e(\lambda) \neq 0$.

\noindent
(b) For any simple object $L_\e(\lambda) \in \Rep_{[0]}(G_\e)$, there is $P\in \CH_\e$ such that $L_\e(\lambda)$ is an composition factor of $P \star \BC$, here $\BC$ is the trivial module.

\noindent
(c) For any $V\in \Rep(G)$, there is $P \in \CH_\e$ such that $\BC$ is a composition factor of $P\star \tFr^*(V)$.
\end{Cor}
\begin{proof}
(a) {\it Step 1:} $\BC[P]$ is a right $\BC[W_{ext}]$-module as follows: $t_\mu H_{g}=t_{g^{-1}(\mu)}$ and $t_\mu H_{t_\lambda}=t_{\mu+\lambda}$ for $g\in W$ and $\mu, \lambda \in P$. On the other hand, the right $W$-representation $\mathsf{sign}$ can be naturally viewed as a right $\BC[W_{ext}]$ representation via the quotient $W_{ext}\twoheadrightarrow W$. Then by tensor-product action, $\mathsf{sign}\otimes \BC[P]$ becomes a right $\BC[W_{ext}]$-module. 

Then we have an isomorphism of right $\BC[W_{ext}]$-modules:  
\[ \phi: \mathsf{sign}\otimes \BC[P]\iso M^{asph} \qquad \text{defined by}~1\otimes t_\mu \mapsto 1 \otimes H_{t_\mu} ~\text{for}~\mu \in P.\]

{\it Step 2:} The object $[L_\e(\lambda)]$ is nonzero in $K_0(\Rep_{[0]}(G_\e)) \iso M^{asph}$, hence we can find a nonzero $x\in \BC[P]$ such that $\phi(1\otimes x)=[L_\e(\lambda)] \in M^{asph}$. Since $\BC[P]$ is an integral domain, the following element is nonzero in $\BC[P]$:
\[ A= \left(\prod_{g\in W} xg\right) \sum_{g\in W} \mathsf{sign}(g)t_{g(\rho)}\]
Moreover, 
\[ 1\otimes A \in (1\otimes x) \BC[W_{ext}] \cap (\mathsf{sign}\otimes \BC[P])^W,\]
here $(\mathsf{sign}\otimes \BC[P])^W$ is the $W$-invariant part, and $(1\otimes x) \BC[W_{ext}]$ is the right $\BC[W_{ext}]$-submodule of $\mathsf{sign}\otimes \BC[P]$. Therefore,
\[\phi(1\otimes A) \in [L_\e(\lambda)]\BC[W_{ext}]\cap (M^{asph})^W.\]
By Lemma \ref{lem: translation 0 to -rho}, $T_0^{-\rho}([L_\e(\lambda)]\BC[W_{ext}])$ contains a nonzero element $T_0^{-\rho}(\phi(1\otimes A))$.  Therefore, we can find $P \in \CH_\e$ such that $[T_0^{-\rho}(P\star L_\e(\lambda))] \neq 0$, hence $P_\e^{-\rho, 0}\star P \star L_\e(\lambda)=T_0^{-\rho}(P\star L_\e(\lambda)) \neq 0$. 

\noindent
(b) Since $[\BC]$ generates the $\BC[W_{ext}]$-module $M^{asph}$, there is $P\in \CH_\e$ such that $P\star \BC$ contains the indecomposable tilting module $T_\e(\lambda)$ as a direct summand. This implies part (b).

\noindent
(c) We can assume $V$ is simple, then by part (b), there is $P \in \CH_\e$ such that $\tFr^*(V^*)$ is a composition factor of $P \star \BC$. Then $\tFr^*(V^*) \otimes \tFr^*(V)$ is a subquotient of $\tFr^*(V) \otimes (P\star \BC)\cong P \star \tFr^*(V)$. This implies that $\BC$ is a composition factor of $P \star \tFr^*(V)$.
\end{proof}

\subsection{Simple Harish-Chandra bimodules}\

In this section, the tensor product $\otimes$ is the tensor product over $\BC$. 

Recall the small quantum group $\fu_\e$ is defined as the Hopf subalgebra of $\cU_\e(\g)$ generated by $\{ \tE_i, \tF_i, K^{2\lambda}\}_{1\leq i \leq r}^{\lambda \in P}$. Let $\uu$ be the quotient algebra of $U_\e^{ev}$ at the point $1 \in  G^d_0 \cong \Spec Z_{Fr}$, which is the same as the quotient of $U_\e^{fin}$ at the point $1$. Then we see that $\uu$ is the quotient  of $\fu_\e$ by the two-sided ideal generated by $I:=\{ K^{2\ell \lambda}-1\}_{\lambda \in P}$. Since $I$ is a Hopf ideal in $\fu_\e$, we see that $\uu$ is also a Hopf algebra. Furthermore, the $\cU_\e(\g)$-action on $U_\e^{ev}$ induces an action of $\cU_\e(\g)$ on $\uu$.

\begin{Rem}\label{rem: irred u-mod}By \cite[Proposition 5.11]{L90} and the assumption \ref{eq: assumption on l} on $\ell$, the irreducible $\uu$-modules are parametrized by $P/\ell P$. Any $V\in \Rep(G_\e)$ is a $\fu_\e$-module by restriction, and then under the assumption \ref{eq: assumption on l} that $\fu_\e$-action on $V$ factors through an action of $\uu$. The collection of $\cU_\e(\g)$-modules $\{L_\e(\lambda)~\}_{ \lambda \in (P/\ell P)_+}$, viewed as $\uu$-modules as above, gives all irreducible $\uu$-modules.
\end{Rem}
\begin{Lem} The Harish-Chandra center gives a decomposition  of categories 
\[ \Rep(\uu)=\bigoplus_{[\lambda] \in P/(W_{ext},\bullet_\ell) } \Rep_{[\lambda]}(\uu).\]
Furthermore, all simple modules in $\Rep_{[\lambda]}(\uu)$ can be obtained  from simple modules $L_\e(\lambda)\in \Rep_{[\lambda]}(G_\e)$ with $ \lambda \in (P/\ell P)_+$ as in Remark \ref{rem: irred u-mod}.
\end{Lem}

\begin{defi}Let $\uu \bimod^{G_\e}$ be the category of $\cU_\e$-equivariant  $\uu$-bimodules with rational $\cU_\e(\g)$-actions.
\end{defi}

By the Hopf algebra morphism $\fu_\e \twoheadrightarrow \uu$, any $\uu$-bimodule $V$ carries the  adjoint action of $\fu_\e$ defined by $x \cdot m=\sum x_{(1)}m S(x_{(2)})$ for $x\in \fu_\e$ and $m \in V$.
\begin{defi} Let $\HC(\uu)$ be the full subcategory of $\uu \bimod^{G_\e}$ consisting of all objects on which the adjoint action of $\fu_\e$ coincides with the action obtained  via restriction from $\cU_\e(\g)$.
\end{defi}

\begin{Rem} \label{rem: objects in HC(u)} For $\lambda\in (P/\ell P)_+$, the left (right) trivial HC-bimodules $L_\e(\lambda)^l, L_\e(\lambda)^r$ belong to $\HC(\uu)$ by Remark \ref{rem: irred u-mod}. Let $\lambda_1, \lambda_2 $ be two weights in $(P/\ell P)_+$. Let $V \in \Rep(G)$. The following object belongs to $\HC(\uu)$: 
\[ L_\e(\lambda_1)^r \otimes L_\e(\lambda_2)^l \otimes \tFr^*(V)\]
 with the  $\cU_\e$-equivariant $\uu \otimes \uu^{op}$-module structure defined as follows:
 \begin{itemize}
     \item $\cU_\e$ acts via the tensor product action.
     \item for $x_1, x_2 \in \uu$ and $v_1 \in L_\e(\lambda_1)^l$, $v_2\in L_\e(\lambda_2)^l$, $ v\in \tFr^*(V)$ then 
     \[ x_1(v_1\otimes v_2\otimes v) x_2=(x_1 v_1)\otimes (v_2 x_2)\otimes v.\]
 \end{itemize}
Equivalently, $L_\e(\lambda_1)^r \otimes L_\e(\lambda_2)^l \otimes \tFr^*(V) \cong L_\e(\lambda_1)^r \otimes \Big(L_\e(\lambda_2)\otimes \tFr^*(V)\Big)^l$.
\end{Rem}
\begin{Prop}\label{prop: simple HCp= simple HCu}Simple Harish-Chandra bimodules in $\HC_u$ are in one-to-one correspondent with simple objects in $\HC(\uu)$. 
\end{Prop}
\begin{proof}Recall $\uchi_u$ be the image of $\chi_u$ under the map $\Spec Z_{Fr} \rightarrow \Spec Z_\cap$. Note that any $M \in \HC_u=U_\e^{fin, \uchi_u}\rmod^{G_\e}$ is a finitely generated module over $Z_{Fr}^{fin, \uchi_u}:= Z_{Fr}^{fin}\otimes_{Z_\cap} Z_{\cap}^{\wedge_{\uchi_u}}$. Let $\m_u$ be the maximal ideal of $Z_\cap$ at $\uchi_u$ then $Z_{Fr}^{fin}/Z_{Fr}^{fin}\m_u$ is the algebra of functions on the unipotent cone $U$ of $G$. Let $\Supp(M)$ be the support of $M$ in the unipotent cone $U$. Then $\Supp(M)$ is a closed subvariety of $U$. The group $G$ acts on $U$ by the conjugation action. For $x\in G$ and $g\in G$ we denote $g.x:=gxg^{-1}\in G$.

{\it Step 1:} We will show that $\Supp(M)$ is a closed $G$-stable subvariety of $U$. Let $x\in U \subset G$ and $\m_x$ be the maximal ideal of $x$ in $\BC[G]$. Assume $M/M \m_x \neq 0$, we want to show that $M/M \m_{g.x} \neq 0$ for all $g\in G$.

Any projective module in $\Rep^{fd}(\cU_\e)$ is also projective in $\Rep(\fu_\e)$. Moreover, simple objects in $\Rep(\fu_\e)$ can be obtained from certain simple modules in $\Rep^{fd}(\cU_\e)$ by restricting the actions from $\cU_\e(\g)$ to $\fu_\e$. Therefore, we can find a projective module $V\in \Rep^{fd}(\cU_\e)$ such that $\Hom_{\ad(\fu_\e)}(V, M/M \m_x) \neq 0$ and we have a short exact sequence
\begin{equation}\label{eq: fiber and taking u-inv} 0\rightarrow \Hom_{\ad(\fu_\e)}(V, M \m_{x'})\rightarrow \Hom_{\ad(\fu_\e)}(V, M) \rightarrow \Hom_{\ad(\fu_\e)}(V, M/M \m_{x'}) \rightarrow 0,
\end{equation}
for all $x'\in U$. In above discussion, we use that  $\fu_\e$ acts trivially on the center $Z(U_\e^{fin})$ of $U_\e^{fin}$.

This implies that $\Hom_{\ad(\fu_\e)}(V, M)$, which is an object in $Z_{Fr}^{fin, \uchi_u}\rmod^{G}$, has a nonzero fiber at $x$. Therefore, the fiber of $\Hom_{\fu_\e}(V,M)$ at $g.x$ is nonzero for all $g\in G$. Combining this with \eqref{eq: fiber and taking u-inv}, we have $\Hom_{\fu_\e}(V, M/M \m_{g.x}) \neq 0$, hence $M/M\m_{g.x} \neq 0$ for all $g\in G$.

{\it Step 2:} There is a fully faithful  functor $\HC(\uu) \rightarrow \HC_u$: let $U_\e^{fin}$ act on the right  of any $M \in \HC(\uu)$ through the quotient $U_\e^{fin} \twoheadrightarrow \uu$ then $M$ becomes an object in $\HC_u$. 

Let $\m_1$ be the maximal ideal of $Z_{Fr}^{fin}$ at $1 \in U$.  For any simple object $L$ in $\HC_u$, by Step 1, $\Supp(L) \subset U$ always contains $1 \in U$. Therefore, $L \m_1=0$.  On the other hand, the left and right action of $Z_{Fr}^{fin}$ on $L$ coincide by Lemma \ref{lem: left and right Zfin-actions coincide}. Hence $L$ is an object in $\HC(\uu)$.
\end{proof}

\begin{Lem}[follow from by Remark \ref{rem: irred u-mod}] \label{lem: simple modules of u xu}The set $\{ L_\e(\lambda_1)^r \otimes L_\e(\lambda_2)^l~|~(\lambda_1, \lambda_2) \in (P/\ell P)_+^{\oplus 2}\}$ classifies all simple $\uu \otimes \uu^{op}$-modules.
\end{Lem}

\begin{Lem}\label{lem: Classification of simple bimods} The set
\[ S:= \Big\{ L_\e(\lambda_1)^r \otimes L_\e(\lambda_2)^l \otimes \tFr^*(V) ~|~ \lambda_1, \lambda_2 \in (P/\ell P)_+, V\in \textnormal{Irr}(G)\Big\}\]
classifies all simple objects in $\HC(\uu)$.
\end{Lem}
\begin{proof} 
Let $M$ be a simple object in $\HC(\uu)$. Then there are $\lambda_1, \lambda_2\in (P/\ell P)_+$ such that 
\begin{equation}\label{eq: Hom(uu,M)} V:=\Hom_{\uu \otimes \uu^{op}}(L_\e(\lambda_1)^r\otimes L_\e(\lambda_2)^l, M) \neq 0.
\end{equation}
The space $V$ in \eqref{eq: Hom(uu,M)} is a $\cU_\e(\g)$-module so that the action of $\cU_\e(\g)$ on it  factors thorough the Frobenius morphism $\tFr: \cU_\e(\g) \rightarrow \cU_\BC(\g)$, see Remark \ref{rem: action on Hom} below.
 Let us consider the following object in $\HC(\uu)$ as defined in Remark \ref{rem: objects in HC(u)}:
 \[ A= L_\e(\lambda_1)^r\otimes L_\e(\lambda_2)^l \otimes \tFr^*(V).\]
 Then 
 \begin{align*}
     \Hom_{\HC(\uu)}(A,M) & \cong \Hom_{\uu \otimes \uu^{op}}(A, M)^{\cU_\e} \\
     &\cong (V\otimes \Hom_{\uu\otimes \uu^{op}}(L_\e(\lambda_1)^r\otimes L_\e(\lambda_2)^l, M)^*)^{\cU_\e}\\
     & \cong (V\otimes V^*)^{\cU_\e} \neq 0.
 \end{align*}
Since $M$ is simple, we must have $A\cong M$, hence $V$ is an irreducible $\cU_\BC(\g)$-representation otherwise $A$ will not be simple.

So we have that if $M$ is a simple object in $\HC(\uu)$ then 
\begin{itemize}
    \item There are $\lambda_1, \lambda_2 \in (P/\ell P)_+$ such that $V=\Hom_{\uu \otimes \uu^{op}}(L_\e(\lambda)^r\otimes L_\e(\lambda_2)^l, M)$ belongs to $\text{Irr}(G)$. Then $M \cong L_\e(\lambda_1)^r \otimes L_\e(\lambda_2)^r \otimes \tFr^*(V)$.
    \item The triple $(\lambda_1, \lambda_2, V)$ is uniquely recovered from $M$.
\end{itemize}
Therefore, the lemma follows if we can show that any object in $S$ is simple.

Let $L:=L_\e(\lambda)^r  \otimes L_\e(\lambda_2)^l \otimes \tFr^*(V)$ be in $S$ and $N$ be a simple suboject of $L$. Then $V':=\Hom_{\uu \otimes \uu^{op}}(L_\e(\lambda_1)^r\otimes L_\e(\lambda_2)^l, N) \in \text{Irr}(G)$ and $N \cong L_\e(\lambda_1)^r\otimes L_\e(\lambda_2)^l \otimes \tFr^*(V')$. On the other hand, we have an inclusion of $G$-representations  $V' \hookrightarrow V$, but $V$ is irreducible hence $V'=V$. Therefore $N=L$, equivalently, $L$ is simple.
\end{proof}
\begin{Rem} \label{rem: action on Hom} In the above proof, for $M,N \in \HC(\uu)$, we equip the space $\Hom_{\uu \otimes \uu^\op}(M,N)$ with a rational $\cU_\e$-module structure. There is a subtlety, see Appendix \ref{append: R-module H}. Nevertheless, 
\[ \Hom_{\uu \otimes \uu^\op}(M,N) \cong \Hom_{\uu\rmod^{\fu_\e}}(M,N)  \quad \text{in $\Hom_{\BC}(M,N)$.}\]
The space $\Hom_{\uu \rmod}(M,N)$ has a natural $\cU_\e(\g)$-action defined by $(h f)(m)=\sum h_{1}f(S({h_2})m)$ for $h \in \cU_\e(\g), f\in \Hom_{\uu \rmod}(M,N)$ and $m\in M$. Then 
\[\Hom_{\uu\rmod^{\fu_\e}}(M,N) \cong \Hom_{\uu \rmod}(M,N) ^{\fu_\e}\]
is closed under this action of $\cU_\e(\g)$ (the reason is that $\fu_\e$ is "normal" in $\cU_\e(\g)$), furthermore, this $\cU_\e(\g)$-action factors through $\cU_\BC(\g)$.
\end{Rem}
\begin{Cor}\label{cor: simple HC}Any simple object in $\HC(\uu)$ is of the form $L_\e(\lambda_1)^r \otimes L_\e(\lambda_2)^l$ for some dominant weights $\lambda_1, \lambda_2$.
\end{Cor}
\begin{proof}By Remark \ref{rem: objects in HC(u)}, for $V=L(\lambda)\in \text{Irr}(G)$ then $L_\e(\lambda)_1^r \otimes L_\e(\lambda_2)^l \otimes \tFr^*(V) \cong L_\e(\lambda_1)^r \otimes \Big(L_\e(\lambda_2)\otimes \tFr^*(V)\Big)^l \cong L_\e(\lambda_1)^r\otimes L_\e(\lambda_2+\ell \lambda)^l$.
\end{proof}
\begin{defi}\label{defi: equivalence on simples}Let $L_1, L_2$ be two simple objects in $\HC_\e(0,0)$. We say $L_1 \prec L_2$ if there are bimodules $P_1, P_2$ in $\CH_\e$ such that $L_1$ is the composition factor of the object $P_1 \star L_2 \star P_2$. We say that $L_1 \sim L_2$ if $L_1 \prec L_2$ and $L_2 \prec L_1$. 
\end{defi}
\begin{Rem}\label{rem: equivalence on simples} Simple objects in $\HC_\e(0,0)$ are also simple objects in $\HC_q(0,0)$. In Definition \ref{defi: equivalence on simples}, we can replace $P_1, P_2\in \CH_\e$ by $P_1, P_2\in \CH_q$ since the action of $\CH_q$ on $\HC_\e(0,0)$ factors through the action of $\CH_\e$.
\end{Rem}
\begin{Thm}\label{thm: simples are in one class} All simple objects in $\HC_\e(0,0)$ are equivalent $ \sim$ to each other.
\end{Thm}

\begin{proof}We will prove that any  simple object $L$ is equivalent to the trivial bimodule $\BC$. 

{\it Step 1:} We will show that $L \prec \BC$. By Corollary  \ref{cor: simple HC}, $L=L_\e(\lambda_1)^r \otimes L_\e(\lambda_2)^l$ for some $\lambda_1, \lambda_2 \in W_{ext} \bullet_\ell 0$. Note that $\BC \cong \BC^r \otimes \BC^l$. By Corollary \ref{cor: Hecke action on Rep(Ge)}.b, there are $P_1, P_2\in \CH_\e$ such that $L_\e(\lambda_1)^r$ is a composition factor of $P_1\star \BC^r$ and $L_\e(\lambda_2)^l$ is a composition factor of $\BC^l \star P_2$. Therefore, $L$ is a composition factor of $P_1 \star \BC \star P_2$, equivalently, $L \prec \BC$.

{\it Step 2:} We will show that $\BC \prec L$.  By Corollary \ref{cor: Hecke action on Rep(Ge)}.a, there are $P_1, P_2 \in \CH_\e$ such that 
\[ P_\e^{-\rho, 0} \star P_1 \star L_\e(\lambda_1)^r \neq 0, \qquad L_\e(\lambda_2)^l \star P_2 \star P_\e^{0,-\rho} \neq 0.\]
Hence, 
\begin{equation}\label{eq: eq1} P_\e^{-\rho, 0} \star P_1 \star L \star P_2 \star P_\e^{0, -\rho} \neq 0.
\end{equation}

{\it Step 2':} We will show that for any simple object $N$ in $\HC_\e(-\rho, -\rho)$, there is $P_3 \in \CH_\e$ such that $\BC$ is the composition factor of $P_3\star P_\e^{0, -\rho} \star N \star P_\e^{-\rho, 0} $. By Lemma \ref{lem: Classification of simple bimods} and Remark \ref{rem: rho block}, we can assume $N:=\St_\e^r \otimes \St_\e^l \otimes \tFr^*(V)$ for some $V\in \textnormal{Irr}(G)$. Then 
\[ P_\e^{0, -\rho} \star N \star P_\e^{-\rho, 0} \cong (T_{\rho \rightarrow 0} \St_\e)^r \otimes (T_{-\rho \rightarrow 0} \St_\e)^l \otimes \tFr^*(V) \]
Since $T_{-\rho \rightarrow 0}\St_\e$ has a subquotient $W_\e(\ell \rho)$, it has a composition factor $L_\e(\ell \rho)$. This implies that $L_\e(\ell\rho)^r \otimes L_\e(\ell \rho)^l \otimes \tFr^*(V)$ is a composition factor of $P_\e^{0, -\rho}\star N \star P_\e^{-\rho, 0}$.

On the other hand, by Lemma \ref{lem: left-right trivial bimodules}-\ref{lem: translation and left-right trivial bimod}
\[ L_\e(\ell \rho)^r \otimes L_\e(\ell \rho)^l \otimes \tFr^*(V)\cong \Big(\tFr^*\big(L(\rho)\otimes L(\rho) \otimes V\big)\Big)^r\]
By Corollary \ref{cor: Hecke action on Rep(Ge)}, there is $P_3\in \CH_\e$ such that $\BC$ is a composition factor of $P_3 \star \tFr^*\big(L(\rho) \otimes L(\rho) \otimes V\big)$. Hence $\BC$ is a composition factor of $P_3 \star P_\e^{0, -\rho} \star N \star P_\e^{-\rho, 0}$.

Now we can finish Step $2$. Pick a composition factor $N$ in \eqref{eq: eq1} and $P_3$ as in Step $2'$, we see that $\BC$ is a composition factor of 
\[ P_3 \star P_\e^{0, -\rho} \star P_\e^{-\rho, 0} \star P_1 \star L \star P_2 \star P_\e^{0, -\rho} \star P_\e^{-\rho, 0}.\]
This implies that $\BC \prec L$.
\end{proof}
\subsection{Projective Harish-Chandra bimodules and the category $\sC_{\hbar, 0}, \sC_0$}\

Let $\Proj(\HC_q(0,0))$ and $\Proj(\HC_\e(0,0))$ be the subcategories of projective objects in $\HC_q(0,0)$ and  $ \HC_\e(0,0)$, respectively. Recall the categories $\sC_{\hbar, 0}, \sC_0$ in Remark \ref{rem: smallest two-sided cell} and Remark \ref{rem: smallest twosided in SB}.
\begin{Rem}\label{rem: semi-monoid} The definition of  {\it monoidal categories} in the paper is defined in \cite[Chapter 2]{EGNO}. If we drop the unit axioms in the definition of monoidal categories, we get {\it semi-monoidal categories}. The categories $\sC_{\hbar, 0}, \sC_0, \Proj(\HC_\e(0,0))$ and $\Proj(\HC_q(0,0))$ are semi-monoidal categories.
\end{Rem}

\begin{Thm}\label{thm: smalles two-sided cell vs proj} The full embeddings of monoidal categories
\[ \SB_\hbar \rightarrow \Hilt_q(0,0), \qquad \SB \rightarrow \Hilt_\e(0,0),\]
restrict to equivalences of semi-monoidal categories
\[ \sC_{\hbar, 0} \cong \Proj(\HC_q(0,0)), \qquad \qquad \sC_0  \cong \Proj(\HC_\e(0,0)).\]
\end{Thm}

\begin{proof}
The statement for $\sC_{ 0}$ is deduced from the statement for $\sC_{\hbar, 0}$. We will now prove the statement for $\sC_{\hbar, 0}$.

{\it Step 1:} We will show that for any indecomposable projective objects $Q_1, Q_2$ in $\HC_q(0,0)$, there are $P_1, P_2\in \CH_q$ such that $Q_1$ is a direct summand of $P_1\star Q_2 \star P_2$.  By Zorn's lemma, there are simple quotients $Q_1\twoheadrightarrow L_{Q_1}$ and $Q_2\twoheadrightarrow L_{Q_2}$. By Theorem \ref{thm: simples are in one class} and Remark \ref{rem: equivalence on simples}, there are $P'_1, P'_2\in \CH_q$ such that $L_{Q_2}$ is a subquotient of $P'_1\star L_{Q_1} \star P'_2$.

Since $T_q(\lambda|\mu)^* \cong T_q(\mu|\lambda)$, by Lemma \ref{lem: adjoint}, there are $P_1, P_2 \in \CH_q$ such that
\[ \Hom_{\HC_q(0,0)}(P_1 \star Q_2 \star P_2, L_{Q_1}) \cong \Hom_{\HC_q(0,0)}(Q_2, P_1' \star L_{Q_1} \star P_2'),\]
which  is nonzero since $L_{Q_2}$ is a subquotient of $P'_1 \star L_{Q_1} \star P'_2$. On the other hand, $\HC_q(0,0)$ is Krull-Schmidt such that any simple object in this category admits a projective cover by Lemma \ref{lem: proj cover of simple}. Therefore, the projective cover $Q_1$ of $L_{Q_1}$ must be a direct summand of $P_1\star Q_2 \star P_2$.

{\it Step 2:} By Step 1 and Remark \ref{rem: smallest two-sided cell} about $\sC_{\hbar, 0}$, it suffices to show  that there is a projective object $Q$ in $\HC_q(0,0)$  such that $Q_\dag$ is in the category $\sC_{\hbar,0}$. This is done by showing that $U_q^{fin, -\rho}$ is projective in $\HC_q(-\rho, -\rho)$. Because then $P_q^{0,-\rho}\star P_q^{-\rho, 0}$ is projective in $\HC_q(0,0)$; moreover, $(P_q^{0,-\rho}\star P_q^{-\rho, 0})_\dag \cong \sR\otimes_{\sR^W} \sR=B_{w_0}$ is contained in $\sC_{\hbar, 0}$ by Proposition \ref{prop: image of translation bimods} and Remark \ref{rem: smallest two-sided cell}.

 Let $R:=\BC[[\hbar]]$. Recall that for  $V\in \Rep^{fd}(\cU_q(\g))$, we set $V^*=\Hom_R(V, R)$.

{\it Step 2.1:}  Recall the isomorphism $U_q^{fin} \cong O_q[G]$. For any $V \in \Rep^{fd}(\cU_q(\g))$, we have a morphism $V\otimes_R V^* \rightarrow O_q[G]$  in $\Rep(\cU_q(\g))$ defined by $v\otimes f \mapsto c_{f, K^{-2\rho} v}$. For $V$ free of finite rank over $R$ with a basis $\{ v_i\}$ and  the dual basis $\{ v_i^*\}$,  the element $c_V:=\sum_i c_{v^*_i, K^{-2\rho} v_i} \in O_q[G]^{\cU_q} \iso \CW_q \subset U_q^{fin}$. See \cite[$\mathsection 7.8$]{LTV} for detail.

    Let $V=W_q(\lambda)$, then the image  of $c_\lambda:= c_{W_q(\lambda)}$ under the Harish-Chandra morphism $\CW_q \cong R\Big[ K^{\pm 2w_1}, \dots K^{\pm 2 \w_r}\Big]^{W_\bullet}$ is 
    \[ \sum_{\mu \in P_{+, \lambda}} \text{rank}\big(W_q(\lambda)_\mu\big)\sum_{\mu'\in W\mu}q^{(\rho, 2\mu')}K^{2\mu'},\]
    here $P_{+, \lambda}$ is the set of dominant weights in $W_q(\lambda)$.

    We see that the evaluation of $c_{\lambda}$ at the point $-\rho \in \Spec \CW_\e$ is  equal to $\text{rank}_R({W_q(\lambda)})  \neq 0$. Hence $c_{\lambda}$ is an invetible element in $\CW_q^{\wedge_{-\rho}}$.

{\it Step 2.2:} Let $\{v_i\}$ be a basis of $\St_q$, then $\{ v^*_i\}$ and $\{ v_i^{**}\}$ be the dual bases in $\St_q^*$ and $\St_q^{**}$, respectively. We have the following morphisms in $\Rep(\cU_q(\g))$:
    \begin{gather*}
        \St_q^*\otimes_R \St_q^{**} \rightarrow O_q[G]\cong U_q^{fin} \\
        R\xrightarrow[]{\text{coev}_{\St_q^*\otimes_R \St_q}} (\St_q^*\otimes_R \St_q) \otimes_R (\St_q^*\otimes_R \St_q)^* \cong (\St_q^*\otimes_R \St_q)\otimes_R (\St_q^*\otimes_R \St_q^{**}) 
    \end{gather*}
    Combining these two morphisms and using the evaluation map $\St_q^*\otimes_R \St_q \rightarrow R$, we get
    \[ R\rightarrow \St_q^*\otimes_R \St_q \otimes_R U_q^{fin} \xrightarrow[]{\text{ev}_{\St_q}\otimes \text{Id}} U_q^{fin},\]
    where the image of $ 1\in R$ is $c_{\St_q^*} \in \CW_q\subset U_q^{fin}$. So we have the  composition of morphisms in $U_q^{fin}\rmod^{G_q}$
    \[ U_q^{fin} \rightarrow \St_q^*\otimes_R \St_q \otimes_R U_q^{fin}\xrightarrow[]{\text{ev}_{\St_q}\otimes \text{Id}} U_q^{fin},\]
    which then gives us the  composition of morphisms  in $U_q^{fin, -\rho} \rmod^{G_q}$
    \begin{equation}\label{eq: composition 1} U_q^{fin, -\rho}\rightarrow \St_q^*\otimes_R \St_q \otimes_R U_q^{fin, -\rho} \rightarrow U_q^{fin, -\rho},
    \end{equation}
     where the image of $1\in U_q^{fin, -\rho}$ is $c_{\St_q^*} \in \CW_q^{\wedge_{-\rho}} \subset U_q^{fin, -\rho}$. Since $c_{\St_q^*} $ is invertible in $\CW_q^{\wedge_{-\rho}}$ by Step 2.1, the composition \eqref{eq: composition 1} is an isomorphism, hence $U_q^{fin, -\rho}$ is a direct summand of $\St_q^*\otimes_R \St_q \otimes_R U_q^{fin, -\rho}$. The latter is projective in $U_q^{fin, -\rho}\rmod^{G_q}$ since $\St^*_q \otimes_R \St_q$ is projective in $\Rep(\cU_q(\g))$ by Proposition \ref{prop: proj in Rep(Uq(g))}. Therefore, $U_q^{fin, -\rho}$ is projective in $U_q^{fin, -\rho}\rmod^{G_q}$ and then projective in $\HC_q(-\rho, -\rho)$.

     This finishes the proof.
\end{proof}
\begin{Cor}\label{cor: C0 to HC} The equivalence $\sC_0 \cong \Proj(\HC_\e(0,0))$ gives us an equivalence of triangulated semi-monoidal categories
\[ \fI^-_1: K^-(\sC_0) \iso D^-(\HC_\e(0,0)).\]
\end{Cor}



\section{Non-commutative Springer resolution}

In this section, we recall the non-commutative Springer resolution $\SA$ constructed in \cite{BM13} then connect it to the category $\HC_\e(0,0)$.

Fix a $G$-invariant non-degenerate form on $\g$ and the corresponding $G$-equivariant isomorphism $\g \cong \g^*$. Let $\g^{reg}$ denote the subset of the regular elements in $\g$ as well as the corresponding subset in $\g^*$ under the isomorphism $\g \cong \g^*$.  Let us recall the Grothendieck-Springer resolution 
\[ \begin{tikzcd} & \tg \arrow[dl, "\pi"'] \arrow[dr] &\\
\g^* \arrow[dr] &&\h^*\arrow[dl]&\\
& \h^*/ W&   
\end{tikzcd}
\]
here, $\tg$ is a resolution of $\g^* \x_{\h^*/W} \h^*$ which is an isomorphism over $\g^{reg} \x_{\h^*/W} \h^*$.
Let $\tg^{reg} \subset \tg$ be the preimage of $\g^{reg} \x_{\h^*/W} \h^*$ and $\Stn:= \tg \times_{\g^*} \tg$ be the Steinberg variety.

For a variety $X$ with an action of $G$, let $D^b\Coh^G(X)$ be the bounded derived categories of $G$-equivariant coherent sheaves on $X$. 

\subsection{Convolution}\label{ssec: convolution}The category $D^b\Coh^G(\St)$ is a monoidal category where the monoidal structure comes from convolution: $\CF_1 *\CF_2:= \pr_{13*}(\pr_{12}^*(\CF_1) \otimes^L \pr_{23}^*(\CF_2))$, where $\pr_{12}, \pr_{23}, \pr_{13}$ are three projections $ \tg \x^R_{\g^*} \tg \x^R_{\g^*} \tg \rightarrow \tg \x_{\g^*} \tg$. This monoidal category acts on $D^b\Coh^G(\tg)$ on the left by $\CF: \CG \rightarrow \pr_{1*}(\CF \otimes^L \pr_2^* \CG)$ and on the right by $\CF: \CG \rightarrow \pr_{2*}(\pr_1^*(\CG) \otimes^L \CF)$. The usage of derived fiber product $\tg \x_{\g^*}^R \tg \x^R_{\g^*} \tg$ and Derived Algebraic Geometry in the construction of the convolution is discussed in \cite[$\mathsection 6.1$]{BR24}. We note that $\tg \x_{\g} \tg =\tg \x_{\g^*}^R \tg$.

\subsection{Extended braid group}\label{sec: extended braid group}

The reference is \cite[$\mathsection 0.2$]{R08}. Let $S$ be the set of simple reflections in $W$ and $S_{aff}$ be the set of simple reflections in $W_{aff}$. To the group $W_{aff}$, the braid group $B_{aff}$ is the group generated by $T_{w}, w\in W_{aff}$ subject to relation $T_{v} T_{w}=T_{vw}$ if $\ell(vw)=\ell(v)+\ell(w)$, here $\ell(v)$ is the length function on $W_{aff}$. The extended braid group $B_{ext}$ associated to $W_{ext}$ is defined similarly. We have $W_{ext}=\Lambda \ltimes W_{aff}$ and $B_{ext}=\Lambda \ltimes B_{aff}$. Here $\Lambda \subset W_{ext}$ is the finite abelian subgroup consisting of elements of length zero.

There is a canonical map $W_{aff}\rightarrow B_{aff}$ defined by sending $w$ to $T_w$.. This map is a section to the canonical map $B_{aff}\rightarrow W_{aff}$ but is not a group homomorphism.  There is similar map for $B_{ext}$.

For $x\in P$, let $x= x_1-x_2$ for dominant weights $x_1, x_2$  and set $\theta_x=T_{x_1}(T_{x_2})^{-1}$, which  only depends on $x$. Then $B_{ext}$ is generated by $\overline{s}, (s\in S)$ and $\theta_\lambda, (\lambda \in P)$ subject to the known relations, see \cite[$\mathsection 0.2$]{R08}. This presentation shows that  there is an anti-involution $\iota: B_{ext} \rightarrow B_{ext}$ such that $\iota(\overline{s})=\overline{s}$ for $s\in S$ and $\iota(\theta_\lambda)=\theta_\lambda$ for $\lambda \in P$.
\begin{Lem}\label{lem: property of iota}Let $w_0$ be the longest element in $W$.

\noindent
(a) If $w\in W$ then $\iota(\overline{w})=\overline{w^{-1}}$.

 \noindent
(b) Let $s_{\a_0}\in S_{aff}-S$ then $\iota(\overline{s}_{\a_0})=\overline{w_0}^{-1} \overline{s}_{\a_0} \overline{w_0}$.

\noindent
(c) Let $b=wt_\lambda \in \Lambda$ then $b'=w_0w^{-1}w_0 t_\lambda \in \Lambda$ and $\iota(\overline{b})=\overline{w_0}^{-1} \overline{b'} \overline{w_0}$

\noindent
(d) The conjugation by  $\overline{w_0}^{-1}$ permutes  the set $\{ \overline{s}~|~ s\in S\}$.
\end{Lem}
\begin{proof}See Appendix \ref{appendix: Lemma iota}.
\end{proof}

\subsection{Braid group actions and reflection functors}\label{ssec: braid group action}

In \cite{BM13}, the authors construct a (weak) homomorphism from the extended braid group $B_{ext}\rightarrow D^b\Coh^G(\St): b\mapsto K_b$ defined as follows:
\begin{itemize}
    \item For $\lambda \in P$, $\theta_\lambda$ corresponds to the direct image of the line bundle $\CO_{\tg}(\lambda)$ under the diagonal embedding.
    \item For a finite simple reflection $s_\a \in W$, $ \overline{s}_\a \in B_{ext}$ corresponds to the structure sheaf $\CO_{\Gamma_{s_\a}}$, where $\Gamma_{s_\a}\subset \St$ is the closure of the graph of the action of $s_\a$ on $\tg_{reg}$.
\end{itemize}
Then $B_{ext}$ weakly acts on $D^b\Coh^G(\tg)$ via $\CG \mapsto K_b * \CG$ and $B^{ op}_{ext}$ weakly acts on $D^b\Coh^G(\tg)$ via $\CG \mapsto \CG * K_b$. Recall that $B_{ext}= \Lambda \ltimes B_{aff}$.

On the other hand, there are reflection functors $\CR_\a,~ (\a \in I_{aff})$ constructed in \cite[$\mathsection 2.3$]{BM13}. Let $\CR_\a$ also denote  the kernel of $\CR_\a$. For a finite simple root $\a \in I$, let $P_\a \supset B$ be a minimal parabolic of type $\a$. Let $\mathfrak{p}_\a, \mathfrak{u}_\a$ be the Lie algebra of $P_\a$ and the unipotent radical. The Grothendieck resolution associated with $\a$ is 
\[ \tg_\a:= G \x_{P_\a} (\g/\mathfrak{u}_\a)^*\]
then we have a natural $G$-equivariant projective morphism $\tg \rightarrow \tg_\a$. Then 
\begin{equation}\label{eq: Ra}
\CR_\a :=\CO(\tg \x_{\tg_\a} \tg)
\end{equation}

For $\b\in I_{aff}-I$, there $x\in B_{ext}$ and $s_{\a}\in S$  such that $T_{s_\b}=b T_{s_\a}b^{-1}$. Fix such elements and set 
\[\CR_\b:=K_{b} * \CR_\a * K_{b^{-1}}.\]

Then we have the following distinguished triangles \cite[Lemma $2.3.1$]{BM13}:
\begin{equation}\label{eq: distinguished triangles} K_{(\overline{s}_\a)^{-1}} \rightarrow \CR_\a \rightarrow \CO_{\Delta\tg}\xrightarrow[]{+1} \quad \text{and} \quad \CO_{\Delta \tg} \rightarrow \CR_\a \rightarrow K_{\overline{s}_\a}\xrightarrow[]{+1}, \qquad \text{for $s_\a \in S_{aff}$}.
\end{equation}

\subsection{Tilting generator and exotic $t$-structure}\label{ssec: tilting generator} There is a tilting bundle $\CE$ on $\tg$ which satisfies the following \cite[$\mathsection 2.5$]{BM13}:
\begin{itemize}
     \item $\CE=\CO_{\tg} \bigoplus \oplus_{\mathcal{J}}  \CR_{\a_1}* \dots * \CR_{\a_k}*\CO_{\tg}$, here $\mathcal{J}=\{(\a_1, \dots, \a_k)\}$ is a finite collection of tuples with entries in $I_{aff}$.
     \item $\Ext^i_{\tg}(\CE, \CE)=0$ for all $i \neq 0$.
    \item Let $\SA:=\End(\CE)^{op}$. The functor $D^b\Coh^G(\tg) \rightarrow D^b(\SA\Mod^G)$ defined by $\CG \mapsto \RHom_{\tg}(\CE, \CG)$ is an equivalence of categories.
\end{itemize}
This equivalence (and its non-equivariant version) is a reason why we call $\SA$ the non-commutative Springer resolution. Note that $\SA$ is an algebra over $\BC[\g^*]\otimes_{\BC[\h^*/W]}\BC[\h^*]$.

The standard $t$-structure on $D^b(\SA \Mod^G)$ gives the {\em exotic $t$-structure} on $D^b\Coh^G(\tg)$, sometimes we will call it {\em $\CE~t$-structure}. We collect some properties of this $t$-structure 
\begin{Lem}\label{lem: exotic t-structure}
(a) The exotic $t$-structure is {\em braid positive}, i.e., the action of $K_{\overline{s}_\a}*-$ is right $t$-exact for all $\a \in I_{aff}$.

\noindent
(b) The action $\CR_{\a}*-$ for $\a \in I_{aff}$ and $K_b *-$ for $b\in \Lambda$ are exotic $t$-exact.

\noindent
(c) The left and right adjoint of the reflection functor $\CR_\a*-$ are both isomorphic to $\CR_\a*-$ for $\a \in I_{aff}$. The left and right adjoint of $K_b*-$ are both isomorphic to $K_{b^{-1}}*-$ for $b \in \Lambda$. 

\noindent
(d) $\SA$ is projective over $\BC[\g^*]$.

\noindent
(e) Let $L(\lambda)$ be the irreducible $G$-representation of the highest weight $\lambda \in P_+$. Then $\CO_{\tg}\otimes L(\lambda)$ is a direct summand of some object of the form $K_b* \CR_{\a_1}*\dots \CR_{\a_k} *\CO_{\tg}$ for some $b\in \Lambda$ and a tuple $\{\a_1, \dots, \a_k\}$ with entries in $I_{aff}$.
\end{Lem}
\begin{proof}Part (a) follows by \cite[Proposition 2.2.1]{BM13}. Part (b) and  part $1$ of (c)  follow by \cite[Remark 1.5.2, Corollary 2.3.1]{BM13}. Part $2$ of (c) holds since $K_b * K_{b^{-1}}\cong K_{b^{-1}}* K_b \cong \CO_{\Delta \tg}$ for $b\in \Lambda$. Part (d) is proved in \cite[Lemma 1.5.3]{BM13}. Part (e) is proved in \cite[Lemma 6.7, $\mathsection$ 6.3]{MR18}.
\end{proof}



\subsubsection{Tilting bundle $\CE^\vee$}  The vector bundle $\CE^\vee:=\mathcal{H}om_{\tg}(\CE, \CO_{\tg})$ is also a tilting generator of $\tg$. Let call the $t$-structure coming from $\CE^\vee$ by {\em transposed $t$-structure}, sometimes we call it the {\em $\CE^\vee~t$-structure}.  In \cite[$\mathsection 1.8$]{BM13}, the authors associate the alcoves with a set of $t$-structures on $D^b\Coh(\tg)$. The $\CE~t$-structure corresponds to the fundamental alcove, while the $\CE^\vee~t$-structure corresponds to the opposite of the fundamental alcove. Furthermore, the $\CE^\vee~t$-structure can be obtained from the $\CE~t$-structure by conjugation with $K_{\overline{w_0}^{-1}}$. For $\a in I_{aff}$, the convolution $K_{\overline{w_0}^{-1}\overline{s}_{\a} \overline{w_0}}*-$ is right $t$-exact with respect to the $\CE^\vee~t$-structure.

Let $\sigma: \St \rightarrow \St$ be the map defined by $(x, y)\mapsto (y,x)$. Then we have $\sigma^*(K_b)\cong K_{\iota(b)}$ for all $b\in B_{ext}$, here $\iota: B_{ext}\rightarrow B_{ext}$ is the anti-involution defined in Section \ref{sec: extended braid group}. In particular, $\sigma^*(K_{\overline{s}_\a})\cong K_{\overline{s}_\a}$ for $\a \in I$ and $\sigma^*(K_{\theta_\lambda})\cong K_{\theta_\lambda}$ for $\lambda \in P$.

For any $\a \in I_{aff}$, let $\CR^t_\a:=\sigma^*(\CR_\a)$ then $\CR^t_\a\cong \CR_\a$ for $\a \in I$. For any $b\in \Lambda$, let $K^t_b:=\sigma^*(K_b)$. Since $\sigma^*(K_b)\cong K_{\iota(b)}$ and Lemma \ref{lem: property of iota}, from \eqref{eq: distinguished triangles} we have the following distinguished triangles:
\begin{equation}\label{eq: transpose R-functor}
\begin{split}
&K_{(\overline{s}_\a)^{-1}} \rightarrow \CR^t_\a \rightarrow \CO_{\Delta\tg} \qquad  \qquad \CO_{\Delta\tg} \rightarrow \CR^t_\a \rightarrow K_{\overline{s}_\a} \qquad \text{for $\a \in I$}\\
&K_{\overline{w}^{-1}_0(\overline{s}_{\a_0})^{-1} \overline{w}_0} \rightarrow \CR^t_{\a_0} \rightarrow \CO_{\Delta\tg} \quad \quad \CO_{\Delta\tg} \rightarrow \CR^t_{\a_0} \rightarrow K_{\overline{w}^{-1}_0 \overline{s}_{\a_0} \overline{w}_0}  \qquad \text{for $\a_0\in I_{aff}-I$}
\end{split}
\end{equation}

Following the construction in \cite[$\mathsection 2$]{BM13}, there is a finite collection of tuples of affine simple reflections $\mathcal{J}'$ such that 
\[ \CE'=\CO_{\tg} \bigoplus \oplus_{\mathcal{J}'} \CR^t_{\a_1} *\dots *\CR^t_{\a_k} *\CO_{\tg}= \CO_{\tg} \bigoplus \oplus_{\mathcal{J}'} \CO_{\tg} *\CR_{\a_k}* \dots *\CR_{\a_1}\]
is a vector bundle and a tilting generator of $D^b\Coh^G(\tg)$ with respect to the transposed $t$-structure. 

\begin{Lem}\label{lem: CE' and CE dual} There are $m,n \in \BZ_{>0}$ such that $\CE'$ is a direct summand of $(\CE^\vee)^{\oplus m}$ and $\CE^\vee$ is a direct summand of $(\CE')^{\oplus n}$.
\end{Lem}
\begin{proof}Since $\CE^\vee$ and $\CE'$ are both tilting generators for the transposed $t$-structure of $D^b\Coh^G(\tg)$, these two vector bundles are {\em equidecomposable} in the sense of \cite[$\mathsection 1.4.2$]{BM13}. Here, two objects $M_1, M_2$ of an additive categories are called {\em equidecomposable} if for $k=1,2$, we have $M_k \cong \bigoplus N_i^{\oplus d^i_k}$ for some $N_i$ and $d^i_k >0$. This implies the lemma.    
\end{proof}
Similarly to Lemma \ref{lem: exotic t-structure} and note that $\sigma^*(\CF)* \CG \cong \CG * \CF$ for $\CF \in D^b\Coh^G(\St)$ and $\CG\in D^b\Coh^G(\tg)$, we have the following lemma
\begin{Lem}\label{lem: tranposed t-structure}
(a)The action $-*K_{\overline{s}_\a}$ is right $t$-exact for all $\a\in I_{aff}$ with respect to the transposed $t$-structure.

\noindent
(b) The action $-*\CR_\a$ for $\a \in I_{aff}$ and $-* K_b$ for $b\in \Lambda$ are tranposed $t$-exact.

\noindent
(c) The left and right adjoint of the functor $-*\CR_\a$ are both isomorphic to $-* \CR_\a$ for $\a \in I_{aff}$. The left and right adjoint of the functor $-*K_b$ are both isomorphic to $-* K_{b^{-1}}$ for $b\in \Lambda$. 
\end{Lem}

\subsection{Results on the Steinberg variety}\label{ssec: result on Steinberg}

The extended affine Weyl group of the group $G \x G$ is $W_{ext} \x W_{ext}$. Let us consider the version of Section \ref{ssec: braid group action}-\ref{ssec: tilting generator} for the group $G \x G$ in which the tilting bundle is given by $\CE \boxtimes \CE^\vee$. Then by applying the base change \cite[$\mathsection 1.3$]{BM13} via the diagonal embedding $\g^*\xrightarrow[]{\Delta} \g^* \x \g^*$ to the Grothendieck -Springer resolution $\tg \x \tg \rightarrow \g^*\x \g^*$, we obtain the following equivalence of monoidal categories:

\[ \Gamma_{\CE}: D^b \Coh^G(\St) \rightarrow D^b(\SA\otimes_{\BC[\g^*]}\SA^{op}\Mod^G)\qquad \CG \mapsto \RHom_{\St}(\CE \boxtimes \CE^\vee|_{\St}, \CG)\]

\begin{Rem} The standard $t$-structure on $D^b(\SA \otimes_{\BC[\g^*]} \SA^{op}\Mod)^G$ gives the {\em exotic $t$-structure} on $D^b\Coh^G(\St)$. Moreover, the functor $\Gamma_\CE$ is monoidal where the monoidal structure on the domain comes from the convolution and the monoidal structure on the codomain coming from the derived tensor $-\overset{L}{\otimes}_\SA-$.
\end{Rem}

\subsubsection{Braid group actions}\

Let us discuss the braid group actions on $D^b\Coh^G (\tg \x \tg)$, $D^b\Coh^G(\St)$ and their compatibilities. Here $G$ acts on $\tg \x \tg$ diagonally. Consider the following diagram:
\begin{align*}
    (\tg \x \tg)\x_{\g^* \x \g^*} (\tg \x \tg) \xleftarrow[]{\Delta_{13} \x \text{Id}_{24}} \tg \x \St \xrightarrow[]{\pr_2} \St \\
    (\tg \x \tg) \x_{\g^* \x \g^*} (\tg \x \tg) \xleftarrow[]{\text{Id}_{13} \x \Delta_{24}} \St \x \tg \xrightarrow[]{\pr_1} \St,
\end{align*}
Here $\Delta_{13}$ means the diagonal embedding of $\tg$ into the first and third components $\tg$, while $\text{Id}_{24}$ means the identity map of $\St$ into the second and forth components $\tg$, same for other notations.

For $\CF \in D^b \Coh^G(\tg \x \tg)$, let $\CF^l, \CF^r \in D^b\Coh^G((\tg \x \tg)\x_{\g^* \x \g^*} (\tg \x \tg))$ defined by 
\[ \CF^l:=(\text{Id}_{13} \x \Delta_{24})_*\pr_1^*(\CF), \qquad \qquad \CF^r:=(\Delta_{13} \x \text{Id}_{24})_* \pr_2^*(\CF).\]
Then the braid group $B_{ext} \x B_{ext}$ acts on $D^b\Coh^G(\tg \x \tg)$ as follows: $(b, 1)$ acts by the kernel $(K_b)^l$, meanwhile $(1,b)$ acts by the kernel $(K_b)^r$.

\begin{Lem}\label{lem: exotic structure on tgxtg}
(a) The functors $(K_b)^l*-, (K_b)^r*-, (\CR_\a)^l*-, (\CR^t_\a)^r*-$ on $D^b\Coh^G(\tg \x \tg)$ are exotic $t$-exact for $b\in \Lambda$ and $\a \in I_{aff}$.

\noindent
(b) The left and right adjoint of $(\CR_\a)^l*-$  are both isomorphic to $(\CR_\a)^l *-$ for $\a \in I_{aff}$. The left and right adjoint of $(\CR^t_\a)^r*-$ are both isomorphic to $(\CR_\a^t)^r *-$ for $\a \in I_{aff}$. The left and right adjoints of $(K_b)^{l/r}*-$ are both isomorphic to $(K_{b^-1})^{l/r}*-$ for $b\in \Lambda$.
\end{Lem}

Following \cite[$\mathsection 5$]{BR12}, the braid group action on $D^b\Coh^G(\St)$ can be obtained from the braid group action on $D^b\Coh^G(\tg \x \tg)$ via the pullback under the closed embedding
\[ \tg \x^R_{\g^*} \tg \x^R_{\g^*} \tg \x^R_{\g^*} \tg \hookrightarrow (\tg \x \tg)\x_{\g^* \x \g^*} (\tg \x \tg).\]
The action can be described similarly as above. Let us consider the following diagram:
\begin{align*} \tg \x^R_{\g^*} \tg \x^R_{\g^*} \tg \x^R_{\g^*} \tg \xleftarrow[]{\Delta_{13} \x \text{Id}_{24}} \tg \x^R_{\g^*} \St \xrightarrow[]{\pr_2} \St\\
\tg \x^R_{\g^*} \tg \x^R_{\g^*} \tg \x^R_{\g^*} \tg \xleftarrow[]{\text{Id}_{13} \x \Delta_{24}} \St \x^R_{\g^*} \tg \xrightarrow[]{\pr_1} \St
\end{align*}
For $\CF \in D^b \Coh^G(\St)$, let $\CF^l, \CF^r \in D^b \Coh(\tg \x^R_{\g^*} \tg \x^R_{\g^*} \tg \x^R_{\g^*} \tg)^G$ defined by 
\[ \CF^l:= (\text{Id}_{13} \x \Delta_{24})_* \pr_1^*(\CF), \qquad \CF^r:=(\Delta_{13} \x \text{Id}_{24})_*\pr_2^*(\CF).\]
Then for $\CG \in D^b\Coh^G(\St)$, we have
\begin{equation}\label{eq: convolutions}
\CF^l * \CG \cong \CF * \CG, \qquad \CF^r* \CG \cong \CG * \sigma^*(\CF),
\end{equation}
here $\sigma: \St \rightarrow \St$ is the map $(x, y) \mapsto (y,x)$.

The braid group $B_{ext} \x B_{ext}$ acts on $D^b\Coh^G(\St)$ as follows: $(b, 1)$ acts by the kernel  $(K_b)^l$, meanwhile $(1, b)$ acts by  the kernel $(K_b)^r$. 
\begin{Lem}\label{lem: compatibility of braid group action}
The braid group actions of $B_{ext} \x B_{ext}$ on $D^b\Coh^G(\tg \x \tg)$ and $D^b\Coh^G(\St)$ commute (up to isomorphism) with the direct and inverse image functors
\[ D^b \Coh^G(\St) \leftrightarrows D^b\Coh^G(\tg \x \tg).\]
\end{Lem}
 
From these discussions, we obtain the following lemma:
\begin{Lem}\label{lem: t-exact} Left and right convolutions with $K_b~ (b\in \Lambda)$ and $\CR_\a~(\a \in I_{aff})$   are exotic $t$-exact on $D^b\Coh^G(\St)$.
\end{Lem}
\begin{proof} Similar to Lemma \ref{lem: exotic structure on tgxtg}(a), the functors $(K_b)^l *-, (K_b)^r*-, (\CR_\a)^l*-, (\CR^t_\a)^r*-$ on $D^b\Coh^G(\St)$ are exotic $t$-exact. Then lemma follows by \eqref{eq: convolutions}.
\end{proof}
\begin{Rem} It is not sure that \ref{lem: exotic structure on tgxtg}(b) also holds for $D^b\Coh^G(\St)$.
\end{Rem}

\begin{Lem}\label{lem: global sections of ref}
Let $\CF$ be one of these objects:  $\CO_{\Delta\tg},$ $ \CR_\a~(\a \in I_{aff}),$ $ K_b~( b\in \Lambda)$. Then,

\noindent
(a) $\CF$ belongs to the heart of the exotic $t$-structure on $D^b\Coh^G(\St)$.

\noindent
(b) $\Gamma_{\CE}(\CF)$ is projective as a left $\SA$ and as a right $\SA$-module. 

\noindent
(c) Left and right convolution with $\CR_\a~(\a\in I_{aff})$ and $K_b~(b\in \Lambda)$ send projective objects to projective objects in the heart of the exotic $t$-structure on $D^b\Coh^G(\St)$.
\end{Lem}
\begin{proof}(a) It is enough to prove  the statement for $\CO_{\Delta \tg}$ since $\CF * \CO_{\Delta\tg} \cong \CF$ and the functor $\CF*-$ is exotic $t$-exact by Lemma \ref{lem: t-exact}. We have
\begin{align*}
    \Gamma_{\CE}(\CO_{\Delta \tg})=\RHom_{\St}(\CE\boxtimes \CE^\vee|_{\St},  \CO_{\Delta\tg})
    &\cong \RHom_{\tg}( \CE\overset{L}{\otimes}_{\tg} \CE^\vee, \CO_{\tg})\\
    &\cong \RHom_{\tg}(\CE, \CE)
\end{align*}
Since $\RHom_{\tg}(\CE, \CE)$ is homologically concentrated in degree $0$, it follows that $\CO_{\Delta \tg}$ belongs to the heart of the exotic $t$-structure of $D^b\Coh^G(\St)$. This proves part (a)

\noindent
(b) Let us consider $\CF= \CR_\a$, proofs for other objects are the same. First, $\CR_\a=(\CR_\a)^l *\CO_{\Delta \tg}$,
\begin{align*}\Gamma_\CE(\CR_\a)=\RHom_{\St}(\CE \boxtimes \CE^\vee|_{\St}, (\CR_\a)^l *\CO_{\Delta \tg}) &\cong \RHom_{\tg \x \tg}(\CE \boxtimes \CE^\vee, (\CR_\a)^l * \CO_{\Delta \tg})\\
&\overset{\text{Lemma \ref{lem: exotic structure on tgxtg}}}{\cong}\RHom_{\tg \x \tg}((\CR_\a)^l *(\CE \boxtimes \CE^\vee), \CO_{\Delta \tg})\\
&\cong \RHom_{\tg \x \tg}((\CR_\a* \CE) \boxtimes \CE^\vee, \CO_{\Delta \tg})\\
&\cong \RHom_{\tg}((\CR_\a*\CE)\otimes^L_{\tg} \CE^\vee, \CO_{\tg})\\
& \cong \RHom_{\tg} (\CR_\a* \CE, \CE).
\end{align*}
By $\SA^{op}=\End_{\tg}(\CE, \CE)$, the last object $\RHom_{\tg}(\CR_\a *\CE, \CE)$ is a right $\SA$-module. meanwhile, by $\SA^{op}=\End_{\tg}(\CE^\vee, \CE^\vee)^{op}$, all other objects in the above isomorphisms are right $\SA$-modules. Then the above isomorphisms are isomorphisms of right $\SA$-modules.

By Lemma \ref{lem: exotic t-structure}, the right adjoint of $\CR_\a*-$ on $D^b\Coh^G(\tg)$ is isomorphic to $\CR_\a*-$, which is exotic $t$-exact. Therefore, $\CR_\a*\CE$ must be projective in the heart of the exotic $t$-structure on $D^b\Coh^G(\tg)$. Then $\CR_\a*\CE$ is a direct summand of  $\CE\otimes_{\BC} V_1$ for some $G$-representation $V_1$. Therefore, $\Gamma_{\CE}(\CR_\a)$ is a direct summand of $\RHom_{\CE}(\CE, \CE)\otimes_{\BC} (V_1)^*=\SA\otimes_{\BC} (V_1)^*$ as a right $\SA$-module. Therefore, $\Gamma_\CE(\CR_\a)$ is a projective right $\SA$-module.  The proof that $\Gamma_\CE(\CR_\a)$ is projective as a left $\SA$-module is similar, when we use $\CR_\a\cong \CO_{\Delta \tg} * \CR_\a \cong (\CR^t_\a)^r* \CO_{\Delta \tg}$.

\noindent
(c) We call projective object in the heart of the exotic $t$-structure on $D^b\Coh^G(\St)$ is the {\it projective exotic object}.  The projective  exotic object $P$ is a direct summand of some object of the form $\CE \boxtimes \CE^\vee|_{\St} \otimes_\BC V$ for some $G$-representation $V$. So it is enough to prove the statement for $P=\CE \boxtimes \CE^\vee|_{\St} \otimes_\BC V$. 
By the compatibility between braid group actions in Lemma \ref{lem: compatibility of braid group action},
\[ \CR_\a * P \cong (\CR_\a)^l* (\CE \boxtimes \CE^\vee|_{\St}\otimes_{\BC} V) \cong ((\CR_\a)^l*(\CE \boxtimes \CE^\vee))|_{\St}\otimes_{\BC} V\cong ((\CR_\a * \CE)\boxtimes \CE^\vee)|_{\St} \otimes_{\BC} V.\]
As in part (b), $\CR_\a * \CE$ is a direct summand  of some object of the form $\CE \otimes_{\BC} V_1$ for some $G$-representation $V_1$. This implies that $\CR_\a * P$ is a projective  exotic object.

Similarly, we have 
\[ P* \CR_\a \cong (\CR_\a^t)^r* P \cong ((\CE \boxtimes (\CR^t_\a * \CE^\vee))|_{\St} \otimes_{\BC} V \cong ((\CE \boxtimes (\CE^\vee * \CR_\a))|{\St} \otimes_{\BC} V.\]
and then $P * \CR_\a$ is a projective exotic object by Lemma \ref{lem: tranposed t-structure}. The statement for $K_b~ (b\in \Lambda)$ is proved similarly with Lemma \ref{lem: exotic t-structure}-\ref{lem: tranposed t-structure}.
\end{proof}

The objects $\Gamma_\CE(\CF)$ in Lemma \ref{lem: global sections of ref}.b belong to the following subcategory:
\begin{defi}Let $\CP_\SA$ be the full subcategory of the category $\SA \otimes_{\BC[\g^*]}\SA^{op}\Mod^G$ consisting of all objects which are projective modules over $\sA$ and over $\sA^{op}$
\end{defi}
\begin{Rem}$\CP_\SA$ is closed under monoidal structure of $\SA\otimes_{\BC[\g^*]}\SA^{op}\Mod^G$.
\end{Rem}

\subsubsection{The functor $\mathfrak{R}$}\

Recall that  $\Stn \xrightarrow[]{p} \h^* \x_{\h^*/W} \g^* \x_{\h^*/W} \h^*$ is proper which is isomorphic over 
\[ X:= \h^* \x_{\h^*/W} \g^{reg} \x_{\h^*/W} \h^*.\]
On the other hand, by construction,  
\[\SA \otimes_{\BC[\g^*]} \SA^{op}=\End_{\Stn}(\CE \boxtimes \CE^\vee|_{\Stn})=p_*\mathcal{E}nd_{\Stn}(\CE \boxtimes \CE^\vee|_{\Stn})\]
Therefore, by proper base change, the restriction of $\SA \otimes_{\BC[\g^*]}\SA^{op}$ to $X$ is a sheaf of Azumaya algebras
with the splitting bundle $\CE\boxtimes \CE^{\vee}|_X$. Let $\CS$ be the Kostant section in $\g^{reg}$. Let $\BJ$ be the group scheme of universal centralizer of $G$ on $\g^*$. Let $\BI$ be the pullback of $\BJ\x_{\g^*} \CS$ under the map $\h^* \x_{\h^*/W} \CS \x_{\h^*/W} \h^* \rightarrow \CS$.

Let us consider the following composition of functors
\begin{multline}\label{eq: functor R}
\fR: \SA \otimes_{\BC[\g^*]}\SA^{op}\Mod^G \xrightarrow[]{j^*} \SA \otimes_{\BC[\g^*]}\SA^{op}|_X\Mod^G \xrightarrow[]{\fR_1}\Coh(X)^G \\\xrightarrow[]{\fR_2}(\h^* \x_{\h^*/W} \CS \x_{\h^*/W} \h^*)\Mod^{\BI}
\end{multline}
in which:
\begin{itemize}
    \item $j^*$ is the pullback along the open embedding $j: X \hookrightarrow \h^* \x_{\h^*/W} \g^* \x_{\h^*/W} \h^*$.
    \item $\fR_1$ is an equivalence via splitting of Azumaya algebras $\SA \otimes_{\BC[\g^*]} \SA^{op}|_X$ with the splitting bundle $\CE\boxtimes \CE^\vee|_X$.
    \item $\fR_2$ is obtained via restriction to the Kostant section $\CS$. The functor $\fR_2$ is an equivalence by \cite[Proposition 3.3.11]{R17}.
\end{itemize}
\begin{Prop}\label{prop: full faithful on R}The functor $\fR$ is fully faithful on the subcategory $\CP_\SA$.
\end{Prop}
\begin{proof}Since $\SA$ is projective over $\BC[\g^*]$, any object in $\CP_{\SA}$ is projective over $\BC[\g^*]$.

    Since $\fR_1$ and $\fR_2$ are equivalences, it is enough to show that $j^*$ is fully faithful on $\CP_\SA$. The right adjoint of $j^*$ is the pushforward $j_*$. So it is enough to show that if $P \in \CP_\SA$ then the natural map $P \rightarrow j_*j^*P$ is an isomorphism. 

     Let $\pr: \h^* \x_{\h^*/W} \g^* \x_{\h^*/W} \h^* \rightarrow \g^*$ and $i: \g^{reg} \hookrightarrow \g^*$. Since both $\g^*$ and $\h^* \x_{\h^*/W} \g^*\x_{\h^*/W} \h^*$ are affine,  $P \rightarrow j_*j^*P$ is an isomorphism if and only if $\pr_*P \rightarrow \pr_*j_*j^* P$ is an isomorphism. On the other hand $\pr_*j_*j^*P \cong i_*i^*\pr_*P$. So we only need to show $\pr_*P \rightarrow i_*i^*\pr_*P$ is an isomorphism.

    Now, since $\g^{reg}$ has codimension $\geq 2$ in the smooth space $\g^*$ and $\pr_*P$ is projective over $\BC[\g^*]$ by the definition of $\CP_\SA$, the natural map $\pr_*P \rightarrow i_*i^*\pr_*P$ is an isomorphism by the Hartog lemma.
\end{proof}
\subsection{Completed version}\label{ssec: complete version}\
 Let 
\begin{gather*}\SA^{\wedge_0}:= \SA\otimes_{\BC[\h^*]} \BC[\h^*]^{\wedge_0}, \qquad \qquad \BC[\g^*]^{\wedge_0}:= \BC[\g^*]\otimes_{\BC[\h^*/W]} \BC[\h^*/W]^{\wedge_0}.
\end{gather*}

For any scheme or algebra $Y$ over $\BC[\h^*/W]$, let $Y^\wedge$ be the scheme or algebra obtained via base change $\BC[\h^*/W] \rightarrow \BC[\h^*/W]^{\wedge_0}$. So we have the following schemes and algebras
\[ (\SA\otimes_{\BC[\g^*]}\SA^{op})^{\wedge}, \qquad X^{\wedge}, \qquad (\h^*\x_{\h^*/W} \CS \x_{\h^*/W} \h^*)^{\wedge}\]
Note that $(\SA \otimes_{\BC[\g^*]}\SA^{op})^\wedge \cong \SA^{\wedge_0} \otimes_{\BC[\g^*]^{\wedge_0}} \SA^{op\wedge_0}$.

The scheme $\St^\wedge$ is obtained from $\tg \x \tg$ by base change $\g^\wedge \rightarrow \g \xrightarrow[]{\Delta} \g \x \g$, which is an exact base change in the sense of \cite[$\mathsection 1.3$]{BM13}. Therefore, $\CE\boxtimes \CE^\vee|_{\St^\wedge}$ is a tilting generator for $\St^\wedge$ and then  we have an equivalence of triangulated categories:
\begin{equation}\label{eq: hat(Gamma) functor} \hat{\Gamma}_{\CE}: D^b \Coh^G(\St^\wedge) \rightarrow D^b((\SA\otimes_{\BC[\g^*]}\SA^{op})^\wedge\Mod^G)\qquad \CG \mapsto \RHom_{\St^\wedge}(\CE \boxtimes \CE^\vee|_{\St^\wedge}, \CG)
\end{equation}

\begin{Rem} The standard $t$-structure on $D^b((\SA \otimes_{\BC[\g^*]} \SA^{op})^\wedge\Mod^G)$ gives the {\em exotic $t$-structure} on $D^b\Coh^G(\St^\wedge)$. Moreover, the functor $\Gamma_\CE$ is monoidal where the monoidal structure on domain comes from the convolution and the monoidal structure on codomain comes from derived tensor product $-\overset{L}{\otimes}_{\SA^{\wedge_0}}-$.
\end{Rem}

\begin{defi}Let $\hat{\CR}_\a~(a\in I_{aff}), \hat{K_b}~ (b\in B_{aff})$ be obtained from $\CR_\a, K_b$ by base change. 
\end{defi}
The restriction of  $(\SA \otimes_{\BC[\g^*]}\SA^{op})^{\wedge}$ to $X^\wedge$ is a sheaf of Azumaya algebras over $X^{\wedge}$
with the splitting bundle $\CE \boxtimes \CE^\vee|_{X^{\wedge}}$. Let $\BI^\wedge$ be the pullback of the group scheme $\BJ\x_{\g^*} \CS$ under   $(\h^* \x_{\h^*/W} \CS \x_{\h^*/W}\h^*)^{\wedge} \rightarrow \CS$ .
We have the following composition of functors as in \eqref{eq: functor R}:
\begin{multline}
     \hat{\fR}: (\SA \otimes_{\BC[\g^*]}\SA^{op})^\wedge \Mod^G \xrightarrow[]{j^*}(\SA \otimes_{\BC[\g^*]}\SA^{op})^\wedge|_{X^\wedge} \Mod^G \xrightarrow[]{\hat{\fR}_1} X\Mod^G \\\xrightarrow[]{\hat{\fR}_2} (\h^* \x_{\h^*/W}\CS \x_{\h^*/W} \h^*)^{\wedge} \Mod^{\BI^\wedge} 
\end{multline}

\begin{defi}\label{defi: complete PA} Let $\hat{\CP}_\SA$ be the image of $\CP_\SA$ under the natural functor $\SA\otimes_{\BC[\g^*]}\SA^{op} \Mod^G \rightarrow (\SA\otimes_{\BC[\g^*]}\SA^{op})^{\wedge}\Mod^G$ defined by  $F\mapsto F\otimes_{\BC[\h^*/W]}\BC[\h^*/W]^{\wedge_0}$. 
\end{defi}

\begin{Lem}The category $\hat{\CP}_\SA$ is closed under the monoidal structure of $(\SA\otimes_{\BC[\g^*]}\SA^{op})^\wedge\Mod^G$. The category $\hat{\CP}_\SA$ contains all projective objects in $(\SA \otimes_{\BC[\g^*]}\SA^{op})^\wedge\Mod^G$.
\end{Lem}
\begin{proof}The first statement is easy to see. A projective object in $(\SA \otimes_{\BC[\g^*]}\SA^{op})^\wedge\Mod^G$  is the direct summand of some object of the form $(\SA \otimes_{\BC[\g^*]}\SA^{op})^\wedge\otimes_{\BC} V$ for some $G$-representation $V$. Then lemma follows since $\SA, \SA^{op}$ are projective over $\BC[\g^*]$.
\end{proof}
\begin{Lem}\label{lem: complete R monoidal and ff on PA} (a) The functor $\hat{\fR}$ is monoidal.

\noindent
(b) The functor $\hat{\fR}$ is fully faithful on the subcategory $\hat{\CP}_\SA$.
\end{Lem}
\begin{proof}(a) We will define the monoidal structures on $X\Mod^G$ and $(\h^* \x_{\h^*/W} \CS \x_{\h^*/W} \h^*)^\wedge \Mod^{\BI^\wedge}$.

The projections $\pr_{12}, \pr_{23}, \pr_{13}$: 
\[ (\g^{reg} \x_{\h^*/W} \h^* \x_{\h^*/W} \h^* \x_{\h^*/W} \h^*)^\wedge \rightarrow (\g^{reg} \x_{\h^*/W} \h^* \x_{\h^*/W} \h^*)^\wedge\]
are finite flat morphisms. For $M,N \in X\Mod^G$, we define the convolution 
\[M * N:= \pr_{13*}(\pr^*_{12}(M) \otimes \pr^*_{23}N),\]
in which we use the original push-pull and tensor product instead of the derived ones. This convolution gives $X\Mod^G$ a monoidal structure. The monoidal structure on $(\h^* \x_{\h^*/W} \CS \x_{\h^*/W} \h^*)^\wedge \Mod^{\BI^\wedge}$ is similarly defined.

\noindent
(b)  follows by Proposition \ref{prop: full faithful on R} and flat base change $\BC[\h^*/W]\rightarrow \BC[\h^*/W]^{\wedge_0}$.
\end{proof}
The following lemma is useful for computations
\begin{Lem}\label{lem: computation for R}Let  $M\in D^b\Coh^G(\St^\wedge)$ belong to the heart of the exotic $t$-structure. Suppose that $M|_{X^\wedge}$ is concentrated in degree $0$, i.e., is a sheaf on $X^\wedge$. Then 
\[ \hat{\fR}_1 \circ j^*(\hat{\Gamma}_{\CE}(M)) \cong M|_{X^\wedge}.\]
\end{Lem}
\begin{proof} Let $\pi: \St^\wedge \rightarrow (\h^* \x_{\h^*/W} \g^* \x_{\h^*/W} \h^*)^\wedge$. Then 
\[\hat{\Gamma}_{\CE}(M)=\pi_*(R\mathcal{H}om_{\St^\wedge}(\CE\boxtimes \CE^\wedge|_{\St^\wedge}, M)).\]
Since $j$ is open and the restriction $\pi|_{X^\wedge}$ is an isomorphism,   
\[ j^*\hat{\Gamma}_{\CE}(M) \cong R\mathcal{H}om_{X^\wedge}(\CE\boxtimes \CE^\vee|_{X^\wedge}, M|_{X^\wedge}).\]
Since $\CE \boxtimes \CE^\vee|_{X^\wedge}$ is a vector bundle and  $M|_{X^\wedge}$ is a coherent sheaf on $X^\wedge$, 
\[ R\mathcal{H}om_{X^\wedge}(\CE\boxtimes \CE^\vee|_{X^\wedge}, M|_{X^\wedge})=\mathcal{H}om_{X^\wedge}(\CE\boxtimes \CE^\vee|_{X^\wedge}, M|_{X^\wedge}).\]
Therefore, 
\[\hat{\fR}_1\circ j^* (\hat{\Gamma}_\CE(M)) \cong M|_{X^\wedge}. \qedhere\]
\end{proof}

\subsection{Soergel bimodules and NCS}\label{ssec: Soergel-NCS}\

We first describe an embedding of monoidal categories of $\ASB$ into $(\h^* \x_{\h^*/W} \CS \x_{\h^*/W} \h^*)^\wedge \Mod^{\BI^\wedge}$. The construction follows \cite[$\mathsection 2.3, 2.4$]{BR22} and hence the details are referred to {\em loc. cit.} We also note that the {\em loc. cit.} works with the $\mathbb{G}_m$-equivariant sheaves on $\h^* \x_{\h^*/W} \h^*$  but the arguments carry to our setting verbatim.
\begin{Rem}We need to emphasize that geometric spaces we are using are schemes, not formal schemes. For example, $\h^{*\wedge_0}:=\Spec(\BC[[\h^*]])$.
\end{Rem}

 Let $\h^*_{\rs}$ be the set of regular elements in $\h^*$, and $\h^{*\wedge_0}_{\rs}:= \h^*_{\rs} \x_{\h^*/W} \h^{*\wedge_0}$. Let us construct the following functor 
\begin{equation}\label{eq: Rep(I) to C}
(\h^* \x_{\h^*/W} \CS \x_{\h^*/W} \h^*)^\wedge \Mod^{\BI^\wedge} \rightarrow \CC,
\end{equation}
where the category $\CC$ is defined in Section \ref{ssec: abe realization}. Any object  in $(\h^* \x_{\h^*/W} \CS \x_{\h^*/W} \h^*)^\wedge \Mod^{\BI^\wedge}$ is in particular a coherent sheaf on $(\h^* \x_{\h^*/W} \h^*)^\wedge$, hence can be regarded as an $\usR$-bimodule.   

For any coherent sheaf $\CF$ on $(\h^* \x_{\h^*/W} \h^*)^\wedge$, the tensor product
\[ \Gamma((\h^* \x_{\h^*/W} \h^*)^\wedge, \CF) \otimes_{\usR} \CO(\h^{*\wedge_0}_{\rs}),\]
admits a canonical  decomposition ( as a $\CO(\h^{*\wedge_0}_\rs)$-bimodule) parametrized by $W$ such that the factor corresponding  to $w\in W$ factors through the quotient
\[ \CO(\h^{*\wedge_0}_{\rs} \x \h^{*\wedge_0}_{\rs}) \rightarrow \CO(\text{Gr}(w, \h^{*\wedge_0}_{\rs})),\]
here $\text{Gr}(w, \h^{*\wedge_0}_{\rs})$ is the graph  of $w$ acting on $\h^{*\wedge_0}_\rs$.

Let $\BI^{\wedge}_w$ be the restriction of $\BI^{\wedge}$ to $\text{Gr}(w, \h^{*\wedge_0}_\rs)$. Identify $\text{Gr}(w, \h^{*\wedge_0}_\rs)$ with $\h^{*\wedge_0}_\rs$ via the first projection, then there is a canonical isomorphism of group schemes:
\[ \BI^{\wedge}_w \cong \h^{*\wedge_0}_\rs \x T.\]
This means that  the category of representation of $\BI^\wedge_w$ on coherent sheaf on $\text{Gr}(w, \h^{*\wedge_0}_\rs)$ is  equivalent  to the category of $P$-graded coherent sheaves on $\h^{*\wedge_0}_\rs$.

So starting with an object $\CF$ in $(\h^* \x_{\h^*/W} \CS \x_{\h^*/W} \h^*)^\wedge\Mod^{\BI^\wedge}$, we obtain a decomposition of $\Gamma(\h^* \x_{\h^*/W} \h^*)^\wedge, \CF) \otimes_{\usR} \CO(\h^{*\wedge_0}_\rs)$ parametrized by $W_{ext}$ by defining, for $\lambda \in P$ and $w\in W$,  the summand associated with $t_\lambda w$ (note the order of $t_\lambda$ and $w$ here) as the $\lambda$-graded  part in  the summand  associated with $w$ (which is a representation of $\BI^\wedge_w$). This finishes the construction  of the functor \eqref{eq: Rep(I) to C}.

Let $(\h^* \x_{\h^*/W} \CS \x_{\h^*/W} \h^*)^{\wedge}\Mod_{\text{fr}}^{\BI^\wedge}$ be the full subcategory of $(\h^* \x_{\h^*/W} \CS \x_{\h^*/W} \h^*)^{\wedge}\Mod^{\BI^\wedge}$ consisting of objects whose underlying coherent sheaves are free with respect to the second projection $(\h^* \x_{\h^*/W} \h^*)^\wedge\rightarrow \h^{*\wedge_0}$. Let us consider the restriction of the functor \eqref{eq: Rep(I) to C}
\begin{equation}\label{eq: Rep(I)-free to C}
(\h^* \x_{\h^*/W} \CS \x_{\h^*/W} \h)^\wedge \Mod_{\text{fr}}^{\BI^\wedge} \rightarrow \CC.
\end{equation}
Now following the same arguments  in \cite[Proposition 2.7, Lemma 2.8-2.9]{BR22}, we have:
\begin{Prop}The monoidal functor \eqref{eq: Rep(I)-free to C} is fully faithful and its image contains $\ASB $.    
\end{Prop}
So we have the following embedding of monoidal categories:
\begin{Cor} \label{lem: SB to Rep(I)} $\mathfrak{L}: \ASB \rightarrow (\h^* \x_{\h^*/W} \CS \x_{\h^*/W} \h^*)^\wedge\Mod^{\BI^\wedge}$.
\end{Cor}
Let us view $\SB \cong \ASB$ as a subcategory of $(\h^* \x_{\h^*/W} \CS \x_{\h^*/W} \h^*)^\wedge \Mod^{\BI^\wedge}$ via the embedding $\mathfrak{L}$. Let us recall the monoidal subcategory $\sC_0 \subset \SB$ corresponding to the smallest two-sided cell in $W_{ext}$ as in Remark \ref{rem: smallest two-sided cell}.

\begin{Lem}\label{lem: computations with R2}(a) $\hat{\CR}_\a|_{X^\wedge}, (\a \in I_{aff})$ and $\hat{K}_b|_{X^\wedge}, ~(b\in B_{ext})$ are sheaves on $X^\wedge$.

\noindent
(b)  $\hat{\fR}_2(\hat{\CR}_\a|_{X^\wedge}) \cong \uBS_{\a}$ for $\a \in I_{aff}$ and $\hat{\fR}_2(\hat{K}_b|_{X^\wedge}) \cong \usR_b$ for $b\in B_{ext}$. Note that in $\usR_b$, we view $b$ as an element in $W_{ext}$ under the natural map $B_{ext}\rightarrow W_{ext}$.
\end{Lem}
\begin{proof}(a) By the construction in Section \ref{ssec: braid group action}, the restrictions of 
\begin{equation}\label{eq: set S}S:=\{\hat{\CR}_\a, ~\a\in I\} \cup \{ \hat{K}_{\overline{s}_\a}, ~ \a \in I\} \cup \{ \hat{K}_{\theta_\lambda}, ~ \lambda \in P\}
\end{equation}
 to $X^\wedge$ are sheaves on $X^\wedge$. These restrictions are free sheaves with respect to the first and second projections $X^\wedge=(\h^*\x_{\h^*/W} \g^{reg} \x_{\h^*/W} \h^*)^\wedge \rightarrow (\g^{reg} \x_{\h^*/W} \h^*)^\wedge$. Therefore,
 \begin{equation}\label{eq: convolution on X}(\CF_1 * \CF_2)|_{X^\wedge} \cong \CF_1|_{X^\wedge} * \CF_2|_{X^\wedge},\qquad \text{  for $\CF_1, \CF_2$ in \eqref{eq: set S}},
 \end{equation} 
 where the convolution on the right hand side is defined in Lemma \ref{lem: complete R monoidal and ff on PA}. So part (a) follows.

 \noindent
 (b) By \eqref{eq: convolution on X}, it is enough to prove the statement for objects in \eqref{eq: set S}, which is straightforward by their construction in Section \ref{ssec: braid group action}.   
\end{proof}
\begin{Rem} Let $L(\lambda)$ be the irreducible $G$-representation of the highest weight $\lambda$. It turns out that $\hat{\fR}(\hat{\Gamma}_\CE(\CO_{\St^\wedge}\otimes L(\lambda))$ is isomorphic to $\uB_{w_0t_\lambda}$, the indecomposable Soergel bimodule in $\ASB$ associated to $w_0t_\lambda$. We will not use this result.
\end{Rem}
\begin{Prop}\label{prop: image of the functor R}
(a) $\hat{\fR}(\hat{\Gamma}_{\CE}(\hat{\CR}_\a))\cong \uBS_{\a}$ for $\a\in I_{aff}$ and $\hat{\fR}(\hat{\Gamma}_{\CE}(\hat{K}_b)) \cong \usR_b$ for $b\in \Lambda$.

\noindent
(b) The subcategory $\SB$ is contained in the image of the subcategory $\hat{\CP}_{\SA}$ under the functor $\hat{\fR}$.

\noindent
(c) The functor $\hat{\fR}$ restricts to an equivalence of semi-monoidal additive categories between the subcategories of projective objects in $(\SA \otimes_{\BC[\g^*]}\SA^{op})^\wedge\Mod^G$ and the category $\sC_0$.
\end{Prop}
\begin{proof}
(a) By Lemma \ref{lem: computations with R2}.(a),  $\hat{\CR}_\a, \hat{K}_b$ satisfy the condition in Lemma \ref{lem: computation for R} so that 
\[ \hat{\fR}_1 \circ j^*(\hat{\Gamma}_\CE(\hat{\CR}_\a)) \cong \hat{\CR}_\a|_{X^\wedge}, \qquad \hat{\fR}_1 \circ j^*(\hat{\Gamma}_\CE(\hat{K}_b)) \cong \hat{K}_b|_{X^\wedge}.\]
But then $\hat{\fR}_2(\hat{\CR}_\a|_{X^\wedge}) \cong \uBS_\a$ for $\a\in I_{aff}$ and $\hat{\fR}_2(\hat{K}_b|_{X^\wedge}) \cong \usR_b$ for $b\in \Lambda$ by Lemma \ref{lem: computations with R2}.(b).

\noindent
(b) Recall that $\SB \cong \ASB$ is the Karoubian category generated by $\{ \uBS_{\a}, \sR_b| \a\in I_{aff}, b\in \Lambda\}$ via taking tensor products and direct summands. Then this part follows by  part (a) and the fact that $\hat{\fR}$ is monoidal and fully faithful on $\hat{\CP}_\SA$ in Lemma \ref{lem: complete R monoidal and ff on PA}.

\noindent
(c) {\em Step 1:} We will show that any object in $\sC_0$ is isomorphic to $\hat{\fR}(P)$ for some projective $P$ in $ (\SA\otimes_{\BC[\g^*]}\SA^{op})^\wedge\Mod^G$.

First, $\hat{\fR}(\hat{\Gamma}_\CE(\CO_{\St^\wedge}))\cong \uB_{w_0}$. By Lemma \ref{lem: global sections of ref}(c), left and right convolutions with $\hat{\Gamma}_\CE(\hat{\CR}_\a), \a \in I_{aff}$ and $\hat{\Gamma}_\CE(\hat{K}_b), b\in \Lambda$ send projective objects in $(\SA \otimes_{\BC[\g^*]} \SA^{\op})\Mod^G$ to projective objects. On the other hand, any object in $\sC_0$ is a direct summand of object $B_1* \uB_{w_0}*B_2$ for some $B_1, B_2\in \SB$. So Step 1 follows since $\hat{\fR}$ is fully faithful on the subcategory of projective objects in $(\SA\otimes_{\BC[\g^*]}\SA^{op})^\wedge \Mod^G$.

\noindent 
{\em Step 2:} We will show that if $P$ is projective in $(\SA\otimes_{\BC[\g^*]} \SA^{op})^\wedge \Mod^G$ then  $\hat{\fR}(P)$ is isomorphic to some object in $\sC_0$. 

First, by Lemma \ref{lem: CE' and CE dual}, $P$ must be a direct summand of  $B_1* \hat{\fR}(\hat{\Gamma}_\CE(\CO_{\St^\wedge}\otimes L(\lambda)))*B_2$ for some $B_1, B_2\in \SB$. Second, by Lemma \ref{lem: exotic t-structure}.e, $\CO_{\St^\wedge} \otimes L(\lambda)$ is a direct summand of  $\hat{K}_b * \hat{\CR}_{\a_1}*\dots * \hat{\CR}_{\a_k} * \CO_{\St^\wedge}$ for some $b\in \Lambda$ and $ \a_1, \dots, \a_k \in I_{aff}$. Hence $\hat{\fR}(\hat{\Gamma}_\CE(\CO_{\St^\wedge} \otimes L(\lambda)))$ is a direct summand of $ B_3 * \uB_{w_0}$ for some $B_3\in \SB$. Therefore, $\hat{\fR}(P)$ is a direct summand of  $B_1'* \uB_{w_0} * B'_2$ for some $B'_1, B'_2 \in \SB$, hence $\hat{\fR}(P)$ belongs to $\sC_0$.

This finishes the proof of part (c).
\end{proof}
Since $\hat{\fR}$ is fully faithful on $\hat{\CP}_\SA$, we have the following corollary
\begin{Cor}\label{Cor: SB and NCS}(a) There is a monoidal embedding $\fI_2: \SB \hookrightarrow (\SA\otimes_{\BC[\g^*]}\SA^{op})^\wedge\Mod^G$. Furthermore, convolution with objects in $\SB$ on the left or on the right give exact endofunctors of $(\SA \otimes_{\BC[\g^*]} \SA^{op})^\wedge \Mod^G$.

\noindent
(b) Under this embedding, $\usR_b \mapsto \hat{\Gamma}_\CE(\hat{K}_b)$ for $ b\in \Lambda$ and $\uBS_\a \mapsto \hat{\Gamma}_\CE(\hat{\CR}_\a)$ for $\a\in I_{aff}$.

\noindent
(c) The restriction of $\fI_2$ on $\sC_0$ give an equivalence of semi-monoidal categories 
\[\sC_0 \iso \Proj((\SA\otimes_{\BC[\g^*]} \SA^{op})^\wedge\Mod^G),\]
where the right hand side is the full subcategory of projective objects. This induces an equivalence of semi-monoidal triangulated categories 
\[\fI^-_2:  K^-(\sC_0) \iso D^-((\SA\otimes_{\BC[\g^*]} \SA^{op})^\wedge \Mod^G).\]

\end{Cor}

\subsection{Main results}\label{ssec: NCS-main results}\

 Recall the functor $\fI_1$ in Corollary \ref{cor: Sb embed} and the functor $\fI_2$ in Corollary \ref{Cor: SB and NCS}.
\begin{Thm}\label{thm: HC and NCS}(a) There is an equivalence of monoidal abelian categories:
\[ \fE: \HC_\e(0,0) \cong (\SA \otimes_{\BC[\g^*]}\SA^{op})^{\wedge}\Mod^G.\]
which fits into the following commutative diagram:
\[\begin{tikzcd} & \SB \arrow[dl, "\fI_1"'] \arrow[dr, "\fI_2"] &\\
\HC_\e(0,0) \arrow[rr, "\fE"] && (\SA \otimes_{\BC[\g^*]} \SA^{op})^\wedge \Mod^G  
\end{tikzcd}
\]
\noindent
(b) There are equivalences of monoidal categories: 
\[ D^b(\HC_\e(0,0)) \cong D^b((\SA\otimes_{\BC[\g^*]} \SA^{op})^\wedge\Mod^G) \cong D^b\Coh^G(\St^\wedge).\]
\end{Thm}
\begin{proof}(a) Let us denote $\HC_0 :=\HC_\e(0,0)$ and $\mathsf{NCS}:= (\SA \otimes_{\BC[\g^*]} \SA^{op})^\wedge\Mod^G$.

By Corollary \ref{cor: C0 to HC} and Corollary \ref{Cor: SB and NCS}, we have the following diagram whose arrows are equivalences of triangulated  semi-monoidal categories
\[\begin{tikzcd} & K^-(\sC_0) \arrow[dl, "\fJ^-_1"'] \arrow[dr, "\fI^-_2"] &\\
D^-(\HC_0) \arrow[rr, "\fE^-= \fI^-_2 \circ (\fI^-_1)^{-1}"] && D^-(\mathsf{NCS})  
\end{tikzcd}
\]
Since $D^-(\HC_0), D^-(\mathsf{NCS})$ are monoidal categories, the equivalence forces $\fE^-$ to send unit to unit, hence $\fE^-$ is a monoidal functor.

{\it Step 1:} We will show that $\fE^-$ maps $\HC_0$ equivalently onto $\mathsf{NCS}$. Since $\HC_0$ has enough projectives, $\HC_0\subset D^-(\HC_0)$ is the full subcategory consisting of all $M$ such that

\begin{enumerate}[label=($\star$)]
\item \label{eq: characterization of HC and NCS}$\RHom_{D^-(\HC_0)}(P,M)$ is homologically concentrated in degree $0$ for all  $P$ in $\Proj(\HC_0)$.
\end{enumerate}
There is a similar description for $\mathsf{NCS}\subset D^-(\mathsf{NCS})$. But $\Proj(\HC_0) \xrightarrow[\cong]{(\fJ^-_1)^{-1}}\sC_0 \xrightarrow[\cong]{\fJ^-_2} \Proj(\mathsf{NCS})$, hence, the property \ref{eq: characterization of HC and NCS} is preserved under the equivalences $\fE^-$. So Step 1 follows and we get an equivalence
of monoidal categories 
\[ \fE: \HC_0 \iso \mathsf{NCS}.\]

{\it Step 2:} Via the monoidal embeddings $\fJ_1$, $B \in \SB$ acts on $\HC_0$ via left convolution $\fI_1(B) *-$ and via right convolution $-* \fI_2(B)$. Similarly, via the monoidal embedding $\fI_2$, $B\in \SB$ acts on $\mathsf{NCS}$ via left and right convolutions. 

These convolutions are exact by Corollary \ref{cor: Sb embed}.a) and Corollary \ref{Cor: SB and NCS}.a), hence they extend to  actions of $\SB$ on $D^-(\HC_0)$ and $D^-(\mathsf{NCS})$. Then the equivalence $\fE^-$ is compatible with these actions of $\SB$, hence so is the equivalence $\fE: \HC_0\iso \mathsf{NCS}$, i.e, for $B \in \SB$ and $X\in\HC_0$, we have $\fE(\fI_1(B)*X)\cong \fI_2(B)* \fE(X)$ and $\fE(X*\fI_1(B)) \cong \fE(X)*\fI_2(B)$. If we choose $X=\uB_1$, the unit in $\SB$, we have $\fE(\fI_1(B)) \cong \fI_2(B)$ for all $B\in \SB$. Hence the desired commutative diagram in part a) exists.

\noindent
(b) By same reasoning as in Step 1, the functor $\fC^-$ maps $D^b(\HC_0)$ equivalently onto $D^b(\mathsf{NCS})$. The monoidal structure on $D^b(\mathsf{NCS})$ is well-defined hence so is the monoidal structure on $D^b(\HC_0)$, i.e., the convolution of two-bounded complexes is bounded. So we obtain the first monoidal equivalence in part b. The second equivalence is  in \eqref{eq: hat(Gamma) functor}.
\end{proof}

\begin{Thm} \label{thm: triangulate equivalence as HCtilt}
The embedding $\SB \hookrightarrow \HC_\e(0,0)$ induces 

\noindent
(a) an equivalence of monoidal triangulatedcategories
\[ K^b(\SB) \cong D^b(\HC_\e(0,0)),\]
(b) an equivalence of monoidal additive categories
\[ \SB \cong \Hilt_\e(0,0).\]
\end{Thm}
\begin{proof}(a) By Theorem \ref{thm: HC and NCS}, we have an embedding of monoidal triangulated categories:
\[ K^b(\SB) \hookrightarrow D^b(\HC_\e(0,0)) \cong D^b((\SA\otimes_{\BC[\g^*]}\SA^{op})^\wedge\Mod^G)\cong D^b\Coh^G(\St^\wedge).\]
Under this embedding
\[ \uB_1 \mapsto \CO_{\Delta\tg^\wedge}, \qquad\usR_b\mapsto \hat{K}_b \;\; (b\in \Lambda), \qquad \uB_\a= \uBS_\a\mapsto \hat{\CR}_\a \;\; (\a\in I_{aff}).\]
Let $\mathcal{D}$ denote the image of $K^b(\SB)$ in $D^b(\Coh^G(\St^\wedge))$ under this embedding. By distinguished triangles \eqref{eq: distinguished triangles},  $\hat{K}_b, \hat{K}_{\tilde{s}_\a}, \hat{K}_{(\tilde{s}_\a)^{-1}}$ belong to $\mathcal{D}$. Note that $b, \tilde{s}_\a, (\tilde{s}_\a)^{-1}$ with $b\in \Lambda, \a\in I_{aff}$ generates the extended braid group $B_{ext}$. Hence $\CO_{\Delta\tg^\wedge}(\lambda)=\hat{K}_{\tilde{t}_\lambda} \in \mathcal{D}$ for all $\lambda \in P$. So $\mathcal{D}$ contains all the following objects
\[ \CO_{\Delta\tg^\wedge}(\lambda) *\hat{\CR}_{\a_1}*\dots *\hat{\CR}_{\a_k}, \qquad (\a_1, \dots \a_k \in I, ~ \lambda \in P).\]
These objects generate $D^b\Coh^G(\St^\wedge)$ as a triangulated category by \cite[Proposition 7.5]{BR24}\footnote{The statement is contained in the proof of the proposition in {\em loc. cit.} and it is enough to consider a finite collection of these objects. Furthermore,  the {\em loc. cit.} works over field of positive characteristics, but the argument is verbatim in characteristic $0$.}. Therefore, we must have $\mathcal{D}=D^b\Coh^G(\St^\wedge)$.

\noindent
(b) We have the composition of two embeddings
\[ K^b(\SB) \hookrightarrow K^b(\Hilt_\e(0,0)) \hookrightarrow D^b(\HC_\e(0,0)).\]
By part (a), we  have an equivalence $K^b(\SB) \cong K^b(\Hilt_\e(0,0))$ which comes from the embedding $\SB \hookrightarrow \Hilt_\e(0,0)$. On the other hand, $\SB$ is Karoubian. Hence part (b) follows.
\end{proof}

\section{Appendices}\label{sec: appendices}

\subsection{Equivariant modules over Hopf algebras}\  \label{append: equivariant modules} 

Let $(H, m, \Delta, \varepsilon, \eta, S)$ be a Hopf algebra over a commutative algebra $k$. The left adjoint action of $H$ on itself is defined as follows:
\[ h \cdot h'= \sum h_{(1)} h' S(h_{(2)}).\]
This left adjoint action makes $H$ an  $H$-module algebra.
\begin{defi}\label{defi: module algebra} A $k$-algebra $A$ with an $H$-module structure is called an $H$-module algebra if the multiplication map $A\otimes_k A \rightarrow A$ is a morphism of $H$-modules.
\end{defi}
\begin{defi}(a) A right $H$-module $M$ equipped with another $H$-action (called {\it equivariant $H$-action}) is {\it  weakly $H$-equivariant} if the multiplication map $M \otimes_k H \rightarrow M$ is a morphism of $H$-modules, where $H$ acts on $H$ by the left adjoint action and $H$ acts on $M$ via the equivariant $H$-action. Let $H\Rmod^H$ denote the category of weakly $H$-equivariant right $H$-modules.

\noindent
(b) The category $H\Lmod^H$ of weakly $H$-equivariant left $H$-modules  and  the category  $H\Bimod^H$ of weakly $H$-equivariant $H$-bimodules are defined similarly.
\end{defi}
\begin{Lem}(a) For any $M\in H\Rmod^H$, the following left $H$-action makes $M$  an object in $H \Bimod^H$:
\begin{equation}\label{eq: left H-action} hm= \sum( h_{(1)}\cdot m)h_{(2)}, \qquad \qquad \text{for $m \in M, h\in H$.}
\end{equation}
here $h \cdot m$ refers to the equivariant $H$-action on $M$. Under this left $H$-action, we have 
\begin{equation}\label{eq: adjoint action} h \cdot m=\sum h_{(1)}mS(h_{(2)}),\qquad \qquad \text{for $m \in M, h \in H$.}
\end{equation}
\noindent
(b) Assume that $S$ is invertible. For any $N \in H\Lmod^H$, the following right $H$-action makes $N$ into an object in $H \Bimod^H$: 
\begin{equation}\label{eq: right H-action}
nh =\sum h_{(2)} (S^{-1}(h_{(1)})\cdot n), \qquad \quad \text{for all $n\in N, h \in H$.}
\end{equation}
Furthermore, under this right $H$-action, the module $N$ satisfies \eqref{eq: adjoint action}.
\end{Lem}
\begin{proof}(a) One needs to show that \eqref{eq: left H-action} defines a $H$-bimodule structure on $M$  then check the equivariant conditions. These verifications are straightforward. (b) The proof is the same.
\end{proof}

\subsubsection{}\label{append: R-module H}Similarly, for a $H$-module algebra $R$, we can form the category $R\Rmod^H,$ $R\Lmod^H$ and $R\Bimod^H$. For $H$-modules $M$ and $N$, there are two $H$-module structures  on $\Hom_k(M,N)$
\begin{enumerate}
    \item[(A1)] $(hf)(m)=\sum h_{(1)}f(S(h_{(2)})m)$ for $h \in H, f\in \Hom_k(M,N)$ and $m\in M$.
    \item[(A2)] $(hf)(m)=\sum h_{(2)}f(S^{-1}(h_{(1)})m)$ for $h \in H, f\in \Hom_k(M,N)$ and $m\in M$.
\end{enumerate}
It turns out that for $M,N \in R\Rmod^H$ then $(A1)$ equips $\Hom_{R\Rmod}(M,N)$ with a $H$-module structure but $(A2)$ does not. On the other hand, for $M,N \in R\Lmod^H$ then $(A2)$ equips $\Hom_{R\Lmod}(M,N)$ with a $H$-module structure but $(A1)$ does not. So for $M,N \in R\Bimod^H$, it is not clear how to equip $\Hom_{R\Bimod}(M,N)$ with a $H$-module structure.

\subsection{Proof of Lemma \ref{lem: complete HCq is abelian}} \label{ssec: proof of abelian}  
We will need the following lemma
\begin{Lem}\label{lem: Criterion for HCq}  Assume $M$ is a $\cU_q$-equivariant right $U_q^{fin, \utheta}$-module such that
\begin{itemize}
    \item $M$ is a rational $\cU_q(\g)$-module. 
    \item The quotient $M/\hbar M$ belongs to $U_\e^{fin, \utheta}\rmod^{G_\e}$.
    \item For any $V_q \in \Rep(\cU_q(\g))$, the space $\Hom_{\cU_q}(V_q, M)$ is a finitely generated module over $\CW_q^{\wedge_{\utheta}}$.
\end{itemize}
Then $M$ is finitely generated over $U_q^{fin, \utheta}$ so that it belongs to $U_q^{fin,\utheta}\rmod^{G_q}$.
\end{Lem}
\begin{proof}There is a projective object $V_q \in \Rep^{fd}(\cU_q(\g))$ with the following commutative diagram 
\[ \begin{tikzcd}  & V_q \otimes_{\BC[[\hbar]]}U_q^{fin, \utheta} \arrow[dl, dashed] \arrow[d, two heads]&\\
M \arrow[r]& M/\hbar M
\end{tikzcd}
\]
We  want to show that the map $V_q\otimes_{\BC[[\hbar]]} U_q^{fin, \utheta}\rightarrow M$ is surjective. It is enough to show that for any projective $V'_q \in \Rep^{fd}(\cU_q(\g))$, the following map is surjective:
\begin{equation}\label{eq: equation1}
 \Hom_{\cU_q}(V'_q, V_q\otimes_{\BC[[\hbar]]} U_q^{fin, \utheta}) \rightarrow \Hom_{\cU_q}(V'_q, M).
\end{equation}
Note that the induced map 
\[ \begin{tikzcd}\Hom_{\cU_q}(V'_q, V_q\otimes_{\BC[[\hbar]]}U_q^{fin, \utheta})/\hbar\Hom_{\cU_q}(V'_q, V_q\otimes_{\BC[[\hbar]]}U_q^{fin, \utheta})\arrow[d, "\cong"] \arrow[r]&  \Hom_{\cU_q}(V'_q, M)/\hbar \Hom_{\cU_q}(V'_q, M)\arrow[d, "\cong"],\\
\Hom_{\cU_q}(V'_q, V_q\otimes_{\BC[[\hbar]]}U_q^{fin, \utheta}/\hbar V_q\otimes_{\BC[[\hbar]]}U_q^{fin, \utheta})\arrow[r]&\Hom_{\cU_q}(V'_q, M/\hbar M)
\end{tikzcd}\]
is surjective by  construction, where the two vertical maps are isomorphisms since $V'_q$ is projective in $\Rep^{fd}(\cU_q(\g))$. On the other hand, both $\Hom_{\cU_q}(V'_q, V_q\otimes_{\BC[[\hbar]]} U_q^{fin, \utheta})$ and $\Hom_{\cU_q}(V'_q, M)$ are finitely generated over $\CW_q^{\wedge_\utheta}$, hence complete and seperated in the $\hbar$-adic topology. Therefore, \eqref{eq: equation1} is surjective.
\end{proof}
\begin{proof}[Proof of Lemma \ref{lem: complete HCq is abelian}]  It is obvious that $\HC_\e(\utheta, \utheta')$ is abelian since $U^{fin, \utheta'}_\e$ is  Noetherian. Let us prove the statement for $\HC_q(\utheta, \utheta')$.  Let $M\in \HC_q(\utheta, \utheta')$. It is enough to show that if $N$ is a subobject of $M$  in $U_q^{fin, \utheta'}\Rmod^{G_q}$ then $N$ belongs to $ U_q^{fin, \utheta'}\rmod^{G_q}$.

{\it Step 1:} For any $V_q\in \Rep(\cU_q)$, we have an inclusion 
\[ \Hom_{\cU_q}(V_q, N) \hookrightarrow \Hom_{\cU_q}(V_q, M).\]
Hence, $\Hom_{\cU_q}(V_q, N)$ is a finitely generated right module over $\CW_q^{\wedge_{\utheta'}}$ for all $V_q\in \Rep(\cU_q)$. 

{\it Step 2:} We will show that $N/\hbar N \in U_\e^{fin, \utheta'}\rmod^{G_\e}$. We have a short exact sequence in $U_q^{fin, \utheta'}\Rmod^{G_q}$
\[ 0\rightarrow N \rightarrow M \rightarrow P \rightarrow 0.\]

{\it Step 2.1:} Suppose that  $P$ is torsion free over $\BC[[\hbar]]$. Then we have a short exact sequence in $U_\e^{fin, \utheta'}\Rmod^{G_\e}$
\[ 0 \rightarrow N/\hbar N \rightarrow M/\hbar M \rightarrow P/\hbar P \rightarrow 0.\]
Since $M/\hbar M \in U_\e^{fin, \utheta'}\rmod^{G_\e}$ and the latter category is abelian, it follows that $N/\hbar N \in U_\e^{fin, \utheta'}\rmod^{G_\e}$. In this case, $N$ satisfies the conditions of Lemma \ref{lem: Criterion for HCq}, hence $N \in U_q^{fin, \utheta'}\rmod^{G_q}$.

{\it Step 2.2:} For general $P$, let $\text{Tor}(P)$ be the $\hbar$-torsion part of $P$. Let $M'=\phi^{-1}(\text{Tor}(P))$ then we have short exact sequences in $U_q^{fin, \utheta'}\Rmod^{G_q}$
\begin{align}
    \label{eq: equation2} 0\rightarrow N \rightarrow M' \rightarrow \text{Tor}(P)\rightarrow 0 \\
    \label{eq: equation3} 0\rightarrow M'\rightarrow M \rightarrow P/\text{Tor}(P)\rightarrow 0.
\end{align}
Since $P/\text{Tor}(P)$ is torsion free over $\BC[[\hbar]]$, by Step $2.1$, we have that $M'\in U_q^{fin, \utheta'}\rmod^{G_q}$. Therefore, $\text{Tor}(P)$ also belongs to $U_q^{fin, \utheta'}\rmod^{G_q}$, in particular, $\text{Tor}(P)$ is finitely generated over $U_q^{fin,\utheta'}$. This implies that there is $k \in \BZ_{\geq 0}$ such that $\hbar^k \text{Tor}(P)=0$, and then $\text{Tor}_{1}(P):=\{ p \in P| \hbar p=0\}$ is finitely generated over $U_\e^{fin, \utheta'}$. The short exact sequence \eqref{eq: equation2} gives us a long exact sequence:
\[ \dots \text{Tor}_1(M')\rightarrow \text{Tor}_1(P) \rightarrow N/\hbar N \rightarrow M' /\hbar M' \rightarrow \text{Tor}(P)/\hbar \text{Tor}(P)\rightarrow 0.\]
Since $M/\hbar M $ and $\text{Tor}_1(P)$ are finitely generated over $U_\e^{fin, \utheta'}$, we see that $N/\hbar N$ is finitely generated over $U_\e^{fin, \utheta'}$, hence $N/\hbar N \in U_\e^{fin, \utheta'}\rmod^{G_\e}$.

{\it Step 3:} By Lemma \ref{lem: Criterion for HCq} with Step 1 and Step 2, we conclude that $N \in U_q^{fin, \utheta'}\rmod^{G_q}$.
\end{proof}

\subsection{Proof of Lemma \ref{lem: equal isogeny}}\ \label{append: equal isogeny} 

{\it Conjugacy classes.} Let $G$ be a {\em simply connected} semisimple algebraic group. We give a nonexhaustive list of some geometric facts about the conjugacy action of $G$ on itself from \cite{St74}.

\begin{Prop}\label{prop: properties of conjugacy classes} Let $G$ act on itself via conjugation and consider  the categorical quotient map $\pi: G\rightarrow G$$\sslash$$ G$. Let $F$ be the  fiber of any closed point $p$ in $G$$\sslash$$G$.

\noindent
(a) There is a natural isomorphism $\BC[G]^G\cong \BC[T]^W$. Furthermore, $G$$\sslash$$G \cong \mathbb{A}^r$, the affine space of dimension $r$ equal to the rank of the Lie algebra $\g$.

 \noindent
 (b) $F$ is a closed, irreducible and normal subvariety of codimension $r$ in $G$. Let $\m_p$ be the maximal ideal of $\BC[G$$\sslash$$G]$ corresponding to $p$ then the defining ideal of $F$ is $\m_p \BC[G]$.

 \noindent
(c)  $F$ contains a unique class of regular elements. This class is open and dense in $F$ and its complement has codimension $ \geq 2$.

\noindent
(d) There is {\em Steinberg section} $S$  that parametrizes conjugacy classes of regular elements and is contained in the regular locus of $G$. The restriction of $\pi$ to $S$ is an isomorphism $S \iso G$$\sslash$$G$.

\end{Prop}

We need the following technical result from  \cite[$\mathsection 3.2$]{IL11} : Let $G$ be a simply connected semisimple algebraic group. Let $H$ be a subgroup of $G$ such that $G/H$ is a quasi-affine and $\BC[G/H]$ is finitely generated. Let $x$ be a point of $G/H$ whose stablizer in $G$ is $H$. Let $H^0$ be the identity component of $H$ and denote by $C(x)=H/H^0$  the component group of $H$. Consider the natural map $\phi: G/H^0\rightarrow G/H$. Let $M$ be a $G$-equivariant vector bundle on $G/H$. The completion $M^{\wedge_x}$ carries natural actions of $\g$ and $H$. Denote this $H$-action by $\rho$. Integrating the $\g$-action on the locally $\g$-finite part $M^{\wedge_x}_{\g\dash fin}$ into the $G$-action then restricting to $H$ we get another $H$-action on $M^{\wedge_x}_{\g\dash fin}$. Denote this action by $\rho'$. Then $\sigma(h)=\rho(h)\rho'(h^{-1})$ defines a new $H$-action on $M^{\wedge_x}_{\g\dash fin}$ which commutes with the  $G$-action. Furthermore, $\sigma(H^0)$ acts trivially so that we have an action of $C(x)$ on $M^{\wedge_x}_{\g\dash fin}$, hence we can define the $C(x)$-invariant part $M^{\wedge_x, C(x)}_{\g\dash fin}$. 
\begin{Lem}[ Proposition $3.2.3$ \cite{IL11}]\label{lem: homogenous section} $M^{\wedge_x}_{\g\dash fin}\cong \Gamma(G/H^0, \phi^*M)$ and $M^{\wedge_x, C(x)}_{\g\dash fin}\cong \Gamma(G/H, M)$.
\end{Lem}

Let $\chi$ be a regular element in $G$ and $\underline{\chi}$  be the image of $\chi$ under the map  $\pi: G\rightarrow G$$\sslash$$G$. Let $\BC[G]^{\wedge_\chi}, \BC[G$$\sslash$$G]^{\wedge_{\underline{\chi}}}$ be the completions of $\BC[G], \BC[G$$\sslash$$G]$ at the closed points $\chi, \underline{\chi}$, respectively.  Denote $I_\chi:=\m_{\uchi} \BC[G]$ and $\BC[G]^{\wedge_{I_\chi}}$ the completions of $\BC[G]$ with respect to the ideal $I_\chi$.

The conjugation action of $G$ on $G$ gives rise to an action of $\g$ on $\BC[G]$. This  action of $\g$ on $\BC[G]$ extends to a $\g$-action on $\BC[G]^{\wedge_\chi}$ by continuity. Let $\BC[G]^{\wedge_\chi}_{\g \dash fin}$ be the locally finite part of this $\g$-action. Integrate the $\g$-action into the $G$-action on $\BC[G]^{\wedge_\chi}_{\g \dash fin}$, and let $\BC[G]^{\wedge_\chi, Z}_{\g \dash fin}$ denote the $Z(G)$-invariant part. We have a natural map $\BC[G]\otimes_{\BC[G\sslash G]} \BC[G$$\sslash$$G]^{\wedge_{\underline{\chi}}} \rightarrow \BC[G]^{\wedge_\chi,Z}_{\g \dash fin}$.
\begin{Lem}\label{lem: fin part C[G] complete at chi}Suppose the natural map $Z(G)\rightarrow C(\chi)$ is surjective. Then the natural map $\BC[G]\otimes_{\BC[G\sslash G]} \BC[G$$\sslash$$G]^{\wedge_{\underline{\chi}}} \rightarrow \BC[G]^{\wedge_\chi,Z(G)}_{\g \dash fin}$ is an isomorphism.
\end{Lem}
\begin{proof}
Since $G$ is Cohen-Macaulay (indeed regular) and $I_\chi$ is generated by codim($I_\chi$) elements,   $I^k_\chi/I^{k+1}_\chi$ is a free module of finite rank over  $\BC[G]/I_\chi=\BC[\overline{G\chi}]$.

Let $M$ be a $G$-equivariant coherent sheaf on $\overline{G\chi}$ such that $Z(G)$ acts on $M$ trivially. Since the natural map $Z(G)\rightarrow C(\chi)$ is surjective,  we have an isomorphism $M^{\wedge_\chi, Z(G)}_{\g\dash fin} \cong M^{\wedge_\chi, C(\chi)}_{\g\dash fin}$. By Lemma $\ref{lem: homogenous section}$, we have $M^{\wedge_\chi, C(\chi)}_{\g\dash fin}\cong \Gamma(G\chi,M|_{G\chi})$. Note that $\overline{G\chi}$ is a normal variety by Proposition \ref{prop: properties of conjugacy classes}.b). Therefore, if $M$ is a free sheaf then $M^{\wedge_\chi, Z(G)}_{\g\dash fin}\cong \Gamma(\overline{G\chi}, M)$. Applying  this analysis to the free $\BC[\overline{G\chi}]$-module $I^k_\chi/I^{k+1}_\chi$ with trivial $Z(G)$-action  we have
\[ I^k_\chi/I^{k+1}_\chi \rightarrow (I^k_\chi/I^{k+1}_\chi)^{\wedge_\chi,Z(G)}_{\g \dash fin}\]
is an isomorphism of $\g$-modules for all $k \geq 0$, where we set $I^0_\chi=\BC[G]$.
Now consider the following commutative diagram in the category of $\g$-modules:
\[\begin{tikzcd} 0\arrow[r]& I^k_\chi/I^{k+1}_\chi \arrow[d]\arrow[r]& \BC[G]/I^{k+1}_\chi \arrow[d]\arrow[r]& \BC[G]/I^k_\chi \arrow[d]\arrow[r]&0\\
0\arrow[r]&  (I^k_\chi/I^{k+1}_\chi)^{\wedge_\chi,Z(G)}_{\g \dash fin} \arrow[r]& (\BC[G]/I^{k+1}_\chi)^{\wedge_\chi,Z(G)}_{\g \dash fin} \arrow[r]& (C[G]/I^k_\chi)^{\wedge_\chi,Z(G)}_{\g \dash fin} 
\end{tikzcd}\]
This diagram allows us to  inductively prove that the natural map
\[ \BC[G]/I^k_\chi \rightarrow (\BC[G]/I^k_\chi)^{\wedge_\chi,Z(G)}_{\g \dash fin}\]
is an isomorphism of $\g$-modules for all $k \geq 1$. 

 Let $V$ be any finite-dimensional representation of $\g$ that is a reprerentation of the adjoint group $G_{\text{ad}}$. Since $\m_{\uchi} \subset \m_\chi$, $\BC[G]^{\wedge_\chi}$ is $\m_{\uchi}$-adically complete. Moreover, $\BC[G]^{\wedge_\chi}/\m^k_{\uchi}\BC[G]^{\wedge_\chi} \cong (\BC[G]/I^k_\chi)^{\wedge_\chi}$. Therefore,  we have
\begin{align*}\Hom_\g(V, \BC[G]^{\wedge_\chi})&\cong \varprojlim \Hom_\g(V, (\BC[G]/I^k_\chi)^{\wedge_\chi})\\
							&\cong \varprojlim \Hom_\g(V, (\BC[G]/I^k_\chi)^{\wedge_\chi,Z(G)}_{\g \dash fin})\\
							&\cong \varprojlim \Hom_\g(V, \BC[G]/I^k_\chi) \\
                            &\cong \varprojlim \Hom_\g(V, \BC[G])/\Hom_\g(V, \BC[G])\m^k_{\uchi}\\
                            &\cong \Hom_\g\Big(V, \BC[G]\Big)\otimes_{\BC[G\sslash G]} \BC[G\sslash G]^{\wedge_{\uchi}}\\
                            & \cong \Hom_\g\Big(V, \BC[G] \otimes_{\BC[G\sslash G]} \BC[G\sslash G]^{\wedge_{\uchi}}\Big).
\end{align*}
In the second last isomorphism, we use that $\Hom_\g(V, \BC[G])$ is finitely generated over $\BC[G\sslash G]$. The last isomorphism holds since $\BC[G\sslash G]^{\wedge_{\uchi}}$ is flat over $\BC[G\sslash G]$.
\end{proof}

\begin{proof}[Proof of Lemma \ref{lem: equal isogeny}]Recall that $V_\e\in \Rep^{fd}(\cU_\e(\g))$ has weights contained in the root lattice $Q$. Since $U_\e^{ev\wedge_\chi}= U_\e^{fin \wedge_\chi}$, we need to show that  the following map is bijective:
\begin{equation}\label{eq: iso of invariant part}
    (V_\e^t \otimes U_\e^{fin, \uchi})^{\cU_\e} \rightarrow (V_\e^t \otimes U_\e^{fin \wedge_\chi})^{\cU_\e},
\end{equation}
here $V_\e^t$ is the right dual. Since $\Rep^{fd}(\cU_\e(\g))$ has enough projectives,  we can assume that $V_\e$ is projective, then $V^t_\e$ is also projective since projective objects in $\Rep^{fd}(\cU_\e(\g))$ are injective.

Let $Z^{fin}_{Fr}\Mod^{G_\e, Q}$ be the category of  finitely generated $Z^{fin}_{Fr}$-modules in $\Rep(\cU_\e(\g))$ whose weights are contained in the root lattice $Q$. 

\noindent
{\it Step 1:}    Since $V^t_\e$ is projective in $\Rep^{fd}(\cU_\e(\g))$ and $U^{fin}_\e$ is a finitely generated projective module over $Z^{fin}_{Fr}$ by Proposition \ref{prop: properties of Ufin}.b, one can show that $V^t_\e\otimes U^{fin}_\e$ is a projective object in $Z^{fin}_{Fr} \Mod^{G_\e, Q}$. So $V^t_\e \otimes U^{fin}_\e$ is a direct summand of $W_\e \otimes Z^{fin}_{Fr}$ for some $W_\e \in \Rep^{fd}(\cU_\e(\g))$ whose weights are contained in the root lattice $Q$. 

\noindent
{\it Step 2:}  We define the following two functors: 
\[ A:~ Z^{fin}_{Fr} \Mod^{G_\e, Q} \rightarrow \text{Vect}_\BC, \qquad \qquad  B: ~ Z^{fin}_{Fr} \Mod^{G_\e,Q} \rightarrow \text{Vect}_\BC\]
Let $M \in Z^{fin}_{Fr} \Mod^{G_\e, Q}$. Then $B(M):= (M\otimes_{Z_\cap} Z_\cap ^{\wedge_{\uchi}})^{\fu_\e}$. 

Let us define $A(M)$. First, we take the completion $ M^{\wedge_\chi}=M\otimes_{Z^{fin}_{Fr}} Z^{fin \wedge_\chi}_{Fr}$. Then take $\fu_\e$-invariants to get $(M^{\wedge_\chi})^{\fu_\e}$ which  is a module over $\cU_\BC(\g)$. Then take the $\g$-finite part $(M^{\wedge_\chi})^{\fu_\e}_{\g \dash fin}$. Finally, take  $Z(G)$-invariants and define $A(M):=((M^{\wedge_\chi})^{\fu_\e}_{\g\dash{fin}})^{Z(G)}$.

So we have a natural transformation $B(M) \rightarrow A(M)$.

\noindent
{ \it Step 3:}  Let $W_\e\in \Rep^{fd}(\cU_\e(\g))$ be such that the weights are contained in the root lattice $Q$. We will show that $B(W_\e\otimes Z^{fin}_{Fr})\rightarrow A(W_\e \otimes Z^{fin}_{Fr})$ is an isomorphism. Indeed we have
\[ B(W_\e\otimes Z^{fin}_{Fr})= (W_\e)^{\fu_\e} \otimes Z_{Fr}^{fin, \uchi}, \qquad \qquad A(W_\e\otimes Z^{fin}_{Fr})=(W_\e)^{\fu_\e} \otimes (Z_{Fr}^{fin \wedge_\chi})_{\g\dash fin}^{Z(G)},\]
here, since the weights of $W_\e$ are contained in $Q$, then the $\fu_\e$-invariant part $(W_\e)^{\fu_\e}$ is a rational representation of $\g$ with a trivial action of $Z(G)$.
Therefore, 
\[ \Big((W_\e)^{\fu_\e} \otimes (Z^{fin\wedge_\chi}_{Fr})_{\g\dash fin}\Big)^{Z(G)} \cong (W_\e)^{\fu_\e} \otimes (Z_{Fr}^{fin \wedge_\chi})_{\g\dash fin}^{Z(G)}.\] 

By Proposition \ref{prop: properties of Ufin}, $Z^{fin}_{Fr}\cong \BC[G]$. By Lemma \ref{lem: fin part C[G] complete at chi}, we have $Z^{fin, \uchi}_{Fr}\iso (Z^{fin \wedge_\chi}_{Fr})_{\g\dash fin}^{Z(G)}$. So, 
\[B(W_\e\otimes Z^{fin}_{Fr}) \iso A(W_\e \otimes Z^{fin}_{Fr}).\]

\noindent
{\it Step 4:} By Steps 1 and 3, we have 
\[ B(V^t_\e \otimes U^{fin, \uchi}_\e )\iso A(V^t_\e \otimes U^{fin, \uchi}_\e).\]
Note that \eqref{eq: iso of invariant part} is obtained from the above isomorphism by taking $\cU_\BC(\g)$-invariants, hence \eqref{eq: iso of invariant part} is an isomorphism. This completes the proof of the lemma.
\end{proof}


\subsection{Proof of Lemma \ref{lem: property of iota}}\label{appendix: Lemma iota} 
(a) is obvious. 

\noindent
(b) Let us fix the notation: for any $w\in W$ then 
\[ \Phi^+_w:=\{\a\in\Delta_+~|~ w(\a)\in \Delta_{+}\} \quad \quad \Phi^-_w :=\{ \a \in \Delta_+~|~ w(\a) \in \Delta_{-}\}.\]
Then $\Phi^+_w=\Phi^-_{w_0w}$ and $\Phi^-_w=\Phi^+_{w_0w}$.

\noindent {\em Step 1:}
$s_{\a_0}=s_{\b} t_{-\b}$ where $\b$ is the highest positive root. Let us  show that $\overline{s}_{\a_0}=\theta_\b \overline{s}^{-1}_\b$. 
\[ 1=\ell(s_{\a_0})=\ell(s_\b t_{-\b})=\sum_{\a \in \Phi^+_{s_\b}}|\< -\b, \a^\vee\>| +\sum_{\a \in \Phi^-_{s_\b}} |\< -\b, \a^\vee\> +1|,\]
therefore,
\begin{equation}\label{eq: sb}\Phi^+_{s_\b}=\{ \a\in \Delta_+~|~ \< \b, \a^\vee\>=0\}, \qquad \Phi^-_{s_\b}=\{ \a\in \Delta_+~|~\< \b, \a^\vee\>=1\} \cup \{ \b\}.
\end{equation}
It follows that $\ell(t_{\b})=\ell(s_\b)+1$, but $\b$ is dominant, hence $\theta_{\b}= \overline{t_{\b}}= \overline{s}_{\a_0}\overline{s}_\b$. So $\overline{s}_{\a_0}=\theta_{\b} \overline{s}^{-1}_\b$.




\noindent
 {\em Step 2:}  We have $w_0s_\b w_0=s_{w_0(\b)}=s_\b$, hence, $s_{\a_0}w_0 s_\b=w_0s_\b s_{\a_0}$.  Since $\ell(s_{\a_0} w_0 s_\b)=\ell(w_0s_\b s_{\a_0})=\ell(w_0 s_\b)+1$ by using \eqref{eq: sb},  we have
 \[\overline{s}_{\a_0} \overline{w_0}~\overline{s}^{-1}_{\b}= \overline{s}_{\a_0} \overline{w_0s_\b}=\overline{s_{\a_0}w_0 s_\b}= \overline{w_0s_\b s_{\a_0}}=\overline{w_0s_\b} ~\overline{s}_{\a_0} =\overline{w_0}~ \overline{s}^{-1}_\b \overline{s}_{\a_0}.\]
 In the other word, $\overline{s}_{\a_0}$ commutes with $\overline{w_0}~\overline{s}^{-1}_\b$. Now we have
 \[ \overline{w_0}^{-1} \overline{s}_{\a_0} \overline{w_0}= \overline{w_0}^{-1} \overline{s}_{\a_0} \overline{w_0}~ \overline{s}^{-1}_{\b} \overline{s}_\b=\overline{s}^{-1}_\b \overline{s}_{\a_0}\overline{s}_\b=\overline{s}^{-1}_\b \theta_\b=\iota(\overline{s}_{\a_0}).\]

(c) {\em Step 1:}  $b=wt_\lambda \in \Lambda$ if and only if $\ell(wt_\lambda)=0$. We first show that $w_0w(\lambda)=\lambda$; $\ell(w_0w^{-1}w_0t_\lambda)=0$ and $\ell(w_0wt_{-\lambda})=\ell(w_0w)+\ell(t_{-\lambda})$. Indeed, 

\[ 0=\ell(wt_\lambda)=\sum_{\a \in \Phi^+_w}|\<\lambda, \a^\vee\>|+\sum_{\a\in \Phi^-_w} |\< \lambda, \a^\vee\> +1|,\]
hence 
\[ \Phi^+_w=\{ \a\in \Delta_+~|~ \< \lambda, \a^\vee\>=0\}, \quad \quad \Phi^-_w=\{ \a\in \Delta_+~|~ \< \lambda, \a^\vee\>=-1\}.\]
From this we see that $\< \lambda, \a^\vee\>=0$ for all $\a\in \Phi^-_{w_0w}$, therefore $w_0w(\lambda)=\lambda$.

The equality $\ell(w_0w^{-1}w_0 t_\lambda)=0$ follows by $\Phi^-_{w_0w^{-1}w_0}=\Phi^-_w$ and $\Phi^+_{w_0w^{-1}w_0}=\Phi^+_w$. Indeed, $w_0w^{-1}w_0(\a)\in \Delta_- \Rightarrow w_0(\a)=w(\gamma) $ for some $\gamma \in \Phi^-_w$. So $-1=\< \lambda, \gamma^\vee\>=\< w_0w(\lambda), \a^\vee\>=\< \lambda, \a^\vee\>$, hence $\a \in \Phi^-_{w}$. Similarly, $w_0w^{-1}w_0(\a)\in \Delta_+ \Rightarrow \a \in \Phi^+_w$. 

Therefore, $\ell(t_{-\lambda})=|\Phi^-_w|; ~\ell(w_0w)=|\Phi^+_w|$, hence  $\ell(w_0wt_{-\lambda})=|\Phi^-_w|+|\Phi^+_w|=\ell(t_{-\lambda})+\ell(w_0w)$.

\noindent
{\em Step 2:} We see that  $-\lambda$ is a dominant weight and $b t_{-\lambda}=w$, hence 
\[ \theta_{-\lambda}=\overline{t}_{-\lambda}, \qquad \overline{b}=\overline{w} \theta^{-1}_{-\lambda}, \qquad \iota(\overline{b})=\theta^{-1}_{-\lambda} \iota(\overline{w})=\theta^{-1}_{-\lambda} \overline{w^{-1}}.\] Similarly, $\overline{b'}=\overline{w_0w^{-1}w_0} \theta^{-1}_{-\lambda}$. Note that $\overline{w_0}=\overline{w_0w^{-1}w_0}~\overline{w_0 w}=\overline{w_0w}~\overline{w^{-1}}$. What we want to show $\iota(\overline{b})=\overline{w_0}^{-1} \overline{b'} \overline{w_0}$ is equivalent to
\begin{equation*}
\overline{w_0}^{-1} \overline{w_0w^{-1}w_0}~\overline{t}_{-\lambda}^{-1} \overline{w_0} =\overline{t}_{-\lambda}^{-1} \overline{w^{-1}} \Leftrightarrow  \overline{t}_{-\lambda}~\overline{w_0w}=\overline{w_0w}~\overline{t}_{-\lambda}.
\end{equation*}
The last equality holds since $t_{-\lambda}w_0w=w_0w t_{-\lambda}$ and $\ell(t_{-\lambda}w_0w)=\ell(t_{-\lambda})+\ell(w_0w)$.

\noindent
(d) It is  because $\overline{w_0}^{-1} \overline{s}_\a \overline{w_0}=\overline{s}_{w_0(\a)}$, which is  equivalent to $\overline{s_\a w_0}=\overline{w_0s_{w_0(\a)}} \Leftrightarrow s_\a w_0=w_0 s_{w_0(\a)}$ for a simple reflection $s_\a \in S$.




\subsection{Proposition \ref{prop: ff on diagonal bimods}}\ \label{append: enhanced version}  We work under the condition \ref{eq: assumption on l}. One of the key points in this paper is Proposition \ref{prop: ff on diagonal bimods} where we assumed that the map $Z(G) \rightarrow C(\chi)$ is surjective. We would like to remove that restriction. It requires to refine the target categories of the functor $\bullet_\dag$ in \eqref{eq: res functor for chi}. 

Let us consider a regular $\chi \in G$ such that the semisimple part $\chi_s$ is contained in $T$. Let $\chi_u$ be the unipotent part of $\chi$.  Let $G_s=\Stab_G(\chi_s)$, the stablizer of $\chi_s$ in $G$.  
\begin{defi}\text{Let $\Xi$ be the center of $G_s$.}
\end{defi}
\begin{Rem}Since $\chi$ is regular, we have that $\Xi$ is a finite abelian group in $ \Stab_{G_s}(\chi_u) =\Stab_G(\chi)$ and the natural map $\Xi \rightarrow \Stab_G(\chi) \rightarrow C(\chi)$ is bijective (since $G$ is simply connected, $G_s$ must be connected). Since $\chi_s\in T$, the maximal torus of $G_s$ is $T$, hence $\Xi \subset T$. 
\end{Rem}

\begin{Rem}\label{rem: Xi-invariant}Lemma \ref{lem: fin part C[G] complete at chi} holds  with the same proof if we replace $\BC[G]_{\g\dash fin}^{\wedge_\chi, Z}$ by $\BC[G]_{\g\dash fin}^{\wedge_\chi, \sigma(\Xi)}$. The latter  is defined as follows : Since $\Xi$ fixes $\chi$, we have $\sigma_1(\Xi)$-action on $\BC[G]_{\g\dash fin}^{\wedge_\chi}$. By integrating the $\g$-action, we obtain a $G$-action on $\BC[G]_{\g\dash fin}^{\wedge_\chi}$ then $\sigma_2(\Xi)$-action is obtained by restricting the $G$-action to $\Xi$. Then $\sigma(x):= \sigma_1(x) \sigma_2(x^{-1})$ defines the $\sigma(\Xi)$-action on $\BC[G]_{\g\dash fin}^{\wedge_\chi}$ that commutes with the action of $G$.

\end{Rem}

Let us define the following categories:
\begin{equation}\label{eq: cat with Xi} 
\Pbim^{\Xi}(U_q^{ev, \uchi}),\;\;\Pbim^\Xi(U_q^{ev\wedge_\chi}), \;\;\Pbim^\Xi(R_\hbar),  \;\;\Pbim^\Xi(\CW_q^{\wedge_{\uchi}})
\end{equation}

\begin{Rem}Let $X\subset P$ be the lattice of all weights on which $\Xi \subset T$ acts trivially. Then giving a $P/X$-grading is  the same as giving an action of $\Xi$.
\end{Rem}
\begin{itemize}
\item The group $\Xi$ acts on $\CW_q^{\wedge_{\uchi}}$ trivially. Then $\Pbim^\Xi(\CW_q^{\wedge_{\uchi}})$ is the category of $\Xi$-graded finitely generated Poisson $\CW_q^{\wedge_\chi}$-bimodules.
\item The $P$-grading on $U_q^{ev, \uchi}$ gives a $\Xi$-action on it. Recall $\phi_{1, \hbar}: U_q^{ev,\uchi} \rightarrow U_\e^{ev, \uchi}$. The subalgebra $P_{1, \hbar}=\phi^{-1}_{1, \hbar}(Z_{Fr}\otimes_{Z_\cap} Z_\cap^{\wedge_{\uchi}})$ is stable under the $\Xi$-action on $U_q^{ev, \uchi}$. So $\Pbim^{\Xi}(U_q^{ev, \uchi})$ is the category of $\Xi$-graded finitely generated Poisson $U_q^{ev, \uchi}$-bimodules.
\item  Since $\Xi$ fixes $\chi$, the maximal ideal $\m_\chi \subset Z_{Fr}$ is homogeneous in the  $P/\ell X$-grading. Since the map $\phi: U_q^{ev} \rightarrow U_\e^{ev}$ is graded, the ideal $\fJ_\chi=\phi^{-1}(\m_\chi)$ also graded. So $U_\e^{ev\wedge_\chi}$ and $U_q^{ev\wedge_\chi}$ have $P/\ell X$-grading, hence naturally have $P/X$-grading, equivalently, a $\Xi$-action. Moreover, $P_\hbar \subset U_q^{ev\wedge_\chi}$ is stable under this $\Xi$-action. So $\Pbim^\Xi(U_q^{ev\wedge_\chi})$ is the category of $\Xi$-graded finitely generated Poisson $U_q^{ev\wedge_\chi}$-bimodules

\item The $\Xi$-action on $(R_\hbar, C_\hbar)$ is more tricky and explained in Remark  \ref{rem: Xi act on R} below. Then $\Pbim^\Xi(R_\hbar)$ is the category of $\Xi$-graded finitely generated Poisson $R_\hbar$-bimodules.
\end{itemize}
\begin{Lem}\label{lem: Xi act on R}We further choose $\chi=\chi_s \chi_u$ such that $\chi_s \in T, \chi_u \in U_+$. Then 

\noindent
(a) There is a complete system of $\Xi$-invariant idempotents $e_1, \dots, e_\sd$ of $U_\e^{ev\wedge_\chi}$ which are matrix units in a decomposition $U_\e^{ev\wedge_\chi}\cong \Mat_\sd(\BC)\otimes_\BC Z^{\wedge_\chi}$.

\noindent
(b) By lifting of  idempotents, there is a complete system of $\Xi$-invariant idempotents $e_1, \dots, e_\sd$ of $U_\e^{ev\wedge_\chi}$ which are matrix units in a decomposition $U_q^{ev\wedge_\chi}\cong \Mat_\sd(\BC)\otimes_\BC R_\hbar$.

We do not claim that $\Xi$ acts on $\Mat_\sd(\BC)$ in both parts.
\end{Lem}
\begin{proof}(a) Recall the isomorphism $Z_{Fr}\cong \tZ^>_{Fr} \otimes \tZ^0_{Fr} \otimes \tZ^<_{Fr} \iso \BC[U_-]\otimes \BC[T] \otimes \BC[U_+]$ in Proposition \ref{prop: F-center vs openBruhat}.  The maximal ideal $\m_\chi \in Z_{Fr}$ is generated by 
\[ \tE_\a^\ell, \;\; K^{2\ell \lambda}-\chi_s(K^{2\ell \lambda}),  \;\; \tF_\a^\ell -\chi(\a) \qquad  \qquad \text{for $\a \in \Delta_+$ and $\lambda \in P$,}\]
here $\chi(\a) \in \BC$ so that $\chi(\a) =0$ if $\a \not \in X$, the set of all weights on which $\Xi \subset T$ acts trivially.

For any $\xi=(\chi, \lambda) \in \Spec Z$, let 
\[ M_\xi=U_\e^{ev}/U_\e^{ev}\< \tE_i, K^{2\mu} -\e^{(\lambda, 2\mu)}, \tF_\a^\ell-\chi(\a)\>.\]
Then 
\begin{equation}\label{eq: fiber}
U_\e^{ev}/U_\e^{ev} \m_\xi \cong \Mat_\BC(M_\xi).
\end{equation}

Since $\chi(\a)=0$ if $\a \not \in X$, there is a $Q/\ell X$-grading hence a $Q/X$-grading on $M_\xi$. Hence $M_\xi$ has an action of $\Xi$ so that the isomorphism \eqref{eq: fiber} is $\Xi$-equivariant. So there is a complete system of $\Xi$-invariant idempotents $e_1, \dots, e_\sd$ in $U_\e^{ev}/U_\e^{ev}\m_\xi$. This system is lifted to a complete system of $\Xi$-invariant idempotents in $U_\e^{ev\wedge_\xi}$.  Since $U_\e^{ev\wedge_\chi}=\prod_{\xi=(\chi, \lambda)} U_\e^{ev\wedge_\xi}$, we can find a complete system of $\Xi$-invariant idempotents in $U_\e^{ev\wedge_\chi}$ as desired.

\noindent
(b) follows by lifting of  idempotents.    
\end{proof}
\begin{Rem}\label{rem: Xi act on R} Let $e=e_1=E_{11}$ then $R_\hbar \cong eU_q^{ev\wedge_\chi}e$ carries an action of $\Xi$. The map $\phi_\hbar: R_\hbar \rightarrow Z^{\wedge_\chi}$ is identified with the map $eU_q^{ev\wedge_\chi}e \rightarrow eU_\e^{\wedge_\chi}e$. Then the map $\phi_\hbar: R_\hbar \rightarrow Z^{\wedge_\chi}$ becomes a $\Xi$-equivariant map. Then $C_\hbar=\phi_\hbar^{-1}(Z_{Fr}^{\wedge_\chi})$ is stable under the $\Xi$-action on $R_\hbar$.
\end{Rem}

The following composition of equivalences
\[ \Pbim^\Xi(U_q^{ev\wedge_\chi}) \xrightarrow[]{\fP_1}\Pbim^\Xi(R_\hbar)\xrightarrow[]{\fP_2} \Pbim^\Xi(\CW_q^{\wedge_{\uchi}})\]
is constructed similarly to  Section \ref{sec: Poisson bimodules}. One needs to show that there is a $\Xi$-equivariant lift $\iota: V\rightarrow R_\hbar$, where $V$ is the cotangent space of the symplectic leave (the conjugacy class of $\chi$) at $\chi$ hence admits a natural action of $\Xi$. 

Following Section \ref{ssec: bullet-complete functor}, there is an enhanced restriction functor
\begin{equation}\label{eq: enhanced bullet}\bullet_\dag: U_q^{fin, \uchi}\rmod^{\cU_q } \rightarrow \Pbim^\Xi(\CW_q^{\wedge_{\uchi}}).
\end{equation}
\begin{Prop}\label{prop: enhanced ff on diagonal bimodules} Proposition \ref{prop: ff on diagonal bimods} holds for the enhanced functor $\bullet_\dag$ in \eqref{eq: enhanced bullet}. 
\end{Prop}
By the same proof, Proposition \ref{prop: enhanced ff on diagonal bimodules} holds by the enhanced version of Lemma \ref{lem: equal isogeny} below:
\begin{Lem}\label{lem: enhanced equal isogeny} For $V_\e \in \Rep(\cU_\e(\g))$, the following natural map is bijective
\[ \Hom_{\cU_\e}(V_\e, U_\e^{fin, \uchi}) \rightarrow \Hom_{\cU_\e \rtimes \Xi}(V_\e, U_\e^{ev\wedge_\chi}).\]
\end{Lem}
\begin{proof}Using that $U_\e^{fin}$ is a finitely generated projective module over $Z_{Fr}$,  we reduce to proving the following is bijective
\[ (W_\e \otimes Z_{Fr}^{fin, \uchi})^{\cU_\e}\rightarrow (W_\e \otimes Z_{Fr}^{\wedge_\chi})^{\cU_\e \rtimes \Xi} \qquad \text{for any $W_\e \in \Rep^{fd}(\cU_\e(\g))$}\]
By taking $\fu_\e$-invariants, we reduce to proving that  the following is bijective for any $V\in \Rep^{fd}(\cU(\g))$:
\begin{equation*}(V\otimes Z_{Fr}^{fin, \uchi})^{\cU(\g)} \rightarrow (V\otimes Z_{Fr}^{\wedge_\chi})^{\cU(\g) \rtimes \Xi}
\end{equation*}
which is equivalent to showing the following is bijective
\begin{equation}\label{eq: invariant iso}  
    (V\otimes Z_{Fr}^{fin, \uchi})^{\cU(\g)} \rightarrow (V\otimes Z_{Fr, \g \dash fin}^{\wedge_\chi})^{\cU(\g) \rtimes \Xi},
\end{equation}
where $Z^{\wedge_\chi}_{Fr, \g\dash fin}$ is the $\g$-locally finite part of $Z^{\wedge_\chi}_{Fr}$.  

On $V$ and $Z^{\wedge_\chi}_{Fr, \g\dash fin} \in \Rep(\cU(\g) \rtimes \Xi)$, there are  $\sigma_1(\Xi)$-actions coming from $\Xi \subset \cU_\e(\g) \rtimes \Xi$. On the other hand, there are the $\sigma_2(\Xi)$-actions obtained by restricting the $G$-actions (integrating the $\g$-actions). Then $\sigma(x):=\sigma_1(x) \sigma_2(x^{-1})$ defines the $\sigma(\Xi)$-actions on $V$ and $ Z^{\wedge_\chi}_{Fr, \g\dash fin}$ that commute with the action of $\cU(\g)$, indeed, $\sigma(\Xi)$ acts trivially on $V$. Then 
\[ (V\otimes Z^{\wedge_\chi}_{Fr, \g \dash fin})^{\cU(\g) \rtimes \Xi} = \left( (V\otimes Z^{\wedge_\chi}_{Fr, \g \dash fin})^{\sigma(\Xi)}\right)^{\cU(\g)}=\left( V\otimes (Z^{\wedge_\chi}_{Fr, \g \dash fin})^{\sigma(\Xi)}\right)^{\cU(\g)},\]
since the elements on the left hand side are fixed by the $\sigma_1(\Xi)$ and $\sigma_2(\Xi)$-actions. By Remark \ref{rem: Xi-invariant}, $Z_{Fr}^{fin, \uchi} \cong (Z_{Fr, \g\dash fin}^{\wedge_\chi})^{\sigma(\Xi)}$, hence \eqref{eq: invariant iso} holds. This finishes the proof.
\end{proof}
\begin{Rem}The improvement  of other results in this paper is either more complicated or unknown, hence will be covered elsewhere.
\end{Rem}

\printbibliography
\end{document}